%% file: HDG_MPET.tex
\documentclass[11pt, a4paper]{amsart}

\usepackage[margin=2.5cm]{geometry}

\usepackage{epsf} 

\usepackage{listings} 
\usepackage{algorithmic,algorithm}
\usepackage{multirow}
\usepackage{enumerate}
\usepackage{booktabs}

\usepackage[]{graphicx}
\usepackage{amscd}
\usepackage[pdftex]{color, colortbl} 
\usepackage[pdftex,colorlinks]{hyperref}
\usepackage{lscape}
\usepackage{indentfirst}
\usepackage{latexsym}
\usepackage{amsmath, amsfonts, amssymb,mathrsfs}
\usepackage{verbatim}

\usepackage{tikz}
\usepackage{pgfplots, pgfplotstable}
\usepgfplotslibrary{groupplots}

\usepackage{caption, subcaption}
\usepackage{todonotes}

\usepackage{bm}

\newtheorem{theorem}{Theorem}
\newtheorem{lemma}[theorem]{Lemma}
\newtheorem{corollary}[theorem]{Corollary}
\newtheorem{remark}{Remark}

\newtheorem{definition}{Definition}

\definecolor{mygray}{gray}{0.95}
\definecolor{myblue}{RGB}{0,102,153}
\definecolor{myorange}{RGB}{235, 174, 52}

\usepackage{siunitx}
\newcommand{\numval}[1]{\num[]{#1}}

\newcommand{\vertiii}[1]{{\left\vert\kern-0.25ex\left\vert\kern-0.25ex\left\vert #1 
    \right\vert\kern-0.25ex\right\vert\kern-0.25ex\right\vert}}

\DeclareMathOperator{\ddiv}{div}
\DeclareMathOperator{\divv}{div}
\renewcommand{\div}{\operatorname{div}}
\DeclareMathOperator{\Divv}{Div}

\def\bHDG{{ \mathbf{HDG}}}
\def\HDG{{\operatorname{HDG}}}

\def\hv{{\hat{\bm v}}}

\def\sig{{\bm \sigma}}
\def\ph{{\varphi}}

\newcommand{\ds}{\mathop{\mathrm{d} s}}
\newcommand{\dx}{\mathop{\mathrm{d} x}}

\newcommand{\T}{{\operatorname{T}}}

\newcommand{\Gammas}{\Gamma_{\operatorname{s}}}
\newcommand{\Gammav}{\Gamma_{\operatorname{v}}}

\def\eps{\boldsymbol \epsilon}
\def\sig{\boldsymbol \sigma}

\def\bU{\boldsymbol U}
\def\bV{\boldsymbol V}
\def\bP{\boldsymbol P}
\def\bW{\boldsymbol W}

\def\bf{\boldsymbol f}

\def\bu{\boldsymbol u}
\def\bv{\boldsymbol v}
\def\bw{\boldsymbol w}

\def\bp{\boldsymbol p}
\def\bq{\boldsymbol q}
\def\bx{\boldsymbol x}

\def\bz{\boldsymbol z}

\def\bn{\boldsymbol n}

\def\tmu{\tilde{\mu}}
\def\tlambda{\tilde{\lambda}}
\def\vh{\boldsymbol{\mathsf{v}}_h}
\def\uh{\boldsymbol{\mathsf{u}}_h}
\def\ph{\boldsymbol{\mathsf{p}}_h}
\def\qh{\boldsymbol{\mathsf{q}}_h}
\def\yh{\boldsymbol{\mathsf{y}}_h}
\def\xh{\boldsymbol{\mathsf{x}}_h}

\def\mA{\boldsymbol A}
\def\mB{\boldsymbol B}
\def\mC{\boldsymbol C}
\def\mF{\boldsymbol F}
\def\mX{\boldsymbol X}
\def\mO{\boldsymbol 0}
\def\mS{\boldsymbol S}
\def\mM{\boldsymbol M}

\title[HDG methods for MPET with medical applications] 
{Hybridized discontinuous Galerkin methods for a multiple network poroelasticity model with medical applications}

  \thanks{
    JK, ML and KO acknowledge the funding by the German Science Fund
    (DFG) - project ``Physics-oriented solvers for
    multicompartmental poromechanics'' under grant number 456235063. \\
    JS acknowledges the funding by the Austrian Science Fund
    (FWF) through the research programm ``Taming complexity in partial
    differential systems'' (F65) - project ``Automated discretization
    in multiphysics'' (P10).\\
    PL acknowledges the support by the Academy of Finland (Decision
    324611).}

  \author[J. Kraus]{Johannes Kraus}
\address{Faculty of Mathematics, University Duisburg-Essen, Germany}
\email{johannes.kraus@uni-due.de}

\author[P.~L.~Lederer]{Philip L. Lederer}
\address{Department of Mathematics and Systems Analysis, Aalto University, Finland}
\email{philip.lederer@aalto.fi}

  \author[M. Lymbery]{Maria Lymbery}
\address{Faculty of Mathematics, University Duisburg-Essen, Germany}
\email{maria.lymbery@uni-due.de}

 \author[K. Osthues]{Kevin Osthues}
\address{Faculty of Mathematics, University Duisburg-Essen, Germany}
\email{kevin.osthues@uni-due.de}

\author[J. Sch\"oberl]{Joachim Sch\"oberl}
\address{Institute for Analysis and Scientific Computing, TU Wien, Austria}
\email{joachim.schoeberl@tuwien.ac.at}

\begin{document}

\begin{abstract}
The quasi-static multiple network poroelastic theory (MPET) model, first introduced in the context of geomechanics, 
has recently found new applications in medicine. 
In practice, the parameters in the MPET equations can vary over several orders of magnitude
which makes their stable discretization and fast solution a challenging task.
Here, a new efficient parameter-robust hybridized discontinuous Galerkin method, which also features fluid mass conservation, 
is proposed for the MPET model. 
Its stability analysis which is crucial for the well-posedness of the discrete problem is performed and
cost-efficient fast parameter-robust preconditioners are derived. 
We present a series of numerical computations for a 4-network MPET model of a human brain 
which support the performance of the new algorithms.
%
  
\end{abstract}


\maketitle

\section{Introduction}\label{sec::intro}

According to a United Nations report, \cite{UNnote2020}, the proportion of people in 2020 worldwide older than 
65 was 9.3\% and is 
expected to reach 16.0\% in 2050. 
As a result of an ageing population health care services 
have to put more focus on age related diseases, in particular those affecting the brain, such as dementia and stroke as these have 
a serious impact on quality of life and consequently healthcare costs. In all this treating brain diseases is rather onerous. 
Being the most complex organ in the human body and difficult to access and study in situ, 
the brain remains substantially an enigma and non-distructive approaches are required
for the study of its physiology. One such approach is naturally 
mathematical modeling. 
A model that has been used to simulate the physiology of neurodegenerative conditions such as brain aneurism is the 
fluid-structure interaction model, see~\cite{Fresca2021real}, which consists of
a two-fields problem, coupling the incompressible Navier-Stokes equations with the
non-linear elastodynamics equation modeling the solid deformation. 
In this paper we consider the Multiple-network PoroElastic Theory (MPET) model that  
has been widely used to investigate blood and tissue fluid flow interactions in the human brain and related brain pathologies 
such as Alzheimer's disease, \cite{Guo2017subject}, acute hydrocephalus, \cite{Chou2016afully,Vardakis2016investigating}, 
ischaemic stroke, \cite{Guo2020amultiple}.  

The MPET model equations in a bounded domain $\Omega\subset \mathbb{R}^d$, $d=1,\,2,\,3$ encompassing 
$n$ fluid networks, for $i=1,\ldots,n$ are given as: 
\begin{subequations}\label{eq::threefield}
\begin{alignat}{2}
- \divv (\sig(\bu)-\bm{\alpha}\cdot\bp \bm{I}) &= \tilde{\bf}, \quad &&\mbox{in } \Omega \times (0,T), \label{eq::threefield-1} \\
K^{-1}_i \bm{w}_i + \nabla p_i &= {\boldsymbol 0}, \quad &&\mbox{in } \Omega \times (0,T),
\label{eq::threefield-2} \\
\frac{\partial}{\partial t}(s_i p_i + \alpha_i \, \divv \, \bu) + \divv \, \bm{w}_i 
+ \sum_{\substack{j=1\\ j \neq i}}^n \xi_{ij}(p_i-p_j)
&=\tilde{g}_i,  \quad &&\mbox{in } \Omega \times (0,T) 
. \label{eq::threefield-3}
\end{alignat}
\end{subequations}
The unknown physical quantities in~\eqref{eq::threefield} are the tissue deformation $\bu=\bu(\bx,t)$, the network fluid pressures 
$p_i=p_i(\bx,t)$ and respective fluxes $\bm{w}_i=\bm{w}_i(\bx,t)$, $i=1,\ldots,n$. 
In~\eqref{eq::threefield-1} we have used the vector notations $\bp=(p_1,\ldots,p_n)$ and $\bm{\alpha}=(\alpha_1,\ldots,\alpha_n)$ 
with $\alpha_i \in (0, 1]$ denoting the Biot-Willis constant for the $i$-th network.
The elastic stress tensor in~\eqref{eq::threefield-1} is defined as
$$
\sig(\bu)=2\tmu \eps(\bu)+\tlambda \divv(\bu)\bm{I},
$$
where $\eps(\bu) := (\nabla \bu + (\nabla \bu)^T)/2$ denotes the elastic strain tensor and $\tlambda$ and $\tmu$
are the Lam\'e parameters, whereas $\tilde{\bf}$ represents a body force density. 
In~\eqref{eq::threefield-2} the tensor $K_i$ represents the hydraulic conductivity for the $i$-th network. 
The constrained specific storage coefficients 
are designated by $s_i\ge 0$ in~\eqref{eq::threefield-3} while $\tilde{g}_i$ 
are fluid sources. 
The coupling of two networks $i$ and $j$ is expressed via the network transfer coefficient 
$\xi_{ij} \geq 0$ for which $\xi_{ij} = \xi_{ji}$ holds.

Solving numerically the MPET system of partial differential equations~\eqref{eq::threefield} is an arduous task since 
the various material parameters involved make it difficult to construct structure-preserving discretizations 
that are robust and stable in all parameter regimes. Only recently have there been some valuable results in this direction.
In~\cite{lee2019spacetime} 
a mixed finite element formulation based on introducing a total pressure variable has been shown to 
be robust for nearly incompressible material, small storage coefficients, and small or vanishing transfer coefficients between
networks. 
A stability analysis of the MPET model~\eqref{eq::threefield} that is fully parameter robust along with 
parameter-robust norm-equivalent preconditioners has been presented in~\cite{Hong2019conservativeMPET}. 
The latter employs discontinuous Galerkin (DG) schemes to achieve strong mass-conservation. 

A major drawback of these schemes, however, remains the high number of globally coupled degrees of freedom 
that arise in the discrete algebraic system and which consequently makes the construction of efficient fast solvers and 
preconditioners difficult. 
An approach to overcome this is hybridization, a process in which continuity constraints
of finite element spaces are removed without altering the solution, 
cf.~\cite{MR2051067,Cockburn2005incompressibleI,Cockburn2005incompressibleII}. 
This allows the construction of methods that take strengths from both continuous Galerkin (CG) and DG methods, 
see e.g.~\cite{CockburnEtAl2009unified,LedererEtAl2018hybrid,Kauffman2017overset,kraus2020uniformly,arnold2002unified}.

In this paper we propose, analyze and numerically test a new efficient parameter-robust 
hybridized discontinuous 
Galerkin method for the MPET model. 
The remainder of the paper consists of four sections. In Section~\ref{sec::problem} the initial and boundary conditions 
completing the MPET system are introduced 
and the weak formulation of the time-discrete model is presented. 
A new hybridized DG (HDG) hybrid mixed
finite element method in space is then proposed in Section~\ref{sec::disc} and 
its parameter-robust well-posedness proven in 
Section~\ref{sec::stability} which is the main theoretical contribution of this paper.
In Section~\ref{sec::numerics} a norm-equivalent preconditioner is proposed and 
the theoretical results are complemented through 
a series of numerical simulations for a 4-network MPET model of a human brain.

\section{Problem formulation}\label{sec::problem}

In this section a time-discrete weak formulation of~\eqref{eq::threefield} is 
delivered. First, 
the equations are posed and tested in suitable function spaces and integration by parts 
is applied. More specifically, the time-dependent trial functions $\bu(t),\bw(t),\bp(t)$ and test 
functions $\bv(t),\bz(t),\bq(t)$ are 
to be taken from Hilbert spaces $\bU, \bW, \bP$ which obey essential homogeneous boundary conditions. 
For the sake of simplicity, the analysis  
in Section~\ref{sec::stability} is presented for the case of 
homogeneous Dirichlet boundary conditions for $\bu$ and homogeneous Neumann conditions for 
the pressures $p_i$, i.e., homogeneous Dirichlet conditions for the fluxes $\bw_i$. 
In this situation $\bU = \bm H_0^1(\Omega)$, $\bW_i = \bm H_0(\div,\Omega)$ and $P_i=L_0^2(\Omega)$, 
$i=1,\ldots,n$.
For short, we write 
$\bw^T=(\bw_1^T,\ldots,\bw_n^T)$, $\bp^T =(p_1,\ldots,p_n)$, $\bz^T=(\bz_1^T,\ldots,\bz_n^T)$, 
$\bq^T =(q_1,\ldots,q_n)$ and also $\bW=\bW_1\times \ldots \times \bW_n $, $\bP = P_1\times \ldots \times P_n$
and the resulting variational problem reads:
For $t \in (0, T)$ and any $(\bv, \bz, \bq)\in \bU\times \bW \times \bP$ find 
$(\bm{u}(t),\bm{w}(t), \bp(t))\in \bU \times \bW \times \bP$ such that
\begin{subequations}\label{eqn:weak-3f} 
\begin{alignat}{2}
\tilde{a}(\bm{u},\bm{v}) - \sum_{i=1}^n(\alpha_i \,p_i,\ddiv \bm{v}) &= (\tilde{\bm f},\bm{v}),
    \label{eqn:weak-3f-1}\\
(K^{-1}_i\bm{w}_i,\bm{z}_i) - (p_i,\ddiv \bm{z}_i) &= 0, \quad i = 1, \ldots, n, 
\label{eqn:weak-3f-2}\\
- (\alpha_i\ddiv \partial_t \bm{u},q_i) - (\ddiv \bm{w}_i,q_i) - (s_i \partial_t  p_i, q_i)
- (\sum_{\substack{j=1\\ j \neq i}}^n \xi_{ij}(p_i-p_j),q_i)
 &= -(\tilde{g}_i,q_i), \quad i = 1, \ldots, n,
\label{eqn:weak-3f-3}
\end{alignat}
\end{subequations}
where
\begin{align}
\tilde{a}(\bm{u},\bm{v}) &:= 2\tmu\int_{\Omega}\eps(\bm{u}):\eps(\bm{v}) \dx + \tlambda\int_{\Omega} \ddiv\bm{u}\,\ddiv\bm{v} \dx.
\label{u_form_2f}
\end{align}
Other scenarios of boundary conditions
in general lead to modifications of~\eqref{eqn:weak-3f} and in certain cases changes in 
the parameter-dependent norms involved in the stability analysis may be required, see~\cite{HongKraus2018}.  The 
well-posedness of the time-continuous problem 
then is guaranteed under the initial conditions for $\bu$ and $\bp$, i.e.,
$\bm{u}(\cdot, 0) = \bm{u}_0(\cdot)$ and
$p_i(\cdot, 0) = p_{i,0}(\cdot)$ for $i=1,\ldots,n$.

Next, we use the implicit Euler method to discretize~\eqref{eqn:weak-3f} in time obtaining the 
time-step variational problem: 
Find $(\bm{u}^k, \bm{w}^k, \bp^k) := (\bm{u}(\bx,t_k),\bm{w}(\bx,t_k), \bp(\bx,t_k))\in \bU \times \bW \times \bP$ solving the system of equations
\begin{subequations}\label{eqn:weak-3f-iE} 
    \begin{alignat}{2}
        \tilde{a}(\bm{u}^k,\bm{v}) - \sum_{i=1}^n(\alpha_i \,p_i^k,\ddiv \bm{v}) &= (\tilde{\bm f}^k,\bm{v}),
        \label{eqn:weak-3f-1-iE}\\
        (K^{-1}_i\bm{w}_i^k,\bm{z}_i) - (p_i^k,\ddiv \bm{z}_i) &= 0, \quad i = 1, \ldots, n, 
        \label{eqn:weak-3f-2-iE}\\
        - (\alpha_i\ddiv \bm{u}^k,q_i) - \tau (\ddiv \bm{w}_i^k,q_i) - (s_i p_i^k, q_i)
        - \tau (\sum_{\substack{j=1\\ j \neq i}}^n \xi_{ij}(p_i^k-p_j^k),q_i)
        &= (\tilde{g}_i^k,q_i), \quad i = 1, \ldots, n,
        \label{eqn:weak-3f-3-iE}
    \end{alignat}
\end{subequations}
for all $ (\bv, \bz, \bq) \in \bU \times \bW \times \bP $, where $\tau$ denotes the time-step parameter and 
  $  \tilde{\bm f}^k  := \tilde{\bm f}(\bx, t_k)$, 
    $\tilde{g}_i^k  := -\tau \tilde{g}_i(\bx,t_k) - \alpha_i\ddiv(\bm{u}^{k-1}) - s_i p_i^{k-1}$
are the right hand sides.

%
In order to obtain a symmetric (perturbed) saddle-point problem that is handy for analysis, we apply 
the following similarity transformations to system~\eqref{eqn:weak-3f-iE} and substitutions: 
we divide \eqref{eqn:weak-3f-1-iE} by $ 2 \tmu $, multiply \eqref{eqn:weak-3f-2-iE} by $ \frac{\alpha_i}{2 \tmu} $ and \eqref{eqn:weak-3f-3-iE} by $ \alpha_i^{-1} $ and introduce
\begin{gather*}\label{scaled_param}
    \begin{gathered}
        \bm{u} := \bm{u}^k, \quad
        \bm{w}_i := \frac{\tau}{\alpha_i} \bm{w}_i^k, \quad
        p_i := \frac{\alpha_i}{2\tmu} p_i^k, \quad
        \bm{f} := \frac{\tilde{\bm{f}}^k}{2\tmu}, \quad
        g_i := \alpha_i^{-1} \tilde{g}_i^k, \\
        \lambda:=\frac{\tlambda}{2\tmu}, \quad
        R^{-1}_i:=\frac{\alpha_i^2}{2 \tmu \tau} K^{-1}_i, \quad
        \alpha_{p_i} := \frac{2\tmu}{\alpha_{i}^2} s_i, \quad
        \xi_{ii} := \sum_{\substack{j=1\\ j \neq i}}^n \xi_{ij}, 
        \\
        \zeta_{ii}:=\alpha_{p_i}+\frac{2\tmu\tau\xi_{ii}}{\alpha_{i}^2}, 
        \quad \zeta_{ij}:=-\frac{2\tmu\tau\xi_{ij}}{\alpha_{i} \alpha_j} \text{ for } i\neq j, \quad
        \bm \zeta_i := (\zeta_{i1},\ldots,\zeta_{in})^T, \quad i=1,\ldots,n.
    \end{gathered}
\end{gather*}
Now problem \eqref{eq::threefield} takes the form, cf.~\cite{Hong2019conservativeMPET}: 
Find $ (\bm{u}, \bm{w}, \bm{p}) \in \bU \times \bW \times \bP $ such that
\begin{subequations}\label{eqn:weak-3f-iE-scaled} 
    \begin{alignat}{2}
        a(\bm{u},\bm{v}) - (\bm 1\cdot \bp,\ddiv \bm{v})
        &= (\bm{f},\bm{v}),
        \label{eqn:weak-3f-1-iE-scaled}\\
        (R^{-1}_i\bm{w}_i,\bm{z}_i) - (p_i,\ddiv \bm{z}_i) &= 0, \qquad i = 1, \ldots, n, 
        \label{eqn:weak-3f-2-iE-scaled}\\
        - (\ddiv \bm{u},q_i) - (\ddiv \bm{w}_i,q_i) - (\bm \zeta_{i}\cdot \bp,q_i)
        &= (g_i,q_i), \qquad i = 1, \ldots, n,
        \label{eqn:weak-3f-3-iE-scaled}
    \end{alignat}
\end{subequations}
holds for all $ (\bv, \bz, \bq)\in \bU\times \bW \times \bP $, where $\bm 1 = (1, \ldots, 1)^T$ and
\begin{align}
    a(\bm{u},\bm{v}) &:= \int_{\Omega}\eps(\bm{u}):\eps(\bm{v}) \dx + \lambda\int_{\Omega} \ddiv\bm{u}\,\ddiv\bm{v} \dx.
    \label{u_form_3f}
\end{align}

\section{Hybridized DG/hybrid mixed methods}\label{sec::disc}



Our goal is to generalize the hybridized DG/hybrid mixed methods, which have already been studied in \cite{kraus2020uniformly} in the case of Biot's consolidation model.
Therefore, we use a hybridized discontinuous Galerkin method for the mechanics subproblem and a hybrid mixed method for the flow subproblem.

We begin with a shape-regular triangulation $\mathcal{T}_h$ with
uniform mesh size $h$ and the associated set of all facets
$\mathcal{F}_h$. Then for $i=1,\ldots,n$, we define the finite element
spaces
\begin{eqnarray*}
\bm U_h&:=&\{\bv \in \bm H_0(\divv, \Omega):\bv|_T \in \bU(T),~T \in \mathcal{T}_h\}, \\
\bm W_{i,h}&:=&\{\bz_i \in \bm H_0(\divv ,\Omega):\bz_i|_T \in \bm W_i(T),~T \in \mathcal{T}_h\}, 
\quad \bm W_h = \bW_{1,h}\times \ldots \times \bW_{n,h},
\\
P_{i,h}&:=&\{q_i \in L_0^2(\Omega):q_i|_T \in P_i(T),~T \in \mathcal{T}_h\}, \quad \bP_{h}=P_{1,h}\times \ldots \times P_{n,h}.
\end{eqnarray*}
Thereby the local spaces $\bm U(T)/\bm W_i(T)/P_i(T)$ can be chosen either as
$\text{BDM}_{\ell}(T)/\text{RT}_{\ell-1}(T)/\text{P}_{\ell-1}(T)$ or as
$\text{BDFM}_{\ell}(T)/\text{RT}_{\ell-1}(T)/\text{P}_{\ell-1}(T)$
where $\text{BD(F)M}_{\ell}(T)$, $\text{RT}_{\ell-1}(T)$, and
$\text{P}_{\ell-1}(T)$ represent the local
Brezzi-Douglas-(Fortin-)Marini space of order $\ell \ge 1$, the
Raviart-Thomas space of order $\ell-1$, and full polynomials of degree
$\ell-1$, respectively, for details, see~\cite{Boffi2013mixed}.

We exploit $\bm H(\divv)$-conforming discretizations for the mechanics subproblem since these, contrary to lower-order
 $H^1$-conforming discretizations, allow to impose incompressibility constraints exactly on the discrete level. In particular,
 pointwise divergence-free solutions can be approximated and pressure robustness achieved,
 see \cite{CockburnEtAl2002local, CockburnEtAl2005locally}.
 
Next, let $\bm{L}^2(\mathcal{F}_h)$ be the space of the vector-valued square integrable functions on the skeleton $\mathcal{F}_h$ 
and $\bm P_{\ell}(F)$ the space of vector-valued, component-wise polynomial functions of order $\ell$ on each facet $F \in \mathcal{F}_h$. 
We now introduce the spaces
\begin{align*}
\widehat{\bU}_h&:=\{ \hat{\bv} \in \bm{L}^2(\mathcal{F}_h): \hat{\bv}|_F \in \bm P_{\ell}(F) \textrm{ and }  
\hat{\bv}|_F \cdot \bn=0, ~ F\in \mathcal{F}_h, ~ \hat{\bv} = \bm 0 \text{ on }\partial{\Omega} \}, \quad \overline{\bU}_h := \bU_h \times
\widehat{\bU}_h, 
\end{align*}
as well as the HDG variable $(\bu_h,\hat{\bu}_h) \in \overline{\bU}_h$.
The unknown $\hat{\bu}_h$ is an approximation of the tangential trace of the displacement field $\bu$ and is used to weakly impose the tangential continuity, as presented in \cite{CockburnEtAl2009unified, LehrenfeldSchoeberl2016high}.
Utilizing the space $\overline{\bU}_h$, 
we replace the first term in the bilinear form $ a(\cdot, \cdot) $ defined in \eqref{u_form_3f} by $ a_h^{\text{HDG}} (\cdot, \cdot) $ given by
\begin{align}\label{a_h_HDG}
    \begin{aligned}
        a_h^{\text{HDG}}((\bu,\hat{\bu}),(\bv,\hat{\bv}))&:=
        \sum_{T\in\mathcal{T}_h}\left[ \int_{T} \eps(\bu) : \eps(\bv) \, \dx + \int_{\partial T} \eps(\bu) \bn \cdot (\hat{\bv}-\bv)_t\, \ds \right. \\
        & \quad + \left. 
        \int_{\partial T} \eps(\bv) \bn \cdot (\hat{\bu}-\bu)_t\, \ds +  \eta \ell^2 h^{-1}\int_{\partial T} (\hat{\bu}-\bu)_t \cdot (\hat{\bv}-\bv)_t \ds \right],
    \end{aligned}
\end{align}
where $(\bu, \hat{\bu})$,
$(\bv, \hat{\bv})\in \overline{\bU}_h$ and $\eta$ represents a suitable stabilization parameter independent of all model and discretization parameters.
The subscript $t$ denotes the tangential component of a vector field on a facet.
Exactly divergence-free HDG methods as well as improvements to these methods are discussed in more depth in \cite{LehrenfeldSchoeberl2016high, LedererEtAl2018hybrid, LedererEtAl2019hybrid}.

Using a hybrid mixed formulation for the flow subproblem allows for more efficient preconditioned iterative methods, since inverting a div-grad operator is more cost-efficient than inverting a grad-div operator.
In order to construct a hybrid mixed method for the $(\bw,\bp)$ system, we break up the normal 
continuity of the $\bw$ variable, i.e., placing it in the space $\bm W_h^{-}=\bm W_{1,h}^-\times \ldots \times \bm W^{-}_{n,h}$ instead of 
the space $\bm W_h$, and reinforce it by a Lagrange multiplier $\hat{\bp}_{h}\in \widehat{\bP}_h=\widehat{P}_{1,h}\times \ldots \times
\widehat{P}_{n,h}$, where  
\begin{align*}
  \bm W_{i,h}^-&:=\{\bz_{i} \in \bm L^2(\Omega):\bz_{i}|_T \in \bm W_i(T),~T \in \mathcal{T}_h \},\quad
  \widehat P_{i,h} :=\prod_{F\in \mathcal{F}_h}\text{P}_{\ell-1}(F),\quad
  \overline P_{i,h} := P_{i,h} \times \widehat P_{i,h}.               
\end{align*}
Here, the local space $\bW_i(T)$ can be chosen as before.
We also introduce the product space $\overline{\bP}_h:=\overline P_{1,h}\times \ldots \times 
\overline P_{n,h}$ and
%
%
define for all $\bm w_{i,h} \in \bm W_{i,h}^-$ and $(p_{i,h}, \hat p_{i,h}) \in \overline P_{i,h}$ the bilinear form
\begin{align}\label{form_b}
  b(\bm w_{i,h},(p_{i,h}, \hat p_{i,h})) = \sum_{T \in \mathcal{T}_h} 
  \left(\int_T \div \bm w_{i,h} \ p_{i,h} \dx - \int_{\partial T} \bm w_{i,h} \cdot \bm n \, \hat p_{i,h} \ds\right), \quad i=1,\ldots,n.
\end{align}
Note that the space $\bm W_{h}^-$ is a discontinuous version of $\bm W_{h}$ and that $\widehat{\bP}_h$ has been chosen as the normal trace space of $\bm W_h$.

With these spaces and bilinear forms at hand, we formulate the following HDG hybrid mixed finite element method 
for~\eqref{eqn:weak-3f-iE-scaled}: 
Find
$((\bu_h, \hat{\bu}_h), \bw_h, (\bp_h,\hat \bp_h) ) \in \overline{\bU}_h
\times \bm W^-_h \times \overline \bP_h$ such that
\begin{subequations}\label{eqn:MPET-disc-3-field-HDG}
    \begin{alignat}{2}
        a_h((\bu_h, \hat{\bu}_h),(\bv_h, \hat{\bv}_h)) - (\bm 1\cdot \bp_h,\ddiv \bv_h)
        &= (\bm{f},\bv_h),
        \label{eqn:MPET-disc-3-field-HDG-1}\\
        (R^{-1}_i\bw_{i,h},\bm{z}_{i,h})  - b( \bz_{i,h},(p_{i,h}, \hat p_{i,h})) &= 0, \qquad i=1,\ldots,n,
        \label{eqn:MPET-disc-3-field-HDG-2}\\
        - (\ddiv \bu_h,q_{i,h}) - b(\bw_{i,h},(q_{i,h},\hat q_{i,h})) - (\bm \zeta_{i}\cdot \bp_h,q_{i,h})
        &= (g_i,q_{i,h}), \qquad i=1,\ldots,n,
        \label{eqn:MPET-disc-3-field-HDG-3}
    \end{alignat}
\end{subequations}
holds for all $ ((\bv_h, \hv_h), \bz_h, (\bq_h,\hat \bq_h)) \in \overline{\bU}_h \times \bW_h^- \times \overline{\bP}_h $, where
\begin{equation}\label{a_h}
    a_{h}((\bu_h, \hat{\bu}_h),(\bv_h, \hat{\bv}_h)) := a_h^{\text{HDG}}((\bu_h,\hat{\bu}_h),(\bv_h,\hat{\bv}_h)) 
    +\lambda \int_{\Omega}  \divv \bu_h \, \divv \bv_h \dx.
\end{equation}

Next, we want to prove the parameter-robust stability for the methods just introduced.

\section{Stability analysis}\label{sec::stability}  

This section is devoted to the well-posedness analysis of the discrete problem~\eqref{eqn:MPET-disc-3-field-HDG} resulting from application 
of the HDG hybrid mixed method. 
To begin with, we introduce in Section~\ref{sec::parameter_norms} parameter-dependant norms, first on the continuous and then on the discrete 
level on which we will also define the related combined norm $\Vert \cdot \Vert_{\overline{\bm X}_h}$. The latter provides a norm-equivalent preconditioner 
for~\eqref{eqn:MPET-disc-3-field-HDG}. To prove parameter-robust well-posedness, 
in Section~\ref{sec::parameter}, and, in particular, inf-sup stability, we use a specific norm fitting technique that has recently been introduced in~\cite{hong2021newframework}. 
This will give rise to a different discrete norm $\Vert \cdot \Vert_{{\bm Y}_h}$ which will be proven to be equivalent to $\Vert \cdot \Vert_{\overline{\bm X}_h}$ at the end of this section. 

\subsection{Parameter-dependent norms}
\label{sec::parameter_norms}

In this subsection, we introduce the parameter-dependent norms 
for the construction of norm-equivalent operator preconditioners.

To this end, let us define the following parameter matrices
\begin{align} \label{eq::lambdamat}
\Lambda_{\zeta} & := \{\zeta_{ij}\}_{i,j=1}^n, &
\Lambda & := \Lambda_{\zeta} + \frac{1}{\lambda_0} \mathbb{I},
\end{align}
where $\lambda_0:=\max\{1,\lambda\}$ and 
$\mathbb{I}$ is the matrix of all ones. By $\Vert \cdot\Vert_0$ we denote the $L^2$-norm for scalar as well as 
vector-valued functions.

Note that $\Lambda+ RI$ is a symmetric positive definite matrix and so is $(\Lambda + RI )^{-1}$, where $R:=\min\{R_1,\ldots,R_n\}$ and $I$ denotes the identity matrix.
The Hilbert spaces $\bm U,\bm W, \bm P$ are equipped with the inner products
\begin{subequations}
\begin{eqnarray}
(\bm u,\bm v)_{\bU}& := & (\eps(\bu),\eps(\bv))+\lambda (\divv\bu,\divv\bv), \\
(\bw,\bz)_{\bW} & : =& \sum_{i=1}^n (R_i^{-1}\bw_i,\bz_i)+((\Lambda+RI)^{-1}\Divv \bw,\Divv \bz),\\
(\bw,\bz)_{\bW^{-}} & := & \sum_{i=1}^n (R_i^{-1}\bw_i,\bz_i), \\
(\bp,\bq)_{\bP} & : =& \sum_{i=1}^{n} R_{i} (p_{i}, q_{i}) + (\Lambda \bp,\bq), 
\end{eqnarray}
\end{subequations}
where $(\Divv \bw)^T = (\divv \bw_1,\ldots, \divv \bw_n)$ and the induced norms 
$\Vert \cdot\Vert_{\bU}$, $\Vert \cdot \Vert_{\bW}$, $\Vert \cdot\Vert_{\bW^-}$ and $\Vert \cdot \Vert_{\bP}$, 
respectively.


Next, we introduce several discrete norms. The hybridized discontinuous Galerkin (HDG) norm on the 
displacement product space $\overline{\bm U}_h$ 
is defined by
\begin{equation}\label{norm_HDG}
\Vert(\bm v_h, \hat{\bm v}_h) \Vert^2_{\bHDG} :=   \sum_{T\in \mathcal{T}_h} \left(\Vert \eps(\bm v_h) \Vert_{0,T}^2 
+  h^{-1}\Vert (\hat{\bm v}_h - \bm v_h)_t \Vert_{0,\partial T}^2
+  h^2\vert  \bm v_h \vert_{2,T}^2
\right),
\end{equation}
and can be used to obtain the following discrete norm 
\begin{equation}\label{norm_Uh_tilde}
\Vert (\bm v_h,\hat{\bm v}_h)\Vert_{\overline{\bm U}_h}^2 := \Vert (\bm v_h, \hat{\bm v}_h) \Vert^2_{\bHDG} 
+ \lambda \Vert \divv \bm v_h \Vert^2_0.
\end{equation} 
Further, we introduce the HDG norm $\Vert \cdot\Vert_{\HDG}$ on the pressure product space 
$\overline{\bP}_h$  by
\begin{align}
  \Vert (q_h, \hat{q}_h)\Vert^2_{\HDG} := 
  \sum_{T\in \mathcal{T}_h} \left(\Vert \nabla q_h \Vert_{0,T}^2 
+  h^{-1}\Vert \hat{ q}_h - q_h \Vert_{0,\partial T}^2
+  h^2\vert  q_h \vert_{2,T}^2\right)
,\label{hdg_p_norm}
\end{align}
and the related norm $\Vert \cdot\Vert_{\overline{\bP}_h}$ by
\begin{align} 
\Vert (\bq_h, \hat{\bq}_h)\Vert^2_{\overline{\bP}_h} :=  \sum_{i=1}^{n} R_i
 \Vert (q_{i, h}, \hat{q}_{i, h})\Vert^2_{\HDG} +  (\Lambda \bq_h, \bq_h).
\label{hdg_p_norm_scaled}
\end{align}
Finally, we denote the product space
\begin{equation}\label{product_space}
\overline{\bm X}_h := \overline{\bm U}_h \times {\bm W_h^{-}}\times \overline{\bP}_h,
\end{equation}
and equip it with the norm
\begin{equation}\label{product_norm}
\Vert((\bm v_h,\hat{\bm v}_h),\bz_h,(\bq_h,\hat{\bq}_h)) \Vert^2_{\overline{\bm X}_h} :=
 \Vert (\bm v_h,\hat{\bm v}_h) \Vert^2 _{\overline{\bm U}_h} + \Vert \bz_h\Vert^2_{\bm W^{-}} 
+\Vert (\bq_h,\hat{\bq}_h)\Vert^2_{\overline{\bP}_h} . 
\end{equation}

\subsection{Parameter-robust well-posedness} 
\label{sec::parameter}
In this subsection, we generalize our result from~\cite{kraus2020uniformly} herewith showing the uniform 
well-posedness 
of~\eqref{eqn:MPET-disc-3-field-HDG}. 

Before we explain what uniformly well-posedness means and why it is important, we recall 
an abstract stability result based on a specific norm fitting, as it has recently been presented in~\cite{hong2021newframework}. 
%
To this end, we
present~\eqref{eqn:MPET-disc-3-field-HDG} as a perturbed saddle-point
problem of the form 
\begin{align}\label{saddle-point} 
\mathsf{A}_h((\uh,\ph),(\vh,\qh)) = \mathsf{F}_h((\vh,\qh)) \qquad \forall \vh \in \bV_h, \forall \qh\in \bm Q_h,
\end{align}
where $\bV_h:= \overline{\bU}_h \times \bW_h^{-}$, $\bm Q_h: =
\overline{\bP}_h$, $\uh:=(({\bu}_h,\hat{\bu}_h), \bw_h)$,
$\vh:=(({\bv}_h,\hat{\bv}_h), \bz_h)$, $\ph:=(\bp_h,\hat{\bp}_h)$,
$\qh:=(\bq_h,\hat{\bq}_h)$ and the bilinear form $ \mathsf{A}_h((\cdot,\cdot),(\cdot,\cdot)) $ has the representation
\begin{align}\label{saddle-point_form}  
    \mathsf{A}_h((\uh,\ph),(\vh,\qh)) 
    & = \mathsf{a}_h(\uh,\vh) 
    + \mathsf{b}_h(\vh,\ph) 
    +\mathsf{b}_h(\uh,\qh) 
    - \mathsf{c}(\ph,\qh), 
\end{align} 
%
in terms of the bilinear forms $\mathsf{a}_h(\cdot,\cdot)$, $\mathsf{b}_h(\cdot,\cdot)$ and $\mathsf{c}(\cdot,\cdot)$. 
Each of these bilinear forms can be associated with a bounded linear operator as given by 
\begin{subequations}
\begin{align}
A: \bV_h\rightarrow \bV_h': \,\, & \langle A\uh,\vh\rangle_{\bV_h'\times \bV_h} 
= \mathsf{a}_h(\uh,\vh), \qquad \forall \uh,\vh\in \bV_h, \\
C:\bm Q_h\rightarrow \bm Q_h': \,\, & \langle C \ph,\qh\rangle_{\bm Q_h'\times \bm Q_h} 
= \mathsf{c}(\ph,\qh), \qquad \forall \ph,\qh\in \bm Q_h, \\
B: \bV_h\rightarrow \bm Q_h': \,\, & \langle B\vh,\qh\rangle_{\bm Q_h'\times \bm Q_h}  
= \mathsf{b}_h(\vh,\qh), \qquad \forall \vh \in \bV_h, \forall \qh\in \bm Q_h, \\ 
B^T: \bm Q_h\rightarrow \bV_h': \,\, & \langle B^T \qh,\vh\rangle_{\bV_h'\times \bV_h} 
= \mathsf{b}_h(\vh,\qh), \qquad \forall \vh \in \bV_h, \forall \qh\in \bm Q_h,
\end{align}
\end{subequations}
where $\bV_h'$ and $\bm Q_h'$ denote the dual spaces of $\bV_h$ and $\bm Q_h$, respectively, and 
$\langle \cdot, \cdot \rangle_{\bV_h'\times \bV_h}$ as well as $\langle \cdot, \cdot \rangle_{\bm Q_h'\times \bm Q_h}$ 
the corresponding duality pairings.  

%
Next we introduce the abstract norm $\Vert \cdot\Vert_{\bm Q_h}$ on the space $\bm Q_h$ and the abstract norm  
$\Vert \cdot\Vert_{\bV_h}$ on $\bm V_h$ 
which have the following representations, see~\cite{hong2021newframework},
\begin{equation} \label{Q_norm}
\Vert \qh\Vert_{\bm Q_h}^2 :=\vert \qh\vert_{\bm Q_h}^2+\mathsf{c}(\qh,\qh) 
=: \langle \bar{Q}\qh,\qh \rangle_{\bm Q_h'\times \bm Q_h},
\end{equation}
and 
\begin{align}\label{V_norm}
\Vert \vh \Vert^2_{\bV_h} &:= \vert \vh\vert^2_{\bV_h} +\vert \vh \vert^2_{\mathsf{b}}.
\end{align}

Assuming that the seminorm $\vert \cdot\vert^2_{\bm Q_h}$ is defined via a symmetric positive semidefinite (SPSD) bilinear 
form $d(\cdot,\cdot)$, i.e., $\vert \qh \vert_{\bm Q_h}^2=d(\qh,\qh)$ and $(c(\cdot,\cdot)+d(\cdot,\cdot))$ is symmetric positive 
definite (SPD), the operator
 $\bar{Q}: \bm Q_h\rightarrow \bm Q_h'$ is a linear invertible operator, the Riesz isomorphism corresponding to the inner product 
 $(\cdot,\cdot)_{\bm Q_h}$ inducing the norm $\Vert \cdot\Vert_{\bm Q_h}$.
Moreover, $\vert \cdot \vert_{\bV_h}$ is a proper seminorm, which is a norm on ${\rm Ker}(B)$ satisfying  
$$
\vert \vh\vert_{\bV_h}^2 \eqsim  \mathsf{a}_h(\vh,\vh) \qquad \forall \vh \in {\rm Ker}(B),
$$ 
and $\vert \cdot\vert_{\mathsf{b}}$ is defined by 
\begin{align}\label{seminorm_b}
\vert \vh \vert^2_{\mathsf{b}}:= \langle B\vh, \bar{Q}^{-1}B\vh \rangle_{\bm Q_h' \times \bm Q_h}.
\end{align}

Note that $\bar{Q}^{-1}:\bm Q_h'\rightarrow \bm Q_h$ is an isometric isomorphism and it is fulfilled that
\begin{align*}
\langle B\vh, \bar{Q}^{-1}B\vh\rangle_{\bm Q_h'\times \bm Q_h}
& = \langle \bar{Q}\bar{Q}^{-1}B\vh,\bar{Q}^{-1}B\vh \rangle_{\bm Q_h'\times\bm Q_h} \\
& = (\bar{Q}^{-1}B\vh,\bar{Q}^{-1}B\vh)_{\bm Q_h}
= \Vert \bar{Q}^{-1}B\vh\Vert_{\bm Q_h}^2
= \Vert B\vh\Vert_{\bm Q_h'}^2.
\end{align*}

Additionally, we introduce the product space $\bm Y_h:=\bV_h\times \bm Q_h$ and equip it  
with the norm $\Vert \cdot \Vert_{\bm Y_h}$ defined by 
\begin{equation}\label{y_norm}
\Vert \yh \Vert_{\bm Y_h}^2
= \Vert \vh\Vert_{\bV_h}^2 
+ \Vert \qh\Vert_{\bm Q_h}^2,\quad \forall \yh = 
(\vh,\qh)\in \bm Y_h. 
\end{equation}

Then the following theorem is an analogue to Brezzi's splitting theorem for the abstract perturbed saddle-point 
problem~\eqref{saddle-point}--\eqref {saddle-point_form}. 
\begin{theorem}\label{our_theorem}
Let $\Vert \cdot\Vert_{\bV_h}$ and $\Vert \cdot \Vert_{\bm Q_h}$ be defined according to~\eqref{Q_norm}--\eqref{V_norm} 
and consider the bilinear form $\mathsf{A}_h((\cdot,\cdot),(\cdot,\cdot))$ from~\eqref{saddle-point_form} 
for which $\mathsf{a}_h(\cdot,\cdot)$ is continuous and $\mathsf{a}_h(\cdot,\cdot)$ and 
$\mathsf{c}(\cdot,\cdot)$ are symmetric positive semidefinite. 

Further, assume that $\mathsf{a}_h(\cdot,\cdot)$ satisfies the coercivity estimate
\begin{equation}\label{coerc_a}
\mathsf{a}_h(\vh,\vh)\ge \underline{C}_{\mathsf{a}} \vert \vh\vert_{\bV_h}^2, \qquad \forall \vh\in {\bV_h},
\end{equation}
and that there exists a constant 
$\underline{\beta}>0$ such that 
\begin{equation}\label{inf_sup_b}
\sup_{\substack{\vh\in {\bm V_h}\\ \vh\neq 0}} 
\frac{\mathsf{b}_h(\vh,\qh)}{\Vert \vh\Vert_{\bm V_h}}\ge \underline{\beta}  \vert \qh\vert_{\bm Q_h}, 
\qquad \forall \qh\in {\bm Q}_h.
\end{equation} 

Then the bilinear form $\mathsf{A}_h((\cdot,\cdot),(\cdot,\cdot))$ is continuous and inf-sup stable under the combined norm 
$\Vert \cdot \Vert_{\bm Y_h}$ defined in~\eqref{y_norm}, i.e., the conditions
\begin{equation}\label{boundedness_A}
\mathsf{A}_h(\xh, \yh) \le \overline{C} \Vert \xh\Vert_{\bm Y_h} \Vert \yh\Vert_{{\bm Y_h}}, 
\qquad \forall \xh, \yh \in \bm Y_h,
\end{equation}
and
\begin{equation}\label{inf_sup_A}
\inf_{\xh \in {\bm Y_h}} \sup_{\yh \in \bm Y_h} \frac{\mathsf{A}_h(\xh,\yh)}
{\Vert \xh\Vert_{\bm Y_h} \Vert \yh\Vert_{\bm Y_h}} \ge \underline{\alpha}>0
\end{equation}
hold, where $\xh:=(\uh,\ph)$.
\end{theorem}

\begin{proof}
The proof of this result is given in~\cite{hong2021newframework}.
\end{proof}



Furthermore, the following Lemmas, the proofs of which can be found in~\cite{kraus2020uniformly}, are essential for showing 
the well-posedness of~\eqref{eqn:MPET-disc-3-field-HDG}.

\begin{lemma}\label{lemma_inf_sup}
There holds the following discrete inf-sup condition
\begin{align}
\inf_{(q_{1,h},\ldots,q_{n,h}) \in \bP_h}\sup_{(\bv_h,\hat{\bv}_h)\in \overline{\bm U}_h}
\frac{(\divv \bv_h,\sum_{i=1}^n q_{i,h})}{\Vert (\bv_h,\hat{\bv}_h)\Vert_{\bHDG}\Vert \sum_{i=1}^n q_{i,h}\Vert_0} 
\ge {\beta}_{S,d}>0,
  \label{hdg_inf_sup_stokes}
\end{align}
where ${\beta}_{S,d}$ is a constant independent of
all problem parameters.
\end{lemma}

\begin{lemma}\label{lemma_HDG-HM}
There hold the following discrete inf-sup conditions
\begin{equation}\label{hdg-hm_inf_sup_darcy}
\inf_{(q_{i,h},\hat{q}_{i,h}) \in \overline{P}_{i,h}}\sup_{\bz_{i,h} \in\bW_{i,h}^-}
\frac{b(\bz_{i,h},(q_{i,h},\hat{q}_{i,h}))}{\Vert \bz_{i,h}\Vert_{0}\Vert (q_{i,h}, \hat q_{i,h})\Vert_{\HDG}}
  \ge {\beta}_{D,d}>0, \qquad i=1,\ldots,n,
 \end{equation} 
where ${\beta}_{D,d}$ is a constant independent
of all problem parameters.
\end{lemma}

For our problem at hand, we identify the small bilinear forms introduced in \eqref{saddle-point_form} by
\begin{subequations}\label{combined_forms}
    \begin{align}
        \mathsf{a}_h(\uh,\vh)  & := a_h(({\bu}_h,\hat{\bu}_h),({\bv}_h,\hat{\bv}_h)) 
        + \sum_{i=1}^n (R_i^{-1}\bw_{i,h},\bz_{i,h}),
        \label{a_form}
        \\
        \mathsf{b}_h(\vh,\ph) & := - (\bm 1\cdot \bp_h,\ddiv \bv_h)  
        - \sum_{i=1}^n b( \bz_{i,h},(p_{i,h},\hat{p}_{i,h}))  \label{b_form},
        \\
        \mathsf{c}(\ph,\qh) & := \sum_{i=1}^n (\bm \zeta_{i}\cdot \bp_h,q_{i,h}) = (\Lambda_{\zeta} \bp_h, \bq_h),
        \label{c_form}
    \end{align}
\end{subequations}
and the linear form by
\begin{align}\label{relation_lin_F}
    \mathsf{F}_h((\vh,\qh))
    =(\bm{f},\bv_h) + \sum_{i=1}^{n} (g_i,q_{i,h}).
\end{align}
Next we specify the fitted norms in which problem~\eqref{eqn:MPET-disc-3-field-HDG} will be analyzed. 
To this end, consider  
\begin{align}
    \vert \qh \vert_{\bm Q_h}^2 &:= \sum_{i=1}^{n} R_{i} \Vert (q_{i, h}, \hat{q}_{i, h}) \Vert^2_{\HDG} + \frac{1}{\lambda_{0}} \Vert \sum_{i=1}^{n} q_{i, h} \Vert_0^{2}, \quad \forall \qh=(\bq_h, \hat{\bq}_h) \in \bm Q_h, 
    \label{Q_seminorm}
    \\ 
    \vert \vh\vert_{\bV_h}^2 &:= a_h^{\text{HDG}}((\bv_h, \hat{\bv}_h), (\bv_h, \hat{\bv}_h)) + \lambda \Vert \divv \bv_h \Vert_0^2 + \sum_{i=1}^{n} R_{i}^{-1} \Vert \bz_{i, h} \Vert_0^2 \label{V_seminorm} \\ 
    & 
    =  \mathsf{a}_h(\vh,\vh), \quad \forall \vh=((\bv_h,\hat{\bv}_h), \bz_h) \in \bV_h, \nonumber
\end{align}
which are to be used together with~\eqref{combined_forms} in the definitions~\eqref{Q_norm} and \eqref{V_norm}. 

\begin{definition}
    We call the problem~\eqref{saddle-point} uniformly well-posed under the norm $\Vert \cdot \Vert_{\bm Y_h}$ 
    defined by~\eqref{y_norm} if the conditions of Theorem~\ref{our_theorem} are satisfied and the constants $\overline{C}$ and $\underline{\alpha}$ in \eqref{boundedness_A} and \eqref{inf_sup_A} are independent of all problem parameters.
\end{definition}

\begin{theorem} 
Problem~\eqref{eqn:MPET-disc-3-field-HDG} is uniformly well-posed under the norm~\eqref{y_norm} where 
$\Vert \cdot \Vert_{\bm Q_h}$ and $\Vert \cdot \Vert_{\bV_h}$ are based on~\eqref{Q_seminorm} and \eqref{V_seminorm}, respectively.
\end{theorem}

\begin{proof}

Obviously, the coercivity estimate~\eqref{coerc_a} is fulfilled with $\underline{C}_{\mathsf{a}}=1$.
We next show that~\eqref{inf_sup_b} holds.
From the inf-sup conditions \eqref{hdg_inf_sup_stokes} and \eqref{hdg-hm_inf_sup_darcy}, we get that for $ (\bq_h, \hat{\bq}_h) \in \bm Q_h $ there exists $ {\vh}_{,0} = ((\bv_{h,0},\hat{\bv}_{h,0}),\bz_{h,0})\in \bV_h$ such that
\begin{align}\label{eq:est_stokes}
    -\divv \bv_{h,0}=\frac{1}{\lambda_0} \sum_{i=1}^{n} q_{i, h} \quad \text{and}\quad \Vert (\bv_{h,0},\hat{\bv}_{h,0})\Vert^2_{\bHDG}
    \le \beta_{S,d}^{-2}\frac{1}{\lambda_0^2}\Vert \sum_{i=1}^{n} q_{i, h}\Vert_0^2,
\end{align}
and
\begin{align}\label{eq:est_darcy}
    -b(\bz_{i, h,0},(q_{i, h},\hat{q}_{i, h}))=R_{i}\Vert (q_{i, h},\hat{q}_{i, h}) \Vert_{\HDG}^2 \quad \text{and} \quad 
    \Vert \bz_{i, h,0}\Vert_0^2 \le \beta_{D,d}^{-2} R_i^2 \Vert (q_{i, h},\hat{q}_{i, h})\Vert^2_{\HDG}.
\end{align}
Now we can infer
\begin{align}\label{eq:est1}
    \beta_{S,d}^{-2}\frac{1}{\lambda_{0}}\Vert \sum_{i=1}^{n} q_{i, h} \Vert_0^2  & \ge \frac{1}{2}\lambda_{0}\left[\Vert (\bv_{h,0},\hat{\bv}_{h,0})\Vert_{\bHDG}^2 
    + \Vert (\bv_{h,0},\hat{\bv}_{h,0})\Vert_{\bHDG}^2
    \right]\\ \nonumber
    & \ge \frac{1}{2\overline{C}_a} \lambda_{0} a_h^{\text{HDG}}((\bv_{h,0},\hat{\bv}_{h,0}),(\bv_{h,0},\hat{\bv}_{h,0}))
    +\frac{1}{2} \lambda_{0} \Vert \divv \bv_{h,0} \Vert_0^2 \\ \nonumber
    & \ge \frac{1}{2C}\left[ a_h^{\text{HDG}}((\bv_{h,0},\hat{\bv}_{h,0}),(\bv_{h,0},\hat{\bv}_{h,0})) + 2 \lambda_{0} \Vert \divv \bv_{h,0}\Vert_0^2\right],
\end{align}
where $C=\max\{2,\overline{C}_a\}$ and $\overline{C}_a$ is the boundedness constant for $a_h^{\text{HDG}}((\cdot,\cdot),(\cdot,\cdot))$.
Futhermore, we get
\begin{align*}
    \beta_{D, d}^{-2} R_{i} \Vert (q_{i,h},\hat{q}_{i,h}) \Vert^2_{\HDG}
    \ge R_{i}^{-1} \Vert \bz_{i,h,0}\Vert^2_0,
\end{align*}
which implies
\begin{align}\label{eq:est2}
    \beta_{D, d}^{-2} \sum_{i=1}^{n} R_{i} \Vert (q_{i,h},\hat{q}_{i,h}) \Vert^2_{\HDG}
    \ge \sum_{i=1}^{n} R_{i}^{-1} \Vert \bz_{i,h,0}\Vert^2_0.
\end{align}
Using the Cauchy-Schwarz inequality, we obtain
\begin{align*}
    \sum_{i=1}^n b( \bz_{i,h},(q_{i,h},\hat{q}_{i,h}))
    & \leq \sum_{i=1}^{n} C_{b} R_{i}^{1/2} \Vert (q_{i, h},\hat{q}_{i, h})\Vert_{\HDG} R_{i}^{-1/2} \Vert \bz_{i, h} \Vert_0 \\
    & \leq C_{b} \left(\sum_{i=1}^{n} R_{i} \Vert (q_{i, h},\hat{q}_{i, h})\Vert_{\HDG}^{2}\right)^{1/2} \left(\sum_{i=1}^{n} R_{i}^{-1} \Vert \bz_{i, h} \Vert_0^{2}\right)^{1/2} \\
    & \leq C_b \Vert\qh\Vert_{\bm Q_h} \left(\sum_{i=1}^{n} R_{i}^{-1} \Vert \bz_{i, h} \Vert_0^{2}\right)^{1/2},
\end{align*}
where $C_b$ is the boundedness constant of $ b(\cdot, (\cdot, \cdot)) $.
With this inequality, we conclude
\begin{align}\label{eq:est3}
    \begin{aligned}
        \vert \vh \vert^2_{\mathsf{b}}
        & = \Vert B \vh \Vert_{\bm Q_h'}^2
        = \left( \sup_{\substack{\qh \in \bm Q_h \\ \qh \neq 0}} \frac{\mathsf{b}_h(\vh, \qh)}{\Vert \qh \Vert_{\bm Q_h}} \right)^{2} \\
        & = \left( \sup_{\substack{\qh \in \bm Q_h \\ \qh \neq 0}} \frac{- (\sum_{i=1}^{n} q_{i, h},\ddiv \bv_h) - \sum_{i=1}^n b( \bz_{i,h},(q_{i,h},\hat{q}_{i,h}))}{ \Vert \qh \Vert_{\bm Q_h}}  \right)^2  \\
        & \le \left( \sup_{\substack{\qh \in \bm Q_h \\ \qh \neq 0}} \frac{\Vert \sum_{i=1}^{n} q_{i, h} \Vert_0 \Vert \divv \bv_h\Vert_0+ C_b \Vert\qh\Vert_{\bm Q_h} (\sum_{i=1}^{n} R_{i}^{-1} \Vert \bz_{i, h} \Vert_0^{2})^{1/2} }{ \Vert \qh \Vert_{\bm Q_h}}  \right)^2 \\
        & \le \left( \sup_{\substack{\qh \in \bm Q_h \\ \qh \neq 0}} \frac{ \overline{C}_b\sqrt{\lambda_{0}}\Vert \qh\Vert_{\bm Q_h} \Vert \divv \bv_h\Vert_0+
        \overline{C}_b \Vert\qh\Vert_{\bm Q_h} (\sum_{i=1}^{n} R_{i}^{-1} \Vert \bz_{i, h} \Vert_0^{2})^{1/2} }{ \Vert \qh \Vert_{\bm Q_h}}  \right)^2  \\
        &\le 2\overline{C}_b^2 \left(\lambda_{0} \Vert \divv \bv_h \Vert_0^2 + \sum_{i=1}^{n} R_{i}^{-1} \Vert \bz_{i, h} \Vert_0^{2} \right),
    \end{aligned}
\end{align}
where $ \overline{C}_{b} = \max\{1, C_{b}\} $.

Utilizing \eqref{eq:est1} -- \eqref{eq:est3} together with the choice $\beta_d:=\min\left\{\beta_{D,d},\frac{\beta_{S,d}}{\sqrt{2 C}}\right\}$, we receive the estimate
\begin{align}
    \beta_d^{-2} \vert \qh \vert_{\bm Q_h}^2
    & = \beta_d^{-2} \left(\sum_{i=1}^{n} R_{i} \Vert (q_{i, h}, \hat{q}_{i, h}) \Vert^2_{\HDG} + \frac{1}{\lambda_{0}} \Vert \sum_{i=1}^{n} q_{i, h} \Vert_0^{2}\right) \nonumber \\
    & \ge \sum_{i=1}^{n} R_{i}^{-1} \Vert \bz_{i,h,0}\Vert^2_0 + a_h^{\text{HDG}}((\bv_{h,0},\hat{\bv}_{h,0}),(\bv_{h,0},\hat{\bv}_{h,0}))+2 \lambda_{0} \Vert \divv \bv_{h,0} \Vert_0^2 \nonumber
    \\
    & \geq \frac{1}{2} \sum_{i=1}^{n} R_{i}^{-1} \Vert \bz_{i,h,0}\Vert^2_0 + a_h((\bv_{h,0},\hat{\bv}_{h,0}),(\bv_{h,0},\hat{\bv}_{h,0})) + 
    \lambda_{0} \Vert \divv \bv_{h,0} \Vert_0^2 \nonumber \\
    & \quad + \frac{1}{2} \sum_{i=1}^{n} R_{i}^{-1} \Vert \bz_{i,h,0}\Vert^2_0 \nonumber\\
    & \ge \frac{1}{2}\left( \mathsf{a}_h({\vh}_{,0},{\vh}_{,0}) + \lambda_{0} \Vert \divv \bv_{h,0} \Vert_0^2 
    + \sum_{i=1}^{n} R_{i}^{-1} \Vert \bz_{i,h,0}\Vert^2_0 \right) \nonumber \\
    & \ge \frac{1}{4\overline{C}_b^2}\left( \vert {\vh}_{,0} \vert_{\bV_h}^2 + \vert {\vh}_{,0} \vert_{\mathsf{b}}^2\right)
    = \frac{1}{4\overline{C}_b^2} \Vert {\vh}_{,0}\Vert_{\bV_h}^2. \label{eq:est4}
\end{align}
Finally, using \eqref{eq:est_stokes} and \eqref{eq:est_darcy} we deduce
\begin{align*}
    \sup_{\substack{\vh \in \bV_h \\ \vh \neq 0}} \frac{ \mathsf{b}_h(\vh,\qh)}{\Vert \vh\Vert_{\bV_h}} &\ge 
    \frac{\mathsf{b}_h({\vh}_{,0},\qh)}{\Vert {\vh}_{,0}\Vert_{\bV_h}} = 
    \frac{\vert \qh\vert^2_{\bm Q_h}}{\Vert {\vh}_{,0}\Vert_{\bV_h}} \ge 
    \frac{1}{2} \overline{C}_b^{-1} \beta_d \frac{\vert \qh\vert^2_{\bm Q_h}}{\vert \qh\vert_{\bm Q_h}}.
\end{align*}
\end{proof}

We now show uniform well-posedness of problem~\eqref{eqn:MPET-disc-3-field-HDG} in the $ \Vert \cdot \Vert_{\overline{\bm X}_h} $ norm. 
To this end, we prove that the norms $ \Vert \cdot \Vert_{\bm Y_h} $ and $ \Vert \cdot \Vert_{\overline{\bm X}_h} $ are equivalent with constants 
independent of all model and discretization parameters. The latter norm is then used in Section~\ref{sec::numerics} to set up parameter-robust preconditioners.

\begin{corollary}\label{cor:norm_equivalence}
    Problem \eqref{eqn:MPET-disc-3-field-HDG} is uniformly well-posed under the norm \eqref{product_norm}.
\end{corollary}
\begin{proof}
    We proof this by showing norm equivalence of the norms $ \Vert
    \cdot \Vert_{\bm Y_h} $ and $ \Vert \cdot \Vert_{\overline{\bm
    X}_h} $ with parameter-independent constants. For the first
    estimate, we exploit \eqref{eq:est3} in order to get
    \begin{align*}
        \Vert (\vh, \qh) \Vert_{\bm Y_h}^{2} & = \vert \qh \vert_{\bm Q_h}^2 + \mathsf{c}(\qh, \qh) + \vert \vh \vert_{\bV_h}^2 + \vert \vh \vert_{\mathsf{b}}^2 \\
        & \leq \sum_{i=1}^{n} R_{i} \Vert (q_{i, h}, \hat{q}_{i, h}) \Vert^2_{\HDG} + \frac{1}{\lambda_{0}} \Vert \sum_{i=1}^{n} q_{i, h} \Vert_0^{2} + (\Lambda_{\zeta} \bq_h, \bq_h) \\
        & \quad +  a_h^{\text{HDG}}((\bv_h, \hat{\bv}_h), (\bv_h, \hat{\bv}_h)) + \lambda \Vert \divv \bv_h \Vert_0^2 + \sum_{i=1}^{n} R_{i}^{-1} \Vert \bz_{i, h} \Vert_0^2 \\
        & \quad + 2 \overline{C}_b^2 \left(\lambda_{0} \Vert \divv \bv_h \Vert_0^2 + \sum_{i=1}^{n} R_{i}^{-1} \Vert \bz_{i, h} \Vert_0^{2} \right) \\
        & \leq \Vert (\bq_h, \hat{\bq}_h) \Vert_{\overline{\bP}_h}^2 + \overline{C}_{a} \Vert (\bv_h, \hat{\bv}_h) \Vert_{\bHDG}^2 + \lambda \Vert \divv \bv_h \Vert_{0}^2 + \Vert \bz_h \Vert_{\bW^-}^2 \\
        & \quad + 2\overline{C}_b^2 \left( \Vert (\bv_h, \hat{\bv}_h) \Vert_{\bHDG}^2 + \lambda \Vert \divv \bv_h \Vert_0^2 + \Vert \bz_h \Vert_{\bW^-}^2 \right) \\
        & \leq C (\Vert (\bq_h, \hat{\bq}_h) \Vert_{\overline{\bP}_h}^{2} + \Vert (\bv_h, \hat{\bv}_h) \Vert_{\overline{\bU}_h}^2 + \Vert \bz_h \Vert_{\bW^-}^2) \\
        & = C\Vert ((\bv_h, \hat{\bv}_h), \bz_h, (\bq_h, \hat{\bq}_h)) \Vert_{\overline{\bm X}_h}^2,
    \end{align*}
    where $ C = \max\{1, \overline{C}_a, 2 \overline{C}_b^2\} $. The second estimate follows directly by the definitions of the norms, which yields
    \begin{align*}
        \Vert (\vh, \qh) \Vert_{\bm Y_h}^{2}
        & \geq \sum_{i=1}^{n} R_{i} \Vert (q_{i, h}, \hat{q}_{i, h}) \Vert^2_{\HDG} + \frac{1}{\lambda_{0}} \Vert \sum_{i=1}^{n} q_{i, h} \Vert_0^{2} + (\Lambda_{\zeta} \bq_h, \bq_h) \\
        & \quad + \underline{C}_c \Vert (\bv_h, \hat{\bv}_h) \Vert_{\bHDG}^2 + \lambda \Vert \divv \bv_h \Vert_0^2 + \sum_{i=1}^{n} R_{i}^{-1} \Vert \bz_{i, h} \Vert_0^2 + \vert \vh \vert_{\mathsf{b}}^2 \\
        & \geq C (\Vert (\bv_h, \hat{\bv}_h) \Vert_{\overline{\bU}_h}^2 + \Vert (\bq_h, \hat{\bq}_h) \Vert_{\overline{\bP}_h}^2 + \Vert \bz_h \Vert_{\bW^-}^2) \\
        & = C \Vert ((\bv_h, \hat{\bv}_h), \bz_h, (\bq_h, \hat{\bq}_h)) \Vert_{\overline{\bm X}_h}^2,
    \end{align*}
    where $C = \min\{1, \underline{C}_c\}$ and $\underline{C}_c $ is the coercivity constant of $a_h^{\text{HDG}}((\cdot, \cdot), (\cdot, \cdot))$.
\end{proof}

\begin{remark}
    In the case $ n= 1$, the definition of norm~\eqref{hdg_p_norm_scaled} reduces to the norm (18b) in \cite{kraus2020uniformly}.
\end{remark}



\section{Numerical examples} \label{sec::numerics}

In this section we demonstrate the applicability of our theoretical results with some
numerical experiments. For this, we first propose a norm-equivalent
preconditioner following similar steps as in
\cite{kraus2020uniformly}. The preconditioner is then extensively
tested with respect to the parameters in several test cases. Finally,
we perform some numerical simulations for a 4-network model of a human brain.
All numerical experiments are realized in the finite
element library Netgen/NGSolve, see \cite{netgen, ngsolve} and
\url{www.ngsolve.org}.

\subsection{Norm equivalent preconditioners}

In \cite{kraus2020uniformly} the authors already provided a detailed
construction of a parameter robust well-posedness theory and
preconditioner of the Biot problem, i.e. the case of the MPET problem
with $n=1$. In the following we give a brief overview of the
corresponding extension to the MPET problem, but refer to
\cite{kraus2020uniformly} for details regarding implementation aspects
and a discussion on using static condensation to eliminate local
(element-wise) unknowns from the system. To this end, let
$\uh=(({\bu}_h,\hat{\bu}_h), \bw_h)$, $\ph=(\bp_h,\hat{\bp}_h)$ be the
solution of \eqref{saddle-point}. Using the same notation for the
coefficient vectors as for the solutions we can rewrite
\eqref{saddle-point} as 
\begin{align} \label{eq::systemmatrix}
    \begin{pmatrix}
      \mA & \mB^\T\\\mB & -\mC
    \end{pmatrix}
    \begin{pmatrix}
        \uh \\ \ph\\
    \end{pmatrix} = 
      \mF,
  \end{align}
where $\mA, \mB, \mC$ and $\mF$ are the corresponding system matrices
and the right hand side vector defined by the bilinear forms
\eqref{combined_forms} and the linear form \eqref{relation_lin_F}.
According to Corollary~\ref{cor:norm_equivalence}, a norm equivalent
preconditioner, compare Section 4.2 in \cite{kraus2020uniformly}, is
then given by  
\begin{align}
    \mathcal{B} := 
    \begin{pmatrix}
      \mX_{\bu, \bw} & \mO \\\mO & \mX_{\bp},
    \end{pmatrix},
  \end{align}
  where $\mX_{\bu, \bw}$ and $\mX_{\bp}$ are the system matrices of
  the bilinear forms induced by the norms used in
  \eqref{product_norm}, i.e. $\Vert \cdot \Vert^2 _{\overline{\bm
  U}_h} + \Vert \cdot \Vert^2_{\bm W^{-}}$ and $\Vert \cdot
  \Vert^2_{\overline{\bP}_h}$, respectively. An alternative
  preconditioner can be proposed following the steps from Section 4.4
  in \cite{kraus2020uniformly}. Since the element-wise block diagonal
  structure of the part of the matrix $\mA$ that corresponds to the
  bilinear form $\sum_{i=1}^n (R_i^{-1}\bw_{i,h},\bz_{i,h})$, see
  second term in \eqref{a_form}, allows a very efficient inversion,
  one can eliminate the velocities $\bw_{i,h}$ from the system matrix.
  Then one solves a reduced problem using the stiffness matrix for the
  displacement (first term in \eqref{a_form}), the divergence
  constraints coupling the displacements and the pressures (first term
  in \eqref{b_form}), and the resulting pressure Schur complement with
  respect to the velocities. Using this technique, we
  introduce the second preconditioner by  
  \begin{align}
    \tilde {\mathcal{B}} := 
    \begin{pmatrix}
      \mX_{\bu} & \mO \\\mO & \tilde{\mX}_{\bp},
    \end{pmatrix}
  \end{align}
  where $\mX_{\bu}$ corresponds to the bilinear form induced by the
  norm $\Vert \cdot \Vert^2 _{\overline{\bm U}_h}$ and
  $\tilde{\mX}_{\bp}$ is the above mentioned resulting pressure Schur
  complement. A detailed construction of $\tilde{\mX}_{\bp}$ is given
  in equation~(63) in \cite{kraus2020uniformly}. Note that there holds
  the spectral equivalence $\mX_{\bp} \eqsim
  \tilde{\mX}_{\bp}$, i.e. both preconditioners ${\mathcal{B}}$ and
  $\tilde {\mathcal{B}}$ are well suited to solve the system with an
  iterative method. However, in accordance with the results from
  \cite{kraus2020uniformly}, the preconditioner $\tilde {\mathcal{B}}$
  performs better, see Section~\ref{sec::comparison}.

\subsection{Parameter-robustness} \label{sec::testproblem}
To study the numerical behavior of the norm equivalent preconditioners
$\mathcal{B}$ and $\tilde{\mathcal{B}}$ defined in the previous
section we consider problem \eqref{eqn:MPET-disc-3-field-HDG} in
the 2-network case on the spatial domain $\Omega = (0, 1)^3$ with
exact solution
\begin{align*}
    \bm u &:= (\varphi, \varphi, \varphi), &
    \bm p &:=
    \begin{pmatrix}
        x^2 (1 - x)^2 y^2 (1 - y)^2 z^2 (1 - z)^2 \\
        \sin(\pi x)^2 \sin(\pi y)^2 \sin(\pi z)^2
    \end{pmatrix}
    - \bm p_{0},
\end{align*}
where $\varphi = \sin(\pi x) \sin(\pi y)\sin(\pi z)$. The constant
vector $\bm p_0$ is chosen in such a way that $\bm p \in
L_{0}^{2}(\Omega)^2$. We solve the problem using a regular mesh
$\mathcal{T}_h$ with varying mesh size. For the interior penalty
stabilization parameter $\eta$ in \eqref{a_h_HDG} we choose the value
$\eta = 10$. All problems were solved by means of the preconditioned
minimal residual method (MinRes) with a fixed tolerance of $10^{-8}$.
The symmetric positive definite (SPD) preconditioner (in the following
section denoted by $\mathcal{B}$ and $\tilde{\mathcal{B}}$) is
inverted using a direct method. To simplify the notation we
introduce symbols for the (scaled) parameters for the pressures when
they coincide, i.e. we have $\alpha_{p} := \alpha_{p_1} =
\alpha_{p_2}$, $R := R_1 = R_2$. If the parameters differ from each
other we use the notation used so far. Further note that since we have
only two pressures we use the symbol $\xi := \xi_{12} = \xi_{21}$.

\subsubsection{Comparison of the preconditioners $\mathcal{B}$ and
$\tilde{\mathcal{B}}$}
\label{sec::comparison}
In this first example we want to compare the performance of the MinRes
solver using either $\mathcal{B}$ or $\tilde{\mathcal{B}}$ as a
preconditioner. To this end, we solve the problem of
Section~\ref{sec::testproblem} using the polynomial order $\ell=2$, a
triangulation with $|\mathcal{T}_h| = 384$ elements, and consider the
parameters $\alpha_{p} = 0, \xi = 0, \lambda = 1$. The value of $R$ is
set to $10^{-i}$ with $i = 0,2,4,6,8$. Note that the test case of
vanishing $R$ and zero values for $\xi, \alpha_p$ is the most
challenging one as observed in Section~\ref{sec::firsttest}. In
Figure~\ref{fig::history_tB} and Figure~\ref{fig::history_B} we plot
the history of the residual (res) over the number of iterations (it)
of the MinRes solver for both preconditioners. Several observations
can be made here. First of all note that the plots show a robustness
with respect to $R$ since there appears an upper bound of the
iterations for $R \rightarrow 0$. To discuss this in more detail we
note that (according to \cite{MR2086829}) the number of iterations of
the MinRes solver particularly depends on the spectral equivalences 
\begin{align} \label{eq::specequi}
    \mX_{\bp} \eqsim \mS_{\bp} \quad  \textrm{or} \quad \tilde{\mX}_{\bp} \eqsim \mS_{\bp},
\end{align} 
 when using $\mathcal{B}$ or $\tilde{\mathcal{B}}$, respectively. Here
$\mS_{\bp} := -\mB \mA^{-1} \mB^\T - \mC$ is the pressure Schur
complement (with respect to the displacement $\bu$). Since we have
chosen $\alpha_p = \xi = 0$, we see that in the limit $R \rightarrow
0$ the ``$\mC$-block'' in the system matrix \eqref{eq::systemmatrix} and
in $\mS_{\bp}$ vanishes. Further, according to definition~\eqref{hdg_p_norm_scaled} both matrices $\mX_{\bp}$ and
$\tilde{\mX}_{\bp}$ then converge to a matrix which is induced by the
bilinear form $(\Lambda \bp, \bq)$ which we denote by
$\mM_{\bp}^\Lambda$ in the following. 
As we have $\Lambda = \frac{1}{\lambda_0} \mathbb{I}$, see
\eqref{eq::lambdamat}, standard estimates (compare \cite{MR768638}) then give the
spectral equivalence, 
\begin{align} \label{eq::schur_spec}
    \mB \mA^{-1} \mB^\T \eqsim \mM_{\bp}^\Lambda.
\end{align}
 This shows that when $R \rightarrow 0$, the spectral
equivalences in \eqref{eq::specequi} are dominated by
\eqref{eq::schur_spec} (as $\mC$ vanishes), and thus the choice of the
elliptic discrete $H^1$-part of the pressure preconditioner given by
the HDG-norm, see \eqref{hdg_p_norm_scaled}, in $\mX_{\bp}$, or by the
elliptic operator resulting from the elimination of the local
velocities (see previous section) in $\tilde{\mX}_{\bp}$, is not
crucial anymore. Summarizing we see that the upper bound of the
iterations is independent of the choice of $\mathcal{B}$ and
$\tilde{\mathcal{B}}$ and robust in $R$. However, since the latter
version provides better results particularly also for moderate
numbers of $R$, we only consider $\tilde{\mathcal{B}}$ for the rest of
this work.

\begin{remark}
    Equations \eqref{eq::specequi} and \eqref{eq::schur_spec} show
    why it is essential that $R > 0$. As $\mM_{\bp}^\Lambda$ is
    not SPD the choice $R=0$ (and $\alpha_p = \xi = 0$) would then
    lead to non SPD preconditioners $\mX_{\bp}$ or
    $\tilde{\mX}_{\bp}$. In other words the norm in
    \eqref{hdg_p_norm_scaled} is only a semi norm for $R = \alpha_p =
    \xi = 0$.
\end{remark}

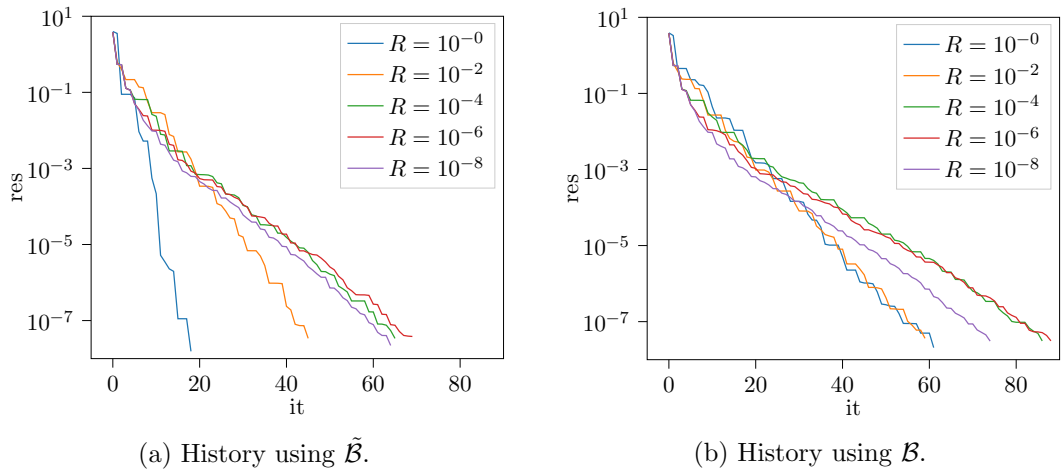
\begin{figure}
    \begin{subfigure}{0.4\textwidth}
        \centering
        \begin{tikzpicture}[scale = 0.8]
            \input{./suppl_tex/history_tB}     
        \end{tikzpicture}        
        \caption{History using  $\tilde{\mathcal{B}}$.}
        \label{fig::history_tB}
    \end{subfigure}
    \hspace{0.5cm}
    \begin{subfigure}{0.4\textwidth}
        \centering
        \begin{tikzpicture}[scale = 0.8]
            \input{./suppl_tex/history_B}     
        \end{tikzpicture} 
        \caption{History using $\mathcal{B}$.}
        \label{fig::history_B}
    \end{subfigure}
    \caption{History of the residual using the preconditioned Minres solver.}
    \label{fig:history}
\end{figure}

\subsubsection{Robustness for equivalent pressure
coefficients}\label{sec::firsttest} 

We solve the problem of Section~\ref{sec::testproblem} using the
polynomial orders $\ell=2,3$ and a fixed mesh with $|\mathcal{T}_h| =
384$ elements. We consider the parameters $\alpha_{p} = 10^{-i}, R =
10^{-i}, \xi = 10^{-i}$ with $i = 0,2,4,6,8$ and $\lambda = 10^{j}$
with $j = 0,4,8$. In Table~\ref{tab::firsttest} we have plotted the
number of iterations of the MinRes solver in orange color for
the polynomial order $\ell=2$ and in blue color for $\ell = 3$. Although
varying slightly, a detailed analysis of the numbers shows a similar
behavior as already discussed in the previous section. While being
very robust in $\lambda$ and $\xi$, smaller values of $\alpha_p$ and
$R$ (i.e. the $\mC$-block gets ``smaller'', compare discussion below
\eqref{eq::specequi}) increase the number of iterations. However, in
accordance with the findings from before, an upper bound is given by
approximately 70 iterations. 

\setlength\tabcolsep{2pt}
  \begin{table}[]
    \centering  
      \input{./suppl_tex/table_2}    
    \caption{Results for the first test case of
    Section~\ref{sec::firsttest}. Orange color is used for $\ell=2$, blue
    for $\ell=3$.}\label{tab::firsttest}
\end{table}

\subsubsection{Robustness for different pressure
coefficients}\label{sec::secondtest} 

Again we solve the problem of Section~\ref{sec::testproblem} using the
polynomial orders $\ell=2,3$ and a fixed mesh with $|\mathcal{T}_h| =
384$ elements. In contrast to the previous section we now fix the
values of $\alpha_{p_1} = R_1 = 10^{-4}$ and vary $\alpha_{p_2} =
10^{-i}, R_2 = 10^{-i}, \xi = 10^{-i}$ with $i = 0,2,4,6,8$ and
$\lambda = 10^{j}$ with $j = 0,4,8$. Again, the results of
Table~\ref{tab::secondtest} show robustness with respect to all
parameters. Comparing the results with the numbers from the previous
section, see Table~\ref{tab::firsttest}, we see that for bigger values
of $\alpha_{p_2}, R_2$ the (relatively) smaller parameters
$\alpha_{p_1}, R_1$ slightly increase the number of iterations (the
$\mC$-block is ``smaller'' as before), and for smaller values of
$\alpha_{p_2}, R_2$ the (relatively) larger parameters $\alpha_{p_1},
R_1$ slightly reduce the number (i.e. the $\mC$-block is ``larger'' as
before). 

\setlength\tabcolsep{2pt}
  \begin{table}[]
    \centering  
      \input{./suppl_tex/table_mixed_2}    
    \caption{Results for the second test case of
    Section~\ref{sec::secondtest}. Orange color is used for $\ell=2$, blue
    for $\ell=3$.} \label{tab::secondtest}
\end{table}

\subsubsection{Robustness with respect to $\mathcal{T}_h$ and the
polynomial order $\ell$.}\label{sec::thirdtest} 

In the last example we discuss the performance of the preconditioner
with respect to the polynomial order and the number of elements of the
triangulation. To this end, we solve the problem of
Section~\ref{sec::testproblem} using the polynomial orders
$\ell=1,\ldots,6$ and several meshes with $|\mathcal{T}_h| = 48, 162, 384,
750, 1296, 2058$ elements. We use constant values $\alpha_{p} = R =
\xi = 10^{-4}$ and $\lambda = 1$. Since we have used a direct solver
to invert $\mathcal{B}$ and $\tilde{\mathcal{B}}$, a dependency with
respect to the mesh size is not expected, and is validated by the
numbers of the right Table~\ref{tab::thirdtest}. In addition to that, the
results also show that the solver is very robust with respect to the
polynomial order.


    \begin{table}[]
            \centering  
            \input{./suppl_tex/table_hconv}    
            \caption{Results for the third test case of Section~\ref{sec::thirdtest}.} \label{tab::thirdtest}
    \end{table}

\subsection{Four-network model of the human brain}
As an application of the MPET formulation~\eqref{eq::threefield}, we
consider a 4-network model of the human brain. This model has already
been studied for the total pressure formulation, e.g. in
\cite{lee2019spacetime}. The brain geometry $\Omega$ originates from
the Colin 27 Adult Brain Atlas FEM mesh. The surface $\partial \Omega$
is split into the outer boundary representing the skull $\Gammas$, and
the inner boundary representing the ventricles $\Gammav$. For the
computations we use the same mesh as in \cite{lee2019spacetime, colinmesh} with
99605 elements and 29037 vertices, see Figure~\ref{fig:brain_mesh}. We
take for the unscaled parameters in problem~\eqref{eq::threefield} the
values summarized in Table~\ref{parameters_MPET4}. Note that the
Lam\'{e} parameters can be expressed in terms of Young's modulus of
elasticity $ E $ and the Poisson ratio $ \nu \in \left[0, 1/2\right) $
via $\lambda  = \frac{\nu E}{(1 + \nu) (1 - 2 \nu)}$, and $\mu  =
\frac{E}{2(1 + \nu)}$. For the
boundary conditions (all values given in mmHg) we use
\begin{alignat*}{5}
    p_1 &= 5+2 \sin (2\pi t) && \quad \textrm{on } \Gammas, &&\quad & p_1 &= 5+2.012\sin(2 \pi t) && \quad \textrm{on } \Gammav, \\ 
    p_2 & = 70+10 \sin (2\pi t) && \quad \textrm{on } \Gammas, &&\quad & \nabla p_2\cdot \bm n & = 0 && \quad \textrm{on } \Gammav, \\
    p_3 & =  6 && \quad \textrm{on } \Gammas, &&\quad & p_3 & = 6 && \quad \textrm{on } \Gammav, \\
    \nabla p_4\cdot \bm n & = 0 &&\quad \textrm{on } \Gammas, &&\quad & \nabla p_4\cdot \bm n & = 0 && \quad \textrm{on } \Gammav, \\
    \bm u  &= 0 &&\quad \textrm{on } \Gammas, &&\quad &   \bm \sigma \bm n  & = -\sum_{i=1}^4 \alpha_i p_{i,V} \bm n && \quad \textrm{on } \Gammav,
\end{alignat*}
where $p_{i,V}$ for $i=1,\ldots,4$ are prescribed pressures on
the ventricles given by 
\begin{align*}
p_{1,V} = 5+2.012\sin(2 \pi t), \quad p_{2,V} = 70+10 \sin (2\pi t), \quad p_{3,V} = 6, \quad p_{4,V} = 38.
\end{align*}
The initial conditions from \cite{lee2019spacetime} for $t = 0$ are
given by $\bm u = 0, p_1 = 5, p_2 = 70, p_3 = 6$, and $p_4 = 38$. We
solved the time dependent problem using different timesteps (discussed
in detail below) and a fixed polynomial order of $\ell=1$. For each
timestep the problem was solved iteratively converging with
approximately 15 iterations. 

As a first test and to compare our results with
\cite{lee2019spacetime} we used a timestep size of $\tau = 0.0125s$ and
solved the problem until $T = 3$. In Figure~\ref{fig::brain} we can
see the absolute value of the displacement and the pressure
distributions at $t = 0.25s$. Further we plotted in
Figure~\ref{fig::eval_absu} and Figure~\ref{fig::brainevals_pres} the
absolute value of the displacement and the pressures at three fixed
points in space given by $P_1 = (89.9,108.9,82.3)$, $P_2 =
(102.2,139.3,82.3)$, and $P_3 = (110.7,108.9,98.5)$, which are located
close to the ventricles. We see that the results are in agreement with
\cite{lee2019spacetime}. The biggest difference appears in the
extra-cellular pressure $p_1$ which in our simulations does not exhibit the
oscillating behavior that can be found in \cite{lee2019spacetime}. To
discuss this in more detail let us recall the unscaled equations
\eqref{eq::threefield-2}--\eqref{eq::threefield-3} for $i=1$
\begin{align} \label{eq::scaledpex}
    \frac{\partial}{\partial t}s_1 p_1  - \divv (K_1 \nabla p_1) + (\xi_{12} + \xi_{13} + \xi_{14}) p_1 &= \tilde{g}_i - \frac{\partial}{\partial t}\alpha_i \divv \bu + \xi_{12} p_2 + \xi_{13} p_3 + \xi_{14} p_4,
\end{align}
with (see Table~\ref{parameters_MPET4})
\begin{align*}
    s_1 = 3.9 \cdot 10^{-4} \quad \textrm{and} \quad K_1 = 1.57 \cdot 10^{-5} \quad \textrm{and} \quad \xi_1 :=  \xi_{12} + \xi_{13} + \xi_{14} = 2 \cdot 10^{-6}.
\end{align*}
This shows that if $p_1$ is stationary ($ {\partial p_1}/{\partial t}=
0$) the magnitude of the diffusive coefficient $K_1$ is bigger as
the source coefficient $\xi_1$, i.e. no boundary layer is expected.
However if the time derivative dominates (in addition to $s_1$ being
ten times bigger as $K_1$) a boundary layer might occur in certain
cases depending on the right hand side and the considered boundary
conditions. 
Although our triangulation does not inherit a boundary refinement
which would be needed to surpass oscillations appearing from non
resolved boundary layers when using continuous finite elements,    
our discretization \eqref{saddle-point} is not effected by non
physical oscillations since boundary conditions are only incorporated
in a weak (Nitsche-like) sense. Indeed, the evaluation of the
extra-cellular pressure $p_1$ in Figure~\ref{fig:brain_cell} appears
to be nearly constant. In contrast to $p_1$ all other pressures do not
show a boundary layer which can be explained by the much bigger
diffusive coefficients $K_2, K_3$ and $K_4$ (compared to corresponding
coefficients $s_i$).

Since the pressures $p_1$ and $p_2$ in Figure~\ref{fig:eval_pone} and
Figure~\ref{fig:eval_pfour} suggest that the system is not in a steady
periodic state yet, we calculated a long term simulation until
$T=2500s$ with a time step $\tau = 0.125s$. In
Figure~\ref{fig::brainevals_pres_mean} we have plotted the periodic
mean values of the pressures at the points $P_j$ defined by
\begin{align*}
    \langle p_i \rangle (P_j, t_k) := \int_{t_k-1/2}^{t_k+1/2} p_i(P_j, s) \ds, 
\end{align*}
and a periodic solution appears at approximately $t = 1500s$. The
definition of the mean value $\langle \cdot \rangle(\cdot, t_k)$ was
appropriately changed for $t_k < 0.5$. Note that although
Figure~\ref{fig:eval_pthree_mean} suggests a non-smooth time
dependency of the venous pressure at the beginning,
Figure~\ref{fig:eval_pthree} reveals that there is a steep but smooth
increase. 
%


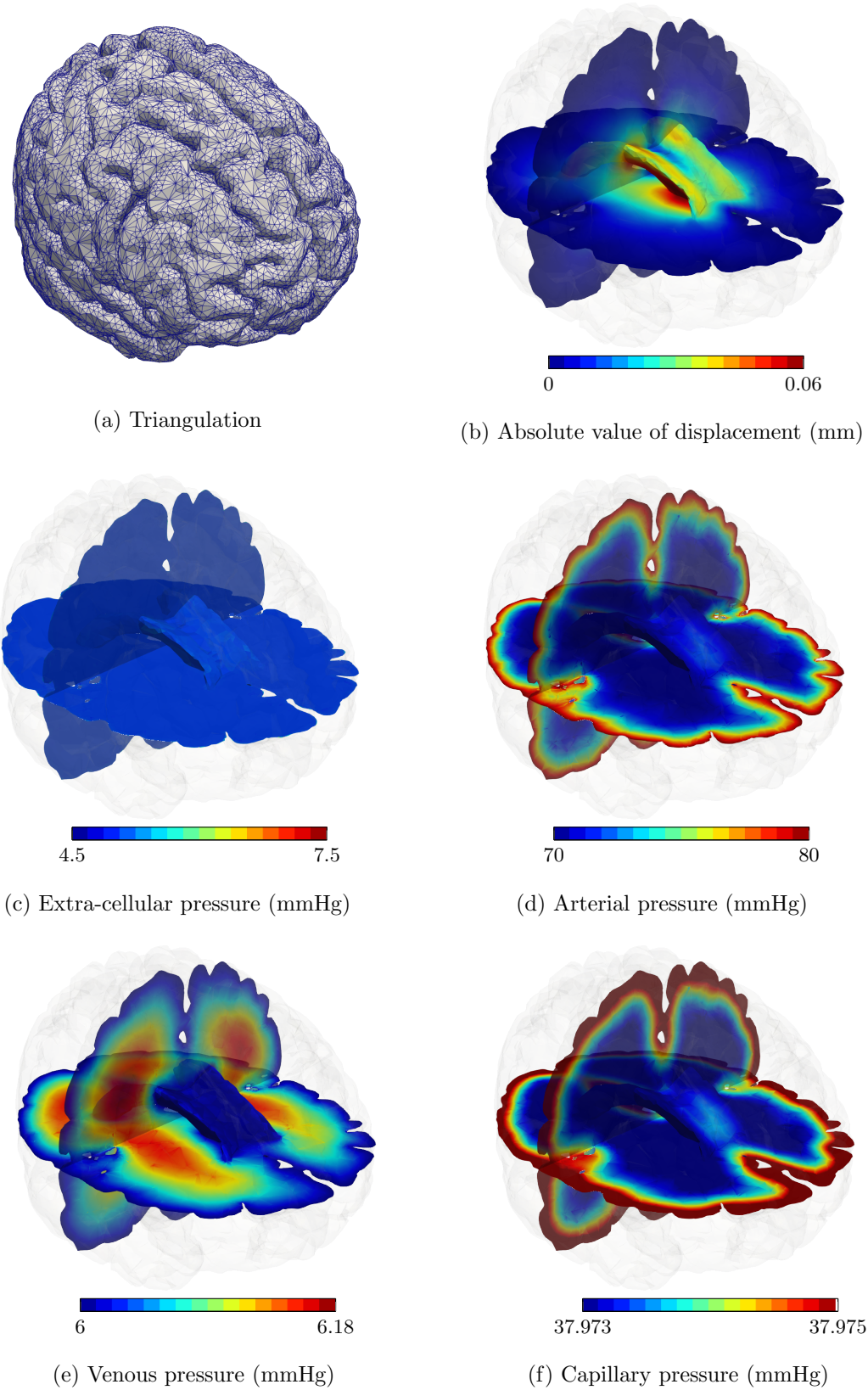
\begin{figure}
    \begin{subfigure}{0.45\textwidth}
        \centering
        \includegraphics[width=.9\linewidth]{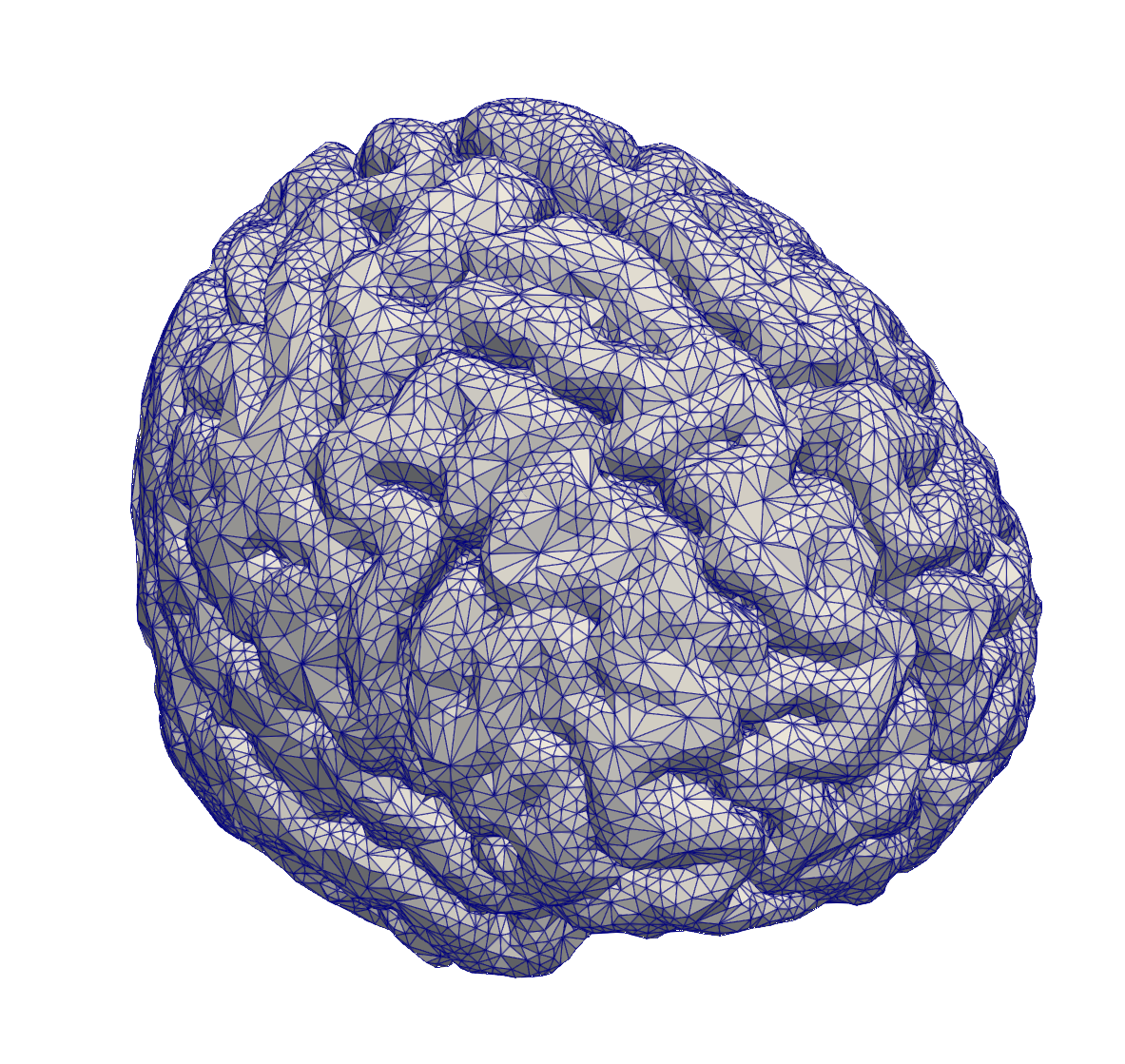}  
        \caption{Triangulation}
        \label{fig:brain_mesh}
    \end{subfigure}
    \begin{subfigure}{0.45\textwidth}
        \centering
        \begin{tikzpicture}
            \node[anchor=south west] (bar1) at (1,-0.5) {
                \pgfplotscolorbardrawstandalone[colormap/jet,colorbar sampled,colorbar horizontal,point meta min=0,point meta max=0.06, 
                colorbar style={samples=17, width=4cm, height = 0.2cm, xtick={0, 0.06}, 
                scaled x ticks=false, 
                xticklabel style={style={font=\footnotesize}, /pgf/number format/fixed,  /pgf/number format/precision=2}}]};
            \node[anchor=south west] (numex1) at (0,0)  {\includegraphics[width=0.8\textwidth]{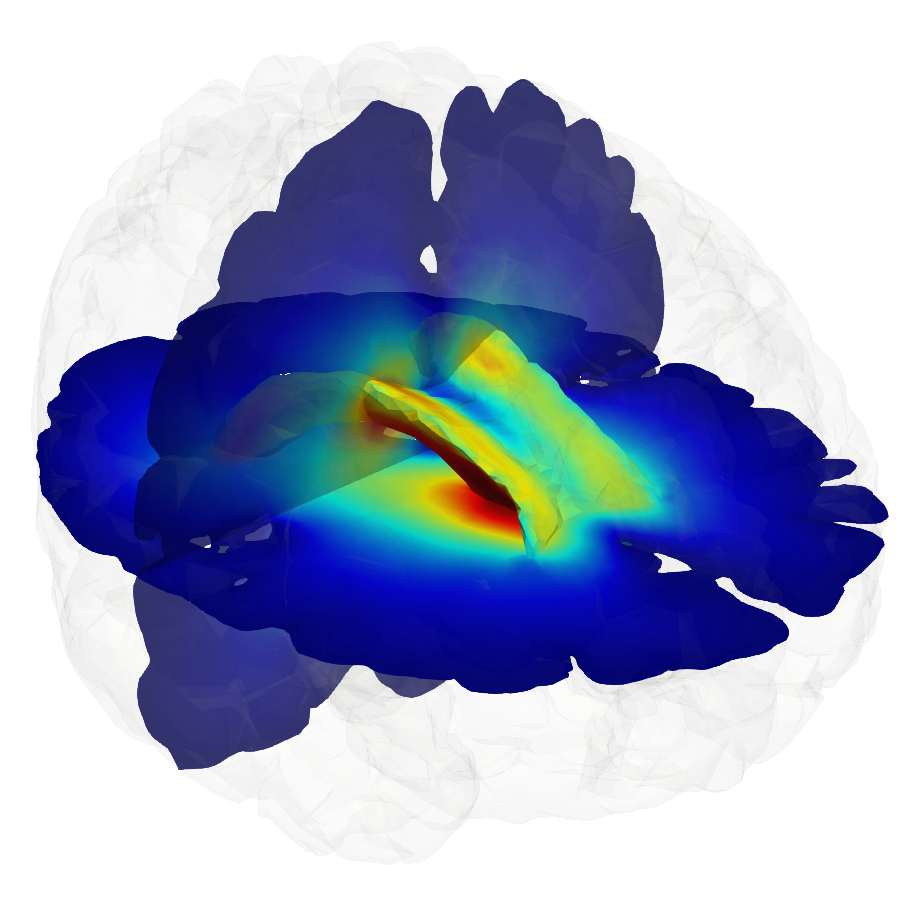}};
          \end{tikzpicture}
        \caption{Absolute value of displacement (mm)}
        \label{fig:brain_displ}
    \end{subfigure}
    \newline
    \begin{subfigure}{0.45\textwidth}
        \centering
        \begin{tikzpicture}
            \node[anchor=south west] (bar1) at (1,-0.5) {
                \pgfplotscolorbardrawstandalone[colormap/jet,colorbar sampled,colorbar horizontal,point meta min=4.5,point meta max=7.5, 
                colorbar style={samples=17, width=4cm, height = 0.2cm, xtick={4.5, 7.5}, 
                scaled x ticks=false, 
                xticklabel style={style={font=\footnotesize}, /pgf/number format/fixed,  /pgf/number format/precision=2}}]};
            \node[anchor=south west] (numex1) at (0,0)  {\includegraphics[width=0.8\textwidth]{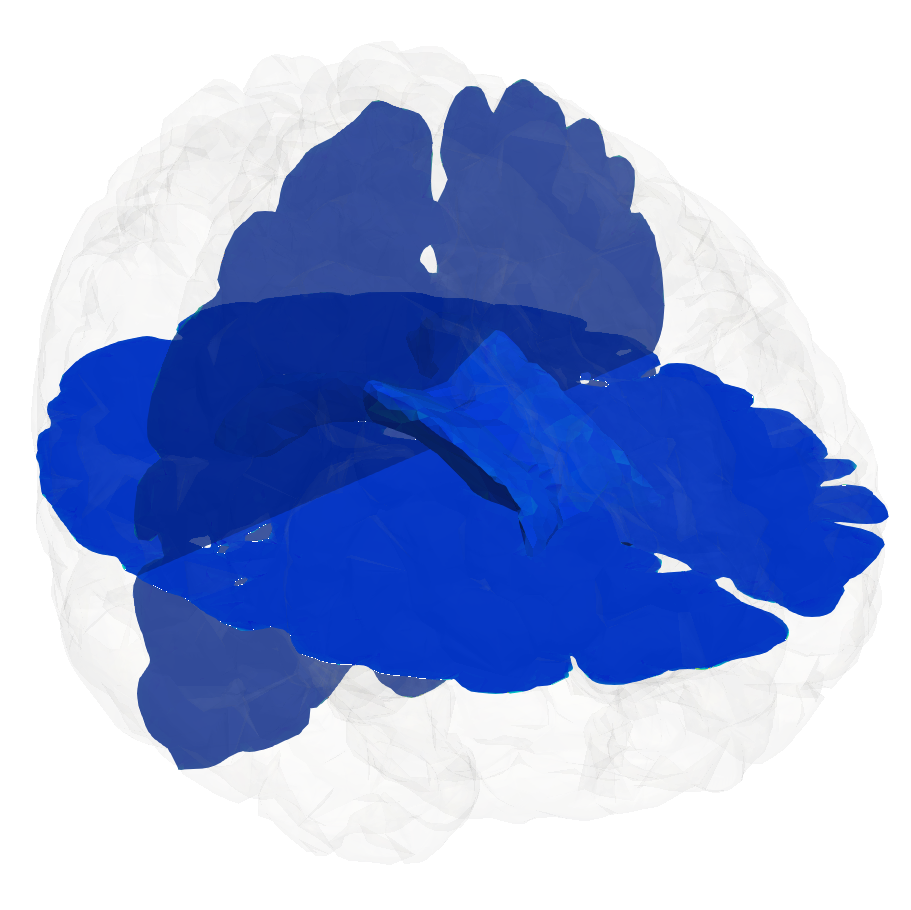}};
          \end{tikzpicture}
        \caption{Extra-cellular pressure (mmHg)}
        \label{fig:brain_cell}
    \end{subfigure}
    \begin{subfigure}{0.45\textwidth}
        \centering
        \begin{tikzpicture}
            \node[anchor=south west] (bar1) at (1,-0.5) {
                \pgfplotscolorbardrawstandalone[colormap/jet,colorbar sampled,colorbar horizontal,point meta min=70,point meta max=80, 
                colorbar style={samples=17, width=4cm, height = 0.2cm, xtick={70, 80}, 
                scaled x ticks=false, 
                xticklabel style={style={font=\footnotesize}, /pgf/number format/fixed,  /pgf/number format/precision=2}}]};
            \node[anchor=south west] (numex1) at (0,0)  {\includegraphics[width=0.8\textwidth]{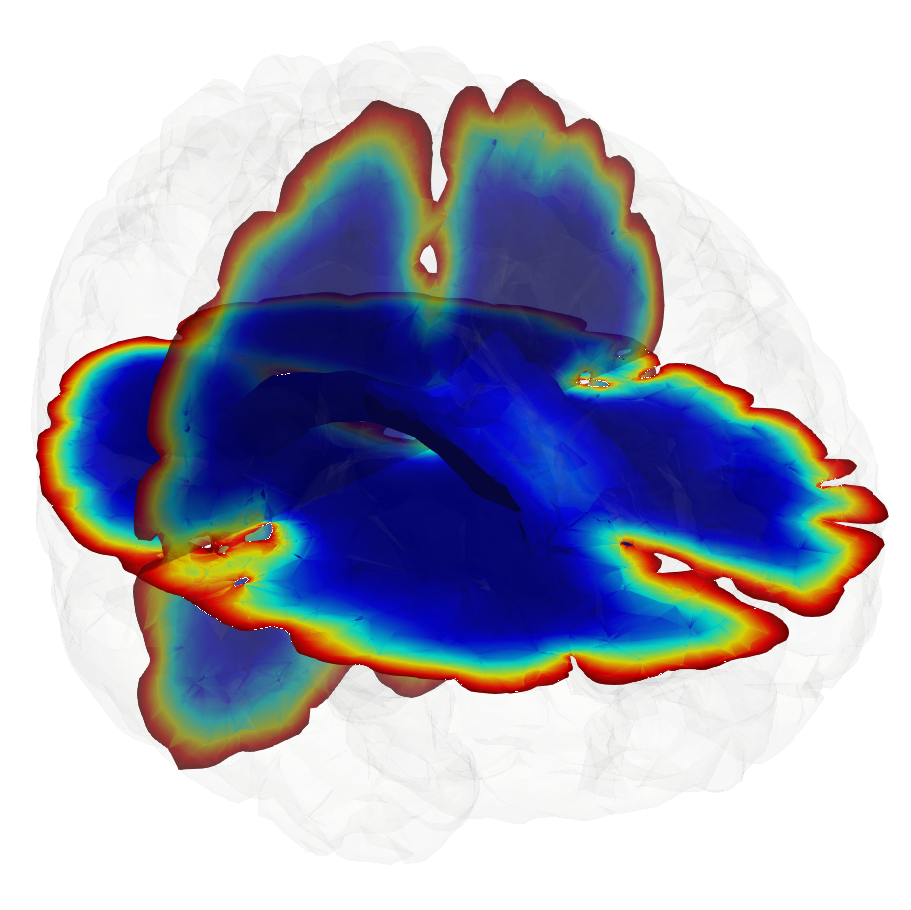}};
          \end{tikzpicture}
        \caption{Arterial pressure (mmHg)}
        \label{fig:brain_art}
    \end{subfigure}
    \newline
    \begin{subfigure}{0.45\textwidth}
        \centering
        \begin{tikzpicture}
            \node[anchor=south west] (bar1) at (1,-0.5) {
                \pgfplotscolorbardrawstandalone[colormap/jet,colorbar sampled,colorbar horizontal,point meta min=6.0,point meta max=6.18, 
                colorbar style={samples=17, width=4cm, height = 0.2cm, xtick={6.0, 6.18}, 
                scaled x ticks=false, 
                xticklabel style={style={font=\footnotesize}, /pgf/number format/fixed,  /pgf/number format/precision=2}}]};
            \node[anchor=south west] (numex1) at (0,0)  {\includegraphics[width=0.8\textwidth]{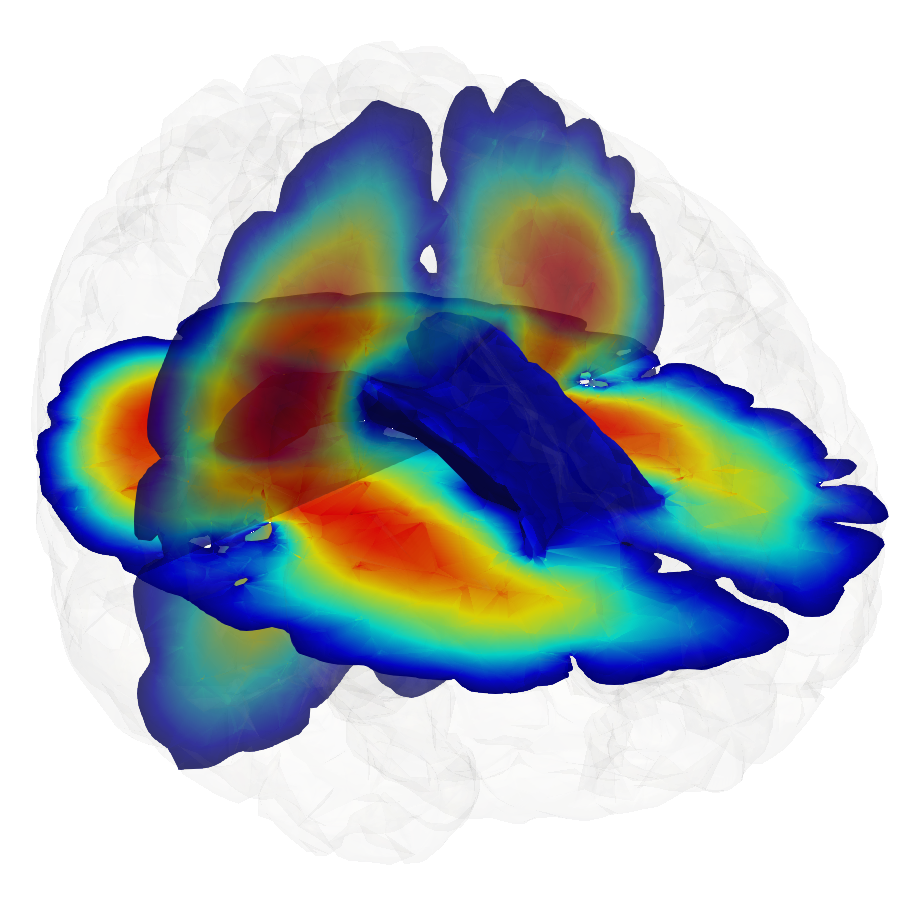}};
          \end{tikzpicture}
        \caption{Venous pressure (mmHg)}
        \label{fig:brain_veins}
    \end{subfigure}
    \begin{subfigure}{0.45\textwidth}
        \centering
        \begin{tikzpicture}
            \node[anchor=south west] (bar1) at (1,-0.5) {
                \pgfplotscolorbardrawstandalone[colormap/jet,colorbar sampled,colorbar horizontal,point meta min=37.973,point meta max=37.975, 
                colorbar style={samples=17, width=4cm, height = 0.2cm, xtick={37.973, 37.975}, 
                scaled x ticks=false, 
                xticklabel style={style={font=\footnotesize}, /pgf/number format/fixed,  /pgf/number format/precision=3}}]};
            \node[anchor=south west] (numex1) at (0,0)  {\includegraphics[width=0.8\textwidth]{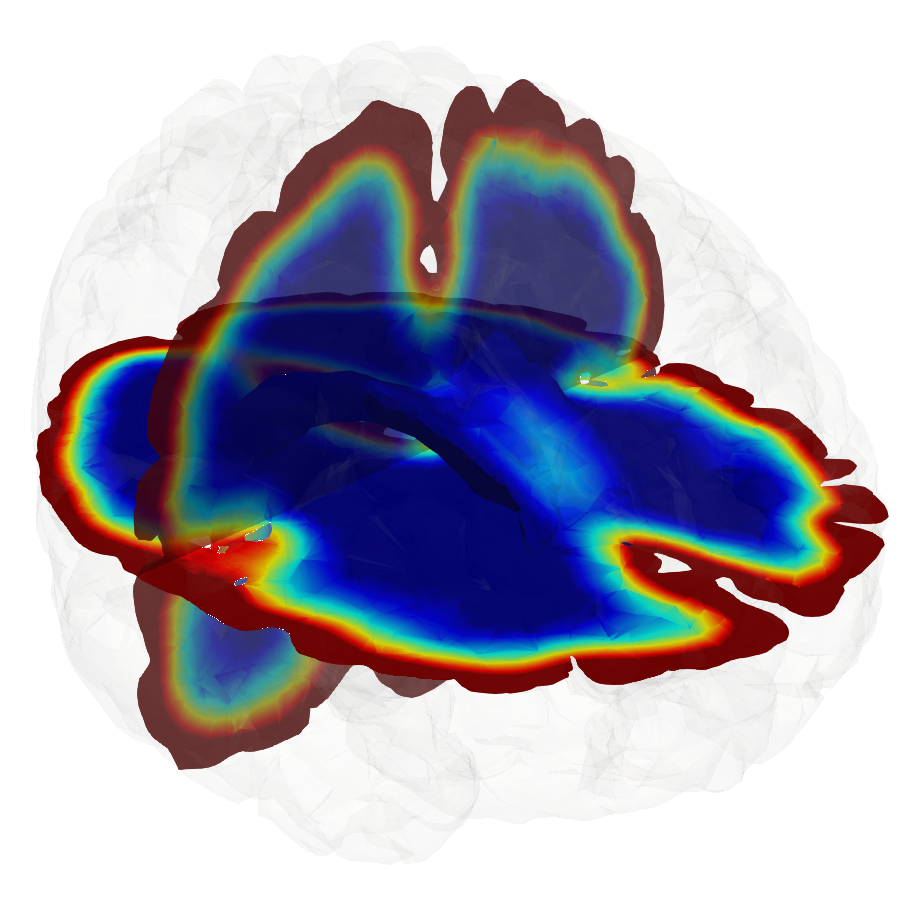}};
          \end{tikzpicture}
        \caption{Capillary pressure (mmHg)}
        \label{fig:brain_cap}
    \end{subfigure}
    \caption{Triangulation of the human brain and the absolute value
    of the displacement and pressure distributions for $t=0.25$.}
    \label{fig::brain}
\end{figure}

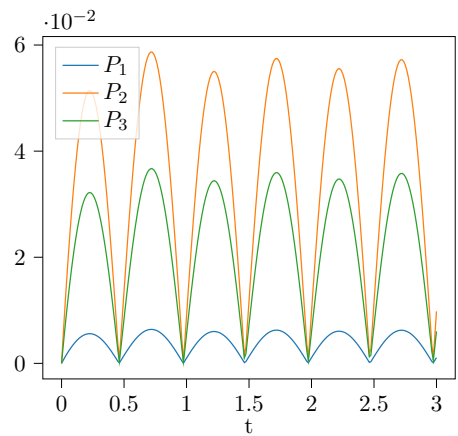
\begin{figure}
        \centering
        \begin{tikzpicture}[scale = 0.8]
            \input{./suppl_tex/abs_u_T3}     
        \end{tikzpicture} 
        \caption{Evaluation of $| \bm u|$ at $P_1$, $P_2$ and $P_3$ (mm).}
        \label{fig::eval_absu}
\end{figure}

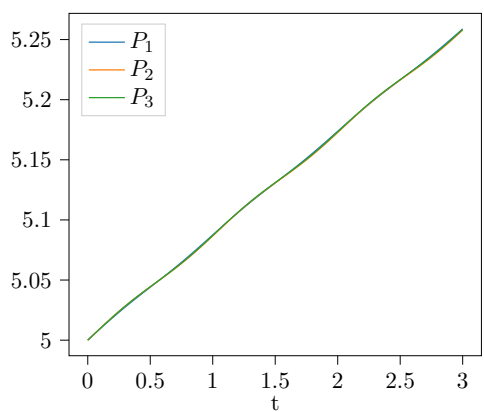
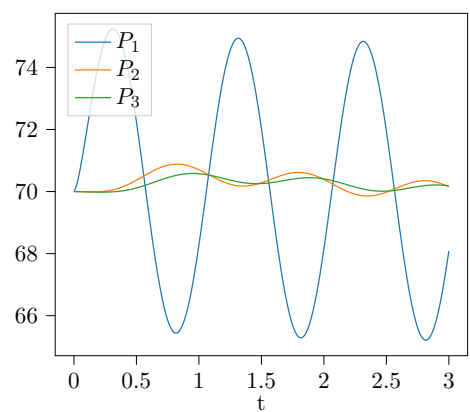
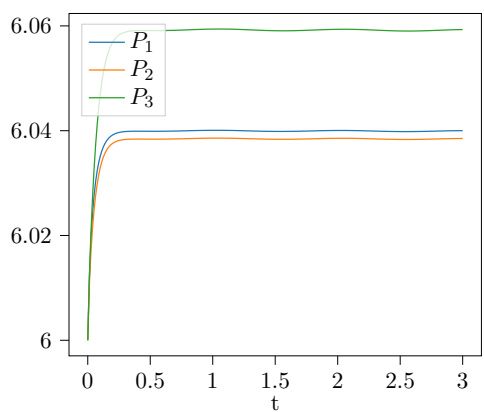
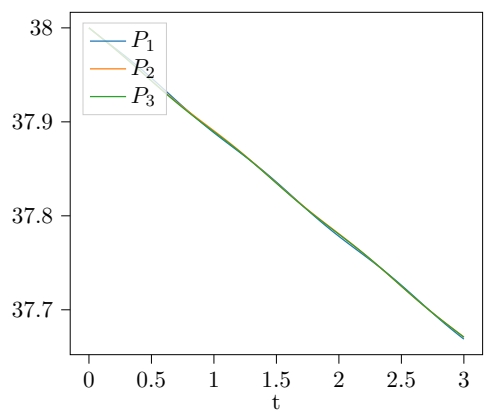
\begin{figure}
    \begin{subfigure}{0.45\textwidth}
        \centering
        \begin{tikzpicture}[scale = 0.8]
            \input{./suppl_tex/p1_T3}     
        \end{tikzpicture}     
        \caption{Evaluation of the extra-cellular pressure}
        \label{fig:eval_pone}
    \end{subfigure}
    \begin{subfigure}{0.45\textwidth}
        \centering
        \begin{tikzpicture}[scale = 0.8]
            \input{./suppl_tex/p2_T3}     
        \end{tikzpicture}     
        \caption{Evaluation of the arterial pressure}
        \label{fig:eval_ptwo}
    \end{subfigure}
    \par\bigskip 
    \begin{subfigure}{0.45\textwidth}
        \centering
        \begin{tikzpicture}[scale = 0.8]
            \input{./suppl_tex/p3_T3}     
        \end{tikzpicture}     
        \caption{Evaluation of the venous pressure}
        \label{fig:eval_pthree}
    \end{subfigure}
    \begin{subfigure}{0.45\textwidth}
        \centering
        \begin{tikzpicture}[scale = 0.8]
            \input{./suppl_tex/p4_T3}     
        \end{tikzpicture}     
        \caption{Evaluation of the capillary pressure}
        \label{fig:eval_pfour}
    \end{subfigure}
    \caption{Evaluation of $p_1,p_2,p_3,p_4$ at $P_1$, $P_2$ and $P_3$ (mmHg).}
    \label{fig::brainevals_pres}
\end{figure}

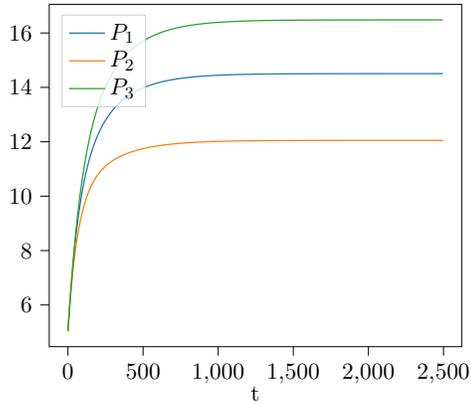
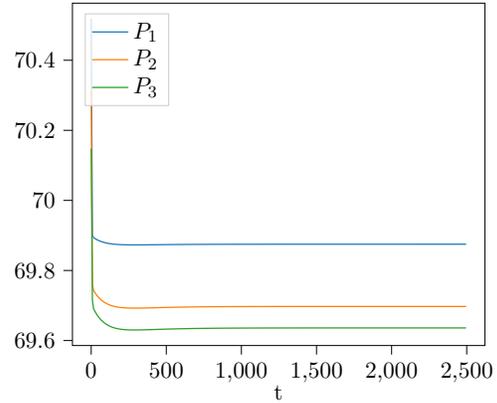
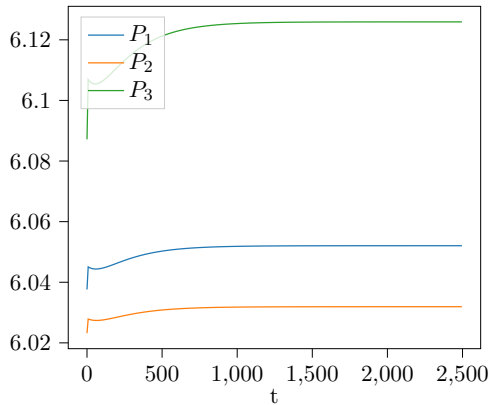
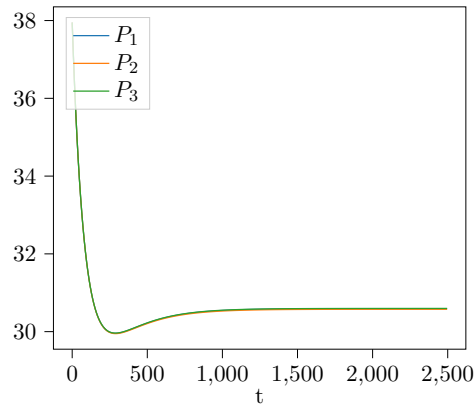
\begin{figure}
    \begin{subfigure}{0.45\textwidth}
        \centering
        \begin{tikzpicture}[scale = 0.8]
            \input{./suppl_tex/p1_mean}     
        \end{tikzpicture}     
        \caption{Time-mean value of the extra-cellular pressure}
        \label{fig:eval_pone_mean}
    \end{subfigure}
    \begin{subfigure}{0.45\textwidth}
        \centering
        \begin{tikzpicture}[scale = 0.8]
            \input{./suppl_tex/p2_mean}     
        \end{tikzpicture}     
        \caption{Time-mean value of the arterial pressure}
        \label{fig:eval_ptwo_mean}
    \end{subfigure}
    \par\bigskip 
    \begin{subfigure}{0.45\textwidth}
        \centering
        \begin{tikzpicture}[scale = 0.8]
            \input{./suppl_tex/p3_mean}     
        \end{tikzpicture}     
        \caption{Time-mean value of the venous pressure}
        \label{fig:eval_pthree_mean}
    \end{subfigure}
    \begin{subfigure}{0.45\textwidth}
        \centering
        \begin{tikzpicture}[scale = 0.8]
            \input{./suppl_tex/p4_mean}     
        \end{tikzpicture}     
        \caption{Time-mean value of the capillary pressure}
        \label{fig:eval_pfour_mean}
    \end{subfigure}
    \caption{Evaluation of the mean values $\langle p_1 \rangle, \langle p_2 \rangle, \langle p_3  \rangle, \langle p_4 \rangle$ at $P_1$, $P_2$ and $P_3$ (mmHg).}
    \label{fig::brainevals_pres_mean}
\end{figure}



\begin{table}[h!]
    \centering
    \begin{tabular}{ccc}
        \hline
        parameter & value & unit \\[0.0ex] \hline
        $\nu$ & $0.4999$ &  \\ [0.0ex]
        $E$ & $1500$ & Nm$^{-2}$ \\ [0.0ex]
        $s_{1}$ & $3.9 *10^{-4}$ & N$^{-1}$m$^2$ \\ [0.0ex]
        $s_{2}=s_{4}$ & $2.9 *10^{-4}$ & N$^{-1}$m$^2$ \\ [0.0ex]
        $s_{3}$ & $1.5 *10^{-5}$ & N$^{-1}$m$^2$ \\ [0.0ex]
        $\alpha_{1}$  & $0.49$ & \\ [0.0ex]
        $\alpha_{2}=\alpha_{4}$  & $0.25$ & \\ [0.0ex]
        $\alpha_{3}$  & $0.01$ & \\ [0.0ex]
        $K_{1}$ & $1.57*10^{-5}$ & mm$^2$N$^{-1}$m$^2$s$^{-1}$ \\[0.0ex]
        $K_{2}=K_{3}=K_{4}$ & $3.75*10^{-2}$ & mm$^2$N$^{-1}$m$^2$s$^{-1}$ \\[0.0ex]
        $\xi_{13}=\xi_{14}$ & $1.0* 10^{-6}$ & N$^{-1}$m$^2$s$^{-1}$ \\ [0.0ex]
        $\xi_{24}=\xi_{34}$ & $1.0* 10^{-6}$ & N$^{-1}$m$^2$s$^{-1}$ \\ [0.0ex]
        $\xi_{12}=\xi_{23}$ & $0.0$ & N$^{-1}$m$^2$s$^{-1}$ \\
        \hline
    \end{tabular}
    \caption{Reference values of model parameters for the four-network brain model taken from \cite{lee2019spacetime}.}
    \label{parameters_MPET4}
\end{table}




\bibliographystyle{plain}
\bibliography{literature}

\end{document}

%% file: suppl_tex/history_tB.tex
\definecolor{color0}{rgb}{0.12156862745098,0.466666666666667,0.705882352941177}
\definecolor{color1}{rgb}{1,0.498039215686275,0.0549019607843137}
\definecolor{color2}{rgb}{0.172549019607843,0.627450980392157,0.172549019607843}
\definecolor{color3}{rgb}{0.83921568627451,0.152941176470588,0.156862745098039}
\definecolor{color4}{rgb}{0.580392156862745,0.403921568627451,0.741176470588235}

\begin{axis}[
legend cell align={left},
legend style={fill opacity=0.8, draw opacity=1, text opacity=1, draw=white!80!black},
log basis y={10},
tick align=outside,
tick pos=left,
x grid style={white!69.0196078431373!black},
xlabel={it},
xmin=-5, xmax=90,
xtick style={color=black},
y grid style={white!69.0196078431373!black},
ylabel={res},
ymin=1e-08, ymax=10,
ymode=log,
ytick style={color=black}
]
\addplot [semithick, color0]
table {%
0 3.98036160125162
1 3.53823804436918
2 0.088862472417382
3 0.0888624481560273
4 0.0886055477353342
5 0.0717283114696207
6 0.0092900272144864
7 0.00523026953208344
8 0.00522850540878111
9 0.000557477130223696
10 0.000215525296250641
11 5.19999657915572e-06
12 3.48435687851475e-06
13 2.29036358532716e-06
14 1.96877596162519e-06
15 1.11786668049308e-07
16 1.11772342938915e-07
17 1.11650503419546e-07
18 1.55309119706391e-08
};
\addlegendentry{$R = 10^{-0}$}
\addplot [semithick, color1]
table {%
0 3.79169257141942
1 0.630774789658195
2 0.404430656763956
3 0.216347312502744
4 0.216320567820145
5 0.21626154925918
6 0.138233327947131
7 0.132250717547772
8 0.0569547903250029
9 0.0291917486916636
10 0.0291432037920707
11 0.0291414281465025
12 0.025996720518108
13 0.00777189462456833
14 0.00692133033792801
15 0.0027258376844912
16 0.00271565296183207
17 0.00271288596743602
18 0.00175530400386148
19 0.000844814362977791
20 0.00034083404763733
21 0.000340579679946123
22 0.000338909697732772
23 0.000316652213592297
24 0.000118966286989855
25 0.00010632860965975
26 7.55220688585642e-05
27 4.97258413783091e-05
28 4.71473470194448e-05
29 1.7599069412442e-05
30 1.57232062360458e-05
31 6.88165944023037e-06
32 6.85217734341979e-06
33 6.84709107011568e-06
34 4.94884371654176e-06
35 2.7946472214367e-06
36 9.62504260377731e-07
37 9.61418071821485e-07
38 9.57355390145255e-07
39 9.36898981737223e-07
40 2.36386048621694e-07
41 1.90899452972685e-07
42 7.98645570125654e-08
43 7.38553755257666e-08
44 7.34699131168345e-08
45 3.40852448916153e-08
};
\addlegendentry{$R = 10^{-2}$}
\addplot [semithick, color2]
table {%
0 3.84088796974329
1 0.543993656581853
2 0.533612340836472
3 0.127919157293391
4 0.115978752693587
5 0.0659252433129717
6 0.0649610357191124
7 0.0646046811059123
8 0.0645145491305431
9 0.0263028647276731
10 0.0235185187169419
11 0.00791015338417054
12 0.00756018091284132
13 0.00288314322373823
14 0.00288247351545449
15 0.0028789723987354
16 0.00281375023109783
17 0.00118884712447681
18 0.0011785268922603
19 0.000705746948831579
20 0.000688267086659152
21 0.000683239387234735
22 0.000670427057730668
23 0.000610790740089986
24 0.000418644357299362
25 0.000394541505323078
26 0.000213420716878768
27 0.000205803944900913
28 0.000199761784586833
29 0.000169840514164771
30 0.000106640290141951
31 9.8511204247939e-05
32 6.02316180241754e-05
33 5.64006452067797e-05
34 3.30988618703116e-05
35 3.293735288412e-05
36 3.20599735023698e-05
37 3.14019936608857e-05
38 1.98566582071076e-05
39 1.982409332367e-05
40 1.53667332243469e-05
41 1.22783414707039e-05
42 1.01721216016728e-05
43 8.18336001797921e-06
44 8.01031090519141e-06
45 6.19881724700383e-06
46 3.77577978363136e-06
47 2.94980895680035e-06
48 1.95323296805673e-06
49 1.92972816159453e-06
50 1.65858724924691e-06
51 1.52310980674789e-06
52 8.07071364335205e-07
53 6.83702934326528e-07
54 5.65376866000865e-07
55 3.27956804026288e-07
56 3.27867241125275e-07
57 3.27804801018461e-07
58 3.27718754261665e-07
59 1.69429627415468e-07
60 1.68787333295213e-07
61 8.12110760950366e-08
62 8.06930140482928e-08
63 7.56091615778294e-08
64 5.45643762659536e-08
65 3.39984084631584e-08
};
\addlegendentry{$R = 10^{-4}$}
\addplot [semithick, color3]
table {%
0 3.84153047914929
1 0.543320299870342
2 0.534070442079492
3 0.12756680543867
4 0.113496730205698
5 0.0506673288481077
6 0.0358814921602802
7 0.0244672092360252
8 0.0238776688746144
9 0.0101179497041384
10 0.0101078503184426
11 0.00999053689084555
12 0.00946024321408094
13 0.00429519835101362
14 0.00402942564798431
15 0.0016837321842754
16 0.00162017605188397
17 0.0013365611100748
18 0.000846781378345519
19 0.000758683334276212
20 0.000539359730150734
21 0.000503833524094905
22 0.00050011009666848
23 0.000495620710139974
24 0.000334245327228442
25 0.000323152415202232
26 0.000213616286880233
27 0.000206376258804555
28 0.000174238900411884
29 0.000156064821394249
30 0.000111995619671053
31 9.50470723984981e-05
32 5.63074354383586e-05
33 5.51598928022553e-05
34 5.29320506696536e-05
35 5.03712074846162e-05
36 3.61202724162992e-05
37 3.06406939639477e-05
38 3.06148129182283e-05
39 1.87192634144222e-05
40 1.87141492082457e-05
41 1.09963654953569e-05
42 1.09078843064919e-05
43 6.81507334981006e-06
44 6.80760527334398e-06
45 6.02453392159136e-06
46 5.44501444680417e-06
47 5.29197610815455e-06
48 4.37581138772571e-06
49 3.46899603762047e-06
50 2.56835889462629e-06
51 2.16552397157695e-06
52 1.46787184113214e-06
53 1.17906298554698e-06
54 7.31423667523868e-07
55 6.09796908266709e-07
56 4.79521022110278e-07
57 4.74313588502579e-07
58 4.73395293683795e-07
59 4.59532031590401e-07
60 2.69779494851109e-07
61 2.63284200201091e-07
62 1.46670108122127e-07
63 1.44174115306777e-07
64 7.8081089367219e-08
65 7.54017555383488e-08
66 5.17959337313778e-08
67 4.06469343230675e-08
68 3.90227166238548e-08
69 3.82190570942684e-08
};
\addlegendentry{$R = 10^{-6}$}
\addplot [semithick, color4]
table {%
0 3.84153693039633
1 0.543313660300645
2 0.534074899859852
3 0.12756472133267
4 0.113471889712839
5 0.0503915815639694
6 0.0341801124730529
7 0.01854077597942
8 0.0132146468047648
9 0.00966791087798666
10 0.00937328725287826
11 0.0043358562590216
12 0.00407455846731819
13 0.00256978931268976
14 0.00165048144587941
15 0.00149728919421374
16 0.000858792112725612
17 0.000767199102783942
18 0.000617482653085374
19 0.000609690164718223
20 0.000451028837600105
21 0.000389633930888781
22 0.000305235185530772
23 0.000264462186808855
24 0.000263157300664383
25 0.000166781015176632
26 0.000166763471854124
27 0.000131604513091852
28 0.000109757655786282
29 9.34323526739235e-05
30 5.98398915581564e-05
31 4.79443717634484e-05
32 3.89371472024859e-05
33 3.85455380748509e-05
34 2.55734455092267e-05
35 2.44090157906154e-05
36 1.54751158500179e-05
37 1.48179475676028e-05
38 1.25041714724306e-05
39 8.94731237601869e-06
40 8.73985977263704e-06
41 5.46844128195233e-06
42 5.31093006559177e-06
43 4.3525554058247e-06
44 3.67895569207719e-06
45 2.61465174229403e-06
46 2.14619986536994e-06
47 1.69286577200256e-06
48 1.37742086935583e-06
49 1.37400217785719e-06
50 7.21256256073218e-07
51 7.13305492685771e-07
52 5.63128095722013e-07
53 3.98103167380806e-07
54 3.35457691335641e-07
55 2.25000602874122e-07
56 2.04553421388062e-07
57 1.53522727339424e-07
58 1.41827466950823e-07
59 8.82338257995222e-08
60 7.83093185792362e-08
61 4.90619694503159e-08
62 4.06912923194307e-08
63 4.036699875582e-08
64 2.21880756963116e-08
};
\addlegendentry{$R = 10^{-8}$}
\end{axis}

%% file: suppl_tex/history_B.tex
\definecolor{color0}{rgb}{0.12156862745098,0.466666666666667,0.705882352941177}
\definecolor{color1}{rgb}{1,0.498039215686275,0.0549019607843137}
\definecolor{color2}{rgb}{0.172549019607843,0.627450980392157,0.172549019607843}
\definecolor{color3}{rgb}{0.83921568627451,0.152941176470588,0.156862745098039}
\definecolor{color4}{rgb}{0.580392156862745,0.403921568627451,0.741176470588235}

\begin{axis}[
legend cell align={left},
legend style={fill opacity=0.8, draw opacity=1, text opacity=1, draw=white!80!black},
log basis y={10},
tick align=outside,
tick pos=left,
x grid style={white!69.0196078431373!black},
xlabel={it},
xmin=-5, xmax=90,
xtick style={color=black},
y grid style={white!69.0196078431373!black},
ylabel={res},
ymin=1e-08, ymax=10,
ymode=log,
ytick style={color=black}
]
\addplot [semithick, color0]
table {%
0 3.91022616860613
1 3.32126373868696
2 0.452259755611432
3 0.452259674473251
4 0.450859719954138
5 0.229944917615605
6 0.228506471286795
7 0.168086499682398
8 0.161309878984204
9 0.107820449005908
10 0.0443305177693658
11 0.0226050729013875
12 0.022593277475391
13 0.0225916810313528
14 0.0214620868780804
15 0.0107583836741742
16 0.0107307404582674
17 0.0106647870036063
18 0.00457534670028782
19 0.00203188148663432
20 0.00149854295678233
21 0.00149058376393876
22 0.00146685118287434
23 0.00133242719313854
24 0.0005757897123953
25 0.000575789410609354
26 0.000575758043047168
27 0.000289362582624717
28 0.000145437272499768
29 0.000145395485540835
30 0.000145395470304303
31 0.000136629156422434
32 6.69135357954659e-05
33 4.69886897094832e-05
34 4.63056196154237e-05
35 3.05247567922116e-05
36 1.0830764569684e-05
37 1.0270128653347e-05
38 1.02673193107229e-05
39 1.02642205983631e-05
40 5.04832243981835e-06
41 2.26556458055361e-06
42 2.26538992255353e-06
43 2.2653831080655e-06
44 1.10931482926403e-06
45 1.03720668625025e-06
46 1.03155250111646e-06
47 9.88713088886784e-07
48 6.06007595173802e-07
49 2.90069877368421e-07
50 2.78044364526649e-07
51 2.56142962393161e-07
52 2.52169274712631e-07
53 1.63714584619982e-07
54 8.97972987154245e-08
55 8.93141888202774e-08
56 8.9258312663411e-08
57 8.81413763516401e-08
58 4.99021466912853e-08
59 4.98985658055832e-08
60 4.98961931493484e-08
61 2.07550938003745e-08
};
\addlegendentry{$R = 10^{-0}$}
\addplot [semithick, color1]
table {%
0 3.78520406459911
1 0.623921428675594
2 0.386552169211908
3 0.239529489229733
4 0.233982748961352
5 0.233980381043381
6 0.134600223633049
7 0.132506461831134
8 0.0621019873490338
9 0.0269905536330213
10 0.0269888691038806
11 0.026985527068897
12 0.0269810034494822
13 0.0107457328968004
14 0.00656274704598309
15 0.00550754308442203
16 0.00524496057179493
17 0.00417653374042649
18 0.00211621749873385
19 0.00200693721995708
20 0.00097138619772473
21 0.000969132854762969
22 0.000966205450653146
23 0.000852742694590502
24 0.000550036013592066
25 0.000271568601539776
26 0.000271528648472436
27 0.000271509747723103
28 0.000270745287234277
29 0.000130407335391188
30 8.11856718763696e-05
31 8.075740443949e-05
32 8.03059824335173e-05
33 6.89194207851835e-05
34 3.25517364657888e-05
35 2.60012661889006e-05
36 1.94642253588544e-05
37 1.78124952900059e-05
38 1.65630100620024e-05
39 8.07142857422646e-06
40 8.02336103286123e-06
41 3.29194496166793e-06
42 3.28282831673877e-06
43 3.28053024826703e-06
44 2.4091223142648e-06
45 1.70312244685663e-06
46 7.94867979409119e-07
47 7.94792237288196e-07
48 7.74805397435502e-07
49 7.25811186597819e-07
50 4.27218792124499e-07
51 2.18039066974381e-07
52 2.15521885049391e-07
53 2.10896325980626e-07
54 2.10822062452235e-07
55 1.02302610333318e-07
56 8.10844022671228e-08
57 5.98572542290953e-08
58 5.98489717862842e-08
59 3.64129274791887e-08
};
\addlegendentry{$R = 10^{-2}$}
\addplot [semithick, color2]
table {%
0 3.84073522106826
1 0.543706050138536
2 0.533153771917201
3 0.129272343536704
4 0.118611888747733
5 0.0661010919692274
6 0.0658983986261899
7 0.0658686531207026
8 0.0649009971823936
9 0.0307478431217059
10 0.0239755231852732
11 0.0195512973019526
12 0.00948662418264015
13 0.00942430482183946
14 0.0094202880013343
15 0.00941729136145106
16 0.00533148741614444
17 0.00436137918549095
18 0.00321749838351268
19 0.00198745284029058
20 0.00194049723307802
21 0.001922753904307
22 0.0019222506034911
23 0.00124220825094951
24 0.0011579468275442
25 0.000828950397794808
26 0.000644223533863892
27 0.000592762753164303
28 0.000534405078071786
29 0.000508772253783105
30 0.000448658299571724
31 0.000434191087152512
32 0.000315924825480386
33 0.000255832221834233
34 0.000237429884099385
35 0.000162816629414886
36 0.00016048058341164
37 0.000159498061683873
38 0.000159482601447321
39 0.000103390374706631
40 8.86020467099786e-05
41 7.24235858032574e-05
42 5.38479480200142e-05
43 5.38330008043629e-05
44 5.38327341695359e-05
45 5.38050899535319e-05
46 3.85302139099929e-05
47 3.38975541085011e-05
48 2.85744688650885e-05
49 2.00244741399929e-05
50 2.00243896488117e-05
51 2.0022226839138e-05
52 1.38610793889968e-05
53 1.31250933500536e-05
54 1.18134482522077e-05
55 9.49192798901156e-06
56 6.8989497003075e-06
57 6.85193550485905e-06
58 6.84605893414301e-06
59 4.59174204933934e-06
60 4.56937323376705e-06
61 4.06455181112645e-06
62 3.24413930657698e-06
63 2.12966576520887e-06
64 2.0341617044176e-06
65 2.00900330402663e-06
66 1.28840869444341e-06
67 1.28021358228927e-06
68 1.24073022324459e-06
69 8.55911666041413e-07
70 6.32644802173959e-07
71 5.77778981810104e-07
72 4.8009260572677e-07
73 3.40636802634524e-07
74 3.39864846144446e-07
75 3.39719463283271e-07
76 2.16246292132646e-07
77 1.70835102608307e-07
78 1.40714997377005e-07
79 1.00938438394904e-07
80 9.6213817739457e-08
81 9.62039767298087e-08
82 9.61351022775988e-08
83 6.78217721037709e-08
84 5.60095266709428e-08
85 4.92478032835832e-08
86 3.10899603085694e-08
};
\addlegendentry{$R = 10^{-4}$}
\addplot [semithick, color3]
table {%
0 3.84152889110027
1 0.543317224107367
2 0.534065589613972
3 0.127580216877009
4 0.113525559479547
5 0.050699280757106
6 0.0359790732270121
7 0.024004008961216
8 0.0236421067694528
9 0.0111154598295709
10 0.0109387385359687
11 0.0104500045084412
12 0.00951134429056974
13 0.00811985033178978
14 0.00444445099204283
15 0.00442969835298193
16 0.00296010942174606
17 0.00239124785752699
18 0.00169426423603794
19 0.00113751013884211
20 0.00107077050055279
21 0.00077481953529069
22 0.000761216080815319
23 0.000742194482324129
24 0.000719526020362777
25 0.000560518148166615
26 0.000474433366619439
27 0.000462103478760068
28 0.000369732960506091
29 0.000355223074115278
30 0.000286775916304127
31 0.000231289443668785
32 0.000223454141244016
33 0.000164498614500961
34 0.000160182353312174
35 0.000147271460731093
36 0.000132406712736319
37 0.000124691545185268
38 0.000103240454212511
39 9.87887821514758e-05
40 6.78028819756195e-05
41 6.55894138679329e-05
42 4.57329756123727e-05
43 3.89346733283654e-05
44 3.49173371941009e-05
45 2.59315097870754e-05
46 2.5929937918135e-05
47 2.34212126720443e-05
48 2.13516876153619e-05
49 1.97419272712046e-05
50 1.70081600769883e-05
51 1.55533629674544e-05
52 1.13911111218114e-05
53 1.13231276914804e-05
54 8.62660443541409e-06
55 7.44737606422479e-06
56 6.77828016878974e-06
57 5.0379304290926e-06
58 4.97169953649304e-06
59 3.6857307266167e-06
60 3.66114779921601e-06
61 3.6305058640937e-06
62 3.06691030297813e-06
63 2.67259613540657e-06
64 1.94326590175657e-06
65 1.94323600678924e-06
66 1.41922860825462e-06
67 1.29643322313385e-06
68 1.07961796165982e-06
69 7.77664973011792e-07
70 7.73658430536803e-07
71 4.97906077652655e-07
72 4.79224659223593e-07
73 4.72783365694791e-07
74 4.22026676408676e-07
75 3.49797633642663e-07
76 2.46610137660606e-07
77 2.46584929529758e-07
78 1.70566527335554e-07
79 1.45733405038479e-07
80 1.29679137355278e-07
81 9.04935146237175e-08
82 9.04359417260847e-08
83 5.9778014569464e-08
84 5.41881934474348e-08
85 5.37196695716156e-08
86 5.14480312686546e-08
87 4.33769944369467e-08
88 3.12669427312335e-08
};
\addlegendentry{$R = 10^{-6}$}
\addplot [semithick, color4]
table {%
0 3.84153691450913
1 0.543313629524299
2 0.534074851304436
3 0.127564856150045
4 0.113472180856831
5 0.0503919804021417
6 0.0341816524143654
7 0.0185179060350993
8 0.0130471455231016
9 0.00954035395027584
10 0.00931704627570478
11 0.00471554663147967
12 0.00408127422123031
13 0.00364191793278765
14 0.00190503482939417
15 0.00189676593657556
16 0.00120537933056595
17 0.0011055349114773
18 0.000929560379052107
19 0.000652737966147485
20 0.000638908937338375
21 0.000507442785036551
22 0.000460341651439318
23 0.000394324812695135
24 0.000314704690393896
25 0.000310715593535355
26 0.000239105033254264
27 0.000221154392166555
28 0.000182574270006758
29 0.000145480839909037
30 0.000144190892922278
31 0.000110053051452207
32 9.90236202378397e-05
33 8.37722306739399e-05
34 6.09733960655333e-05
35 6.08706407964681e-05
36 4.11994998350067e-05
37 3.72061670077324e-05
38 3.19089595405386e-05
39 2.45763852478914e-05
40 2.41604949101833e-05
41 1.76125594812829e-05
42 1.63445999804636e-05
43 1.34655366469064e-05
44 1.05006708278316e-05
45 1.04266174946809e-05
46 7.9119340956698e-06
47 7.30220475204885e-06
48 6.26938934995108e-06
49 4.53593954554399e-06
50 4.51651619568596e-06
51 3.34215967187945e-06
52 2.98833417869048e-06
53 2.50378719246645e-06
54 1.83356773257305e-06
55 1.82309909093329e-06
56 1.45888669704297e-06
57 1.28549006803829e-06
58 1.04432410126878e-06
59 7.2212590588704e-07
60 7.11797538113741e-07
61 4.55128205461211e-07
62 4.16884687748244e-07
63 3.18838764114185e-07
64 2.284335671687e-07
65 2.23116392763775e-07
66 1.66752469898488e-07
67 1.46282691715014e-07
68 1.2534368889155e-07
69 8.70317174161454e-08
70 8.70144698896713e-08
71 6.06986346377305e-08
72 5.23474637266938e-08
73 4.50646075858204e-08
74 3.08063987873219e-08
};
\addlegendentry{$R = 10^{-8}$}
\end{axis}

%% file: suppl_tex/table_2.tex
\begin{tabular}{c|ccccc|ccccc|ccccc} 
\toprule\toprule
\multicolumn{16}{c}{$\xi = $\numval{1}} \\
\multicolumn{1}{c}{~}&\multicolumn{5}{c}{$\lambda = $\numval{1}}&\multicolumn{5}{c}{$\lambda = $\numval{10000}}&\multicolumn{5}{c}{$\lambda = $\numval{100000000}}\\
$\alpha_p \!\!\setminus \!\!R^{-1} $ &\numval{1}&\numval{100}&\numval{10000}&\numval{1000000}&\numval{100000000}&\numval{1}&\numval{100}&\numval{10000}&\numval{1000000}&\numval{100000000}&\numval{1}&\numval{100}&\numval{10000}&\numval{1000000}&\numval{100000000}\\
\hline
\numval{1}&{\color{myorange}11},{\color{myblue}10}&{\color{myorange}22},{\color{myblue}22}&{\color{myorange}21},{\color{myblue}21}&{\color{myorange}21},{\color{myblue}21}&{\color{myorange}21},{\color{myblue}21}&{\color{myorange}11},{\color{myblue}10}&{\color{myorange}22},{\color{myblue}22}&{\color{myorange}21},{\color{myblue}21}&{\color{myorange}21},{\color{myblue}21}&{\color{myorange}21},{\color{myblue}21}&{\color{myorange}11},{\color{myblue}10}&{\color{myorange}22},{\color{myblue}22}&{\color{myorange}21},{\color{myblue}21}&{\color{myorange}21},{\color{myblue}21}&{\color{myorange}21},{\color{myblue}21}\\
\numval{0.01}&{\color{myorange}11},{\color{myblue}10}&{\color{myorange}30},{\color{myblue}30}&{\color{myorange}47},{\color{myblue}43}&{\color{myorange}53},{\color{myblue}53}&{\color{myorange}53},{\color{myblue}53}&{\color{myorange}11},{\color{myblue}10}&{\color{myorange}30},{\color{myblue}30}&{\color{myorange}47},{\color{myblue}43}&{\color{myorange}53},{\color{myblue}53}&{\color{myorange}53},{\color{myblue}53}&{\color{myorange}11},{\color{myblue}10}&{\color{myorange}30},{\color{myblue}30}&{\color{myorange}47},{\color{myblue}43}&{\color{myorange}53},{\color{myblue}53}&{\color{myorange}53},{\color{myblue}53}\\
\numval{0.0001}&{\color{myorange}11},{\color{myblue}10}&{\color{myorange}30},{\color{myblue}30}&{\color{myorange}48},{\color{myblue}44}&{\color{myorange}55},{\color{myblue}56}&{\color{myorange}55},{\color{myblue}57}&{\color{myorange}11},{\color{myblue}10}&{\color{myorange}30},{\color{myblue}30}&{\color{myorange}48},{\color{myblue}44}&{\color{myorange}55},{\color{myblue}56}&{\color{myorange}55},{\color{myblue}57}&{\color{myorange}11},{\color{myblue}10}&{\color{myorange}30},{\color{myblue}30}&{\color{myorange}48},{\color{myblue}44}&{\color{myorange}55},{\color{myblue}56}&{\color{myorange}55},{\color{myblue}57}\\
\numval{1e-06}&{\color{myorange}11},{\color{myblue}10}&{\color{myorange}30},{\color{myblue}30}&{\color{myorange}48},{\color{myblue}44}&{\color{myorange}55},{\color{myblue}56}&{\color{myorange}55},{\color{myblue}57}&{\color{myorange}11},{\color{myblue}10}&{\color{myorange}30},{\color{myblue}30}&{\color{myorange}48},{\color{myblue}44}&{\color{myorange}55},{\color{myblue}56}&{\color{myorange}55},{\color{myblue}57}&{\color{myorange}11},{\color{myblue}10}&{\color{myorange}30},{\color{myblue}30}&{\color{myorange}48},{\color{myblue}44}&{\color{myorange}55},{\color{myblue}56}&{\color{myorange}55},{\color{myblue}57}\\
\numval{1e-08}&{\color{myorange}11},{\color{myblue}10}&{\color{myorange}30},{\color{myblue}30}&{\color{myorange}48},{\color{myblue}44}&{\color{myorange}55},{\color{myblue}56}&{\color{myorange}55},{\color{myblue}57}&{\color{myorange}11},{\color{myblue}10}&{\color{myorange}30},{\color{myblue}30}&{\color{myorange}48},{\color{myblue}44}&{\color{myorange}55},{\color{myblue}56}&{\color{myorange}55},{\color{myblue}57}&{\color{myorange}11},{\color{myblue}10}&{\color{myorange}30},{\color{myblue}30}&{\color{myorange}48},{\color{myblue}44}&{\color{myorange}55},{\color{myblue}56}&{\color{myorange}55},{\color{myblue}57}\\
\midrule
\midrule
\rowcolor{mygray}
\multicolumn{16}{c}{$\xi = $\numval{0.01}} \\
\rowcolor{mygray}
\multicolumn{1}{c}{~}&\multicolumn{5}{c}{$\lambda = $\numval{1}}&\multicolumn{5}{c}{$\lambda = $\numval{10000}}&\multicolumn{5}{c}{$\lambda = $\numval{100000000}}\\
\rowcolor{mygray}
$\alpha_p \!\!\setminus \!\!R^{-1} $ &\numval{1}&\numval{100}&\numval{10000}&\numval{1000000}&\numval{100000000}&\numval{1}&\numval{100}&\numval{10000}&\numval{1000000}&\numval{100000000}&\numval{1}&\numval{100}&\numval{10000}&\numval{1000000}&\numval{100000000}\\
\rowcolor{mygray}
\hline
\rowcolor{mygray}
\numval{1}&{\color{myorange}11},{\color{myblue}10}&{\color{myorange}22},{\color{myblue}22}&{\color{myorange}21},{\color{myblue}21}&{\color{myorange}21},{\color{myblue}21}&{\color{myorange}21},{\color{myblue}21}&{\color{myorange}11},{\color{myblue}10}&{\color{myorange}22},{\color{myblue}22}&{\color{myorange}21},{\color{myblue}21}&{\color{myorange}21},{\color{myblue}21}&{\color{myorange}21},{\color{myblue}21}&{\color{myorange}11},{\color{myblue}10}&{\color{myorange}22},{\color{myblue}22}&{\color{myorange}21},{\color{myblue}21}&{\color{myorange}21},{\color{myblue}21}&{\color{myorange}21},{\color{myblue}21}\\
\rowcolor{mygray}
\numval{0.01}&{\color{myorange}13},{\color{myblue}12}&{\color{myorange}30},{\color{myblue}30}&{\color{myorange}48},{\color{myblue}44}&{\color{myorange}54},{\color{myblue}54}&{\color{myorange}53},{\color{myblue}53}&{\color{myorange}13},{\color{myblue}12}&{\color{myorange}30},{\color{myblue}30}&{\color{myorange}48},{\color{myblue}44}&{\color{myorange}54},{\color{myblue}54}&{\color{myorange}53},{\color{myblue}53}&{\color{myorange}13},{\color{myblue}12}&{\color{myorange}30},{\color{myblue}30}&{\color{myorange}48},{\color{myblue}44}&{\color{myorange}54},{\color{myblue}54}&{\color{myorange}53},{\color{myblue}53}\\
\rowcolor{mygray}
\numval{0.0001}&{\color{myorange}13},{\color{myblue}12}&{\color{myorange}30},{\color{myblue}30}&{\color{myorange}48},{\color{myblue}44}&{\color{myorange}56},{\color{myblue}56}&{\color{myorange}55},{\color{myblue}57}&{\color{myorange}13},{\color{myblue}12}&{\color{myorange}30},{\color{myblue}30}&{\color{myorange}48},{\color{myblue}44}&{\color{myorange}56},{\color{myblue}56}&{\color{myorange}55},{\color{myblue}57}&{\color{myorange}13},{\color{myblue}12}&{\color{myorange}30},{\color{myblue}30}&{\color{myorange}48},{\color{myblue}44}&{\color{myorange}56},{\color{myblue}56}&{\color{myorange}55},{\color{myblue}57}\\
\rowcolor{mygray}
\numval{1e-06}&{\color{myorange}13},{\color{myblue}12}&{\color{myorange}30},{\color{myblue}30}&{\color{myorange}48},{\color{myblue}44}&{\color{myorange}56},{\color{myblue}56}&{\color{myorange}55},{\color{myblue}57}&{\color{myorange}13},{\color{myblue}12}&{\color{myorange}30},{\color{myblue}30}&{\color{myorange}48},{\color{myblue}44}&{\color{myorange}56},{\color{myblue}56}&{\color{myorange}55},{\color{myblue}57}&{\color{myorange}13},{\color{myblue}12}&{\color{myorange}30},{\color{myblue}30}&{\color{myorange}48},{\color{myblue}44}&{\color{myorange}56},{\color{myblue}56}&{\color{myorange}55},{\color{myblue}57}\\
\rowcolor{mygray}
\numval{1e-08}&{\color{myorange}13},{\color{myblue}12}&{\color{myorange}30},{\color{myblue}30}&{\color{myorange}48},{\color{myblue}44}&{\color{myorange}56},{\color{myblue}56}&{\color{myorange}55},{\color{myblue}57}&{\color{myorange}13},{\color{myblue}12}&{\color{myorange}30},{\color{myblue}30}&{\color{myorange}48},{\color{myblue}44}&{\color{myorange}56},{\color{myblue}56}&{\color{myorange}55},{\color{myblue}57}&{\color{myorange}13},{\color{myblue}12}&{\color{myorange}30},{\color{myblue}30}&{\color{myorange}48},{\color{myblue}44}&{\color{myorange}56},{\color{myblue}56}&{\color{myorange}55},{\color{myblue}57}\\
\midrule
\midrule
\multicolumn{16}{c}{$\xi = $\numval{0.0001}} \\
\multicolumn{1}{c}{~}&\multicolumn{5}{c}{$\lambda = $\numval{1}}&\multicolumn{5}{c}{$\lambda = $\numval{10000}}&\multicolumn{5}{c}{$\lambda = $\numval{100000000}}\\
$\alpha_p \!\!\setminus \!\!R^{-1} $ &\numval{1}&\numval{100}&\numval{10000}&\numval{1000000}&\numval{100000000}&\numval{1}&\numval{100}&\numval{10000}&\numval{1000000}&\numval{100000000}&\numval{1}&\numval{100}&\numval{10000}&\numval{1000000}&\numval{100000000}\\
\hline
\numval{1}&{\color{myorange}11},{\color{myblue}10}&{\color{myorange}22},{\color{myblue}22}&{\color{myorange}21},{\color{myblue}21}&{\color{myorange}21},{\color{myblue}21}&{\color{myorange}21},{\color{myblue}21}&{\color{myorange}11},{\color{myblue}10}&{\color{myorange}22},{\color{myblue}22}&{\color{myorange}21},{\color{myblue}21}&{\color{myorange}21},{\color{myblue}21}&{\color{myorange}21},{\color{myblue}21}&{\color{myorange}11},{\color{myblue}10}&{\color{myorange}22},{\color{myblue}22}&{\color{myorange}21},{\color{myblue}21}&{\color{myorange}21},{\color{myblue}21}&{\color{myorange}21},{\color{myblue}21}\\
\numval{0.01}&{\color{myorange}13},{\color{myblue}12}&{\color{myorange}30},{\color{myblue}30}&{\color{myorange}48},{\color{myblue}44}&{\color{myorange}54},{\color{myblue}54}&{\color{myorange}53},{\color{myblue}53}&{\color{myorange}13},{\color{myblue}12}&{\color{myorange}30},{\color{myblue}30}&{\color{myorange}48},{\color{myblue}44}&{\color{myorange}54},{\color{myblue}54}&{\color{myorange}53},{\color{myblue}53}&{\color{myorange}13},{\color{myblue}12}&{\color{myorange}30},{\color{myblue}30}&{\color{myorange}48},{\color{myblue}44}&{\color{myorange}54},{\color{myblue}54}&{\color{myorange}53},{\color{myblue}53}\\
\numval{0.0001}&{\color{myorange}13},{\color{myblue}12}&{\color{myorange}34},{\color{myblue}34}&{\color{myorange}48},{\color{myblue}45}&{\color{myorange}57},{\color{myblue}58}&{\color{myorange}56},{\color{myblue}58}&{\color{myorange}13},{\color{myblue}12}&{\color{myorange}34},{\color{myblue}34}&{\color{myorange}48},{\color{myblue}45}&{\color{myorange}57},{\color{myblue}58}&{\color{myorange}56},{\color{myblue}58}&{\color{myorange}13},{\color{myblue}12}&{\color{myorange}34},{\color{myblue}34}&{\color{myorange}48},{\color{myblue}45}&{\color{myorange}57},{\color{myblue}58}&{\color{myorange}56},{\color{myblue}58}\\
\numval{1e-06}&{\color{myorange}15},{\color{myblue}12}&{\color{myorange}34},{\color{myblue}34}&{\color{myorange}48},{\color{myblue}45}&{\color{myorange}57},{\color{myblue}58}&{\color{myorange}56},{\color{myblue}58}&{\color{myorange}15},{\color{myblue}12}&{\color{myorange}34},{\color{myblue}34}&{\color{myorange}48},{\color{myblue}45}&{\color{myorange}57},{\color{myblue}58}&{\color{myorange}56},{\color{myblue}58}&{\color{myorange}15},{\color{myblue}12}&{\color{myorange}34},{\color{myblue}34}&{\color{myorange}48},{\color{myblue}45}&{\color{myorange}57},{\color{myblue}58}&{\color{myorange}56},{\color{myblue}58}\\
\numval{1e-08}&{\color{myorange}13},{\color{myblue}14}&{\color{myorange}34},{\color{myblue}34}&{\color{myorange}48},{\color{myblue}45}&{\color{myorange}57},{\color{myblue}58}&{\color{myorange}56},{\color{myblue}58}&{\color{myorange}13},{\color{myblue}14}&{\color{myorange}34},{\color{myblue}34}&{\color{myorange}48},{\color{myblue}45}&{\color{myorange}57},{\color{myblue}58}&{\color{myorange}56},{\color{myblue}58}&{\color{myorange}13},{\color{myblue}14}&{\color{myorange}34},{\color{myblue}34}&{\color{myorange}48},{\color{myblue}45}&{\color{myorange}57},{\color{myblue}58}&{\color{myorange}56},{\color{myblue}58}\\
\midrule
\midrule
\rowcolor{mygray}
\multicolumn{16}{c}{$\xi = $\numval{1e-06}} \\
\rowcolor{mygray}
\multicolumn{1}{c}{~}&\multicolumn{5}{c}{$\lambda = $\numval{1}}&\multicolumn{5}{c}{$\lambda = $\numval{10000}}&\multicolumn{5}{c}{$\lambda = $\numval{100000000}}\\
\rowcolor{mygray}
$\alpha_p \!\!\setminus \!\!R^{-1} $ &\numval{1}&\numval{100}&\numval{10000}&\numval{1000000}&\numval{100000000}&\numval{1}&\numval{100}&\numval{10000}&\numval{1000000}&\numval{100000000}&\numval{1}&\numval{100}&\numval{10000}&\numval{1000000}&\numval{100000000}\\
\rowcolor{mygray}
\hline
\rowcolor{mygray}
\numval{1}&{\color{myorange}11},{\color{myblue}10}&{\color{myorange}22},{\color{myblue}22}&{\color{myorange}21},{\color{myblue}21}&{\color{myorange}21},{\color{myblue}21}&{\color{myorange}21},{\color{myblue}21}&{\color{myorange}11},{\color{myblue}10}&{\color{myorange}22},{\color{myblue}22}&{\color{myorange}21},{\color{myblue}21}&{\color{myorange}21},{\color{myblue}21}&{\color{myorange}21},{\color{myblue}21}&{\color{myorange}11},{\color{myblue}10}&{\color{myorange}22},{\color{myblue}22}&{\color{myorange}21},{\color{myblue}21}&{\color{myorange}21},{\color{myblue}21}&{\color{myorange}21},{\color{myblue}21}\\
\rowcolor{mygray}
\numval{0.01}&{\color{myorange}13},{\color{myblue}12}&{\color{myorange}31},{\color{myblue}30}&{\color{myorange}48},{\color{myblue}44}&{\color{myorange}54},{\color{myblue}54}&{\color{myorange}53},{\color{myblue}53}&{\color{myorange}13},{\color{myblue}12}&{\color{myorange}31},{\color{myblue}30}&{\color{myorange}48},{\color{myblue}44}&{\color{myorange}54},{\color{myblue}54}&{\color{myorange}53},{\color{myblue}53}&{\color{myorange}13},{\color{myblue}12}&{\color{myorange}31},{\color{myblue}30}&{\color{myorange}48},{\color{myblue}44}&{\color{myorange}54},{\color{myblue}54}&{\color{myorange}53},{\color{myblue}53}\\
\rowcolor{mygray}
\numval{0.0001}&{\color{myorange}15},{\color{myblue}14}&{\color{myorange}35},{\color{myblue}36}&{\color{myorange}49},{\color{myblue}45}&{\color{myorange}57},{\color{myblue}58}&{\color{myorange}56},{\color{myblue}58}&{\color{myorange}15},{\color{myblue}14}&{\color{myorange}35},{\color{myblue}36}&{\color{myorange}49},{\color{myblue}45}&{\color{myorange}57},{\color{myblue}58}&{\color{myorange}56},{\color{myblue}58}&{\color{myorange}15},{\color{myblue}14}&{\color{myorange}35},{\color{myblue}36}&{\color{myorange}49},{\color{myblue}45}&{\color{myorange}57},{\color{myblue}58}&{\color{myorange}56},{\color{myblue}58}\\
\rowcolor{mygray}
\numval{1e-06}&{\color{myorange}15},{\color{myblue}14}&{\color{myorange}39},{\color{myblue}38}&{\color{myorange}54},{\color{myblue}50}&{\color{myorange}61},{\color{myblue}63}&{\color{myorange}59},{\color{myblue}59}&{\color{myorange}15},{\color{myblue}14}&{\color{myorange}39},{\color{myblue}38}&{\color{myorange}54},{\color{myblue}50}&{\color{myorange}61},{\color{myblue}63}&{\color{myorange}59},{\color{myblue}59}&{\color{myorange}15},{\color{myblue}14}&{\color{myorange}39},{\color{myblue}38}&{\color{myorange}54},{\color{myblue}50}&{\color{myorange}61},{\color{myblue}63}&{\color{myorange}59},{\color{myblue}59}\\
\rowcolor{mygray}
\numval{1e-08}&{\color{myorange}15},{\color{myblue}14}&{\color{myorange}40},{\color{myblue}40}&{\color{myorange}56},{\color{myblue}51}&{\color{myorange}61},{\color{myblue}62}&{\color{myorange}59},{\color{myblue}59}&{\color{myorange}15},{\color{myblue}14}&{\color{myorange}40},{\color{myblue}40}&{\color{myorange}56},{\color{myblue}51}&{\color{myorange}61},{\color{myblue}62}&{\color{myorange}59},{\color{myblue}59}&{\color{myorange}15},{\color{myblue}14}&{\color{myorange}40},{\color{myblue}40}&{\color{myorange}56},{\color{myblue}51}&{\color{myorange}61},{\color{myblue}62}&{\color{myorange}59},{\color{myblue}59}\\
\midrule
\midrule
\multicolumn{16}{c}{$\xi = $\numval{1e-08}} \\
\multicolumn{1}{c}{~}&\multicolumn{5}{c}{$\lambda = $\numval{1}}&\multicolumn{5}{c}{$\lambda = $\numval{10000}}&\multicolumn{5}{c}{$\lambda = $\numval{100000000}}\\
$\alpha_p \!\!\setminus \!\!R^{-1} $ &\numval{1}&\numval{100}&\numval{10000}&\numval{1000000}&\numval{100000000}&\numval{1}&\numval{100}&\numval{10000}&\numval{1000000}&\numval{100000000}&\numval{1}&\numval{100}&\numval{10000}&\numval{1000000}&\numval{100000000}\\
\hline
\numval{1}&{\color{myorange}11},{\color{myblue}10}&{\color{myorange}22},{\color{myblue}22}&{\color{myorange}21},{\color{myblue}21}&{\color{myorange}21},{\color{myblue}21}&{\color{myorange}21},{\color{myblue}21}&{\color{myorange}11},{\color{myblue}10}&{\color{myorange}22},{\color{myblue}22}&{\color{myorange}21},{\color{myblue}21}&{\color{myorange}21},{\color{myblue}21}&{\color{myorange}21},{\color{myblue}21}&{\color{myorange}11},{\color{myblue}10}&{\color{myorange}22},{\color{myblue}22}&{\color{myorange}21},{\color{myblue}21}&{\color{myorange}21},{\color{myblue}21}&{\color{myorange}21},{\color{myblue}21}\\
\numval{0.01}&{\color{myorange}13},{\color{myblue}12}&{\color{myorange}31},{\color{myblue}30}&{\color{myorange}48},{\color{myblue}44}&{\color{myorange}54},{\color{myblue}54}&{\color{myorange}53},{\color{myblue}53}&{\color{myorange}13},{\color{myblue}12}&{\color{myorange}31},{\color{myblue}30}&{\color{myorange}48},{\color{myblue}44}&{\color{myorange}54},{\color{myblue}54}&{\color{myorange}53},{\color{myblue}53}&{\color{myorange}13},{\color{myblue}12}&{\color{myorange}31},{\color{myblue}30}&{\color{myorange}48},{\color{myblue}44}&{\color{myorange}54},{\color{myblue}54}&{\color{myorange}53},{\color{myblue}53}\\
\numval{0.0001}&{\color{myorange}15},{\color{myblue}14}&{\color{myorange}36},{\color{myblue}34}&{\color{myorange}49},{\color{myblue}45}&{\color{myorange}57},{\color{myblue}58}&{\color{myorange}56},{\color{myblue}58}&{\color{myorange}15},{\color{myblue}14}&{\color{myorange}36},{\color{myblue}34}&{\color{myorange}49},{\color{myblue}45}&{\color{myorange}57},{\color{myblue}58}&{\color{myorange}56},{\color{myblue}58}&{\color{myorange}15},{\color{myblue}14}&{\color{myorange}36},{\color{myblue}34}&{\color{myorange}49},{\color{myblue}45}&{\color{myorange}57},{\color{myblue}58}&{\color{myorange}56},{\color{myblue}58}\\
\numval{1e-06}&{\color{myorange}17},{\color{myblue}14}&{\color{myorange}40},{\color{myblue}40}&{\color{myorange}56},{\color{myblue}52}&{\color{myorange}61},{\color{myblue}61}&{\color{myorange}59},{\color{myblue}59}&{\color{myorange}17},{\color{myblue}14}&{\color{myorange}40},{\color{myblue}40}&{\color{myorange}56},{\color{myblue}52}&{\color{myorange}61},{\color{myblue}61}&{\color{myorange}59},{\color{myblue}59}&{\color{myorange}17},{\color{myblue}14}&{\color{myorange}40},{\color{myblue}40}&{\color{myorange}56},{\color{myblue}52}&{\color{myorange}61},{\color{myblue}61}&{\color{myorange}59},{\color{myblue}59}\\
\numval{1e-08}&{\color{myorange}19},{\color{myblue}16}&{\color{myorange}44},{\color{myblue}45}&{\color{myorange}63},{\color{myblue}58}&{\color{myorange}66},{\color{myblue}67}&{\color{myorange}63},{\color{myblue}64}&{\color{myorange}19},{\color{myblue}16}&{\color{myorange}44},{\color{myblue}45}&{\color{myorange}63},{\color{myblue}58}&{\color{myorange}66},{\color{myblue}67}&{\color{myorange}63},{\color{myblue}64}&{\color{myorange}19},{\color{myblue}16}&{\color{myorange}44},{\color{myblue}45}&{\color{myorange}63},{\color{myblue}58}&{\color{myorange}66},{\color{myblue}67}&{\color{myorange}63},{\color{myblue}64}\\
\bottomrule\bottomrule\end{tabular}

%% file: suppl_tex/table_mixed_2.tex
\begin{tabular}{c|ccccc|ccccc|ccccc} 
\toprule\toprule
\multicolumn{16}{c}{$\xi = $\numval{1}} \\
\multicolumn{1}{c}{~}&\multicolumn{5}{c}{$\lambda = $\numval{1}}&\multicolumn{5}{c}{$\lambda = $\numval{10000}}&\multicolumn{5}{c}{$\lambda = $\numval{100000000}}\\
$\alpha_{p_2} \!\!\setminus \!\!R_2^{-1} $ &\numval{1}&\numval{100}&\numval{10000}&\numval{1000000}&\numval{100000000}&\numval{1}&\numval{100}&\numval{10000}&\numval{1000000}&\numval{100000000}&\numval{1}&\numval{100}&\numval{10000}&\numval{1000000}&\numval{100000000}\\
\hline
\numval{1}&{\color{myorange}18},{\color{myblue}18}&{\color{myorange}25},{\color{myblue}26}&{\color{myorange}27},{\color{myblue}27}&{\color{myorange}27},{\color{myblue}27}&{\color{myorange}27},{\color{myblue}27}&{\color{myorange}18},{\color{myblue}18}&{\color{myorange}25},{\color{myblue}26}&{\color{myorange}27},{\color{myblue}27}&{\color{myorange}27},{\color{myblue}27}&{\color{myorange}27},{\color{myblue}27}&{\color{myorange}18},{\color{myblue}18}&{\color{myorange}25},{\color{myblue}26}&{\color{myorange}27},{\color{myblue}27}&{\color{myorange}27},{\color{myblue}27}&{\color{myorange}27},{\color{myblue}27}\\
\numval{0.01}&{\color{myorange}18},{\color{myblue}18}&{\color{myorange}31},{\color{myblue}31}&{\color{myorange}47},{\color{myblue}43}&{\color{myorange}50},{\color{myblue}47}&{\color{myorange}50},{\color{myblue}47}&{\color{myorange}18},{\color{myblue}18}&{\color{myorange}31},{\color{myblue}31}&{\color{myorange}47},{\color{myblue}43}&{\color{myorange}50},{\color{myblue}47}&{\color{myorange}50},{\color{myblue}47}&{\color{myorange}18},{\color{myblue}18}&{\color{myorange}31},{\color{myblue}31}&{\color{myorange}47},{\color{myblue}43}&{\color{myorange}50},{\color{myblue}47}&{\color{myorange}50},{\color{myblue}47}\\
\numval{0.0001}&{\color{myorange}18},{\color{myblue}18}&{\color{myorange}31},{\color{myblue}31}&{\color{myorange}48},{\color{myblue}44}&{\color{myorange}51},{\color{myblue}48}&{\color{myorange}51},{\color{myblue}48}&{\color{myorange}18},{\color{myblue}18}&{\color{myorange}31},{\color{myblue}31}&{\color{myorange}48},{\color{myblue}44}&{\color{myorange}51},{\color{myblue}48}&{\color{myorange}51},{\color{myblue}48}&{\color{myorange}18},{\color{myblue}18}&{\color{myorange}31},{\color{myblue}31}&{\color{myorange}48},{\color{myblue}44}&{\color{myorange}51},{\color{myblue}48}&{\color{myorange}51},{\color{myblue}48}\\
\numval{1e-06}&{\color{myorange}18},{\color{myblue}18}&{\color{myorange}31},{\color{myblue}31}&{\color{myorange}48},{\color{myblue}44}&{\color{myorange}51},{\color{myblue}48}&{\color{myorange}51},{\color{myblue}48}&{\color{myorange}18},{\color{myblue}18}&{\color{myorange}31},{\color{myblue}31}&{\color{myorange}48},{\color{myblue}44}&{\color{myorange}51},{\color{myblue}48}&{\color{myorange}51},{\color{myblue}48}&{\color{myorange}18},{\color{myblue}18}&{\color{myorange}31},{\color{myblue}31}&{\color{myorange}48},{\color{myblue}44}&{\color{myorange}51},{\color{myblue}48}&{\color{myorange}51},{\color{myblue}48}\\
\numval{1e-08}&{\color{myorange}18},{\color{myblue}18}&{\color{myorange}31},{\color{myblue}31}&{\color{myorange}48},{\color{myblue}44}&{\color{myorange}51},{\color{myblue}48}&{\color{myorange}51},{\color{myblue}48}&{\color{myorange}18},{\color{myblue}18}&{\color{myorange}31},{\color{myblue}31}&{\color{myorange}48},{\color{myblue}44}&{\color{myorange}51},{\color{myblue}48}&{\color{myorange}51},{\color{myblue}48}&{\color{myorange}18},{\color{myblue}18}&{\color{myorange}31},{\color{myblue}31}&{\color{myorange}48},{\color{myblue}44}&{\color{myorange}51},{\color{myblue}48}&{\color{myorange}51},{\color{myblue}48}\\
\midrule
\midrule
\rowcolor{mygray}
\multicolumn{16}{c}{$\xi = $\numval{0.01}} \\
\rowcolor{mygray}
\multicolumn{1}{c}{~}&\multicolumn{5}{c}{$\lambda = $\numval{1}}&\multicolumn{5}{c}{$\lambda = $\numval{10000}}&\multicolumn{5}{c}{$\lambda = $\numval{100000000}}\\
\rowcolor{mygray}
$\alpha_{p_2} \!\!\setminus \!\!R_2^{-1} $ &\numval{1}&\numval{100}&\numval{10000}&\numval{1000000}&\numval{100000000}&\numval{1}&\numval{100}&\numval{10000}&\numval{1000000}&\numval{100000000}&\numval{1}&\numval{100}&\numval{10000}&\numval{1000000}&\numval{100000000}\\
\rowcolor{mygray}
\hline
\rowcolor{mygray}
\numval{1}&{\color{myorange}42},{\color{myblue}41}&{\color{myorange}43},{\color{myblue}41}&{\color{myorange}43},{\color{myblue}40}&{\color{myorange}43},{\color{myblue}39}&{\color{myorange}43},{\color{myblue}39}&{\color{myorange}42},{\color{myblue}41}&{\color{myorange}43},{\color{myblue}41}&{\color{myorange}43},{\color{myblue}40}&{\color{myorange}43},{\color{myblue}39}&{\color{myorange}43},{\color{myblue}39}&{\color{myorange}42},{\color{myblue}41}&{\color{myorange}43},{\color{myblue}41}&{\color{myorange}43},{\color{myblue}40}&{\color{myorange}43},{\color{myblue}39}&{\color{myorange}43},{\color{myblue}39}\\
\rowcolor{mygray}
\numval{0.01}&{\color{myorange}46},{\color{myblue}45}&{\color{myorange}44},{\color{myblue}41}&{\color{myorange}48},{\color{myblue}44}&{\color{myorange}52},{\color{myblue}52}&{\color{myorange}52},{\color{myblue}52}&{\color{myorange}46},{\color{myblue}45}&{\color{myorange}44},{\color{myblue}41}&{\color{myorange}48},{\color{myblue}44}&{\color{myorange}52},{\color{myblue}52}&{\color{myorange}52},{\color{myblue}52}&{\color{myorange}46},{\color{myblue}45}&{\color{myorange}44},{\color{myblue}41}&{\color{myorange}48},{\color{myblue}44}&{\color{myorange}52},{\color{myblue}52}&{\color{myorange}52},{\color{myblue}52}\\
\rowcolor{mygray}
\numval{0.0001}&{\color{myorange}47},{\color{myblue}45}&{\color{myorange}44},{\color{myblue}41}&{\color{myorange}48},{\color{myblue}44}&{\color{myorange}54},{\color{myblue}54}&{\color{myorange}54},{\color{myblue}54}&{\color{myorange}47},{\color{myblue}45}&{\color{myorange}44},{\color{myblue}41}&{\color{myorange}48},{\color{myblue}44}&{\color{myorange}54},{\color{myblue}54}&{\color{myorange}54},{\color{myblue}54}&{\color{myorange}47},{\color{myblue}45}&{\color{myorange}44},{\color{myblue}41}&{\color{myorange}48},{\color{myblue}44}&{\color{myorange}54},{\color{myblue}54}&{\color{myorange}54},{\color{myblue}54}\\
\rowcolor{mygray}
\numval{1e-06}&{\color{myorange}47},{\color{myblue}45}&{\color{myorange}44},{\color{myblue}41}&{\color{myorange}48},{\color{myblue}44}&{\color{myorange}54},{\color{myblue}54}&{\color{myorange}54},{\color{myblue}54}&{\color{myorange}47},{\color{myblue}45}&{\color{myorange}44},{\color{myblue}41}&{\color{myorange}48},{\color{myblue}44}&{\color{myorange}54},{\color{myblue}54}&{\color{myorange}54},{\color{myblue}54}&{\color{myorange}47},{\color{myblue}45}&{\color{myorange}44},{\color{myblue}41}&{\color{myorange}48},{\color{myblue}44}&{\color{myorange}54},{\color{myblue}54}&{\color{myorange}54},{\color{myblue}54}\\
\rowcolor{mygray}
\numval{1e-08}&{\color{myorange}46},{\color{myblue}45}&{\color{myorange}44},{\color{myblue}41}&{\color{myorange}48},{\color{myblue}44}&{\color{myorange}54},{\color{myblue}54}&{\color{myorange}54},{\color{myblue}54}&{\color{myorange}46},{\color{myblue}45}&{\color{myorange}44},{\color{myblue}41}&{\color{myorange}48},{\color{myblue}44}&{\color{myorange}54},{\color{myblue}54}&{\color{myorange}54},{\color{myblue}54}&{\color{myorange}46},{\color{myblue}45}&{\color{myorange}44},{\color{myblue}41}&{\color{myorange}48},{\color{myblue}44}&{\color{myorange}54},{\color{myblue}54}&{\color{myorange}54},{\color{myblue}54}\\
\midrule
\midrule
\multicolumn{16}{c}{$\xi = $\numval{0.0001}} \\
\multicolumn{1}{c}{~}&\multicolumn{5}{c}{$\lambda = $\numval{1}}&\multicolumn{5}{c}{$\lambda = $\numval{10000}}&\multicolumn{5}{c}{$\lambda = $\numval{100000000}}\\
$\alpha_{p_2} \!\!\setminus \!\!R_2^{-1} $ &\numval{1}&\numval{100}&\numval{10000}&\numval{1000000}&\numval{100000000}&\numval{1}&\numval{100}&\numval{10000}&\numval{1000000}&\numval{100000000}&\numval{1}&\numval{100}&\numval{10000}&\numval{1000000}&\numval{100000000}\\
\hline
\numval{1}&{\color{myorange}45},{\color{myblue}41}&{\color{myorange}44},{\color{myblue}41}&{\color{myorange}45},{\color{myblue}40}&{\color{myorange}45},{\color{myblue}40}&{\color{myorange}45},{\color{myblue}40}&{\color{myorange}45},{\color{myblue}41}&{\color{myorange}44},{\color{myblue}41}&{\color{myorange}45},{\color{myblue}40}&{\color{myorange}45},{\color{myblue}40}&{\color{myorange}45},{\color{myblue}40}&{\color{myorange}45},{\color{myblue}41}&{\color{myorange}44},{\color{myblue}41}&{\color{myorange}45},{\color{myblue}40}&{\color{myorange}45},{\color{myblue}40}&{\color{myorange}45},{\color{myblue}40}\\
\numval{0.01}&{\color{myorange}51},{\color{myblue}48}&{\color{myorange}46},{\color{myblue}41}&{\color{myorange}48},{\color{myblue}45}&{\color{myorange}53},{\color{myblue}52}&{\color{myorange}54},{\color{myblue}52}&{\color{myorange}51},{\color{myblue}48}&{\color{myorange}46},{\color{myblue}41}&{\color{myorange}48},{\color{myblue}45}&{\color{myorange}53},{\color{myblue}52}&{\color{myorange}54},{\color{myblue}52}&{\color{myorange}51},{\color{myblue}48}&{\color{myorange}46},{\color{myblue}41}&{\color{myorange}48},{\color{myblue}45}&{\color{myorange}53},{\color{myblue}52}&{\color{myorange}54},{\color{myblue}52}\\
\numval{0.0001}&{\color{myorange}55},{\color{myblue}52}&{\color{myorange}49},{\color{myblue}46}&{\color{myorange}48},{\color{myblue}45}&{\color{myorange}57},{\color{myblue}59}&{\color{myorange}57},{\color{myblue}59}&{\color{myorange}55},{\color{myblue}52}&{\color{myorange}49},{\color{myblue}46}&{\color{myorange}48},{\color{myblue}45}&{\color{myorange}57},{\color{myblue}59}&{\color{myorange}57},{\color{myblue}59}&{\color{myorange}55},{\color{myblue}52}&{\color{myorange}49},{\color{myblue}46}&{\color{myorange}48},{\color{myblue}45}&{\color{myorange}57},{\color{myblue}59}&{\color{myorange}57},{\color{myblue}59}\\
\numval{1e-06}&{\color{myorange}57},{\color{myblue}52}&{\color{myorange}49},{\color{myblue}46}&{\color{myorange}48},{\color{myblue}45}&{\color{myorange}57},{\color{myblue}59}&{\color{myorange}59},{\color{myblue}58}&{\color{myorange}57},{\color{myblue}52}&{\color{myorange}49},{\color{myblue}46}&{\color{myorange}48},{\color{myblue}45}&{\color{myorange}57},{\color{myblue}59}&{\color{myorange}59},{\color{myblue}58}&{\color{myorange}57},{\color{myblue}52}&{\color{myorange}49},{\color{myblue}46}&{\color{myorange}48},{\color{myblue}45}&{\color{myorange}57},{\color{myblue}59}&{\color{myorange}59},{\color{myblue}58}\\
\numval{1e-08}&{\color{myorange}55},{\color{myblue}52}&{\color{myorange}49},{\color{myblue}46}&{\color{myorange}48},{\color{myblue}45}&{\color{myorange}57},{\color{myblue}59}&{\color{myorange}57},{\color{myblue}59}&{\color{myorange}55},{\color{myblue}52}&{\color{myorange}49},{\color{myblue}46}&{\color{myorange}48},{\color{myblue}45}&{\color{myorange}57},{\color{myblue}59}&{\color{myorange}57},{\color{myblue}59}&{\color{myorange}55},{\color{myblue}52}&{\color{myorange}49},{\color{myblue}46}&{\color{myorange}48},{\color{myblue}45}&{\color{myorange}57},{\color{myblue}59}&{\color{myorange}57},{\color{myblue}59}\\
\midrule
\midrule
\rowcolor{mygray}
\multicolumn{16}{c}{$\xi = $\numval{1e-06}} \\
\rowcolor{mygray}
\multicolumn{1}{c}{~}&\multicolumn{5}{c}{$\lambda = $\numval{1}}&\multicolumn{5}{c}{$\lambda = $\numval{10000}}&\multicolumn{5}{c}{$\lambda = $\numval{100000000}}\\
\rowcolor{mygray}
$\alpha_{p_2} \!\!\setminus \!\!R_2^{-1} $ &\numval{1}&\numval{100}&\numval{10000}&\numval{1000000}&\numval{100000000}&\numval{1}&\numval{100}&\numval{10000}&\numval{1000000}&\numval{100000000}&\numval{1}&\numval{100}&\numval{10000}&\numval{1000000}&\numval{100000000}\\
\rowcolor{mygray}
\hline
\rowcolor{mygray}
\numval{1}&{\color{myorange}46},{\color{myblue}41}&{\color{myorange}44},{\color{myblue}41}&{\color{myorange}45},{\color{myblue}40}&{\color{myorange}45},{\color{myblue}40}&{\color{myorange}45},{\color{myblue}40}&{\color{myorange}46},{\color{myblue}41}&{\color{myorange}44},{\color{myblue}41}&{\color{myorange}45},{\color{myblue}40}&{\color{myorange}45},{\color{myblue}40}&{\color{myorange}45},{\color{myblue}40}&{\color{myorange}46},{\color{myblue}41}&{\color{myorange}44},{\color{myblue}41}&{\color{myorange}45},{\color{myblue}40}&{\color{myorange}45},{\color{myblue}40}&{\color{myorange}45},{\color{myblue}40}\\
\rowcolor{mygray}
\numval{0.01}&{\color{myorange}51},{\color{myblue}48}&{\color{myorange}47},{\color{myblue}41}&{\color{myorange}48},{\color{myblue}45}&{\color{myorange}53},{\color{myblue}52}&{\color{myorange}54},{\color{myblue}53}&{\color{myorange}51},{\color{myblue}48}&{\color{myorange}47},{\color{myblue}41}&{\color{myorange}48},{\color{myblue}45}&{\color{myorange}53},{\color{myblue}52}&{\color{myorange}54},{\color{myblue}53}&{\color{myorange}51},{\color{myblue}48}&{\color{myorange}47},{\color{myblue}41}&{\color{myorange}48},{\color{myblue}45}&{\color{myorange}53},{\color{myblue}52}&{\color{myorange}54},{\color{myblue}53}\\
\rowcolor{mygray}
\numval{0.0001}&{\color{myorange}57},{\color{myblue}54}&{\color{myorange}51},{\color{myblue}47}&{\color{myorange}49},{\color{myblue}45}&{\color{myorange}60},{\color{myblue}60}&{\color{myorange}60},{\color{myblue}62}&{\color{myorange}57},{\color{myblue}54}&{\color{myorange}51},{\color{myblue}47}&{\color{myorange}49},{\color{myblue}45}&{\color{myorange}60},{\color{myblue}60}&{\color{myorange}60},{\color{myblue}62}&{\color{myorange}57},{\color{myblue}54}&{\color{myorange}51},{\color{myblue}47}&{\color{myorange}49},{\color{myblue}45}&{\color{myorange}60},{\color{myblue}60}&{\color{myorange}60},{\color{myblue}62}\\
\rowcolor{mygray}
\numval{1e-06}&{\color{myorange}58},{\color{myblue}55}&{\color{myorange}52},{\color{myblue}49}&{\color{myorange}49},{\color{myblue}46}&{\color{myorange}60},{\color{myblue}60}&{\color{myorange}60},{\color{myblue}61}&{\color{myorange}58},{\color{myblue}55}&{\color{myorange}52},{\color{myblue}49}&{\color{myorange}49},{\color{myblue}46}&{\color{myorange}60},{\color{myblue}60}&{\color{myorange}60},{\color{myblue}61}&{\color{myorange}58},{\color{myblue}55}&{\color{myorange}52},{\color{myblue}49}&{\color{myorange}49},{\color{myblue}46}&{\color{myorange}60},{\color{myblue}60}&{\color{myorange}60},{\color{myblue}61}\\
\rowcolor{mygray}
\numval{1e-08}&{\color{myorange}59},{\color{myblue}56}&{\color{myorange}52},{\color{myblue}49}&{\color{myorange}49},{\color{myblue}46}&{\color{myorange}60},{\color{myblue}60}&{\color{myorange}60},{\color{myblue}61}&{\color{myorange}59},{\color{myblue}56}&{\color{myorange}52},{\color{myblue}49}&{\color{myorange}49},{\color{myblue}46}&{\color{myorange}60},{\color{myblue}60}&{\color{myorange}60},{\color{myblue}61}&{\color{myorange}59},{\color{myblue}56}&{\color{myorange}52},{\color{myblue}49}&{\color{myorange}49},{\color{myblue}46}&{\color{myorange}60},{\color{myblue}60}&{\color{myorange}60},{\color{myblue}61}\\
\midrule
\midrule
\multicolumn{16}{c}{$\xi = $\numval{1e-08}} \\
\multicolumn{1}{c}{~}&\multicolumn{5}{c}{$\lambda = $\numval{1}}&\multicolumn{5}{c}{$\lambda = $\numval{10000}}&\multicolumn{5}{c}{$\lambda = $\numval{100000000}}\\
$\alpha_{p_2} \!\!\setminus \!\!R_2^{-1} $ &\numval{1}&\numval{100}&\numval{10000}&\numval{1000000}&\numval{100000000}&\numval{1}&\numval{100}&\numval{10000}&\numval{1000000}&\numval{100000000}&\numval{1}&\numval{100}&\numval{10000}&\numval{1000000}&\numval{100000000}\\
\hline
\numval{1}&{\color{myorange}46},{\color{myblue}41}&{\color{myorange}44},{\color{myblue}41}&{\color{myorange}45},{\color{myblue}40}&{\color{myorange}45},{\color{myblue}40}&{\color{myorange}45},{\color{myblue}40}&{\color{myorange}46},{\color{myblue}41}&{\color{myorange}44},{\color{myblue}41}&{\color{myorange}45},{\color{myblue}40}&{\color{myorange}45},{\color{myblue}40}&{\color{myorange}45},{\color{myblue}40}&{\color{myorange}46},{\color{myblue}41}&{\color{myorange}44},{\color{myblue}41}&{\color{myorange}45},{\color{myblue}40}&{\color{myorange}45},{\color{myblue}40}&{\color{myorange}45},{\color{myblue}40}\\
\numval{0.01}&{\color{myorange}51},{\color{myblue}48}&{\color{myorange}46},{\color{myblue}41}&{\color{myorange}48},{\color{myblue}45}&{\color{myorange}53},{\color{myblue}52}&{\color{myorange}54},{\color{myblue}53}&{\color{myorange}51},{\color{myblue}48}&{\color{myorange}46},{\color{myblue}41}&{\color{myorange}48},{\color{myblue}45}&{\color{myorange}53},{\color{myblue}52}&{\color{myorange}54},{\color{myblue}53}&{\color{myorange}51},{\color{myblue}48}&{\color{myorange}46},{\color{myblue}41}&{\color{myorange}48},{\color{myblue}45}&{\color{myorange}53},{\color{myblue}52}&{\color{myorange}54},{\color{myblue}53}\\
\numval{0.0001}&{\color{myorange}58},{\color{myblue}54}&{\color{myorange}51},{\color{myblue}47}&{\color{myorange}49},{\color{myblue}45}&{\color{myorange}60},{\color{myblue}60}&{\color{myorange}60},{\color{myblue}62}&{\color{myorange}58},{\color{myblue}54}&{\color{myorange}51},{\color{myblue}47}&{\color{myorange}49},{\color{myblue}45}&{\color{myorange}60},{\color{myblue}60}&{\color{myorange}60},{\color{myblue}62}&{\color{myorange}58},{\color{myblue}54}&{\color{myorange}51},{\color{myblue}47}&{\color{myorange}49},{\color{myblue}45}&{\color{myorange}60},{\color{myblue}60}&{\color{myorange}60},{\color{myblue}62}\\
\numval{1e-06}&{\color{myorange}59},{\color{myblue}56}&{\color{myorange}52},{\color{myblue}49}&{\color{myorange}49},{\color{myblue}46}&{\color{myorange}60},{\color{myblue}60}&{\color{myorange}60},{\color{myblue}61}&{\color{myorange}59},{\color{myblue}56}&{\color{myorange}52},{\color{myblue}49}&{\color{myorange}49},{\color{myblue}46}&{\color{myorange}60},{\color{myblue}60}&{\color{myorange}60},{\color{myblue}61}&{\color{myorange}59},{\color{myblue}56}&{\color{myorange}52},{\color{myblue}49}&{\color{myorange}49},{\color{myblue}46}&{\color{myorange}60},{\color{myblue}60}&{\color{myorange}60},{\color{myblue}61}\\
\numval{1e-08}&{\color{myorange}58},{\color{myblue}56}&{\color{myorange}52},{\color{myblue}49}&{\color{myorange}49},{\color{myblue}46}&{\color{myorange}60},{\color{myblue}60}&{\color{myorange}60},{\color{myblue}61}&{\color{myorange}58},{\color{myblue}56}&{\color{myorange}52},{\color{myblue}49}&{\color{myorange}49},{\color{myblue}46}&{\color{myorange}60},{\color{myblue}60}&{\color{myorange}60},{\color{myblue}61}&{\color{myorange}58},{\color{myblue}56}&{\color{myorange}52},{\color{myblue}49}&{\color{myorange}49},{\color{myblue}46}&{\color{myorange}60},{\color{myblue}60}&{\color{myorange}60},{\color{myblue}61}\\
\bottomrule\bottomrule\end{tabular}

%% file: suppl_tex/table_hconv.tex
\begin{tabular}{c|cccccc} 

$\ell \!\setminus \!\!| \mathcal{T}_h | $ &48&162&384&750&1296&2058\\
\midrule
1&19&37&50&53&53&52\\
\rowcolor{mygray}
2&46&52&48&46&44&44\\
3&50&48&45&43&41&41\\
\rowcolor{mygray}
4&52&49&46&43&41&41\\
5&51&46&43&41&41&41\\
\rowcolor{mygray}
6&47&43&42&41&42&42\\
\end{tabular}

%% file: suppl_tex/abs_u_T3.tex
\definecolor{color0}{rgb}{0.12156862745098,0.466666666666667,0.705882352941177}
\definecolor{color1}{rgb}{1,0.498039215686275,0.0549019607843137}
\definecolor{color2}{rgb}{0.172549019607843,0.627450980392157,0.172549019607843}

\begin{axis}[
legend cell align={left},
legend style={
  fill opacity=0.8,
  draw opacity=1,
  text opacity=1,
  at={(0.03,0.97)},
  anchor=north west,
  draw=white!80!black
},
tick align=outside,
tick pos=left,
x grid style={white!69.0196078431373!black},
xlabel={t},
xmin=-0.15, xmax=3.15000000000001,
xtick style={color=black},
y grid style={white!69.0196078431373!black},
ymin=-0.00293317789511062, ymax=0.0615967357973231,
ytick style={color=black}
]
\addplot [semithick, color0]
table {%
0 0
0.0125 0.000549384181047677
0.025 0.00107377293376786
0.0375 0.00157576106052353
0.05 0.00205522966012373
0.0625 0.00251107695382206
0.075 0.00294178727953246
0.0875 0.00334567516049338
0.1 0.0037210120878631
0.1125 0.0040661026050061
0.125 0.00437933452104842
0.1375 0.0046592142542568
0.15 0.0049043940802059
0.1625 0.00511369282665802
0.175 0.00528611207485488
0.1875 0.00542084937816995
0.2 0.00551730843106168
0.2125 0.00557510669939647
0.225 0.00559408116979697
0.2375 0.00557429179940482
0.25 0.00551602370674133
0.2625 0.00541978669532565
0.275 0.00528631376617747
0.2875 0.00511655796102604
0.3 0.00491168741310883
0.3125 0.00467307908893354
0.325 0.00440231094380312
0.3375 0.00410115331883658
0.35 0.00377155814165454
0.3625 0.00341564770222792
0.375 0.00303570256412306
0.3875 0.00263414766909895
0.4 0.00221353913219264
0.4125 0.00177655083215259
0.425 0.00132596556372921
0.4375 0.000864688526719437
0.45 0.000395978697292943
0.4625 8.50862523186046e-05
0.475 0.000556513065022178
0.4875 0.0010313019181533
0.5 0.00150230501142629
0.5125 0.00196639217180984
0.525 0.00242061324142129
0.5375 0.00286210668484183
0.55 0.00328809900271934
0.5625 0.00369591685787992
0.575 0.00408300214415896
0.5875 0.00444692673284307
0.6 0.00478540706087135
0.6125 0.00509631807303965
0.625 0.00537770610003792
0.6375 0.00562780089812326
0.65 0.00584502613361991
0.6625 0.00602800966425587
0.675 0.00617559121691336
0.6875 0.00628683004171805
0.7 0.00636101033013494
0.712499999999999 0.00639764568744737
0.724999999999999 0.00639648208110263
0.737499999999999 0.00635749949769994
0.749999999999999 0.00628091180632747
0.762499999999999 0.00616716546071215
0.774999999999999 0.00601693695644864
0.787499999999999 0.00583112833237173
0.799999999999999 0.00561086162544669
0.812499999999999 0.00535747208947328
0.824999999999999 0.00507249991330938
0.837499999999999 0.00475768045127038
0.849999999999999 0.00441493398111591
0.862499999999999 0.00404635340019403
0.874999999999999 0.00365419166161657
0.887499999999999 0.00324084793819114
0.899999999999999 0.00280885266439109
0.912499999999999 0.0023608529054765
0.924999999999999 0.00189959652211711
0.937499999999999 0.00142791867147476
0.949999999999999 0.000948736855821381
0.962499999999999 0.000465124797238412
0.974999999999999 3.03038268806368e-05
0.987499999999998 0.000505568276928635
0.999999999999998 0.00098521274060546
1.0125 0.00145776063182199
1.025 0.00192019948862487
1.0375 0.00236966871684938
1.05 0.00280340303383897
1.0625 0.00321873842666568
1.075 0.00361312568812594
1.0875 0.00398414518934252
1.1 0.00432952148930714
1.1125 0.00464713733991922
1.125 0.00493504615888081
1.1375 0.00519148457178896
1.15 0.00541488298916601
1.1625 0.00560387517926984
1.175 0.00575730694263294
1.1875 0.00587424294645876
1.2 0.00595397282941775
1.2125 0.00599601522189817
1.225 0.00600012105605306
1.2375 0.00596627474955308
1.25 0.00589469467036049
1.2625 0.00578583156791871
1.275 0.00564036577548567
1.2875 0.00545920329172831
1.3 0.0052434698568451
1.3125 0.00499450423405852
1.325 0.00471384995528251
1.3375 0.00440324577882521
1.35 0.00406461486640671
1.3625 0.00370005338647832
1.375 0.00331181700647319
1.3875 0.00290230774867168
1.4 0.0024740590657097
1.4125 0.00202972091402795
1.425 0.00157204525948953
1.4375 0.001103875163298
1.45 0.000628161903421469
1.4625 0.000148653184453756
1.475 0.000335534027097613
1.4875 0.000816172814610255
1.5 0.00129283594085239
1.5125 0.0017623673090516
1.525 0.00222183455237995
1.5375 0.00266838966390874
1.55 0.0030992696803844
1.5625 0.00351181073205869
1.575 0.00390346264421208
1.5875 0.00427180473431859
1.6 0.00461456043616647
1.6125 0.00492961096994991
1.625 0.00521500867508807
1.6375 0.00546898878782211
1.65 0.00568998053437987
1.6625 0.00587661642548685
1.675 0.00602774112534765
1.6875 0.00614241819143021
1.7 0.00621993597191704
1.7125 0.00625981220375238
1.725 0.00626179673507784
1.7375 0.00622587290598878
1.75 0.00615225812891117
1.7625 0.00604140223913069
1.775 0.00589398464886111
1.7875 0.00571091037619479
1.8 0.00549330439968085
1.8125 0.00524250463457825
1.825 0.00496005380965358
1.8375 0.00464768987720139
1.85 0.00430733537990406
1.8625 0.00394108560515135
1.875 0.00355119573979947
1.8875 0.00314006723140756
1.9 0.00271023289294193
1.9125 0.00226434241460403
1.925 0.00180514696603087
1.9375 0.0013354878389394
1.95 0.000858299837212409
1.9625 0.000376779834964117
1.97499999999999 0.000110437337552254
1.98749999999999 0.000590550697968799
1.99999999999999 0.00106842357157283
2.01249999999999 0.00153923025191415
2.025 0.00199996941161658
2.0375 0.0024477789824149
2.05 0.00287989204395796
2.0625 0.00329364310143715
2.075 0.00368648159223995
2.0875 0.00405598672241182
2.1 0.00439988196176552
2.1125 0.00471604888220203
2.125 0.00500254007953843
2.1375 0.00525759120706799
2.15 0.00547963171287272
2.1625 0.00566729453865725
2.175 0.00581942454612042
2.1875 0.00593508588978126
2.2 0.00601356713522844
2.2125 0.00605438632174623
2.225 0.00605729362738083
2.2375 0.0060222727751461
2.25 0.00594954148282724
2.2625 0.00583954985390385
2.275 0.00569297767816666
2.2875 0.00551073021573846
2.3 0.00529393276962677
2.3125 0.00504392344318524
2.325 0.00476224523244145
2.3375 0.00445063647701569
2.35 0.0041110197315824
2.3625 0.00374549059733314
2.375 0.00335630439942005
2.3875 0.00294586259092971
2.4 0.00251669810718647
2.4125 0.00207146032237408
2.425 0.00161290042501613
2.4375 0.00114385957443997
2.45 0.000667280624665331
2.4625 0.000186667700461831
2.475 0.000298116649075406
2.4875 0.000779506040045622
2.5 0.00125695753526113
2.5125 0.00172726796161072
2.525 0.00218750196450043
2.5375 0.00263481113174355
2.55 0.00306643290119651
2.5625 0.00347970325350677
2.575 0.00387207262739091
2.5875 0.00424112038413574
2.6 0.00458457023596497
2.6125 0.00490030374725168
2.625 0.00518637343564116
2.6375 0.0054410148388037
2.65 0.00566265728165594
2.6625 0.00584993373720971
2.675 0.00600168888610605
2.6875 0.0061169865451625
2.7 0.00619511549744409
2.7125 0.00623559336626385
2.725 0.00623817028218678
2.73750000000001 0.00620282994301481
2.75000000000001 0.00612978978720932
2.76250000000001 0.00601949987955416
2.77500000000001 0.0058726398105092
2.78750000000001 0.00569011483061006
2.80000000000001 0.00547304998612565
2.81250000000001 0.00522278340632234
2.82500000000001 0.00494085792504604
2.83750000000001 0.00462901173313847
2.85000000000001 0.00428916744099552
2.86250000000001 0.00392342043083477
2.87500000000001 0.00353402618537798
2.88750000000001 0.00312338608506179
2.90000000000001 0.00269403313308893
2.91250000000001 0.0022486168713918
2.92500000000001 0.00178988861630055
2.93750000000001 0.00132068924387416
2.95000000000001 0.000843952826075323
2.96250000000001 0.000362878738496397
2.97500000000001 0.000123430116121206
2.98750000000001 0.00060351102677287
3.00000000000001 0.00108097319622003
};
\addlegendentry{$P_1$}
\addplot [semithick, color1]
table {%
0 0
0.0125 0.00504690224719199
0.025 0.00986156354870986
0.0375 0.0144701165776255
0.05 0.0188720866721568
0.0625 0.0230575674624841
0.075 0.0270127311724038
0.0875 0.0307221473792038
0.1 0.034169971675239
0.1125 0.0373406502834659
0.125 0.0402193835502483
0.1375 0.0427924563911608
0.15 0.0450474846960135
0.1625 0.0469736067634094
0.175 0.0485616344324812
0.1875 0.0498041727124909
0.2 0.0506957138940558
0.2125 0.0512327093971773
0.225 0.0514136217651849
0.2375 0.0512389585406141
0.25 0.0507112889073455
0.2625 0.0498352444475534
0.275 0.0486175046163292
0.2875 0.0470667670505286
0.3 0.0451937044343368
0.3125 0.0430109077012009
0.325 0.0405328167660738
0.3375 0.0377756393393252
0.35 0.0347572599053572
0.3625 0.0314971394803735
0.375 0.028016209776565
0.3875 0.0243367688336722
0.4 0.0204823919716197
0.4125 0.0164779052725915
0.425 0.01234959027845
0.4375 0.00812650755709789
0.45 0.00385179105512003
0.4625 0.000868941854014916
0.475 0.00497235882537729
0.4875 0.0093123643657906
0.5 0.0136295374301239
0.5125 0.0178870419246765
0.525 0.0220560151838431
0.5375 0.0261095723670031
0.55 0.0300219819716198
0.5625 0.0337685630288692
0.575 0.0373257463653802
0.5875 0.040671182471565
0.6 0.0437838616008674
0.6125 0.0466442338454112
0.625 0.0492343243277627
0.6375 0.0515378407463662
0.65 0.0535402717990998
0.6625 0.0552289745885013
0.675 0.0565932520703996
0.6875 0.0576244175961653
0.7 0.0583158481512232
0.712499999999999 0.0586630243277433
0.724999999999999 0.0586635579022125
0.737499999999999 0.058317205940328
0.749999999999999 0.0576258722047725
0.762499999999999 0.0565935951571849
0.774999999999999 0.0552265224720697
0.787499999999999 0.0535328736399753
0.799999999999999 0.0515228889844934
0.812499999999999 0.0492087666560303
0.824999999999999 0.0466045879925971
0.837499999999999 0.0437262315637675
0.849999999999999 0.0405912764463942
0.862499999999999 0.0372188966769043
0.874999999999999 0.0336297474016692
0.887499999999999 0.0298458460473675
0.899999999999999 0.0258904542008199
0.912499999999999 0.0217879717070287
0.924999999999999 0.017563881941604
0.937499999999999 0.013244879048358
0.949999999999999 0.00885984127224028
0.962499999999999 0.00444758170382388
0.974999999999999 0.000649449303729733
0.987499999999998 0.00453832953112743
0.999999999999998 0.00891708270486013
1.0125 0.013245311178592
1.025 0.0174850914048713
1.0375 0.0216080801458147
1.05 0.0255881936732584
1.0625 0.0294006721750472
1.075 0.0330219553511586
1.0875 0.0364297348449123
1.1 0.039603054207838
1.1125 0.0425224209516975
1.125 0.0451699185486301
1.1375 0.0475293118402051
1.15 0.0495861447867974
1.1625 0.0513278278820871
1.175 0.0527437150085218
1.1875 0.0538251684846918
1.2 0.0545656116435577
1.2125 0.0549605696233613
1.225 0.0550076963845595
1.2375 0.0547067895726414
1.25 0.054059791421587
1.2625 0.0530707770394035
1.275 0.0517459294916671
1.2875 0.0500935016459666
1.3 0.0481237659123422
1.3125 0.045848951266732
1.325 0.0432831686796546
1.3375 0.0404423251377928
1.35 0.0373440278213208
1.3625 0.0340074784154755
1.375 0.0304533611196741
1.3875 0.0267037261256572
1.4 0.0227818780383183
1.4125 0.0187122894571503
1.425 0.0145206101087481
1.4375 0.0102340770139221
1.45 0.00588445815748589
1.4625 0.00156418618731684
1.475 0.00303129421314307
1.4875 0.00740852323325958
1.5 0.0117732965931404
1.5125 0.0160777179953715
1.525 0.0202921161406279
1.5375 0.0243895365804733
1.55 0.028344293321929
1.5625 0.0321317720016672
1.575 0.0357284724980132
1.5875 0.039112111558976
1.6 0.0422617417345699
1.6125 0.0451578717695481
1.625 0.0477825817315289
1.6375 0.050119630979865
1.65 0.0521545564265396
1.6625 0.053874761044581
1.675 0.0552695906245634
1.6875 0.0563303992221508
1.7 0.0570506021278409
1.7125 0.0574257161375788
1.725 0.0574533872001559
1.7375 0.0571334051739395
1.75 0.0564677045543983
1.7625 0.0554603528947444
1.775 0.0541175259934961
1.7875 0.0524474697597613
1.8 0.0504604496503966
1.8125 0.0481686880074584
1.825 0.0455862893600846
1.8375 0.0427291544627216
1.85 0.0396148841529406
1.8625 0.0362626736862295
1.875 0.032693199859214
1.8875 0.0289285031270146
1.9 0.024991871898025
1.9125 0.0209077432736047
1.925 0.0167016670168855
1.9375 0.0124005040265483
1.95 0.00803378913306836
1.9625 0.00364616583795901
1.97499999999999 0.00106582084466257
1.98749999999999 0.00531123443437187
1.99999999999999 0.00967949606541162
2.01249999999999 0.0139936108615381
2.025 0.0182189574635771
2.0375 0.0223276387523113
2.05 0.026293678492006
2.0625 0.0300923484746512
2.075 0.033700096740686
2.0875 0.0370946144115665
2.1 0.0402549409369655
2.1125 0.0431615782909148
2.125 0.0457966034453348
2.1375 0.0481437746599093
2.15 0.0501886292647125
2.1625 0.0519185713277252
2.175 0.0533229481783703
2.1875 0.0543931156765832
2.2 0.0551224913747258
2.2125 0.055506594339569
2.225 0.0555430729736303
2.2375 0.0552317193947432
2.25 0.054574470569489
2.2625 0.0535753965448746
2.275 0.0522406754639948
2.2875 0.0505785556190184
2.3 0.0485993049089629
2.3125 0.0463151481657903
2.325 0.0437401922937147
2.3375 0.0408903405567642
2.35 0.0377831967835517
2.3625 0.0344379596880419
2.375 0.030875310765113
2.3875 0.0271172982857614
2.4 0.0231872254992267
2.4125 0.0191095640557916
2.425 0.0149099586810837
2.4375 0.0106156115272401
2.45 0.00625793782777164
2.4625 0.00191446578271804
2.475 0.00268560052611988
2.4875 0.00706069511715048
2.5 0.0114312985573005
2.5125 0.0157423262874004
2.525 0.0199634733787822
2.5375 0.0240676376546702
2.55 0.028029082675623
2.5625 0.031823173968851
2.575 0.0354264019666221
2.5875 0.0388164794087509
2.6 0.0419724572480424
2.6125 0.044874843837073
2.625 0.0475057197564463
2.6375 0.0498488450916917
2.65 0.0518897580649965
2.6625 0.0536158627413493
2.675 0.0550165064021828
2.6875 0.0560830444687848
2.7 0.0568088934701605
2.7125 0.0571895718232286
2.725 0.0572227268759397
2.73750000000001 0.0569081496350286
2.75000000000001 0.0562477760397285
2.76250000000001 0.055245674874337
2.77500000000001 0.0539080231290261
2.78750000000001 0.0522430676606053
2.80000000000001 0.0502610751028959
2.81250000000001 0.0479742685903528
2.82500000000001 0.0453967533317382
2.83750000000001 0.0425444304227813
2.85000000000001 0.0394349008478592
2.86250000000001 0.036087359450177
2.87500000000001 0.0325224815872373
2.88750000000001 0.0287623051934695
2.90000000000001 0.0248301135567772
2.91250000000001 0.0207503341398292
2.92500000000001 0.0165484970518707
2.93750000000001 0.012251420269631
2.95000000000001 0.00788853325980638
2.96250000000001 0.00350441950190282
2.97500000000001 0.00116254841212997
2.98750000000001 0.00543853202431064
3.00000000000001 0.00980471675479633
};
\addlegendentry{$P_2$}
\addplot [semithick, color2]
table {%
0 0
0.0125 0.00315068753538691
0.025 0.00615857651310765
0.0375 0.00903927106683262
0.05 0.0117921094240799
0.0625 0.0144106621001794
0.075 0.016886110418169
0.0875 0.019208669982209
0.1 0.021368321356957
0.1125 0.0233552424573038
0.125 0.0251600941729454
0.1375 0.0267742237791223
0.15 0.0281898164873369
0.1625 0.0294000139187679
0.175 0.0303990077239668
0.1875 0.0311821134294351
0.2 0.0317458287913042
0.2125 0.0320878785824959
0.225 0.0322072467767135
0.2375 0.0321041977901756
0.25 0.0317802863796879
0.2625 0.0312383588783141
0.275 0.0304825433685475
0.2875 0.029518231579452
0.3 0.0283520516207974
0.3125 0.0269918322359409
0.325 0.0254465595437732
0.3375 0.0237263252979879
0.35 0.0218422689011456
0.3625 0.0198065126372932
0.375 0.017632091086238
0.3875 0.0153328772118186
0.4 0.012923507099262
0.4125 0.0104193127373843
0.425 0.00783629346590846
0.4375 0.00519128458884743
0.45 0.00250384482788528
0.4625 0.000324893100880586
0.475 0.00297236814912472
0.4875 0.00569751057622604
0.5 0.00840333648881594
0.5125 0.0110708344136823
0.525 0.0136828507304995
0.5375 0.0162228731774284
0.55 0.0186749289288974
0.5625 0.0210236296881258
0.575 0.0232542479563043
0.5875 0.0253528003228901
0.6 0.0273061295843593
0.6125 0.0291019839711119
0.625 0.030729091352749
0.6375 0.0321772274889565
0.65 0.0334372789381967
0.6625 0.0345012982801762
0.675 0.0353625529256424
0.6875 0.0360155662359622
0.7 0.0364561508081131
0.712499999999999 0.0366814343457726
0.724999999999999 0.0366898768057057
0.737499999999999 0.0364812797944235
0.749999999999999 0.0360567875866491
0.762499999999999 0.0354188798593389
0.774999999999999 0.0345713564527243
0.787499999999999 0.033519313354801
0.799999999999999 0.0322691116853565
0.812499999999999 0.0308283382465535
0.824999999999999 0.0292057585303888
0.837499999999999 0.0274112631961373
0.849999999999999 0.0254558068810272
0.862499999999999 0.0233513414791973
0.874999999999999 0.0211107431205346
0.887499999999999 0.0187477347657746
0.899999999999999 0.016276805324663
0.912499999999999 0.0137131283253261
0.924999999999999 0.0110724889704345
0.937499999999999 0.00837124802159057
0.949999999999999 0.0056264821403125
0.962499999999999 0.00285752768992128
0.974999999999999 0.000249822664442889
0.987499999999998 0.00272595419515932
0.999999999999998 0.0054749040532403
1.0125 0.00818697899408261
1.025 0.0108427892772027
1.0375 0.0134254935319398
1.05 0.0159190623292797
1.0625 0.0183081174954632
1.075 0.0205779626032083
1.0875 0.022714651711771
1.1 0.0247050670738307
1.1125 0.0265369958910793
1.125 0.0281992040315511
1.1375 0.0296815038984279
1.15 0.0309748169915213
1.1625 0.0320712293037915
1.175 0.0329640400181876
1.1875 0.0336478027363826
1.2 0.0341183590240994
1.2125 0.0343728638182539
1.225 0.0344098033154288
1.2375 0.0342290040187737
1.25 0.0338316338439012
1.2625 0.0332201951675662
1.275 0.0323985092764062
1.2875 0.0313716928152029
1.3 0.0301461265506514
1.3125 0.0287294159682789
1.325 0.0271303444803786
1.3375 0.0253588198492832
1.35 0.0234258132888425
1.3625 0.0213432923743207
1.375 0.019124148885476
1.3875 0.016782121263004
1.4 0.0143317151369151
1.4125 0.0117881253154016
1.425 0.00916717462087117
1.4375 0.00648532982514965
1.45 0.00376021232586839
1.4625 0.00102105747060671
1.475 0.00178071300798436
1.4875 0.00453324355277809
1.5 0.00726843141883047
1.5125 0.00996451072833046
1.525 0.0126041468604785
1.5375 0.0151708331591208
1.55 0.0176486323201704
1.5625 0.0200221970498333
1.575 0.0222768409082086
1.5875 0.0243986196314765
1.6 0.0263744132681254
1.6125 0.0281920052142695
1.625 0.0298401563574646
1.6375 0.0313086737463441
1.65 0.0325884732909896
1.6625 0.0336716354877528
1.675 0.0345514538600867
1.6875 0.0352224765392278
1.7 0.0356805396711945
1.7125 0.035922792951663
1.725 0.0359477174175509
1.7375 0.0357551344917744
1.75 0.0353462072899739
1.7625 0.0347234334858442
1.775 0.0338906297108319
1.7875 0.0328529081847114
1.8 0.0316166453409649
1.8125 0.0301894425294752
1.825 0.0285800790503534
1.8375 0.026798458726287
1.85 0.0248555488131966
1.8625 0.0227633133057121
1.875 0.0205346398458581
1.8875 0.0181832625713653
1.9 0.0157236816006082
1.9125 0.0131710823477464
1.925 0.0105412654703598
1.9375 0.00785062064782284
1.95 0.00511632896211057
1.9625 0.00235869198406369
1.97499999999999 0.000487629969516779
1.98749999999999 0.00320457650887862
1.99999999999999 0.00594556970973027
2.01249999999999 0.0086487777714743
2.025 0.0112957274145476
2.0375 0.013869703165146
2.05 0.0163547054836246
2.0625 0.0187353636964695
2.075 0.0209969815102214
2.0875 0.0231256107284764
2.1 0.0251081300888252
2.1125 0.0269323228210472
2.125 0.0285869506091904
2.1375 0.0300618216875656
2.15 0.0313478532976939
2.1625 0.0324371275072298
2.175 0.0333229394796029
2.1875 0.0339998389723853
2.2 0.0344636638289884
2.2125 0.0347115655068724
2.225 0.0347420265558757
2.2375 0.0345548702371472
2.25 0.0341512612345777
2.2625 0.0335336987983319
2.275 0.0327060011658823
2.2875 0.0316732822629767
2.3 0.0304419199279401
2.3125 0.0290195170788846
2.325 0.0274148547030745
2.3375 0.0256378380405266
2.35 0.0236994361682747
2.3625 0.0216116146155549
2.375 0.0193872633856208
2.3875 0.0170401194309976
2.4 0.0145846872870443
2.4125 0.0120361616538679
2.425 0.00941036586369233
2.4375 0.00672376611835004
2.45 0.00399394338404829
2.4625 0.00124768524211042
2.475 0.00155978804541841
2.4875 0.00431425386957977
2.5 0.00705348929520404
2.5125 0.00975378035441323
2.525 0.0123976279997326
2.5375 0.0149684873831057
2.55 0.0174504087912888
2.5625 0.0198280402497168
2.575 0.0220866935104805
2.5875 0.0242124240635546
2.6 0.026192112128344
2.6125 0.0280135417134288
2.625 0.0296654743847113
2.6375 0.0311377183568097
2.65 0.0324211903094752
2.6625 0.0335079716173827
2.675 0.0343913569217279
2.6875 0.035065895298911
2.7 0.0355274237316708
2.7125 0.0357730930636439
2.725 0.0358013850972664
2.73750000000001 0.0356121221454162
2.75000000000001 0.0352064683286696
2.76250000000001 0.0345869220400621
2.77500000000001 0.0337573007304732
2.78750000000001 0.0327227175345394
2.80000000000001 0.031489549479746
2.81250000000001 0.0300653986157844
2.82500000000001 0.0284590450106287
2.83750000000001 0.0266803930203771
2.85000000000001 0.0247404103719846
2.86250000000001 0.0226510614498241
2.87500000000001 0.0204252340083383
2.88750000000001 0.0180766621012293
2.90000000000001 0.0156198450352782
2.91250000000001 0.0130699666899556
2.92500000000001 0.0104428235469801
2.93750000000001 0.00775479619405693
2.95000000000001 0.00502304216940472
2.96250000000001 0.00226783236553265
2.97500000000001 0.000562191093839728
2.98750000000001 0.00328889545307345
3.00000000000001 0.00602785112598997
};
\addlegendentry{$P_3$}
\end{axis}

%% file: suppl_tex/p1_T3.tex
\definecolor{color0}{rgb}{0.12156862745098,0.466666666666667,0.705882352941177}
\definecolor{color1}{rgb}{1,0.498039215686275,0.0549019607843137}
\definecolor{color2}{rgb}{0.172549019607843,0.627450980392157,0.172549019607843}

\begin{axis}[
legend cell align={left},
legend style={
  fill opacity=0.8,
  draw opacity=1,
  text opacity=1,
  at={(0.03,0.97)},
  anchor=north west,
  draw=white!80!black
},
tick align=outside,
tick pos=left,
x grid style={white!69.0196078431373!black},
xlabel={t},
xmin=-0.15, xmax=3.15000000000001,
xtick style={color=black},
y grid style={white!69.0196078431373!black},
ymin=4.98705739113556, ymax=5.27179478615323,
ytick style={color=black}
]
\addplot [semithick, color0]
table {%
0 5
0.0125 5.00121425688222
0.025 5.00242058746316
0.0375 5.00362081619101
0.05 5.00481588975615
0.0625 5.00600622827243
0.075 5.00719194985265
0.0875 5.00837299957895
0.1 5.00954922350247
0.1125 5.01072041225607
0.125 5.01188632766346
0.1375 5.01304672079785
0.15 5.01420134370633
0.1625 5.01534995837706
0.175 5.0164923434489
0.1875 5.01762829943894
0.2 5.01875765289263
0.2125 5.01988025969843
0.225 5.02099600772885
0.2375 5.02210481892799
0.25 5.02320665090499
0.2625 5.02430149810406
0.275 5.02538939257892
0.2875 5.02647040441078
0.3 5.02754464178821
0.3125 5.02861225077066
0.325 5.02967341475343
0.3375 5.03072835364733
0.35 5.03177732278965
0.3625 5.03282061159877
0.375 5.03385854198737
0.3875 5.0348914665501
0.4 5.03591976653892
0.4125 5.03694384964416
0.425 5.03796414759048
0.4375 5.03898111353708
0.45 5.03999521964407
0.4625 5.0410069536888
0.475 5.04201681657261
0.4875 5.04302531916357
0.5 5.04403297919889
0.5125 5.04504031813942
0.525 5.04604785799564
0.5375 5.04705611814544
0.55 5.04806561216392
0.5625 5.04907684468545
0.575 5.05009030831794
0.5875 5.05110648062759
0.6 5.05212582121509
0.6125 5.05314876889909
0.625 5.05417573902682
0.6375 5.05520712092804
0.65 5.05624327552622
0.6625 5.05728453312552
0.675 5.05833119138426
0.6875 5.05938351348881
0.7 5.06044172653922
0.712499999999999 5.06150602015691
0.724999999999999 5.06257654532172
0.737499999999999 5.06365341344755
0.749999999999999 5.06473669570094
0.762499999999999 5.06582642256749
0.774999999999999 5.06692258366845
0.787499999999999 5.06802512782897
0.799999999999999 5.0691339633975
0.812499999999999 5.07024895881465
0.824999999999999 5.07136994342725
0.837499999999999 5.07249670854406
0.849999999999999 5.0736290087249
0.862499999999999 5.0747665632966
0.874999999999999 5.07590905808637
0.887499999999999 5.07705614736039
0.899999999999999 5.07820745595827
0.912499999999999 5.07936258160777
0.924999999999999 5.08052109740692
0.937499999999999 5.08168255445792
0.949999999999999 5.08284648463655
0.962499999999999 5.08401240347987
0.974999999999999 5.08517981317471
0.987499999999998 5.08634820562848
0.999999999999998 5.08751706560319
1.0125 5.0886858738932
1.025 5.08985411052794
1.0375 5.09102125797872
1.05 5.09218680435112
1.0625 5.0933502465416
1.075 5.09451109334083
1.0875 5.09566886846328
1.1 5.09682311348444
1.1125 5.09797339066864
1.125 5.09911928566837
1.1375 5.10026041008013
1.15 5.10139640387506
1.1625 5.10252693737952
1.175 5.1036517138002
1.1875 5.10477047065072
1.2 5.10588298148726
1.2125 5.10698905730679
1.225 5.10808854771225
1.2375 5.109181341841
1.25 5.11026736904844
1.2625 5.11134659934458
1.275 5.11241904358119
1.2875 5.11348475338192
1.3 5.11454382082973
1.3125 5.11559637789429
1.325 5.11664259562305
1.3375 5.11768268308573
1.35 5.11871688608623
1.3625 5.11974548564722
1.375 5.12076879627943
1.3875 5.12178716404011
1.4 5.12280096440436
1.4125 5.12381059990589
1.425 5.12481649792357
1.4375 5.12581910751254
1.45 5.12681889707127
1.4625 5.12781635132793
1.475 5.12881196835306
1.4875 5.1298062564885
1.5 5.13079973121407
1.5125 5.13179291196978
1.525 5.13278631895365
1.5375 5.13378046991499
1.55 5.13477587696249
1.5625 5.13577304340751
1.575 5.13677246066097
1.5875 5.13777460520434
1.6 5.13877993565195
1.6125 5.13978888992409
1.625 5.14080188254704
1.6375 5.1418193020981
1.65 5.14284150880967
1.6625 5.14386883234887
1.675 5.14490156978522
1.6875 5.14593998375937
1.7 5.1469843008637
1.7125 5.14803471024626
1.725 5.14909136244437
1.7375 5.15015436845628
1.75 5.15122379905819
1.7625 5.15229968436676
1.775 5.15338201365465
1.7875 5.15447073541614
1.8 5.15556575768545
1.8125 5.15666694860322
1.825 5.15777413723013
1.8375 5.1588871146009
1.85 5.16000563501247
1.8625 5.16112941753917
1.875 5.16225814776486
1.8875 5.16339147972155
1.9 5.16452903802236
1.9125 5.16567042017619
1.925 5.16681519906911
1.9375 5.16796292559822
1.95 5.16911313143966
1.9625 5.17026533193732
1.97499999999999 5.17141902908983
1.98749999999999 5.17257371462182
1.99999999999999 5.17372887311718
2.01249999999999 5.17488398519699
2.025 5.17603853072183
2.0375 5.17719199199841
2.05 5.17834385697165
2.0625 5.17949362238153
2.075 5.18064079686591
2.0875 5.1817849039902
2.1 5.18292548518444
2.1125 5.18406210257079
2.125 5.18519434166309
2.1375 5.18632181392255
2.15 5.18744415915286
2.1625 5.1885610477556
2.175 5.18967218252571
2.1875 5.19077730096917
2.2 5.19187617652389
2.2125 5.19296862006865
2.225 5.19405448109165
2.2375 5.19513364861871
2.25 5.19620605189614
2.2625 5.19727166082812
2.275 5.19833048616256
2.2875 5.19938257942124
2.3 5.20042803258809
2.3125 5.2014669775377
2.325 5.20249958522295
2.3375 5.2035260646218
2.35 5.20454666144937
2.3625 5.20556165664174
2.375 5.20657136462366
2.3875 5.20757613137219
2.4 5.20857633224057
2.4125 5.20957236990774
2.425 5.21056467136785
2.4375 5.21155368572775
2.45 5.21253988131414
2.4625 5.2135237427854
2.475 5.2145057681439
2.4875 5.21548646566558
2.5 5.21646635076607
2.5125 5.217445942823
2.525 5.21842576197387
2.5375 5.21940632590938
2.55 5.22038814668113
2.5625 5.22137172754505
2.575 5.22235755985856
2.5875 5.22334612005111
2.6 5.2243378666866
2.6125 5.2253332376365
2.625 5.22633264737994
2.6375 5.22733648444826
2.65 5.22834510902955
2.6625 5.22935885074814
2.675 5.23037800663212
2.6875 5.23140283928179
2.7 5.23243357525131
2.7125 5.2334704036509
2.725 5.23451347498171
2.73750000000001 5.23556290020745
2.75000000000001 5.23661875007041
2.76250000000001 5.23768105465513
2.77500000000001 5.2387498032027
2.78750000000001 5.23982494417753
2.80000000000001 5.2409063855847
2.81250000000001 5.24199399553719
2.82500000000001 5.24308760306872
2.83750000000001 5.24418699918825
2.85000000000001 5.24529193816793
2.86250000000001 5.24640213905852
2.87500000000001 5.24751728742104
2.88750000000001 5.24863703726554
2.90000000000001 5.24976101318436
2.91250000000001 5.25088881266639
2.92500000000001 5.25202000857868
2.93750000000001 5.25315415179973
2.95000000000001 5.25429077398865
2.96250000000001 5.25542939047242
2.97500000000001 5.25656950323401
2.98750000000001 5.25771060398295
3.00000000000001 5.25885217728879
};
\addlegendentry{$P_1$}
\addplot [semithick, color1]
table {%
0 5
0.0125 5.00122920474031
0.025 5.00246056015482
0.0375 5.00369310877667
0.05 5.00492592190546
0.0625 5.00615808197987
0.075 5.00738868149199
0.0875 5.00861682583577
0.1 5.00984163734336
0.1125 5.01106225976907
0.125 5.01227786287301
0.1375 5.01348764717318
0.15 5.01469084854923
0.1625 5.01588674278904
0.175 5.0170746499618
0.1875 5.01825393857701
0.2 5.01942402948194
0.2125 5.02058439946107
0.225 5.0217345845047
0.2375 5.02287418271627
0.25 5.02400285684022
0.2625 5.02512033638705
0.275 5.02622641934403
0.2875 5.027320973458
0.3 5.02840393708281
0.3125 5.02947531958634
0.325 5.03053520131422
0.3375 5.03158373311172
0.35 5.03262113540618
0.3625 5.03364769685724
0.375 5.03466377258311
0.3875 5.03566978197378
0.4 5.03666620610617
0.4125 5.03765358477633
0.425 5.03863251316924
0.4375 5.03960363816592
0.45 5.04056765441293
0.4625 5.04152529995193
0.475 5.04247735170776
0.4875 5.04342462068126
0.5 5.04436794694139
0.5125 5.04530819443764
0.525 5.04624624566509
0.5375 5.04718299621495
0.55 5.04811934924451
0.5625 5.04905620990042
0.575 5.04999447972973
0.5875 5.05093505111336
0.6 5.05187880175626
0.6125 5.05282658926813
0.625 5.0537792458683
0.6375 5.05473757324729
0.65 5.05570233761701
0.6625 5.05667426497945
0.675 5.05765403664408
0.6875 5.05864228502045
0.7 5.05963958971289
0.712499999999999 5.06064647394073
0.724999999999999 5.06166340130638
0.737499999999999 5.06269077293119
0.749999999999999 5.06372892497673
0.762499999999999 5.06477812656666
0.774999999999999 5.06583857812208
0.787499999999999 5.06691041012063
0.799999999999999 5.06799368228705
0.812499999999999 5.06908838321985
0.824999999999999 5.07019443045686
0.837499999999999 5.07131167097896
0.849999999999999 5.07243988214881
0.862499999999999 5.07357877307912
0.874999999999999 5.07472798642163
0.887499999999999 5.07588710056602
0.899999999999999 5.07705563223519
0.912499999999999 5.07823303946062
0.924999999999999 5.07941872491996
0.937499999999999 5.08061203961568
0.949999999999999 5.08181228687269
0.962499999999999 5.08301872662976
0.974999999999999 5.08423057999855
0.987499999999998 5.08544703406234
0.999999999999998 5.08666724688416
1.0125 5.08789035269444
1.025 5.0891154672253
1.0375 5.09034169315905
1.05 5.09156812565724
1.0625 5.09279385793621
1.075 5.09401798685462
1.0875 5.09523961847846
1.1 5.09645787358966
1.1125 5.09767189310385
1.125 5.09888084336388
1.1375 5.10008392127711
1.15 5.10128035927906
1.1625 5.1024694300097
1.175 5.10365045086366
1.1875 5.10482278815847
1.2 5.10598586104684
1.2125 5.10713914509475
1.225 5.10828217551276
1.2375 5.10941455001891
1.25 5.11053593131709
1.2625 5.1116460491741
1.275 5.11274470208322
1.2875 5.11383175850554
1.3 5.1149071576774
1.3125 5.11597090998458
1.325 5.11702309689522
1.3375 5.11806387045607
1.35 5.11909345235309
1.3625 5.12011213254313
1.375 5.12112026746404
1.3875 5.1221182778362
1.4 5.12310664606597
1.4125 5.12408591325389
1.425 5.12505667591557
1.4375 5.12601958222691
1.45 5.12697532806824
1.4625 5.1279246527241
1.475 5.12886833432396
1.4875 5.1298071850436
1.5 5.13074204609637
1.5125 5.13167378254562
1.525 5.13260327797008
1.5375 5.13353142901501
1.55 5.1344591398631
1.5625 5.13538731665855
1.575 5.13631686191952
1.5875 5.1372486689726
1.6 5.13818361644417
1.6125 5.13912256284248
1.625 5.14006634126364
1.6375 5.14101575425442
1.65 5.14197156886337
1.6625 5.14293451191093
1.675 5.14390526550745
1.6875 5.14488446284702
1.7 5.14587268430284
1.7125 5.14687045384823
1.725 5.14787823582555
1.7375 5.14889643208273
1.75 5.14992537949493
1.7625 5.15096534788709
1.775 5.15201653836979
1.7875 5.1530790820987
1.8 5.15415303946539
1.8125 5.15523839972454
1.825 5.15633508105965
1.8375 5.15744293108696
1.85 5.15856172779453
1.8625 5.15969118091054
1.875 5.16083093369252
1.8875 5.16198056512631
1.9 5.16313959252152
1.9125 5.16430747448691
1.925 5.16548361426812
1.9375 5.16666736342631
1.95 5.16785802583594
1.9625 5.16905486197603
1.97499999999999 5.17025709348955
1.98749999999999 5.17146390798179
1.99999999999999 5.17267446402891
2.01249999999999 5.17388789636529
2.025 5.17510332121804
2.0375 5.17631984175543
2.05 5.17753655361613
2.0625 5.17875255048456
2.075 5.17996692967858
2.0875 5.18117879771455
2.1 5.18238727581593
2.1125 5.18359150533104
2.125 5.18479065302682
2.1375 5.18598391622588
2.15 5.18717052775519
2.1625 5.18834976069071
2.175 5.18952093278692
2.1875 5.19068341075003
2.2 5.19183661410472
2.2125 5.19298001878165
2.225 5.19411316034758
2.2375 5.19523563686862
2.25 5.19634711138859
2.2625 5.19744731400616
2.275 5.19853604353878
2.2875 5.19961316876395
2.3 5.20067862922675
2.3125 5.20173243561361
2.325 5.20277466968638
2.3375 5.20380548377802
2.35 5.2048250998532
2.3625 5.20583380814034
2.375 5.20683196534221
2.3875 5.20781999243633
2.4 5.20879837206404
2.4125 5.20976764561091
2.425 5.21072840979685
2.4375 5.21168131303702
2.45 5.21262705143608
2.4625 5.21356636449678
2.475 5.21450003056072
2.4875 5.21542886200973
2.5 5.21635370025733
2.5125 5.21727541056121
2.525 5.21819487668876
2.5375 5.21911299546834
2.55 5.22003067126019
2.5625 5.22094881038089
2.575 5.2218683155155
2.5875 5.22279008015258
2.6 5.22371498307542
2.6125 5.22464388294425
2.625 5.22557761300239
2.6375 5.22651697593912
2.65 5.22746273894101
2.6625 5.2284156289619
2.675 5.22937632824133
2.6875 5.23034547009822
2.7 5.23132363502644
2.7125 5.23231134711603
2.725 5.23330907082203
2.73750000000001 5.23431720810113
2.75000000000001 5.23533609593361
2.76250000000001 5.23636600424579
2.77500000000001 5.23740713424606
2.78750000000001 5.23845961718438
2.80000000000001 5.2395235135433
2.81250000000001 5.2405988126651
2.82500000000001 5.2416854328177
2.83750000000001 5.24278322169855
2.85000000000001 5.24389195737397
2.86250000000001 5.24501134964735
2.87500000000001 5.24614104184853
2.88750000000001 5.24728061303292
2.90000000000001 5.24842958057688
2.91250000000001 5.24958740315335
2.92500000000001 5.25075348406945
2.93750000000001 5.25192717494542
2.95000000000001 5.25310777971222
2.96250000000001 5.25429455890316
2.97500000000001 5.25548673421301
2.98750000000001 5.25668349329679
3.00000000000001 5.25788399477805
};
\addlegendentry{$P_2$}
\addplot [semithick, color2]
table {%
0 5
0.0125 5.0012259546192
0.025 5.00245371939431
0.0375 5.00368236476434
0.05 5.00491099989062
0.0625 5.00613875039364
0.075 5.00736475592082
0.0875 5.00858817247441
0.1 5.00980817610298
0.1125 5.01102396712922
0.125 5.01223477383159
0.1375 5.01343985714127
0.15 5.01463851435105
0.1625 5.01583008304346
0.175 5.01701394479598
0.1875 5.01818952868682
0.2 5.01935631457257
0.2125 5.02051383610667
0.225 5.02166168347233
0.2375 5.02279950580568
0.25 5.02392701328446
0.2625 5.02504397886469
0.275 5.02615023964813
0.2875 5.02724569786817
0.3 5.02833032148396
0.3125 5.02940414437614
0.325 5.03046726614094
0.3375 5.03151985148242
0.35 5.03256212920449
0.3625 5.03359439080879
0.375 5.03461698870635
0.3875 5.03563033405286
0.4 5.03663489422183
0.4125 5.03763118993027
0.425 5.03861979203374
0.4375 5.03960131787028
0.45 5.04057642794343
0.4625 5.04154582125572
0.475 5.04251023131712
0.4875 5.04347042156749
0.5 5.04442718067223
0.5125 5.04538131766201
0.525 5.04633365694662
0.5375 5.04728503323348
0.55 5.04823628638281
0.5625 5.04918825623068
0.575 5.05014177741178
0.5875 5.05109767421467
0.6 5.05205675550037
0.6125 5.05301980971597
0.625 5.05398760003416
0.6375 5.05496085964842
0.65 5.05594028725312
0.6625 5.05692654273583
0.675 5.05792024310954
0.6875 5.05892195870831
0.7 5.05993220967129
0.712499999999999 5.06095146273598
0.724999999999999 5.06198012836029
0.737499999999999 5.06301855819186
0.749999999999999 5.0640670428999
0.762499999999999 5.06512581038262
0.774999999999999 5.06619502436185
0.787499999999999 5.06727478337294
0.799999999999999 5.06836512015683
0.812499999999999 5.06946600145737
0.824999999999999 5.07057732822536
0.837499999999999 5.07169893622873
0.849999999999999 5.0728305970635
0.862499999999999 5.07397201956135
0.874999999999999 5.07512285158332
0.887499999999999 5.07628268218948
0.899999999999999 5.07745104417154
0.912499999999999 5.07862741693136
0.924999999999999 5.07981122968942
0.937499999999999 5.08100186500212
0.949999999999999 5.08219866256685
0.962499999999999 5.08340092329126
0.974999999999999 5.0846079136014
0.987499999999998 5.08581886996289
0.999999999999998 5.08703300358606
1.0125 5.08824950528738
1.025 5.08946755047624
1.0375 5.09068630423662
1.05 5.09190492647268
1.0625 5.09312257708606
1.075 5.09433842115303
1.0875 5.09555163407005
1.1 5.09676140663539
1.1125 5.0979669500361
1.125 5.09916750070861
1.1375 5.10036232504452
1.15 5.10155072405216
1.1625 5.10273203719035
1.175 5.10390564701569
1.1875 5.10507098268906
1.2 5.1062275235461
1.2125 5.10737480230511
1.225 5.10851240794557
1.2375 5.10963998823989
1.25 5.11075725192298
1.2625 5.11186397048549
1.275 5.11295997958086
1.2875 5.1140451800357
1.3 5.11511953846003
1.3125 5.11618308744985
1.325 5.11723592538497
1.3375 5.11827821581997
1.35 5.11931018647372
1.3625 5.12033212782365
1.375 5.12134439131254
1.3875 5.1223473871797
1.4 5.12334158193024
1.4125 5.12432749531371
1.425 5.12530569760057
1.4375 5.1262768054705
1.45 5.12724147851956
1.4625 5.12820041513437
1.475 5.12915434818503
1.4875 5.13010404050472
1.5 5.13105028018495
1.5125 5.13199387571507
1.525 5.13293565099598
1.5375 5.13387644025888
1.55 5.13481708292043
1.5625 5.13575841840573
1.575 5.1367012809717
1.5875 5.1376464945619
1.6 5.13859486772507
1.6125 5.13954718862914
1.625 5.14050422020009
1.6375 5.14146669541702
1.65 5.14243531279169
1.6625 5.14341073206077
1.675 5.14439357011679
1.6875 5.14538439720378
1.7 5.14638373340012
1.7125 5.14739204541084
1.725 5.14840974368923
1.7375 5.14943717990489
1.75 5.15047464477456
1.7625 5.1515223662685
1.775 5.1525805082041
1.7875 5.15364916923463
1.8 5.15472838224008
1.8125 5.15581811412339
1.825 5.15691826601349
1.8375 5.15802867387402
1.85 5.15914910951355
1.8625 5.1602792819917
1.875 5.16141883941188
1.8875 5.16256737108989
1.9 5.1637244100854
1.9125 5.16488943607947
1.925 5.16606187858214
1.9375 5.16724112044854
1.95 5.16842650168312
1.9625 5.16961732350809
1.97499999999999 5.17081285267089
1.98749999999999 5.17201232596382
1.99999999999999 5.17321495492939
2.01249999999999 5.17441993072013
2.025 5.17562642908519
2.0375 5.17683361545113
2.05 5.17804065006693
2.0625 5.17924669318061
2.075 5.18045091021597
2.0875 5.18165247691743
2.1 5.18285058443129
2.1125 5.18404444429205
2.125 5.18523329328287
2.1375 5.18641639814075
2.15 5.18759306007656
2.1625 5.18876261922377
2.175 5.18992445823364
2.1875 5.19107800665678
2.2 5.1922227441626
2.2125 5.19335820379973
2.225 5.19448397487471
2.2375 5.19559970548352
2.25 5.19670510468076
2.2625 5.19779994427293
2.275 5.19888406022518
2.2875 5.19995735367131
2.3 5.20101979152441
2.3125 5.20207140667927
2.325 5.20311229780951
2.3375 5.20414262875912
2.35 5.20516262753175
2.3625 5.20617258488472
2.375 5.20717285253576
2.3875 5.20816384099493
2.4 5.20914601689008
2.4125 5.21011990057145
2.425 5.21108606231144
2.4375 5.21204511909746
2.45 5.21299773077146
2.4625 5.21394459596145
2.475 5.21488644777381
2.4875 5.21582404927345
2.5 5.21675818877885
2.5125 5.21768967500157
2.525 5.21861933206023
2.5375 5.21954799439902
2.55 5.22047650164295
2.5625 5.22140569342121
2.575 5.22233640419005
2.5875 5.22326945808825
2.6 5.22420566385535
2.6125 5.22514580984566
2.625 5.22609065916748
2.6375 5.22704094497811
2.65 5.22799736596333
2.6625 5.22896058202977
2.675 5.22993121023619
2.6875 5.23090982098879
2.7 5.23189693452425
2.7125 5.2328930177025
2.725 5.23389848112767
2.73750000000001 5.23491367661678
2.75000000000001 5.23593889503048
2.76250000000001 5.23697436447948
2.77500000000001 5.23802024891809
2.78750000000001 5.23907664713334
2.80000000000001 5.24014359213563
2.81250000000001 5.24122105095518
2.82500000000001 5.24230892484506
2.83750000000001 5.24340704988992
2.85000000000001 5.24451519801648
2.86250000000001 5.24563307839956
2.87500000000001 5.24676033925476
2.88750000000001 5.2478965700075
2.90000000000001 5.24904130382413
2.91250000000001 5.25019402048996
2.92500000000001 5.25135414961643
2.93750000000001 5.25252107415758
2.95000000000001 5.25369413421447
2.96250000000001 5.25487263110326
2.97500000000001 5.256055831663
2.98750000000001 5.25724297277546
3.00000000000001 5.25843326607006
};
\addlegendentry{$P_3$}
\end{axis}

%% file: suppl_tex/p2_T3.tex
\definecolor{color0}{rgb}{0.12156862745098,0.466666666666667,0.705882352941177}
\definecolor{color1}{rgb}{1,0.498039215686275,0.0549019607843137}
\definecolor{color2}{rgb}{0.172549019607843,0.627450980392157,0.172549019607843}

\begin{axis}[
legend cell align={left},
legend style={
  fill opacity=0.8,
  draw opacity=1,
  text opacity=1,
  at={(0.03,0.97)},
  anchor=north west,
  draw=white!80!black
},
tick align=outside,
tick pos=left,
x grid style={white!69.0196078431373!black},
xlabel={t},
xmin=-0.15, xmax=3.15000000000001,
xtick style={color=black},
y grid style={white!69.0196078431373!black},
ymin=64.7083121081905, ymax=75.7436013972821,
ytick style={color=black}
]
\addplot [semithick, color0]
table {%
0 70
0.0125 70.0971843252621
0.025 70.2645028987071
0.0375 70.4802603931172
0.05 70.7301303437302
0.0625 71.0041862616179
0.075 71.295074620913
0.0875 71.596990834828
0.1 71.9051088000068
0.1125 72.2152594491651
0.125 72.5237451891169
0.1375 72.8272295871542
0.15 73.1226698630945
0.1625 73.4072747926066
0.175 73.6784783070162
0.1875 73.9339231451265
0.2 74.1714511065934
0.2125 74.389097691372
0.225 74.585089638898
0.2375 74.7578443331618
0.25 74.9059703336396
0.2625 75.0282684906644
0.275 75.1237332426023
0.2875 75.1915537918769
0.3 75.2311149302251
0.3125 75.241997338687
0.325 75.2239772301595
0.3375 75.1770252354536
0.35 75.1013044602737
0.3625 74.9971676622454
0.375 74.8651535150312
0.3875 74.7059819420057
0.4 74.5205485151621
0.4125 74.3099179266053
0.425 74.0753165503661
0.4375 73.8181241118995
0.45 73.5398645492517
0.4625 73.2421960123569
0.475 72.9269001457719
0.4875 72.5958706656124
0.5 72.2511013008238
0.5125 71.8946731660121
0.525 71.5287416375565
0.5375 71.1555228085597
0.55 70.7772796014546
0.5625 70.3963076198611
0.575 70.0149208233207
0.5875 69.6354371101919
0.6 69.2601638948608
0.6125 68.8913837658525
0.625 68.5313403112201
0.6375 68.1822241967792
0.65 67.8461595814764
0.6625 67.5251909522709
0.675 67.2212704584847
0.6875 66.9362458226781
0.7 66.6718489016589
0.712499999999999 66.4296849673537
0.724999999999999 66.2112227729159
0.737499999999999 66.0177854647328
0.749999999999999 65.8505423958001
0.762499999999999 65.7105018904818
0.774999999999999 65.598505004856
0.787499999999999 65.5152203207556
0.799999999999999 65.4611398052954
0.812499999999999 65.4365757611377
0.824999999999999 65.4416588860569
0.837499999999999 65.4763374535723
0.849999999999999 65.5403776195092
0.862499999999999 65.6333648524455
0.874999999999999 65.7547064790873
0.887499999999999 65.9036353287612
0.899999999999999 66.0792144544531
0.912499999999999 66.2803429011819
0.924999999999999 66.5057624860942
0.937499999999999 66.7540655483691
0.949999999999999 67.0237036211404
0.962499999999999 67.3129969718852
0.974999999999999 67.6201449524133
0.987499999999998 67.9432370946391
0.999999999999998 68.2802648836661
1.0125 68.6291341355692
1.025 68.9876779035619
1.0375 69.3536698329235
1.05 69.7248378823665
1.0625 70.0988783271878
1.075 70.4734699578964
1.0875 70.8462883867587
1.1 71.2150203740145
1.1125 71.5773780854811
1.125 71.9311131935893
1.1375 72.2740307349408
1.15 72.6040026493127
1.1625 72.9189808644344
1.175 73.2170099502849
1.1875 73.4962391615829
1.2 73.7549338440064
1.2125 73.9914861250283
1.225 74.2044248232847
1.2375 74.3924245152873
1.25 74.5543137035532
1.2625 74.6890820357901
1.275 74.7958865306132
1.2875 74.8740567714242
1.3 74.9230990364279
1.3125 74.9426993393161
1.325 74.9327253618834
1.3375 74.8932272666505
1.35 74.8244373844823
1.3625 74.7267687791555
1.375 74.600812697669
1.3875 74.4473349220857
1.4 74.2672710453339
1.4125 74.0617206903463
1.425 73.8319407573052
1.4375 73.5793376425492
1.45 73.3054585738406
1.4625 73.0119820690016
1.475 72.700707584887
1.4875 72.3735444206789
1.5 72.0324999439694
1.5125 71.6796672122755
1.525 71.3172120663011
1.5375 70.9473597745128
1.55 70.5723813113862
1.5625 70.194579353916
1.575 69.816274082708
1.5875 69.4397888752194
1.6 69.0674359793357
1.6125 68.7015022556182
1.625 68.3442350761148
1.6375 67.9978284666896
1.65 67.6644095782826
1.6625 67.3460255705434
1.675 67.0446309886619
1.6875 66.762075711263
1.7 66.5000935436368
1.7125 66.2602915266474
1.725 66.0441400272365
1.7375 65.8529636716009
1.75 65.6879331769534
1.7625 65.5500581322346
1.775 65.4401807722572
1.7875 65.3589707837043
1.8 65.3069211749572
1.8125 65.2843452352701
1.825 65.2913746019849
1.8375 65.3279584477438
1.85 65.3938637926836
1.8625 65.488676939699
1.875 65.6118060239213
1.8875 65.7624846606856
1.9 65.9397766695269
1.9125 66.1425818450185
1.925 66.3696427389374
1.9375 66.619552411911
1.95 66.8907631067376
1.9625 67.1815957899455
1.97499999999999 67.4902505027286
1.98749999999999 67.8148174574509
1.99999999999999 68.1532888113253
2.01249999999999 68.5035710446398
2.025 68.8634978672498
2.0375 69.2308435737403
2.05 69.6033367649281
2.0625 69.9786743511052
2.075 70.3545357506898
2.0875 70.728597196731
2.1 71.0985460630841
2.1125 71.4620951219106
2.125 71.8169966446227
2.1375 72.1610562592945
2.15 72.4921464791615
2.1625 72.8082198291196
2.175 73.1073214371676
2.1875 73.3876011175049
2.2 73.6473247675212
2.2125 73.884885058211
2.225 74.0988113432839
2.2375 74.2877787257179
2.25 74.4506162257783
2.2625 74.5863140000884
2.275 74.6940295672714
2.2875 74.7730930017412
2.3 74.8230110636482
2.3125 74.8434702395149
2.325 74.8343386747982
2.3375 74.7956669864818
2.35 74.7276879506795
2.3625 74.6308150671718
2.375 74.5056400097381
2.3875 74.3529289779888
2.4 74.173617963385
2.4125 73.9688070078915
2.425 73.7397533923253
2.4375 73.4878638934143
2.45 73.2146861106936
2.4625 72.9218989249536
2.475 72.6113021471722
2.4875 72.2848054218445
2.5 71.9444164531556
2.5125 71.5922286265819
2.525 71.2304081022598
2.5375 70.8611804596668
2.55 70.4868169759796
2.5625 70.1096206226912
2.575 69.7319118668345
2.5875 69.356014364341
2.6 68.9842406337384
2.6125 68.6188777985312
2.625 68.2621734861458
2.6375 67.9163219703801
2.65 67.5834506428176
2.6625 67.2656068965797
2.675 66.9647455033119
2.6875 66.6827165611995
2.7 66.4212540883523
2.7125 66.1819653318484
2.725 65.9663208583755
2.73750000000001 65.775645487554
2.75000000000001 65.6111101238395
2.76250000000001 65.4737245373684
2.77500000000001 65.3643311382502
2.78750000000001 65.2835997826909
2.80000000000001 65.2320236429804
2.81250000000001 65.2099161667856
2.82500000000001 65.2174091445107
2.83750000000001 65.2544518966449
2.85000000000001 65.3208115860816
2.86250000000001 65.4160746535272
2.87500000000001 65.5396493670945
2.88750000000001 65.6907694704178
2.90000000000001 65.8684989067524
2.91250000000001 66.0717375899592
2.92500000000001 66.2992281867893
2.93750000000001 66.5495638686375
2.95000000000001 66.8211969849989
2.96250000000001 67.1124486051389
2.97500000000001 67.4215188691414
2.98750000000001 67.746498084539
3.00000000000001 68.0853785000804
};
\addlegendentry{$P_1$}
\addplot [semithick, color1]
table {%
0 70
0.0125 69.9987143375215
0.025 69.9974290578247
0.0375 69.9961462827902
0.05 69.9948714409159
0.0625 69.9936148534705
0.075 69.9923934164849
0.0875 69.9912321468361
0.1 69.9901653327476
0.1125 69.9892371107716
0.125 69.9885014252764
0.1375 69.988021402911
0.15 69.98786824646
0.1625 69.988119792338
0.175 69.9888588511644
0.1875 69.990171443172
0.2 69.9921450169377
0.2125 69.9948667168454
0.225 69.9984217443515
0.2375 70.0028918419544
0.25 70.0083539163545
0.2625 70.014878808611
0.275 70.022530213073
0.2875 70.0313637430614
0.3 70.0414261389475
0.3125 70.0527546129926
0.325 70.0653763247504
0.3375 70.0793079806315
0.35 70.094555551282
0.3625 70.1111141006645
0.375 70.1289677208404
0.3875 70.1480895667761
0.4 70.1684419855909
0.4125 70.1899767348556
0.425 70.2126352846106
0.4375 70.2363492004392
0.45 70.2610405923456
0.4625 70.2866226350262
0.475 70.3130001582439
0.4875 70.3400702867948
0.5 70.3677231334112
0.5125 70.3958425383934
0.525 70.424306850145
0.5375 70.4529897408826
0.55 70.4817610517495
0.5625 70.5104876616049
0.575 70.5390343736268
0.5875 70.5672648139618
0.6 70.5950423365837
0.6125 70.6222309285831
0.625 70.6486961101836
0.6375 70.6743058237909
0.65 70.6989313065659
0.6625 70.7224479410938
0.675 70.7447360788935
0.6875 70.7656818317043
0.7 70.785177825694
0.712499999999999 70.8031239139771
0.724999999999999 70.8194278430786
0.737499999999999 70.8340058693187
0.749999999999999 70.8467833213344
0.762499999999999 70.8576951053517
0.774999999999999 70.8666861501357
0.787499999999999 70.8737117889404
0.799999999999999 70.8787380761546
0.812499999999999 70.8817420367517
0.824999999999999 70.8827118470431
0.837499999999999 70.8816469456897
0.849999999999999 70.8785580743152
0.862499999999999 70.8734672475277
0.874999999999999 70.8664076525644
0.887499999999999 70.8574234792289
0.899999999999999 70.8465696811932
0.912499999999999 70.833911670166
0.924999999999999 70.8195249448667
0.937499999999999 70.8034946570753
0.949999999999999 70.7859151175124
0.962499999999999 70.7668892445803
0.974999999999999 70.7465279593812
0.987499999999998 70.7249495307983
0.999999999999998 70.7022788746632
1.0125 70.6786468113425
1.025 70.6541892863686
1.0375 70.6290465588886
1.05 70.6033623629941
1.0625 70.5772830470608
1.075 70.5509566964552
1.0875 70.5245322450239
1.1 70.4981585808116
1.1125 70.4719836515896
1.125 70.4461535756573
1.1375 70.4208117634541
1.15 70.3960980524539
1.1625 70.3721478723088
1.175 70.3490914314008
1.1875 70.3270529288539
1.2 70.306149810588
1.2125 70.2864920654605
1.225 70.2681815665531
1.2375 70.2513114619212
1.25 70.2359656186475
1.2625 70.2222181236815
1.275 70.2101328445384
1.2875 70.1997630525772
1.3 70.1911511112117
1.3125 70.1843282309407
1.325 70.17931429275
1.3375 70.1761177409654
1.35 70.1747355462107
1.3625 70.1751532387399
1.375 70.1773450118937
1.3875 70.1812738950949
1.4 70.1868919952896
1.4125 70.1941408080615
1.425 70.2029515860203
1.4375 70.2132457739487
1.45 70.2249355111851
1.4625 70.237924183423
1.475 70.2521070291938
1.4875 70.2673717966513
1.5 70.2835994463311
1.5125 70.3006648954955
1.525 70.3184377994367
1.5375 70.3367833649175
1.55 70.3555631907396
1.5625 70.3746361302704
1.575 70.3938591705893
1.5875 70.4130883228724
1.6 70.43217951851
1.6125 70.4509895054509
1.625 70.4693767392369
1.6375 70.4872022632652
1.65 70.5043305728246
1.6625 70.5206304576216
1.675 70.5359758175718
1.6875 70.5502464468628
1.7 70.563328781435
1.7125 70.5751166052925
1.725 70.5855117112982
1.7375 70.5944245123665
1.75 70.6017745993039
1.7625 70.6074912418721
1.775 70.61151382996
1.7875 70.613792252187
1.8 70.6142872095689
1.8125 70.6129704623593
1.825 70.6098250085102
1.8375 70.6048451926864
1.85 70.5980367451472
1.8625 70.5894167502682
1.875 70.5790135449041
1.8875 70.5668665472089
1.9 70.5530260170101
1.9125 70.537552749154
1.925 70.520517701786
1.9375 70.5020015618276
1.95 70.4820942503187
1.9625 70.4608943707186
1.97499999999999 70.4385086035364
1.98749999999999 70.4150510510375
1.99999999999999 70.3906425360868
2.01249999999999 70.3654098594223
2.025 70.3394850199748
2.0375 70.3130044030179
2.05 70.2861079411737
2.0625 70.2589382534437
2.075 70.2316397675809
2.0875 70.2043578312032
2.1 70.1772378171508
2.1125 70.1504242285927
2.125 70.1240598094134
2.1375 70.0982846653445
2.15 70.0732354013023
2.1625 70.0490442772709
2.175 70.0258383996808
2.1875 70.0037389390562
2.2 69.9828603780227
2.2125 69.9633098078945
2.225 69.9451862696727
2.2375 69.9285801442271
2.25 69.9135725956827
2.2625 69.9002350715156
2.275 69.8886288624984
2.2875 69.8788047251783
2.3 69.8708025692723
2.3125 69.864651211878
2.325 69.8603682000167
2.3375 69.8579597026214
2.35 69.8574204726284
2.3625 69.8587338793996
2.375 69.8618720112912
2.3875 69.8667958477311
2.4 69.873455502461
2.4125 69.8817905258323
2.425 69.8917302763258
2.4375 69.9031943618754
2.45 69.9160931336718
2.4625 69.930328238079
2.475 69.9457932225673
2.4875 69.9623741916817
2.5 69.9799505089155
2.5125 69.9983955401285
2.525 70.0175774339258
2.5375 70.0373599341701
2.55 70.057603219627
2.5625 70.0781647655653
2.575 70.0989002220083
2.5875 70.1196643032094
2.6 70.14031168287
2.6125 70.1606978896072
2.625 70.1806801971119
2.6375 70.2001185035354
2.65 70.2188761946862
2.6625 70.2368209856964
2.675 70.2538257359963
2.6875 70.2697692325371
2.7 70.2845369364767
2.7125 70.2980216886865
2.725 70.3101243697586
2.73750000000001 70.3207545104289
2.75000000000001 70.3298308486587
2.76250000000001 70.337281829944
2.77500000000001 70.3430460477596
2.78750000000001 70.3470726214278
2.80000000000001 70.3493215090936
2.81250000000001 70.3497637538564
2.82500000000001 70.3483816615515
2.83750000000001 70.3451689091009
2.85000000000001 70.3401305827256
2.86250000000001 70.333283145839
2.87500000000001 70.3246543367641
2.88750000000001 70.3142829969543
2.90000000000001 70.3022188307399
2.91250000000001 70.2885220981043
2.92500000000001 70.2732632423792
2.93750000000001 70.256522455142
2.95000000000001 70.2383891810202
2.96250000000001 70.2189615654451
2.97500000000001 70.1983458487501
2.98750000000001 70.1766557103721
3.00000000000001 70.1540115671688
};
\addlegendentry{$P_2$}
\addplot [semithick, color2]
table {%
0 70
0.0125 69.9987128298484
0.025 69.9974255810085
0.0375 69.99613823766
0.05 69.9948514437165
0.0625 69.9935666358856
0.075 69.9922862901942
0.0875 69.9910142765999
0.1 69.9897563100382
0.1125 69.9885204815196
0.125 69.987317693847
0.1375 69.9861621943802
0.15 69.9850718873331
0.1625 69.9840685368777
0.175 69.9831778141291
0.1875 69.9824291876577
0.2 69.9818556690104
0.2125 69.9814934325513
0.225 69.9813813330612
0.2375 69.9815603459431
0.25 69.9820729541648
0.2625 69.982962504072
0.275 69.9842725494655
0.2875 69.9860462002956
0.3 69.9883254893113
0.3125 69.9911507671729
0.325 69.994560134048
0.3375 69.9985889135587
0.35 70.0032691731344
0.3625 70.0086292934538
0.375 70.0146935883997
0.3875 70.0214819761679
0.4 70.0290097014309
0.4125 70.0372871079828
0.425 70.0463194609352
0.4375 70.0561067955453
0.45 70.066643905893
0.4625 70.0779202283003
0.475 70.0899198784748
0.4875 70.1026216999965
0.5 70.1159993530896
0.5125 70.1300214419493
0.525 70.1446516786978
0.5375 70.1598490820065
0.55 70.1755682083445
0.5625 70.1917594138152
0.575 70.2083691443875
0.5875 70.22534025234
0.6 70.242612336608
0.6125 70.2601221046682
0.625 70.2778037535502
0.6375 70.2955893674338
0.65 70.3134093293003
0.6625 70.331192744011
0.675 70.3488678701468
0.6875 70.3663625579225
0.7 70.3836046904473
0.712499999999999 70.4005226256146
0.724999999999999 70.4170456358716
0.737499999999999 70.4331043431973
0.749999999999999 70.4486311465729
0.762499999999999 70.4635606393382
0.774999999999999 70.4778300138523
0.787499999999999 70.4913794509691
0.799999999999999 70.5041524919232
0.812499999999999 70.5160963903334
0.824999999999999 70.5271624421324
0.837499999999999 70.5373062913946
0.849999999999999 70.5464882101447
0.862499999999999 70.5546733504166
0.874999999999999 70.5618319669744
0.887499999999999 70.5679396093153
0.899999999999999 70.572977281727
0.912499999999999 70.5769315703876
0.924999999999999 70.5797947367113
0.937499999999999 70.5815647762957
0.949999999999999 70.5822454431185
0.962499999999999 70.5818462387698
0.974999999999999 70.580382366744
0.987499999999998 70.5778746520748
0.999999999999998 70.5743494267427
1.0125 70.5698383815233
1.025 70.5643783851839
1.0375 70.5580112720714
1.05 70.5507835994053
1.0625 70.5427463756945
1.075 70.5339547619562
1.0875 70.524467747543
1.1 70.5143478025137
1.1125 70.5036605087104
1.125 70.49247417174
1.1375 70.4808594162634
1.15 70.4688887896874
1.1625 70.4566362573103
1.175 70.4441768450603
1.1875 70.4315861619701
1.2 70.4189399506473
1.2125 70.4063136369098
1.225 70.3937818812107
1.2375 70.3814181346437
1.25 70.3692942022713
1.2625 70.3574798164769
1.275 70.3460422229386
1.2875 70.3350457817948
1.3 70.3245515864316
1.3125 70.3146171022395
1.325 70.3052958275532
1.3375 70.2966369788634
1.35 70.2886852022157
1.3625 70.2814803126074
1.375 70.2750570629217
1.3875 70.2694449438895
1.4 70.2646680162368
1.4125 70.2607447543288
1.425 70.2576880172575
1.4375 70.2555049031756
1.45 70.2541967492582
1.4625 70.2537591360964
1.475 70.2541819272985
1.4875 70.2554493442113
1.5 70.2575400752232
1.5125 70.2604274189709
1.525 70.2640794605416
1.5375 70.2684592795822
1.55 70.2735251890328
1.5625 70.2792310030179
1.575 70.2855263322247
1.5875 70.2923569049712
1.6 70.2996649119711
1.6125 70.307389372683
1.625 70.3154665209688
1.6375 70.323830207713
1.65 70.3324123178911
1.6625 70.3411431995538
1.675 70.3499521020483
1.6875 70.358767620813
1.7 70.3675181459758
1.7125 70.3761323120116
1.725 70.3845394456913
1.7375 70.3926700095483
1.75 70.4004560381438
1.7625 70.4078315644519
1.775 70.4147330337122
1.7875 70.4210997022375
1.8 70.4268740186821
1.8125 70.4320019854665
1.825 70.4364334980921
1.8375 70.4401226602913
1.85 70.4430280730498
1.8625 70.4451130957358
1.875 70.4463460777264
1.8875 70.4467005590984
1.9 70.4461554391791
1.9125 70.4446951118658
1.925 70.4423095669428
1.9375 70.4389944567466
1.95 70.4347511277569
1.9625 70.4295866169586
1.97499999999999 70.4235136129689
1.98749999999999 70.4165503821818
1.99999999999999 70.4087206604051
2.01249999999999 70.4000535106283
2.025 70.3905831478373
2.0375 70.3803487319272
2.05 70.369394130015
2.0625 70.357767649612
2.075 70.3455217443069
2.0875 70.3327126937666
2.1 70.3194002600466
2.1125 70.3056473223211
2.125 70.2915194922978
2.1375 70.2770847126668
2.15 70.2624128411053
2.1625 70.2475752449405
2.175 70.232644289951
2.1875 70.2176929788766
2.2 70.2027944724217
2.2125 70.1880216407889
2.225 70.173446616949
2.2375 70.159140354277
2.25 70.1451721913455
2.2625 70.131609426557
2.275 70.1185169052803
2.2875 70.1059566220219
2.3 70.0939873400917
2.3125 70.0826642311325
2.325 70.0720385367115
2.3375 70.0621572540803
2.35 70.0530628480329
2.3625 70.0447929906463
2.375 70.0373803305125
2.3875 70.030852292872
2.4 70.0252308901068
2.4125 70.0205326588708
2.425 70.0167684796475
2.4375 70.0139435415821
2.45 70.0120573114572
2.4625 70.0111035378801
2.475 70.0110702907723
2.4875 70.0119400358731
2.5 70.0136897437766
2.5125 70.0162910328112
2.525 70.01971034489
2.5375 70.0239091532292
2.55 70.0288442006665
2.5625 70.0344677671057
2.575 70.0407279644458
2.5875 70.0475690571649
2.6 70.0549318065844
2.6125 70.0627538367144
2.625 70.0709700193907
2.6375 70.0795128763513
2.65 70.0883129957785
2.6625 70.0972994607194
2.675 70.1064002867702
2.6875 70.1155428662939
2.7 70.1246544164713
2.7125 70.1336624283899
2.725 70.1424951144259
2.73750000000001 70.1510818511506
2.75000000000001 70.1593536150332
2.76250000000001 70.1672434082598
2.77500000000001 70.1746866720327
2.78750000000001 70.1816216848117
2.80000000000001 70.1879899430473
2.81250000000001 70.1937365220449
2.82500000000001 70.1988104147417
2.83750000000001 70.2031648463238
2.85000000000001 70.2067575627068
2.86250000000001 70.2095510911538
2.87500000000001 70.2115129713691
2.88750000000001 70.212615955695
2.90000000000001 70.2128381771363
2.91250000000001 70.2121632842029
2.92500000000001 70.2105805417334
2.93750000000001 70.2080848970651
2.95000000000001 70.2046770111664
2.96250000000001 70.200363254527
2.97500000000001 70.1951556678245
2.98750000000001 70.189071887631
3.00000000000001 70.182135037588
};
\addlegendentry{$P_3$}
\end{axis}

%% file: suppl_tex/p3_T3.tex
\definecolor{color0}{rgb}{0.12156862745098,0.466666666666667,0.705882352941177}
\definecolor{color1}{rgb}{1,0.498039215686275,0.0549019607843137}
\definecolor{color2}{rgb}{0.172549019607843,0.627450980392157,0.172549019607843}

\begin{axis}[
legend cell align={left},
legend style={
  fill opacity=0.8,
  draw opacity=1,
  text opacity=1,
  at={(0.03,0.97)},
  anchor=north west,
  draw=white!80!black
},
tick align=outside,
tick pos=left,
x grid style={white!69.0196078431373!black},
xlabel={t},
xmin=-0.15, xmax=3.15000000000001,
xtick style={color=black},
y grid style={white!69.0196078431373!black},
ymin=5.997030060359, ymax=6.06236873246091,
ytick style={color=black}
]
\addplot [semithick, color0]
table {%
0 6
0.0125 6.01173755870887
0.025 6.01893883168357
0.0375 6.0238941666425
0.05 6.02750101553843
0.0625 6.03021018158562
0.075 6.03228455279884
0.0875 6.0338931467446
0.1 6.03515164010748
0.1125 6.03614251078226
0.125 6.03692625250053
0.1375 6.03754818198772
0.15 6.03804279253206
0.1625 6.03843667113276
0.175 6.03875051430655
0.1875 6.0390005590189
0.2 6.03919962189094
0.2125 6.03935786875885
0.225 6.03948339398643
0.2375 6.03958266255087
0.25 6.03966085118331
0.2625 6.03972211397321
0.275 6.03976979063452
0.2875 6.03980657071946
0.3 6.03983462365823
0.3125 6.03985570208055
0.325 6.03987122411889
0.3375 6.03988233909611
0.35 6.03988998002913
0.3625 6.03989490564274
0.375 6.03989773401432
0.3875 6.03989896953307
0.4 6.03989902451204
0.4125 6.03989823651666
0.425 6.03989688226568
0.4375 6.03989518651681
0.45 6.0398933394729
0.4625 6.03989149262914
0.475 6.03988977049334
0.4875 6.0398882746456
0.5 6.03988708727668
0.5125 6.03988627412895
0.525 6.03988588689611
0.5375 6.03988596517432
0.55 6.03988653804812
0.5625 6.03988762539033
0.575 6.0398892389284
0.5875 6.03989138313392
0.6 6.03989405596939
0.6125 6.03989724952444
0.625 6.0399009505704
0.6375 6.03990514104663
0.65 6.03990979850053
0.6625 6.03991489649055
0.675 6.03992040496273
0.6875 6.03992629060818
0.7 6.0399325172079
0.712499999999999 6.03993904596874
0.724999999999999 6.03994583585161
0.737499999999999 6.03995284389986
0.749999999999999 6.03996002556182
0.762499999999999 6.03996733501382
0.774999999999999 6.0399747254814
0.787499999999999 6.039982149558
0.799999999999999 6.03998955952361
0.812499999999999 6.03999690765843
0.824999999999999 6.04000414655424
0.837499999999999 6.0400112294191
0.849999999999999 6.04001811037756
0.862499999999999 6.04002474476157
0.874999999999999 6.04003108939279
0.887499999999999 6.04003710285511
0.899999999999999 6.04004274575434
0.912499999999999 6.04004798096404
0.924999999999999 6.04005277385927
0.937499999999999 6.04005709253004
0.949999999999999 6.04006090798206
0.962499999999999 6.04006419431682
0.974999999999999 6.04006692889121
0.987499999999998 6.04006909245963
0.999999999999998 6.04007066929312
1.0125 6.04007164727644
1.025 6.04007201798426
1.0375 6.04007177673259
1.05 6.0400709226093
1.0625 6.04006945847875
1.075 6.04006739096479
1.0875 6.04006473041097
1.1 6.04006149081538
1.1125 6.04005768974644
1.125 6.04005334823244
1.1375 6.04004849063396
1.15 6.04004314681205
1.1625 6.04003734326048
1.175 6.04003111460806
1.1875 6.040024497216
1.2 6.04001752957672
1.2125 6.04001025226534
1.225 6.04000270774433
1.2375 6.03999494011763
1.25 6.03998699486371
1.2625 6.03997891855682
1.275 6.03997075857893
1.2875 6.03996256282897
1.3 6.03995437942586
1.3125 6.03994625641355
1.325 6.03993824146323
1.3375 6.03993038158053
1.35 6.03992272281394
1.3625 6.03991530997339
1.375 6.03990818635129
1.3875 6.03990139345848
1.4 6.03989497076511
1.4125 6.03988895319045
1.425 6.03988337937248
1.4375 6.03987828005783
1.45 6.0398736839107
1.4625 6.03986961675578
1.475 6.03986610123239
1.4875 6.03986315660128
1.5 6.03986079860906
1.5125 6.03985903938665
1.525 6.03985788737317
1.5375 6.03985734726334
1.55 6.03985741997869
1.5625 6.03985810266244
1.575 6.03985938869676
1.5875 6.03986126774346
1.6 6.03986372580807
1.6125 6.03986674532699
1.625 6.03987030527391
1.6375 6.03987438129163
1.65 6.03987894584021
1.6625 6.03988396836858
1.675 6.03988941550125
1.6875 6.03989525124558
1.7 6.03990143721297
1.7125 6.03990793285461
1.725 6.03991469571366
1.7375 6.03992168168534
1.75 6.03992884528882
1.7625 6.03993613994896
1.775 6.03994351828194
1.7875 6.03995093238861
1.8 6.03995833414854
1.8125 6.03996567551814
1.825 6.03997290882546
1.8375 6.03997998706529
1.85 6.03998686418826
1.8625 6.03999349538484
1.875 6.039999837362
1.8875 6.04000584860903
1.9 6.04001148965581
1.9125 6.04001672331296
1.925 6.04002151490343
1.9375 6.04002583247599
1.95 6.04002964700139
1.9625 6.04003293255199
1.97499999999999 6.04003566646163
1.98749999999999 6.04003782946466
1.99999999999999 6.04003940581586
2.01249999999999 6.04004038338645
2.025 6.04004075373959
2.0375 6.04004051218191
2.05 6.04003965779284
2.0625 6.04003819342989
2.075 6.0400361257118
2.0875 6.04003346497651
2.1 6.04003022521839
2.1125 6.04002642400166
2.125 6.04002208235238
2.1375 6.04001722462792
2.15 6.04001187836806
2.1625 6.04000607644019
2.175 6.03999984816386
2.1875 6.03999323078445
2.2 6.03998626305663
2.2125 6.03997898563098
2.225 6.0399714409912
2.2375 6.03996367324605
2.25 6.03995572787489
2.2625 6.03994765145047
2.275 6.03993949135426
2.2875 6.0399312954849
2.3 6.03992311196075
2.3125 6.03991498882515
2.325 6.03990697374939
2.3375 6.03989911373904
2.35 6.03989145484329
2.3625 6.03988404187152
2.375 6.03987691811751
2.3875 6.03987012509108
2.4 6.03986369999409
2.4125 6.03985768397884
2.425 6.03985211051568
2.4375 6.0398470111924
2.45 6.03984241493562
2.4625 6.03983834764667
2.475 6.03983483198815
2.4875 6.0398318872273
2.5 6.03982952911272
2.5125 6.03982776977571
2.525 6.0398266176559
2.5375 6.03982607744824
2.55 6.03982615007504
2.5625 6.03982683267966
2.575 6.03982811864469
2.5875 6.03982999763333
2.6 6.03983245565172
2.6125 6.03983547513675
2.625 6.03983903506376
2.6375 6.03984311107547
2.65 6.03984767563356
2.6625 6.03985269818717
2.675 6.0398581453625
2.6875 6.03986398116725
2.7 6.03987016721362
2.7125 6.03987666295451
2.725 6.03988342593335
2.73750000000001 6.03989041204621
2.75000000000001 6.03989757581353
2.76250000000001 6.03990487066074
2.77500000000001 6.03991224920519
2.78750000000001 6.0399196635479
2.80000000000001 6.03992706557009
2.81250000000001 6.03993440722856
2.82500000000001 6.03994164085232
2.83750000000001 6.03994871943727
2.85000000000001 6.03995559693435
2.86250000000001 6.03996222853531
2.87500000000001 6.03996857094734
2.88750000000001 6.03997458266126
2.90000000000001 6.03998022420665
2.91250000000001 6.03998545839562
2.92500000000001 6.03999025055209
2.93750000000001 6.03999456872497
2.95000000000001 6.03999838388588
2.96250000000001 6.04000167010803
2.97500000000001 6.0400044047256
2.98750000000001 6.04000656847434
3.00000000000001 6.04000814560868
};
\addlegendentry{$P_1$}
\addplot [semithick, color1]
table {%
0 6
0.0125 6.01038566905409
0.025 6.01664531137549
0.0375 6.02119872626291
0.05 6.02472826448047
0.0625 6.0275225606057
0.075 6.02974957568993
0.0875 6.03152724731822
0.1 6.03294597452231
0.1125 6.03407746577574
0.125 6.03497923675493
0.1375 6.03569747760718
0.15 6.03626923073532
0.1625 6.03672413443415
0.175 6.03708585486226
0.1875 6.03737326174848
0.2 6.03760138736301
0.2125 6.03778220286853
0.225 6.03792524271694
0.2375 6.03803810427263
0.25 6.03812684607407
0.2625 6.03819630448103
0.275 6.0382503449983
0.2875 6.03829206156149
0.3 6.03832393448627
0.3125 6.03834795564645
0.325 6.03836572770529
0.3375 6.03837854280721
0.35 6.03838744501574
0.3625 6.03839327988832
0.375 6.03839673385732
0.3875 6.03839836554074
0.4 6.03839863064902
0.4125 6.03839790181266
0.425 6.03839648437158
0.4375 6.03839463276395
0.45 6.03839254652631
0.4625 6.03839039683878
0.475 6.03838832331137
0.4875 6.03838644003428
0.5 6.03838483994611
0.5125 6.03838359818937
0.525 6.03838277476914
0.5375 6.03838241668601
0.55 6.03838255965607
0.5625 6.03838322951251
0.575 6.03838444334907
0.5875 6.03838621046832
0.6 6.03838853317304
0.6125 6.03839140743653
0.625 6.03839482348315
0.6375 6.03839876629385
0.65 6.0384032160611
0.6625 6.03840814860351
0.675 6.03841353575101
0.6875 6.03841934571006
0.7 6.03842554341435
0.712499999999999 6.03843209086712
0.724999999999999 6.03843894747513
0.737499999999999 6.03844607038326
0.749999999999999 6.03845341480397
0.762499999999999 6.03846093434819
0.774999999999999 6.03846858135483
0.787499999999999 6.03847630721961
0.799999999999999 6.03848406272434
0.812499999999999 6.03849179836344
0.824999999999999 6.03849946466792
0.837499999999999 6.03850701252572
0.849999999999999 6.03851439349746
0.862499999999999 6.03852156012469
0.874999999999999 6.03852846623071
0.887499999999999 6.03853506721256
0.899999999999999 6.0385413203211
0.912499999999999 6.03854718492834
0.924999999999999 6.03855262278319
0.937499999999999 6.0385575982479
0.949999999999999 6.03856207852232
0.962499999999999 6.03856603384757
0.974999999999999 6.03856943768964
0.987499999999998 6.03857226690601
0.999999999999998 6.03857450188916
1.0125 6.03857612668699
1.025 6.03857712910382
1.0375 6.03857750077424
1.05 6.03857723721742
1.0625 6.03857633786286
1.075 6.03857480605511
1.0875 6.03857264903466
1.1 6.03856987789122
1.1125 6.03856650749802
1.125 6.03856255641798
1.1375 6.03855804679105
1.15 6.03855300026966
1.1625 6.03854745234007
1.175 6.03854143308121
1.1875 6.0385349768139
1.2 6.03852812094032
1.2125 6.03852090543518
1.225 6.03851337252292
1.2375 6.03850556639418
1.25 6.03849753292484
1.2625 6.03848931939167
1.275 6.03848097418035
1.2875 6.03847254648384
1.3 6.03846408600377
1.3125 6.0384556426391
1.325 6.03844726618162
1.3375 6.03843900600865
1.35 6.03843091077731
1.3625 6.03842302812643
1.375 6.03841540438167
1.3875 6.03840808426931
1.4 6.0384011106421
1.4125 6.03839452806723
1.425 6.03838836834172
1.4375 6.03838266904321
1.45 6.03837746357322
1.4625 6.03837278190037
1.475 6.03836865065128
1.4875 6.03836509302297
1.5 6.03836212866196
1.5125 6.03835977355105
1.525 6.03835803991358
1.5375 6.03835693613889
1.55 6.03835646673075
1.5625 6.03835663228089
1.575 6.03835742946388
1.5875 6.03835885105857
1.6 6.03836088599237
1.6125 6.03836351941027
1.625 6.03836673276498
1.6375 6.0383705039332
1.65 6.03837480735003
1.6625 6.03837961416849
1.675 6.03838489243603
1.6875 6.03839060729256
1.7 6.0383967211848
1.7125 6.03840319409751
1.725 6.03840998380158
1.7375 6.03841704611308
1.75 6.03842433516596
1.7625 6.03843180369576
1.775 6.03843940332983
1.7875 6.03844708488658
1.8 6.03845479867759
1.8125 6.03846249481551
1.825 6.0384701235201
1.8375 6.03847763542636
1.85 6.03848498188832
1.8625 6.03849211527901
1.875 6.03849898928451
1.8875 6.03850555918852
1.9 6.03851178215027
1.9125 6.03851761746567
1.925 6.03852302682018
1.9375 6.03852797452552
1.95 6.03853242773795
1.9625 6.03853635666233
1.97499999999999 6.0385397347352
1.98749999999999 6.03854253878812
1.99999999999999 6.03854474919175
2.01249999999999 6.03854634997572
2.025 6.0385473289279
2.0375 6.03854767766919
2.05 6.03854739170549
2.0625 6.03854647045519
2.075 6.03854491725357
2.0875 6.03854273933093
2.1 6.038539947769
2.1125 6.03853655743229
2.125 6.03853258687697
2.1375 6.03852805823524
2.15 6.03852299708051
2.1625 6.03851742833977
2.175 6.03851139060005
2.1875 6.03850491684103
2.2 6.03849804406999
2.2125 6.03849081214403
2.225 6.03848326324621
2.2375 6.03847544154907
2.25 6.038467392919
2.2625 6.03845916462534
2.275 6.03845080504804
2.2875 6.03844236337399
2.3 6.0384338892983
2.3125 6.03842543271575
2.325 6.03841704341218
2.3375 6.03840877075891
2.35 6.03840066340872
2.3625 6.03839276899464
2.375 6.03838513383641
2.3875 6.03837780265617
2.4 6.03837082215862
2.4125 6.03836422651924
2.425 6.03835805593551
2.4375 6.03835234665121
2.45 6.03834713167796
2.4625 6.03834244086961
2.475 6.038338300814
2.4875 6.03833473469138
2.5 6.03833176213945
2.5125 6.03832939913469
2.525 6.03832765789544
2.5375 6.03832654680619
2.55 6.03832607036653
2.5625 6.03832622916277
2.575 6.03832701986543
2.5875 6.03832843524922
2.6 6.03833046423682
2.6125 6.03833309196904
2.625 6.03833629989482
2.6375 6.03834006588577
2.65 6.03834436437372
2.6625 6.0383491665069
2.675 6.03835444032957
2.6875 6.038360150977
2.7 6.03836626089231
2.7125 6.03837273005677
2.725 6.03837951623702
2.73750000000001 6.03838657524558
2.75000000000001 6.03839386121289
2.76250000000001 6.03840132687074
2.77500000000001 6.0384089238431
2.78750000000001 6.03841660294426
2.80000000000001 6.03842431448311
2.81250000000001 6.0384320085685
2.82500000000001 6.03843963541688
2.83750000000001 6.03844714566035
2.85000000000001 6.03845449064911
2.86250000000001 6.03846162275358
2.87500000000001 6.03846849565608
2.88750000000001 6.03847506463801
2.90000000000001 6.03848128685445
2.91250000000001 6.038487121599
2.92500000000001 6.03849253055469
2.93750000000001 6.03849747802943
2.95000000000001 6.03850193117696
2.96250000000001 6.03850586019945
2.97500000000001 6.0385092385303
2.98750000000001 6.03851204299922
3.00000000000001 6.0385142539731
};
\addlegendentry{$P_2$}
\addplot [semithick, color2]
table {%
0 6
0.0125 6.0118584220116
0.025 6.02018239028751
0.0375 6.02687795007368
0.05 6.03245556074404
0.0625 6.03713230857132
0.075 6.04104623352223
0.0875 6.04430807145922
0.1 6.04701451807695
0.1125 6.04925128158554
0.125 6.05109367833619
0.1375 6.05260699423724
0.15 6.05384708014948
0.1625 6.05486120234519
0.175 6.05568902401472
0.1875 6.05636361249262
0.2 6.05691240051184
0.2125 6.05735806184189
0.225 6.05771928455628
0.2375 6.0580114395044
0.25 6.05824714959815
0.2625 6.05843676945244
0.275 6.05858878634415
0.2875 6.05871015345604
0.3 6.05880656563372
0.3125 6.05888268681014
0.325 6.05894233708273
0.3375 6.05898864628417
0.35 6.05902417983824
0.3625 6.05905104176592
0.375 6.05907095888923
0.3875 6.05908534960174
0.4 6.05909537998626
0.4125 6.05910200957257
0.425 6.05910602862602
0.4375 6.05910808701995
0.45 6.0591087250204
0.4625 6.05910838415096
0.475 6.05910742974844
0.4875 6.05910616283334
0.5 6.05910483081944
0.5125 6.05910363630992
0.525 6.05910274425864
0.5375 6.05910228780157
0.55 6.05910237300911
0.5625 6.05910308277229
0.575 6.05910447998399
0.5875 6.05910661015889
0.6 6.05910950360012
0.6125 6.05911317720305
0.625 6.05911763597425
0.6375 6.05912287431757
0.65 6.05912887714295
0.6625 6.05913562083277
0.675 6.05914307409789
0.6875 6.05915119874885
0.7 6.05915995040158
0.712499999999999 6.05916927913431
0.724999999999999 6.05917913010524
0.737499999999999 6.05918944414553
0.749999999999999 6.05920015832901
0.762499999999999 6.05921120652831
0.774999999999999 6.05922251995873
0.787499999999999 6.05923402771254
0.799999999999999 6.05924565728737
0.812499999999999 6.05925733510588
0.824999999999999 6.05926898702929
0.837499999999999 6.05928053886264
0.849999999999999 6.05929191685228
0.862499999999999 6.05930304817221
0.874999999999999 6.059313861399
0.887499999999999 6.05932428697406
0.899999999999999 6.05933425764938
0.912499999999999 6.05934370891559
0.924999999999999 6.05935257941255
0.937499999999999 6.05936081131499
0.949999999999999 6.05936835069831
0.962499999999999 6.05937514787687
0.974999999999999 6.05938115771411
0.987499999999998 6.05938633990624
0.999999999999998 6.05939065923386
1.0125 6.05939408578085
1.025 6.05939659512236
1.0375 6.05939816847585
1.05 6.05939879281992
1.0625 6.05939846097426
1.075 6.05939717164535
1.0875 6.05939492943635
1.1 6.05939174481737
1.1125 6.05938763406359
1.125 6.05938261915352
1.1375 6.05937672763509
1.15 6.05936999404018
1.1625 6.05936245326061
1.175 6.05935415005493
1.1875 6.05934513229898
1.2 6.05933545216766
1.2125 6.05932516592499
1.225 6.05931433357812
1.2375 6.05930301850216
1.25 6.05929128704699
1.2625 6.05927920812751
1.275 6.05926685279712
1.2875 6.05925429381028
1.3 6.05924160517339
1.3125 6.05922886168896
1.325 6.05921613849389
1.3375 6.05920351059714
1.35 6.05919105241583
1.3625 6.0591788373186
1.375 6.0591669371712
1.3875 6.05915542189532
1.4 6.05914435903573
1.4125 6.05913381185616
1.425 6.05912384500004
1.4375 6.05911451486186
1.45 6.05910587543883
1.4625 6.05909797656315
1.475 6.05909086348646
1.4875 6.0590845766028
1.5 6.05907915120407
1.5125 6.05907461726569
1.525 6.05907099926319
1.5375 6.05906831602181
1.55 6.0590665806007
1.5625 6.05906580021352
1.575 6.05906597618294
1.5875 6.05906710393302
1.6 6.05906917301741
1.6125 6.05907216718449
1.625 6.05907606447609
1.6375 6.05908083736481
1.65 6.05908645292178
1.6625 6.05909287302148
1.675 6.05910005457554
1.6875 6.05910794979933
1.7 6.05911650650615
1.7125 6.0591256684285
1.725 6.05913537556652
1.7375 6.05914556455654
1.75 6.0591561690621
1.7625 6.05916712018364
1.775 6.0591783468825
1.7875 6.05918977641952
1.8 6.05920133480281
1.8125 6.05921294724504
1.825 6.05922453862341
1.8375 6.05923603394375
1.85 6.05924735880241
1.8625 6.05925843984521
1.875 6.05926920521986
1.8875 6.05927958501818
1.9 6.05928951170903
1.9125 6.0592989205522
1.925 6.05930774999908
1.9375 6.05931594207267
1.95 6.05932344272339
1.9625 6.05933020216386
1.97499999999999 6.05933617517521
1.98749999999999 6.05934132138559
1.99999999999999 6.05934560552003
2.01249999999999 6.05934899761701
2.025 6.05935147321399
2.0375 6.0593530134976
2.05 6.05935360542009
2.0625 6.05935324177955
2.075 6.05935192126492
2.0875 6.05934964846316
2.1 6.05934643383162
2.1125 6.0593422936334
2.125 6.05933724983762
2.1375 6.05933132998275
2.15 6.05932456700837
2.1625 6.05931700062265
2.175 6.0593086706916
2.1875 6.05929962665996
2.2 6.05928992079722
2.2125 6.05927960935209
2.225 6.05926875231626
2.2375 6.05925741305407
2.25 6.05924565790819
2.2625 6.05923355578652
2.275 6.05922117773738
2.2875 6.05920859651021
2.3 6.05919588610609
2.3125 6.0591831213227
2.325 6.05917037729225
2.3375 6.05915772901865
2.35 6.05914525091487
2.3625 6.05913301634407
2.375 6.05912109716789
2.3875 6.0591095633025
2.4 6.05909848080254
2.4125 6.05908791749578
2.425 6.05907793334085
2.4375 6.05906858625987
2.45 6.05905993034186
2.4625 6.0590520154027
2.475 6.05904488667926
2.4875 6.05903858455542
2.5 6.05903314431589
2.5125 6.05902859593018
2.525 6.05902496386884
2.5375 6.05902226695228
2.55 6.05902051823546
2.5625 6.05901972492688
2.575 6.05901988834491
2.5875 6.05902100390955
2.6 6.05902306116979
2.6125 6.05902604386968
2.625 6.05902993004726
2.6375 6.05903469216999
2.65 6.0590402973055
2.6625 6.0590467073234
2.675 6.05905387913176
2.6875 6.05906176494133
2.7 6.0590703125614
2.7125 6.05907946572086
2.725 6.05908916441529
2.73750000000001 6.05909934527714
2.75000000000001 6.05910994196611
2.76250000000001 6.05912088557868
2.77500000000001 6.05913210507235
2.78750000000001 6.05914352770369
2.80000000000001 6.05915507947756
2.81250000000001 6.05916668560254
2.82500000000001 6.05917827095213
2.83750000000001 6.0591897605288
2.85000000000001 6.05920107992484
2.86250000000001 6.05921215578289
2.87500000000001 6.05922291624653
2.88750000000001 6.05923329140478
2.90000000000001 6.05924321372206
2.91250000000001 6.05925261845524
2.92500000000001 6.05926144405285
2.93750000000001 6.05926963253363
2.95000000000001 6.05927712984514
2.96250000000001 6.05928388619669
2.97500000000001 6.05928985636584
2.98750000000001 6.05929499997833
3.00000000000001 6.05929928175502
};
\addlegendentry{$P_3$}
\end{axis}

%% file: suppl_tex/p4_T3.tex
\definecolor{color0}{rgb}{0.12156862745098,0.466666666666667,0.705882352941177}
\definecolor{color1}{rgb}{1,0.498039215686275,0.0549019607843137}
\definecolor{color2}{rgb}{0.172549019607843,0.627450980392157,0.172549019607843}

\begin{axis}[
legend cell align={left},
legend style={
  fill opacity=0.8,
  draw opacity=1,
  text opacity=1,
  at={(0.03,0.97)},
  anchor=north west,
  draw=white!80!black
},
tick align=outside,
tick pos=left,
x grid style={white!69.0196078431373!black},
xlabel={t},
xmin=-0.15, xmax=3.15000000000001,
xtick style={color=black},
y grid style={white!69.0196078431373!black},
ymin=37.6522427580189, ymax=38.0165598686658,
ytick style={color=black}
]
\addplot [semithick, color0]
table {%
0 38
0.0125 37.9986662533282
0.025 37.9973347604173
0.0375 37.9960085197237
0.05 37.9946886788056
0.0625 37.9933754948696
0.075 37.9920687902977
0.0875 37.9907681546486
0.1 37.9894730355184
0.1125 37.9881827898911
0.125 37.9868966932893
0.1375 37.9856139722593
0.15 37.9843338155843
0.1625 37.9830553858855
0.175 37.9817778296027
0.1875 37.9805002859215
0.2 37.9792218948433
0.2125 37.9779418045804
0.225 37.9766591783695
0.2375 37.9753732007989
0.25 37.9740830836915
0.2625 37.9727880715783
0.275 37.9714874467775
0.2875 37.9701805340876
0.3 37.9688667050927
0.3125 37.9675453820741
0.325 37.9662160415253
0.3375 37.9648782172647
0.35 37.9635315031286
0.3625 37.9621755552605
0.375 37.9608100939697
0.3875 37.9594349051709
0.4 37.9580498413965
0.4125 37.956654822384
0.425 37.9552498352452
0.4375 37.9538349287249
0.45 37.9524102272206
0.4625 37.9509759178668
0.475 37.9495322512291
0.4875 37.9480795395744
0.5 37.9466181548453
0.5125 37.9451485261726
0.525 37.9436711369942
0.5375 37.9421865218317
0.55 37.9406952627481
0.5625 37.9391979855367
0.575 37.9376953556403
0.5875 37.9361880738488
0.6 37.9346768717907
0.6125 37.9331625072457
0.625 37.9316457593191
0.6375 37.930127423487
0.65 37.928608306559
0.6625 37.927089221582
0.675 37.9255709827161
0.6875 37.9240544001156
0.7 37.922540274843
0.712499999999999 37.9210293938511
0.724999999999999 37.9195225250538
0.737499999999999 37.9180204125326
0.749999999999999 37.9165237718904
0.762499999999999 37.9150332857895
0.774999999999999 37.9135495997004
0.787499999999999 37.9120733178829
0.799999999999999 37.9106049996311
0.812499999999999 37.9091451557972
0.824999999999999 37.9076942456199
0.837499999999999 37.9062526738748
0.849999999999999 37.9048207883655
0.862499999999999 37.9033988777684
0.874999999999999 37.9019871698434
0.887499999999999 37.900585830028
0.899999999999999 37.8991949604148
0.912499999999999 37.8978145991196
0.924999999999999 37.896444720053
0.937499999999999 37.89508523308
0.949999999999999 37.8937359845859
0.962499999999999 37.8923967584312
0.974999999999999 37.8910672772873
0.987499999999998 37.8897472043596
0.999999999999998 37.8884361454757
1.0125 37.8871336515218
1.025 37.8858392212282
1.0375 37.8845523042673
1.05 37.8832723046663
1.0625 37.8819985844875
1.075 37.8807304677755
1.0875 37.8794672447432
1.1 37.8782081761541
1.1125 37.8769524979004
1.125 37.8756994257198
1.1375 37.8744481600494
1.15 37.8731978975342
1.1625 37.8719478180398
1.175 37.8706971048401
1.1875 37.8694449464922
1.2 37.8681905416395
1.2125 37.8669331037592
1.225 37.8656718659321
1.2375 37.8644060855411
1.25 37.8631350488443
1.2625 37.8618580753836
1.275 37.8605745221844
1.2875 37.8592837877311
1.3 37.8579853156841
1.3125 37.8566785983186
1.325 37.8553631796617
1.3375 37.8540386583137
1.35 37.8527046899223
1.3625 37.8513609893206
1.375 37.8500073322815
1.3875 37.8486435569122
1.4 37.84726956465
1.4125 37.8458853152153
1.425 37.8444908420683
1.4375 37.8430862398316
1.45 37.841671666032
1.4625 37.8402473402896
1.475 37.8388135431601
1.4875 37.8373706145039
1.5 37.8359189514299
1.5125 37.834459005865
1.525 37.8329912817735
1.5375 37.8315163320499
1.55 37.8300347551086
1.5625 37.8285471911968
1.575 37.8270543184421
1.5875 37.8255568486713
1.6 37.8240555230192
1.6125 37.8225511073617
1.625 37.8210443875855
1.6375 37.819536164748
1.65 37.81802725013
1.6625 37.8165184602359
1.675 37.8150106117547
1.6875 37.8135045165246
1.7 37.8120009765229
1.7125 37.8105007789169
1.725 37.8090046912126
1.7375 37.8075134565155
1.75 37.8060277889449
1.7625 37.8045483692321
1.775 37.8030758405149
1.7875 37.8016108043732
1.8 37.8001538171105
1.8125 37.7987053863263
1.825 37.797265967778
1.8375 37.7958359625692
1.85 37.7944157146722
1.8625 37.7930055088031
1.875 37.7916055686628
1.8875 37.7902160555503
1.9 37.7888370673727
1.9125 37.7874686380268
1.925 37.7861107371913
1.9375 37.7847632705112
1.95 37.783426080172
1.9625 37.7820989458704
1.97499999999999 37.7807815861685
1.98749999999999 37.7794736602224
1.99999999999999 37.7781747698846
2.01249999999999 37.7768844621496
2.025 37.7756022319467
2.0375 37.7743275252481
2.05 37.7730597424816
2.0625 37.7717982422228
2.075 37.7705423451488
2.0875 37.7692913382192
2.1 37.7680444790709
2.1125 37.7668010005904
2.125 37.7655601156435
2.1375 37.7643210219183
2.15 37.7630829068762
2.1625 37.7618449592961
2.175 37.7606063563963
2.1875 37.7593662839505
2.2 37.7581239381556
2.2125 37.7568785303752
2.225 37.755629291796
2.2375 37.7543754780699
2.25 37.7531163738656
2.2625 37.751851297258
2.275 37.7505796039297
2.2875 37.7493006911392
2.3 37.7480140014337
2.3125 37.7467190260874
2.325 37.7454153082362
2.3375 37.7441024456941
2.35 37.7427800934312
2.3625 37.7414479656977
2.375 37.740105837787
2.3875 37.7387535474169
2.4 37.7373909900424
2.4125 37.7360181347811
2.425 37.7346350121574
2.4375 37.7332417162539
2.45 37.7318384042873
2.4625 37.7304252958385
2.475 37.7290026715932
2.4875 37.7275708716503
2.5 37.7261302934411
2.5125 37.7246813892822
2.525 37.7232246635918
2.5375 37.7217606697735
2.55 37.720290006808
2.5625 37.7188133155549
2.575 37.7173312748063
2.5875 37.7158445971009
2.6 37.714354024327
2.6125 37.7128603231554
2.625 37.7113642803092
2.6375 37.7098666977142
2.65 37.7083683875568
2.6625 37.7068701672759
2.675 37.7053728545281
2.6875 37.7038772621418
2.7 37.7023841931115
2.7125 37.7008944356459
2.725 37.6994087583105
2.73750000000001 37.6979279052894
2.75000000000001 37.6964525917992
2.76250000000001 37.6949834996815
2.77500000000001 37.6935212732
2.78750000000001 37.6920665150671
2.80000000000001 37.6906197827342
2.81250000000001 37.6891815849541
2.82500000000001 37.6877523786443
2.83750000000001 37.686332566076
2.85000000000001 37.6849224923903
2.86250000000001 37.6835224434777
2.87500000000001 37.6821326442108
2.88750000000001 37.680753257066
2.90000000000001 37.6793843811197
2.91250000000001 37.6780260514414
2.92500000000001 37.67667823888
2.93750000000001 37.6753408502428
2.95000000000001 37.6740137288732
2.96250000000001 37.6726966556217
2.97500000000001 37.6713893501946
2.98750000000001 37.6700914728894
3.00000000000001 37.6688026266847
};
\addlegendentry{$P_1$}
\addplot [semithick, color1]
table {%
0 38
0.0125 37.9986717415721
0.025 37.997343356317
0.0375 37.9960141970136
0.05 37.9946836650852
0.0625 37.9933511858284
0.075 37.9920162050372
0.0875 37.9906781911663
0.1 37.9893366391536
0.1125 37.9879910667802
0.125 37.9866410401404
0.1375 37.9852861625462
0.15 37.9839260778421
0.1625 37.9825604752259
0.175 37.9811890924514
0.1875 37.9798117180956
0.2 37.9784281931332
0.2125 37.9770384119433
0.225 37.9756423227727
0.2375 37.9742399277506
0.25 37.9728312824426
0.2625 37.9714164950106
0.275 37.9699957249872
0.2875 37.9685691817046
0.3 37.9671371224004
0.3125 37.9656998500153
0.325 37.9642577107185
0.3375 37.9628110911751
0.35 37.9613604155713
0.3625 37.959906142442
0.375 37.9584487612894
0.3875 37.956988789044
0.4 37.9555267663754
0.4125 37.9540632538824
0.425 37.952598828175
0.4375 37.9511340856439
0.45 37.949669617923
0.4625 37.9482060223077
0.475 37.946743898341
0.4875 37.945283842133
0.5 37.9438264417891
0.5125 37.9423722732789
0.525 37.9409218965248
0.5375 37.9394758516537
0.55 37.9380346553903
0.5625 37.9365987976259
0.575 37.9351687381511
0.5875 37.9337449035936
0.6 37.9323276845665
0.6125 37.930917433045
0.625 37.929514460001
0.6375 37.9281190332862
0.65 37.9267313757963
0.6625 37.9253516639194
0.675 37.9239800262778
0.6875 37.9226165427715
0.7 37.92126124393
0.712499999999999 37.9199141105768
0.724999999999999 37.9185750737981
0.737499999999999 37.917244015237
0.749999999999999 37.9159207676889
0.762499999999999 37.9146051160086
0.774999999999999 37.9132967983168
0.787499999999999 37.9119955074977
0.799999999999999 37.9107008929853
0.812499999999999 37.9094125628211
0.824999999999999 37.9081300859692
0.837499999999999 37.9068529948797
0.849999999999999 37.9055807882821
0.862499999999999 37.9043129341892
0.874999999999999 37.9030488730946
0.887499999999999 37.9017880213475
0.899999999999999 37.9005297746777
0.912499999999999 37.8992735118492
0.924999999999999 37.8980185984315
0.937499999999999 37.8967643906412
0.949999999999999 37.8955102392608
0.962499999999999 37.8942554935879
0.974999999999999 37.8929995053901
0.987499999999998 37.8917416328594
0.999999999999998 37.8904812445276
1.0125 37.8892177231126
1.025 37.8879504692938
1.0375 37.8866789053665
1.05 37.8854024787809
1.0625 37.8841206655124
1.075 37.8828329732685
1.0875 37.8815389445065
1.1 37.8802381592259
1.1125 37.8789302375488
1.125 37.8776148420387
1.1375 37.8762916797729
1.15 37.8749604959944
1.1625 37.8736210972963
1.175 37.8722733371789
1.1875 37.8709171166653
1.2 37.8695523865642
1.2125 37.8681791483387
1.225 37.8667974542833
1.2375 37.8654074072596
1.25 37.8640091600739
1.2625 37.8626029145408
1.275 37.8611889202417
1.2875 37.8597674729804
1.3 37.858338912975
1.3125 37.8569036227574
1.325 37.8554620248287
1.3375 37.8540145790673
1.35 37.8525617799015
1.3625 37.8511041532827
1.375 37.8496422534544
1.3875 37.8481766595506
1.4 37.8467079720445
1.4125 37.8452368169033
1.425 37.8437638210636
1.4375 37.8422896232587
1.45 37.8408148706946
1.4625 37.8393402134511
1.475 37.8378662999236
1.4875 37.8363937726649
1.5 37.8349232643876
1.5125 37.8334553940784
1.525 37.8319907632088
1.5375 37.830529952059
1.55 37.8290735161705
1.5625 37.8276219829574
1.575 37.8261758484792
1.5875 37.824735574415
1.6 37.8233015852446
1.6125 37.8218742656638
1.625 37.8204539582345
1.6375 37.819040961308
1.65 37.8176355272063
1.6625 37.8162378607
1.675 37.8148481177697
1.6875 37.8134664046753
1.7 37.8120927773267
1.7125 37.8107272409651
1.725 37.8093697501626
1.7375 37.8080202091263
1.75 37.8066784723164
1.7625 37.8053443453739
1.775 37.8040175863429
1.7875 37.8026979071926
1.8 37.801384975616
1.8125 37.8000784171104
1.825 37.798777817309
1.8375 37.7974827245649
1.85 37.7961926527618
1.8625 37.7949070843361
1.875 37.7936254734962
1.8875 37.7923472496088
1.9 37.7910718207528
1.9125 37.7897985773826
1.925 37.7885268961166
1.9375 37.7872561436097
1.95 37.7859856804761
1.9625 37.7847148652612
1.97499999999999 37.7834430584192
1.98749999999999 37.7821696262786
1.99999999999999 37.7808939449771
2.01249999999999 37.7796154043277
2.025 37.7783334116079
2.0375 37.7770473952347
2.05 37.7757568083131
2.0625 37.7744611320297
2.075 37.773159878877
2.0875 37.7718525956776
2.1 37.7705388664022
2.1125 37.7692183147556
2.125 37.7678906065217
2.1375 37.7665554516385
2.15 37.7652126060151
2.1625 37.7638618649507
2.175 37.7625030859227
2.1875 37.761136173969
2.2 37.7597610819827
2.2125 37.7583778126727
2.225 37.7569864191302
2.2375 37.7555870047056
2.25 37.754179722447
2.2625 37.7527647741837
2.275 37.7513424093046
2.2875 37.7499129232196
2.3 37.7484766555581
2.3125 37.7470339880881
2.325 37.7455853423707
2.3375 37.7441311771774
2.35 37.7426719856788
2.3625 37.741208292416
2.375 37.7397406500818
2.3875 37.7382696361299
2.4 37.7367958570429
2.4125 37.7353199239904
2.425 37.7338424636833
2.4375 37.7323641151397
2.45 37.7308855241233
2.4625 37.7294073386112
2.475 37.7279302046255
2.4875 37.7264547622076
2.5 37.7249816414756
2.5125 37.723511458755
2.525 37.7220448128023
2.5375 37.7205822811309
2.55 37.7191244164705
2.5625 37.7176717433743
2.575 37.716224755005
2.5875 37.7147839101084
2.6 37.7133496301955
2.6125 37.711922296964
2.625 37.710502249954
2.6375 37.7090897844657
2.65 37.7076851497558
2.6625 37.7062885475073
2.675 37.7049001306049
2.6875 37.7035200021955
2.7 37.7021482150709
2.7125 37.7007847713482
2.725 37.6994296224694
2.73750000000001 37.6980826695119
2.75000000000001 37.6967437638072
2.76250000000001 37.6954127078703
2.77500000000001 37.6940892566248
2.78750000000001 37.6927731189232
2.80000000000001 37.6914639593545
2.81250000000001 37.6901614003202
2.82500000000001 37.6888650243696
2.83750000000001 37.6875743767876
2.85000000000001 37.6862889684005
2.86250000000001 37.6850082786079
2.87500000000001 37.6837317585938
2.88750000000001 37.6824588347252
2.90000000000001 37.6811889120928
2.91250000000001 37.6799213781895
2.92500000000001 37.6786556066952
2.93750000000001 37.6773909613427
2.95000000000001 37.6761267998511
2.96250000000001 37.6748624778942
2.97500000000001 37.6735973530777
2.98750000000001 37.6723307889115
3.00000000000001 37.6710621587336
};
\addlegendentry{$P_2$}
\addplot [semithick, color2]
table {%
0 38
0.0125 37.9986704627616
0.025 37.9973412044798
0.0375 37.996011479258
0.05 37.9946806483671
0.0625 37.9933481131386
0.075 37.9920132993603
0.0875 37.990675655174
0.1 37.9893346528245
0.1125 37.9879898094982
0.125 37.9866406304612
0.1375 37.9852866803429
0.15 37.9839275595136
0.1625 37.9825629071982
0.175 37.9811924043781
0.1875 37.9798157763432
0.2 37.9784327948745
0.2125 37.9770432800706
0.225 37.975647101791
0.2375 37.9742441807072
0.25 37.9728344889609
0.2625 37.9714180504218
0.275 37.9699949405566
0.2875 37.9685652859148
0.3 37.9671292632483
0.3125 37.9656870982747
0.325 37.9642390641051
0.3375 37.9627854793554
0.35 37.9613267059505
0.3625 37.9598631466633
0.375 37.9583952423837
0.3875 37.9569234691557
0.4 37.955448334998
0.4125 37.953970376528
0.425 37.9524901554265
0.4375 37.9510082375576
0.45 37.9495252468844
0.4625 37.9480417956023
0.475 37.946558506248
0.4875 37.9450760074315
0.5 37.9435949295551
0.5125 37.9421159006242
0.525 37.94063954213
0.5375 37.9391664650044
0.55 37.9376972656483
0.5625 37.9362325220731
0.575 37.9347727901501
0.5875 37.933318600009
0.6 37.9318704525973
0.6125 37.9304288164177
0.625 37.928994124479
0.6375 37.9275667714558
0.65 37.9261471110919
0.6625 37.9247354538573
0.675 37.9233320648705
0.6875 37.9219371621049
0.7 37.9205509148837
0.712499999999999 37.9191734426809
0.724999999999999 37.9178048142209
0.737499999999999 37.9164450469047
0.749999999999999 37.9150941065455
0.762499999999999 37.9137519074281
0.774999999999999 37.9124183126857
0.787499999999999 37.9110931349936
0.799999999999999 37.9097761375787
0.812499999999999 37.9084670355351
0.824999999999999 37.9071654974373
0.837499999999999 37.9058711472449
0.849999999999999 37.9045835664845
0.862499999999999 37.9033022966958
0.874999999999999 37.9020268421262
0.887499999999999 37.9007566726618
0.899999999999999 37.8994912269712
0.912499999999999 37.8982299158422
0.924999999999999 37.8969721257022
0.937499999999999 37.8957172222799
0.949999999999999 37.8944645544108
0.962499999999999 37.8932134579453
0.974999999999999 37.8919632597335
0.987499999999998 37.8907132816805
0.999999999999998 37.8894628448317
1.0125 37.8882112734608
1.025 37.8869578991518
1.0375 37.8857020648297
1.05 37.8844431287362
1.0625 37.8831804683013
1.075 37.8819134839069
1.0875 37.880641602518
1.1 37.8793642811394
1.1125 37.878081010106
1.125 37.8767913161566
1.1375 37.8754947652977
1.15 37.8741909834829
1.1625 37.8728795985838
1.175 37.871560311075
1.1875 37.8702328689255
1.2 37.8688970693045
1.2125 37.867552759997
1.225 37.866199840417
1.2375 37.8648382622511
1.25 37.8634680297585
1.2625 37.8620891997399
1.275 37.8607018811698
1.2875 37.8593062345139
1.3 37.8579024707211
1.3125 37.85649084991
1.325 37.855071679748
1.3375 37.8536453135423
1.35 37.8522121480427
1.3625 37.8507726209927
1.375 37.8493272084102
1.3875 37.8478764216526
1.4 37.8464208042468
1.4125 37.8449609112693
1.425 37.8434973638443
1.4375 37.8420307791579
1.45 37.8405617928708
1.4625 37.8390910549808
1.475 37.8376192255949
1.4875 37.8361469707499
1.5 37.8346749582457
1.5125 37.8332038534964
1.525 37.8317343154088
1.5375 37.8302669923093
1.55 37.8288025179401
1.5625 37.8273415075555
1.575 37.8258845541263
1.5875 37.8244322246893
1.6 37.8229850568579
1.6125 37.8215435555174
1.625 37.8201081897166
1.6375 37.8186793897917
1.65 37.8172575447182
1.6625 37.8158429997299
1.675 37.8144360541991
1.6875 37.8130369598072
1.7 37.811645919005
1.7125 37.8102630837773
1.725 37.8088885547256
1.7375 37.8075223804583
1.75 37.8061645573074
1.7625 37.8048150293702
1.775 37.8034736888667
1.7875 37.8021403768236
1.8 37.8008148840679
1.8125 37.7994969525371
1.825 37.7981862768844
1.8375 37.7968825063812
1.85 37.7955852470961
1.8625 37.7942940643411
1.875 37.793008485371
1.8875 37.7917280023119
1.9 37.7904520753198
1.9125 37.7891801359164
1.925 37.7879115905172
1.9375 37.7866458241127
1.95 37.7853822040724
1.9625 37.7841200840684
1.97499999999999 37.7828588080779
1.98749999999999 37.7815977144427
1.99999999999999 37.780336139973
2.01249999999999 37.7790734240495
2.025 37.777808912719
2.0375 37.776541962742
2.05 37.7752719455778
2.0625 37.7739982512777
2.075 37.7727202922667
2.0875 37.7714375069804
2.1 37.7701493633478
2.1125 37.768855362089
2.125 37.767555039814
2.1375 37.7662479718927
2.15 37.7649337750967
2.1625 37.7636121280078
2.175 37.7622827128173
2.1875 37.7609452856711
2.2 37.7595996512963
2.2125 37.7582456644253
2.225 37.7568832309149
2.2375 37.7555123084512
2.25 37.7541329068928
2.2625 37.7527450882575
2.275 37.7513489663838
2.2875 37.7499447062628
2.3 37.7485325230417
2.3125 37.7471126807282
2.325 37.7456854905831
2.3375 37.744251309223
2.35 37.742810536439
2.3625 37.7413636127505
2.375 37.7399110167092
2.3875 37.738453261961
2.4 37.73699087678
2.4125 37.7355244589137
2.425 37.7340546055952
2.4375 37.7325819361547
2.45 37.7311070879948
2.4625 37.7296307124604
2.475 37.7281534707256
2.4875 37.7266760296699
2.5 37.7251990577463
2.5125 37.7237232208496
2.525 37.7222491782122
2.5375 37.7207775783396
2.55 37.719309055019
2.5625 37.7178442234174
2.575 37.7163836763016
2.5875 37.7149279803925
2.6 37.7134776728785
2.6125 37.7120332581204
2.625 37.7105952045501
2.6375 37.7091639417927
2.65 37.7077398580331
2.6625 37.7063232976307
2.675 37.7049145590138
2.6875 37.7035138928447
2.7 37.7021215004916
2.7125 37.7007375327956
2.725 37.6993620891528
2.73750000000001 37.6979952169141
2.75000000000001 37.6966369111033
2.76250000000001 37.6952871144608
2.77500000000001 37.6939457178063
2.78750000000001 37.6926125607217
2.80000000000001 37.6912874325544
2.81250000000001 37.6899700737254
2.82500000000001 37.6886601773376
2.83750000000001 37.6873573910844
2.85000000000001 37.6860613194249
2.86250000000001 37.6847715260395
2.87500000000001 37.6834875365244
2.88750000000001 37.6822088413306
2.90000000000001 37.6809348989134
2.91250000000001 37.6796651390812
2.92500000000001 37.6783989665226
2.93750000000001 37.6771357644822
2.95000000000001 37.6758748985747
2.96250000000001 37.6746157207067
2.97500000000001 37.6733575730793
2.98750000000001 37.6720997922557
3.00000000000001 37.6708417132542
};
\addlegendentry{$P_3$}
\end{axis}

%% file: suppl_tex/p1_mean.tex
\definecolor{color0}{rgb}{0.12156862745098,0.466666666666667,0.705882352941177}
\definecolor{color1}{rgb}{1,0.498039215686275,0.0549019607843137}
\definecolor{color2}{rgb}{0.172549019607843,0.627450980392157,0.172549019607843}

\begin{axis}[
legend cell align={left},
legend style={
  fill opacity=0.8,
  draw opacity=1,
  text opacity=1,
  at={(0.03,0.97)},
  anchor=north west,
  draw=white!80!black
},
tick align=outside,
tick pos=left,
x grid style={white!69.0196078431373!black},
xlabel={t},
xmin=-124.8, xmax=2620.8,
xtick style={color=black},
y grid style={white!69.0196078431373!black},
ymin=4.46582102042157, ymax=17.054043622605,
ytick style={color=black}
]
\addplot [semithick, color0]
table {%
0 5.03832898310163
8 5.70686674647645
16 6.3219388128795
24 6.88561825628256
32 7.4009284849947
40 7.87143170092802
48 8.30090739018441
56 8.69312325564038
64 9.05168613766762
72 9.37995638935155
80 9.68100776271475
88 9.95761729599536
96 10.212273478111
104 10.4471945457758
112 10.6643515910404
120 10.8654931809181
128 11.0521695581219
136 11.2257553764455
144 11.3874704755935
152 11.5383985310367
160 11.6795036038059
168 11.811644716196
176 11.9355886267235
184 12.0520209934523
192 12.1615561127235
200 12.2647454088411
208 12.3620848344403
216 12.4540213240195
224 12.540958426119
232 12.6232612237193
240 12.7012606380052
248 12.7752571978137
256 12.8455243458364
264 12.9123113428463
272 12.9758458227603
280 13.0363360440447
288 13.0939728767059
296 13.1489315587137
304 13.2013732510984
312 13.2514464169883
320 13.2992880464639
328 13.345024746179
336 13.3887737101987
344 13.4306435863411
352 13.4707352504556
360 13.5091424994744
368 13.5459526726844
376 13.5812472094818
384 13.6151021508367
392 13.6475885907978
400 13.6787730835984
408 13.7087180112436
416 13.7374819158791
424 13.7651198007269
432 13.7916834029321
440 13.8172214412716
448 13.84177984134
456 13.8654019405282
464 13.8881286748472
472 13.9099987494256
480 13.9310487943003
488 13.9513135069508
496 13.9708257828624
504 13.9896168352739
512 14.0077163051353
520 14.0251523621997
528 14.0419517980731
536 14.058140111964
544 14.0737415897958
552 14.0887793772827
560 14.1032755475058
568 14.1172511634732
576 14.1307263361043
584 14.14372027803
592 14.1562513535689
600 14.1683371251987
608 14.1799943968192
616 14.1912392540678
624 14.2020871019317
632 14.2125526998718
640 14.2226501946591
648 14.2323931511028
656 14.241794580835
664 14.2508669693011
672 14.2596223010941
680 14.2680720837549
688 14.2762273701554
696 14.2840987795659
704 14.2916965175051
712 14.2990303944562
720 14.3061098435336
728 14.3129439371698
736 14.319541402892
744 14.3259106382504
752 14.3320597249539
760 14.3379964422675
768 14.3437282797177
776 14.3492624491527
784 14.3546058961965
792 14.3597653111359
800 14.3647471392762
808 14.3695575907968
816 14.3742026501369
824 14.3786880849404
832 14.3830194545851
840 14.3872021183186
848 14.3912412430283
856 14.3951418106585
864 14.3989086253016
872 14.4025463199761
880 14.4060593631107
888 14.4094520647498
896 14.4127285824943
904 14.4158929271914
912 14.4189489683882
920 14.4219004395575
928 14.4247509431104
936 14.4275039552043
944 14.430162830357
952 14.4327308058752
960 14.4352110061088
968 14.4376064465346
976 14.4399200376813
984 14.442154588903
992 14.4443128120035
1000 14.4463973247242
1008 14.4484106540967
1016 14.4503552396696
1024 14.4522334366119
1032 14.4540475187004
1040 14.4557996811952
1048 14.4574920436081
1056 14.4591266523687
1064 14.4607054833911
1072 14.4622304445476
1080 14.4637033780508
1088 14.4651260627491
1096 14.4665002163391
1104 14.4678274974969
1112 14.4691095079345
1120 14.470347794379
1128 14.4715438504846
1136 14.472699118672
1144 14.4738149919066
1152 14.4748928154091
1160 14.4759338883089
1168 14.476939465236
1176 14.4779107578596
1184 14.4788489363692
1192 14.4797551309061
1200 14.4806304329433
1208 14.4814758966174
1216 14.4822925400135
1224 14.4830813464058
1232 14.4838432654542
1240 14.4845792143594
1248 14.4852900789773
1256 14.4859767148962
1264 14.4866399484747
1272 14.4872805778443
1280 14.487899373878
1288 14.4884970811235
1296 14.4890744187064
1304 14.4896320812005
1312 14.4901707394685
1320 14.4906910414739
1328 14.4911936130645
1336 14.4916790587294
1344 14.4921479623296
1352 14.4926008878033
1360 14.4930383798474
1368 14.4934609645754
1376 14.4938691501519
1384 14.4942634274067
1392 14.4946442704272
1400 14.49501213713
1408 14.4953674698129
1416 14.4957106956897
1424 14.4960422274037
1432 14.4963624635267
1440 14.4966717890381
1448 14.4969705757899
1456 14.4972591829544
1464 14.4975379574572
1472 14.4978072343949
1480 14.4980673374389
1488 14.4983185792255
1496 14.498561261732
1504 14.4987956766408
1512 14.4990221056903
1520 14.4992408210143
1528 14.4994520854697
1536 14.499656152953
1544 14.4998532687053
1552 14.5000436696086
1560 14.5002275844701
1568 14.5004052342976
1576 14.5005768325659
1584 14.5007425854734
1592 14.5009026921906
1600 14.5010573450991
1608 14.5012067300236
1616 14.5013510264554
1624 14.5014904077683
1632 14.5016250414273
1640 14.5017550891895
1648 14.5018807072996
1656 14.5020020466774
1664 14.5021192530989
1672 14.5022324673726
1680 14.5023418255079
1688 14.5024474588795
1696 14.5025494943848
1704 14.5026480545969
1712 14.5027432579119
1720 14.5028352186912
1728 14.5029240473992
1736 14.503009850736
1744 14.5030927317658
1752 14.5031727900407
1760 14.5032501217208
1768 14.5033248196891
1776 14.503396973664
1784 14.5034666703064
1792 14.5035339933251
1800 14.5035990235755
1808 14.5036618391589
1816 14.5037225155152
1824 14.5037811255137
1832 14.5038377395411
1840 14.5038924255855
1848 14.5039452493192
1856 14.5039962741762
1864 14.5040455614297
1872 14.5040931702653
1880 14.504139157852
1888 14.5041835794111
1896 14.5042264882823
1904 14.5042679359887
1912 14.5043079722974
1920 14.5043466452803
1928 14.5043840013714
1936 14.504420085423
1944 14.5044549407592
1952 14.5044886092285
1960 14.5045211312533
1968 14.5045525458792
1976 14.5045828908217
1984 14.5046122025113
1992 14.5046405161377
2000 14.5046678656917
2008 14.5046942840062
2016 14.5047198027957
2024 14.5047444526944
2032 14.5047682632929
2040 14.504791263174
2048 14.5048134799467
2056 14.5048349402795
2064 14.5048556699327
2072 14.5048756937891
2080 14.5048950358838
2088 14.5049137194336
2096 14.5049317668641
2104 14.5049491998371
2112 14.5049660392769
2120 14.5049823053945
2128 14.5049980177131
2136 14.5050131950909
2144 14.5050278557433
2152 14.505042017266
2160 14.5050556966547
2168 14.5050689103267
2176 14.5050816741403
2184 14.5050940034131
2192 14.5051059129414
2200 14.5051174170177
2208 14.5051285294475
2216 14.5051392635663
2224 14.5051496322551
2232 14.5051596479567
2240 14.50516932269
2248 14.5051786680644
2256 14.505187695294
2264 14.5051964152111
2272 14.5052048382791
2280 14.5052129746051
2288 14.5052208339517
2296 14.5052284257495
2304 14.5052357591078
2312 14.5052428428255
2320 14.5052496854021
2328 14.5052562950475
2336 14.5052626796921
2344 14.5052688469962
2352 14.5052748043596
2360 14.5052805589297
2368 14.5052861176105
2376 14.5052914870712
2384 14.5052966737538
2392 14.505301683881
2400 14.5053065234635
2408 14.5053111983074
2416 14.505315714021
2424 14.5053200760219
2432 14.5053242895427
2440 14.5053283596386
2448 14.505332291192
2456 14.5053360889195
2464 14.5053397573767
2472 14.5053433009646
2480 14.505346723934
2488 14.5053500303908
2496 14.5053532243018
};
\addlegendentry{$P_1$}
\addplot [semithick, color1]
table {%
0 5.04029252221767
8 5.68904562328921
16 6.25931347990043
24 6.76222695315433
32 7.20724845762266
40 7.60238328637118
48 7.95438072524796
56 8.26893274038866
64 8.55085238145116
72 8.80422586429223
80 9.03253863697572
88 9.23877833741131
96 9.42551830149004
104 9.59498518133005
112 9.74911379571214
120 9.88959180847906
128 10.0178963283931
136 10.1353240891655
144 10.2430165109772
152 10.3419806597857
160 10.4331068975021
168 10.5171838432085
176 10.5949111323792
184 10.6669103586679
192 10.7337345040438
200 10.7958761022723
208 10.8537743336265
216 10.9078212119854
224 10.9583669966437
232 11.0057249383567
240 11.0501754509621
248 11.091969785308
256 11.131333270367
264 11.1684681767365
272 11.2035562497447
280 11.2367609527615
288 11.2682294557722
296 11.2980943996133
304 11.3264754623151
312 11.3534807506339
320 11.3792080369727
328 11.4037458594098
336 11.4271745004109
344 11.4495668579436
352 11.4709892210924
360 11.4915019608637
368 11.5111601456327
376 11.5300140896004
384 11.5481098416777
392 11.5654896213717
400 11.5821922075127
408 11.5982532850055
416 11.6137057542084
424 11.6285800070374
432 11.6429041734332
440 11.6567043414307
448 11.6700047537148
456 11.6828279832241
464 11.6951950900898
472 11.7071257619407
480 11.7186384393862
488 11.7297504282904
496 11.7404780002732
504 11.750836482719
512 11.7608403394332
520 11.7705032429612
528 11.7798381394761
536 11.7888573070411
544 11.7975724079644
552 11.8059945358869
560 11.8141342581736
568 11.822001654115
576 11.8296063493921
584 11.8369575472084
592 11.8440640564462
600 11.8509343171686
608 11.8575764237513
616 11.8639981458973
624 11.8702069477612
632 11.8762100053832
640 11.8820142226127
648 11.8876262456798
656 11.8930524765573
664 11.8982990852394
672 11.9033720210491
680 11.9082770230744
688 11.9130196298227
696 11.9176051881721
704 11.9220388616901
712 11.9263256383837
720 11.930470337934
728 11.9344776184687
736 11.9383519829132
744 11.9420977849621
752 11.9457192347053
760 11.94922040394
768 11.9526052311965
776 11.9558775265023
784 11.9590409759076
792 11.9620991457908
800 11.9650554869623
808 11.9679133385826
816 11.9706759319071
824 11.9733463938735
832 11.975927750539
840 11.978422930381
848 11.9808347674684
856 11.9831660045116
864 11.9854192958005
872 11.9875972100344
880 11.989702233052
888 11.9917367704659
896 11.993703150207
904 11.9956036249828
912 11.9974403746543
920 11.9992155085349
928 12.0009310676146
936 12.002589026713
944 12.0041912965644
952 12.0057397258362
960 12.0072361030855
968 12.0086821586537
976 12.0100795665031
984 12.0114299459967
992 12.0127348636239
1000 12.0139958346727
1008 12.0152143248511
1016 12.0163917518593
1024 12.0175294869139
1032 12.0186288562258
1040 12.0196911424325
1048 12.0207175859879
1056 12.0217093865089
1064 12.0226677040812
1072 12.0235936605257
1080 12.0244883406255
1088 12.0253527933155
1096 12.0261880328361
1104 12.0269950398511
1112 12.0277747625308
1120 12.0285281176028
1128 12.0292559913688
1136 12.0299592406916
1144 12.0306386939504
1152 12.0312951519662
1160 12.0319293888998
1168 12.0325421531199
1176 12.0331341680453
1184 12.0337061329604
1192 12.0342587238045
1200 12.0347925939373
1208 12.035308374879
1216 12.0358066770277
1224 12.0362880903545
1232 12.0367531850749
1240 12.0372025123007
1248 12.0376366046695
1256 12.0380559769553
1264 12.0384611266588
1272 12.0388525345789
1280 12.0392306653659
1288 12.0395959680569
1296 12.0399488765937
1304 12.0402898103242
1312 12.0406191744877
1320 12.0409373606836
1328 12.0412447473259
1336 12.0415417000825
1344 12.0418285723
1352 12.0421057054147
1360 12.0423734293504
1368 12.0426320629031
1376 12.0428819141128
1384 12.0431232806232
1392 12.0433564500302
1400 12.0435817002179
1408 12.0437992996841
1416 12.0440095078554
1424 12.044212575391
1432 12.0444087444772
1440 12.0445982491118
1448 12.0447813153795
1456 12.0449581617173
1464 12.0451289991721
1472 12.0452940316494
1480 12.045453456153
1488 12.0456074630182
1496 12.0457562361354
1504 12.0458999531682
1512 12.0460387857626
1520 12.0461728997503
1528 12.0463024553444
1536 12.0464276073294
1544 12.0465485052439
1552 12.046665293558
1560 12.0467781118446
1568 12.0468870949442
1576 12.0469923731253
1584 12.0470940722389
1592 12.0471923138675
1600 12.0472872154697
1608 12.0473788905194
1616 12.0474674486409
1624 12.0475529957389
1632 12.0476356341243
1640 12.0477154626361
1648 12.0477925767587
1656 12.0478670687356
1664 12.047939027679
1672 12.048008539676
1680 12.0480756878912
1688 12.0481405526651
1696 12.0482032116108
1704 12.0482637397052
1712 12.0483222093797
1720 12.0483786906052
1728 12.0484332509765
1736 12.0484859557923
1744 12.0485368681333
1752 12.0485860489375
1760 12.0486335570725
1768 12.0486794494061
1776 12.0487237808739
1784 12.0487666045452
1792 12.0488079716859
1800 12.0488479318202
1808 12.048886532789
1816 12.0489238208077
1824 12.0489598405213
1832 12.0489946350572
1840 12.0490282460771
1848 12.049060713827
1856 12.0490920771847
1864 12.0491223737068
1872 12.0491516396731
1880 12.0491799101303
1888 12.0492072189337
1896 12.0492335987879
1904 12.0492590812854
1912 12.049283696945
1920 12.0493074752479
1928 12.0493304446728
1936 12.0493526327303
1944 12.0493740659957
1952 12.0493947701408
1960 12.0494147699645
1968 12.0494340894224
1976 12.0494527516556
1984 12.0494707790184
1992 12.049488193105
2000 12.0495050147753
2008 12.0495212641797
2016 12.0495369607836
2024 12.0495521233903
2032 12.0495667701638
2040 12.0495809186504
2048 12.0495945857996
2056 12.0496077879845
2064 12.0496205410215
2072 12.0496328601888
2080 12.0496447602455
2088 12.0496562554484
2096 12.0496673595695
2104 12.0496780859123
2112 12.0496884473281
2120 12.0496984562307
2128 12.049708124612
2136 12.049717464056
2144 12.0497264857527
2152 12.0497352005113
2160 12.0497436187736
2168 12.0497517506261
2176 12.0497596058123
2184 12.049767193744
2192 12.0497745235134
2200 12.049781603903
2208 12.0497884433966
2216 12.0497950501897
2224 12.0498014321986
2232 12.0498075970708
2240 12.0498135521933
2248 12.049819304702
2256 12.0498248614899
2264 12.0498302292158
2272 12.0498354143116
2280 12.0498404229909
2288 12.0498452612554
2296 12.0498499349029
2304 12.049854449534
2312 12.0498588105587
2320 12.0498630232031
2328 12.0498670925153
2336 12.0498710233717
2344 12.0498748204829
2352 12.0498784883993
2360 12.0498820315165
2368 12.0498854540803
2376 12.0498887601925
2384 12.0498919538148
2392 12.0498950387748
2400 12.0498980187693
2408 12.0499008973696
2416 12.0499036780255
2424 12.0499063640694
2432 12.0499089587203
2440 12.0499114650876
2448 12.0499138861752
2456 12.0499162248845
2464 12.0499184840184
2472 12.0499206662844
2480 12.0499227742978
2488 12.0499248105849
2496 12.0499267775863
};
\addlegendentry{$P_2$}
\addplot [semithick, color2]
table {%
0 5.03801295688445
8 5.71243844342182
16 6.34197676356078
24 6.92829093884511
32 7.47364883269318
40 7.98061818831705
48 8.45186387657647
56 8.89002400573743
64 9.29763763666503
72 9.67710513125027
80 10.0306690157346
88 10.3604076118366
96 10.6682364705253
104 10.9559144229446
112 11.2250522075571
120 11.4771223721873
128 11.7134696294157
136 11.9353211557973
144 12.1437965284588
152 12.3399171246059
160 12.5246148948123
168 12.6987404756098
176 12.8630706413067
184 13.0183151160067
192 13.165122779046
200 13.3040873035436
208 13.4357522704939
216 13.5606158011719
224 13.6791347494717
232 13.7917284937742
240 13.8987823654163
248 14.0006507481043
256 14.0976598798232
264 14.1901103860613
272 14.2782795705669
280 14.3624234873974
288 14.4427788157553
296 14.519564557011
304 14.5929835714082
312 14.663223970209
320 14.7304603774644
328 14.7948550741772
336 14.8565590363417
344 14.9157128771935
352 14.9724477029644
360 15.0268858905078
368 15.0791417943214
376 15.1293223897436
384 15.1775278584258
392 15.2238521215767
400 15.2683833259311
408 15.3112042869074
416 15.3523928929798
424 15.3920224748965
432 15.4301621430217
440 15.4668770957586
448 15.5022289017294
456 15.5362757581219
464 15.5690727273899
472 15.600671954278
480 15.6311228649572
488 15.6604723498858
496 15.6887649318604
504 15.716042920578
512 15.7423465549129
520 15.7677141339927
528 15.7921821380614
536 15.8157853400229
544 15.838556908476
552 15.8605285029787
560 15.8817303622081
568 15.9021913856257
576 15.9219392091969
584 15.9410002756696
592 15.9593998998649
600 15.9771623293969
608 15.9943108011972
616 16.0108675941896
624 16.0268540784269
632 16.0422907609749
640 16.0571973288031
648 16.0715926889208
656 16.0854950059729
664 16.0989217374931
672 16.1118896669954
680 16.1244149350682
688 16.1365130686204
696 16.1481990084184
704 16.1594871350397
712 16.1703912933568
720 16.1809248156589
728 16.191100543507
736 16.2009308484115
744 16.210427651413
752 16.2196024416425
760 16.2284662939277
768 16.2370298855108
776 16.2453035119335
784 16.2532971021446
792 16.261020232879
800 16.2684821423529
808 16.2756917433185
816 16.2826576355152
824 16.2893881175547
832 16.2958911982732
840 16.3021746075797
848 16.3082458068311
856 16.3141119987596
864 16.3197801369759
872 16.3252569350734
880 16.3305488753522
888 16.3356622171849
896 16.3406030050406
904 16.3453770761864
912 16.3499900680806
920 16.354447425474
928 16.3587544072333
936 16.3629160928988
944 16.36693738899
952 16.3708230350703
960 16.3745776095819
968 16.3782055354605
976 16.3817110855408
984 16.385098387761
992 16.3883714301754
1000 16.3915340657826
1008 16.3945900171787
1016 16.3975428810405
1024 16.4003961324477
1032 16.4031531290494
1040 16.405817115082
1048 16.4083912252433
1056 16.4108784884302
1064 16.4132818313438
1072 16.4156040819675
1080 16.4178479729231
1088 16.4200161447099
1096 16.4221111488305
1104 16.4241354508079
1112 16.4260914330981
1120 16.4279813979021
1128 16.4298075698805
1136 16.4315720987746
1144 16.4332770619377
1152 16.4349244667796
1160 16.4365162531268
1168 16.4380542955028
1176 16.4395404053298
1184 16.4409763330561
1192 16.4423637702102
1200 16.4437043513859
1208 16.4449996561594
1216 16.4462512109419
1224 16.4474604907687
1232 16.4486289210284
1240 16.449757879133
1248 16.4508486961322
1256 16.4519026582724
1264 16.4529210085039
1272 16.4539049479365
1280 16.4548556372468
1288 16.4557741980367
1296 16.4566617141482
1304 16.457519232932
1312 16.4583477664746
1320 16.4591482927835
1328 16.4599217569327
1336 16.46066907217
1344 16.4613911209865
1352 16.4620887561505
1360 16.4627628017069
1368 16.4634140539424
1376 16.464043282319
1384 16.4646512303752
1392 16.4652386165981
1400 16.4658061352655
1408 16.4663544572595
1416 16.4668842308537
1424 16.4673960824728
1432 16.467890617428
1440 16.4683684206266
1448 16.4688300572581
1456 16.469276073458
1464 16.469706996948
1472 16.470123337656
1480 16.4705255883145
1488 16.4709142250391
1496 16.4712897078874
1504 16.4716524814
1512 16.4720029751219
1520 16.4723416041079
1528 16.4726687694097
1536 16.4729848585479
1544 16.4732902459668
1552 16.4735852934754
1560 16.473870350672
1568 16.4741457553562
1576 16.4744118339256
1584 16.4746689017603
1592 16.4749172635935
1600 16.4751572138707
1608 16.4753890370955
1616 16.4756130081657
1624 16.4758293926959
1632 16.4760384473307
1640 16.4762404200469
1648 16.4764355504461
1656 16.4766240700365
1664 16.4768062025056
1672 16.4769821639845
1680 16.4771521633018
1688 16.4773164022306
1696 16.4774750757262
1704 16.4776283721555
1712 16.4777764735201
1720 16.4779195556704
1728 16.4780577885129
1736 16.4781913362115
1744 16.4783203573807
1752 16.4784450052726
1760 16.4785654279585
1768 16.4786817685032
1776 16.4787941651343
1784 16.4789027514052
1792 16.4790076563532
1800 16.4791090046516
1808 16.4792069167571
1816 16.4793015090524
1824 16.4793928939836
1832 16.4794811801929
1840 16.4795664726476
1848 16.4796488727637
1856 16.479728478526
1864 16.4798053846041
1872 16.4798796824644
1880 16.479951460478
1888 16.4800208040259
1896 16.4800877955992
1904 16.4801525148974
1912 16.4802150389224
1920 16.48027544207
1928 16.4803337962174
1936 16.4803901708089
1944 16.4804446329379
1952 16.4804972474263
1960 16.4805480769016
1968 16.4805971818704
1976 16.4806446207909
1984 16.4806904501418
1992 16.4807347244891
2000 16.4807774965512
2008 16.4808188172608
2016 16.4808587358257
2024 16.4808972997872
2032 16.4809345550762
2040 16.4809705460678
2048 16.481005315634
2056 16.4810389051943
2064 16.4810713547652
2072 16.4811027030075
2080 16.4811329872719
2088 16.481162243644
2096 16.4811905069865
2104 16.4812178109806
2112 16.4812441881664
2120 16.4812696699811
2128 16.4812942867968
2136 16.481318067956
2144 16.4813410418069
2152 16.4813632357368
2160 16.4813846762048
2168 16.4814053887731
2176 16.4814253981374
2184 16.4814447281563
2192 16.4814634018796
2200 16.4814814415755
2208 16.4814988687576
2216 16.4815157042098
2224 16.4815319680115
2232 16.4815476795612
2240 16.4815628575997
2248 16.4815775202324
2256 16.4815916849506
2264 16.4816053686524
2272 16.4816185876632
2280 16.4816313577544
2288 16.4816436941625
2296 16.4816556116076
2304 16.4816671243102
2312 16.4816782460086
2320 16.4816889899751
2328 16.4816993690318
2336 16.4817093955659
2344 16.4817190815444
2352 16.4817284385282
2360 16.4817374776861
2368 16.4817462098081
2376 16.4817546453179
2384 16.4817627942859
2392 16.4817706664404
2400 16.48177827118
2408 16.4817856175842
2416 16.4817927144245
2424 16.4817995701747
2432 16.481806193021
2440 16.481812590872
2448 16.4818187713677
2456 16.481824741889
2464 16.4818305095662
2472 16.4818360812877
2480 16.4818414637081
2488 16.481846663256
2496 16.4818516861421
};
\addlegendentry{$P_3$}
\end{axis}

%% file: suppl_tex/p2_mean.tex
\definecolor{color0}{rgb}{0.12156862745098,0.466666666666667,0.705882352941177}
\definecolor{color1}{rgb}{1,0.498039215686275,0.0549019607843137}
\definecolor{color2}{rgb}{0.172549019607843,0.627450980392157,0.172549019607843}

\begin{axis}[
legend cell align={left},
legend style={
  fill opacity=0.8,
  draw opacity=1,
  text opacity=1,
  at={(0.03,0.97)},
  anchor=north west,
  draw=white!80!black
},
tick align=outside,
tick pos=left,
x grid style={white!69.0196078431373!black},
xlabel={t},
xmin=-124.8, xmax=2620.8,
xtick style={color=black},
y grid style={white!69.0196078431373!black},
ymin=69.5857456804918, ymax=70.5629201272821,
ytick style={color=black}
]
\addplot [semithick, color0]
table {%
0 70.5185031069735
8 69.9008139107911
16 69.8936723923335
24 69.8913647694675
32 69.8893236728137
40 69.8874871644228
48 69.8858354955669
56 69.884351120197
64 69.8830181547963
72 69.8818222114948
80 69.8807502530696
88 69.8797904634784
96 69.87893213174
104 69.8781655474508
112 69.877481906644
120 69.876873226806
128 69.8763322700945
136 69.8758524739076
144 69.8754278880567
152 69.8750531178766
160 69.8747232726939
168 69.8744339191404
176 69.874181038821
184 69.8739609899242
192 69.8737704724261
200 69.8736064965019
208 69.8734663538796
216 69.873347591831
224 69.8732479895799
232 69.8731655368784
240 69.8730984145593
248 69.8730449768789
256 69.8730037354834
264 69.8729733448518
272 69.8729525890813
280 69.8729403698706
288 69.8729356956219
296 69.8729376715362
304 69.8729454906197
312 69.8729584255091
320 69.8729758210511
328 69.8729970875671
336 69.873021694732
344 69.8730491660011
352 69.8730790735749
360 69.8731110338082
368 69.8731447030344
376 69.8731797738017
384 69.873215971415
392 69.8732530508351
400 69.8732907938424
408 69.8733290064548
416 69.8733675166081
424 69.8734061720311
432 69.8734448383183
440 69.8734833971913
448 69.873521744913
456 69.8735597908502
464 69.8735974561795
472 69.8736346726863
480 69.8736713817135
488 69.8737075331804
496 69.8737430847055
504 69.8737780008065
512 69.8738122521828
520 69.8738458150553
528 69.8738786705849
536 69.8739108043165
544 69.8739422057241
552 69.8739728677436
560 69.874002786396
568 69.8740319604267
576 69.8740603909716
584 69.8740880812785
592 69.8741150364412
600 69.8741412631567
608 69.874166769523
616 69.8741915648384
624 69.8742156594309
632 69.874239064508
640 69.8742617920105
648 69.8742838544942
656 69.8743052650183
664 69.8743260370433
672 69.8743461843401
680 69.874365720918
688 69.8743846609506
696 69.8744030187077
704 69.8744208085133
712 69.8744380446841
720 69.8744547414923
728 69.8744709131325
736 69.8744865736758
744 69.8745017370539
752 69.8745164170321
760 69.8745306271804
768 69.8745443808637
776 69.8745576912245
784 69.8745705711642
792 69.8745830333446
800 69.8745950901672
808 69.874606753771
816 69.8746180360287
824 69.8746289485437
832 69.8746395026437
840 69.8746497093816
848 69.8746595795333
856 69.8746691236064
864 69.8746783518272
872 69.8746872741515
880 69.8746959002705
888 69.8747042396057
896 69.8747123013165
904 69.8747200943052
912 69.8747276272134
920 69.8747349084426
928 69.8747419461368
936 69.8747487482106
944 69.8747553223347
952 69.874761675951
960 69.8747678162806
968 69.8747737503172
976 69.8747794848436
984 69.874785026432
992 69.87479038145
1000 69.8747955560657
1008 69.8748005562544
1016 69.8748053878008
1024 69.8748100563089
1032 69.8748145672006
1040 69.874818925725
1048 69.8748231369676
1056 69.8748272058434
1064 69.8748311371095
1072 69.874834935376
1080 69.8748386050924
1088 69.874842150572
1096 69.87484557598
1104 69.8748488853492
1112 69.874852082582
1120 69.8748551714493
1128 69.8748581555949
1136 69.8748610385466
1144 69.8748638237178
1152 69.8748665144017
1160 69.8748691137864
1168 69.8748716249538
1176 69.8748740508844
1184 69.8748763944571
1192 69.8748786584535
1200 69.874880845566
1208 69.8748829583943
1216 69.8748849994543
1224 69.8748869711748
1232 69.8748888759001
1240 69.8748907159033
1248 69.8748924933747
1256 69.8748942104353
1264 69.8748958691315
1272 69.874897471445
1280 69.8748990192835
1288 69.8749005144975
1296 69.8749019588745
1304 69.8749033541356
1312 69.8749047019484
1320 69.8749060039264
1328 69.8749072616251
1336 69.8749084765475
1344 69.8749096501435
1352 69.874910783823
1360 69.8749118789359
1368 69.8749129367964
1376 69.8749139586675
1384 69.8749149457756
1392 69.8749158993011
1400 69.8749168203827
1408 69.8749177101282
1416 69.8749185695994
1424 69.8749193998239
1432 69.8749202018008
1440 69.8749209764881
1448 69.8749217248142
1456 69.8749224476733
1464 69.8749231459354
1472 69.8749238204364
1480 69.8749244719821
1488 69.8749251013532
1496 69.8749257093094
1504 69.8749262965726
1512 69.8749268638495
1520 69.8749274118218
1528 69.874927941143
1536 69.8749284524491
1544 69.8749289463534
1552 69.8749294234493
1560 69.8749298843063
1568 69.8749303294779
1576 69.8749307594963
1584 69.8749311748817
1592 69.8749315761279
1600 69.8749319637137
1608 69.8749323381132
1616 69.8749326997673
1624 69.8749330491104
1632 69.8749333865632
1640 69.8749337125342
1648 69.8749340274066
1656 69.8749343315615
1664 69.8749346253669
1672 69.8749349091705
1680 69.8749351833095
1688 69.874935448124
1696 69.8749357039227
1704 69.8749359510149
1712 69.8749361896973
1720 69.8749364202544
1728 69.8749366429646
1736 69.8749368580951
1744 69.8749370659004
1752 69.8749372666371
1760 69.8749374605393
1768 69.8749376478393
1776 69.8749378287675
1784 69.874938003535
1792 69.874938172354
1800 69.8749383354284
1808 69.874938492951
1816 69.8749386451125
1824 69.8749387920964
1832 69.8749389340743
1840 69.8749390712228
1848 69.8749392037018
1856 69.8749393316699
1864 69.8749394552872
1872 69.8749395746933
1880 69.8749396900346
1888 69.8749398014523
1896 69.8749399090763
1904 69.8749400130397
1912 69.874940113461
1920 69.8749402104655
1928 69.8749403041711
1936 69.8749403946828
1944 69.8749404821164
1952 69.8749405665752
1960 69.8749406481582
1968 69.8749407269632
1976 69.8749408030876
1984 69.8749408766206
1992 69.8749409476507
2000 69.8749410162634
2008 69.8749410825438
2016 69.8749411465638
2024 69.8749412084058
2032 69.8749412681451
2040 69.8749413258493
2048 69.8749413815886
2056 69.8749414354319
2064 69.8749414874437
2072 69.8749415376846
2080 69.8749415862176
2088 69.8749416330945
2096 69.8749416783803
2104 69.8749417221226
2112 69.8749417643748
2120 69.8749418051931
2128 69.8749418446182
2136 69.8749418827025
2144 69.8749419194915
2152 69.8749419550277
2160 69.8749419893553
2168 69.8749420225122
2176 69.8749420545451
2184 69.8749420854856
2192 69.8749421153709
2200 69.8749421442419
2208 69.8749421721294
2216 69.8749421990675
2224 69.8749422250879
2232 69.8749422502249
2240 69.8749422745059
2248 69.8749422979596
2256 69.8749423206152
2264 69.8749423425007
2272 69.8749423636429
2280 69.8749423840611
2288 69.8749424037865
2296 69.8749424228418
2304 69.8749424412466
2312 69.8749424590267
2320 69.8749424762013
2328 69.8749424927904
2336 69.8749425088155
2344 69.8749425242959
2352 69.8749425392494
2360 69.8749425536946
2368 69.8749425676448
2376 69.8749425811246
2384 69.8749425941443
2392 69.8749426067182
2400 69.8749426188684
2408 69.8749426306022
2416 69.8749426419382
2424 69.8749426528861
2432 69.8749426634637
2440 69.8749426736818
2448 69.8749426835494
2456 69.8749426930823
2464 69.8749427022945
2472 69.8749427111889
2480 69.8749427197815
2488 69.8749427280818
2496 69.8749427361007
};
\addlegendentry{$P_1$}
\addplot [semithick, color1]
table {%
0 70.3129586183939
8 69.7590067248077
16 69.7427095088503
24 69.7371287323018
32 69.7321874178117
40 69.7277409736187
48 69.7237416826583
56 69.7201471276169
64 69.7169189072524
72 69.7140222337069
80 69.7114255822543
88 69.7091003786524
96 69.7070207188145
104 69.7051631166188
112 69.7035062766673
120 69.7020308891338
128 69.7007194443773
136 69.6995560652744
144 69.6985263554715
152 69.6976172619492
160 69.6968169505178
168 69.6961146929917
176 69.6955007648809
184 69.6949663525972
192 69.6945034693382
200 69.694104878712
208 69.6937640254782
216 69.6934749726568
224 69.6932323444692
232 69.6930312745343
240 69.6928673588281
248 69.6927366129761
256 69.6926354334743
264 69.6925605624697
272 69.6925090557974
280 69.6924782539048
288 69.6924657554941
296 69.6924693935494
304 69.6924872135828
312 69.692517453853
320 69.6925585274101
328 69.692609005799
336 69.6926676042424
344 69.6927331681646
352 69.6928046609948
360 69.692881153058
368 69.6929618114861
376 69.6930458910797
384 69.6931327259655
392 69.693221722079
400 69.693312350294
408 69.6934041401868
416 69.693496674409
424 69.6935895835521
432 69.6936825414794
440 69.6937752611122
448 69.6938674905997
456 69.6939590098357
464 69.694049627318
472 69.6941391772564
480 69.6942275170129
488 69.6943145247305
496 69.6944000972172
504 69.6944841479959
512 69.6945666055754
520 69.6946474118521
528 69.6947265206923
536 69.6948038966022
544 69.6948795136081
552 69.694953354139
560 69.6950254081073
568 69.695095672032
576 69.6951641482395
584 69.6952308441827
592 69.6952957717936
600 69.695358946902
608 69.6954203887424
616 69.6954801194621
624 69.6955381637267
632 69.6955945483252
640 69.695649301847
648 69.6957024543783
656 69.6957540372286
664 69.6958040826939
672 69.695852623833
680 69.69589969429
688 69.6959453281122
696 69.6959895595958
704 69.6960324231717
712 69.6960739532584
720 69.6961141841872
728 69.6961531500938
736 69.6961908848329
744 69.6962274219278
752 69.6962627944993
760 69.696297035211
768 69.6963301762306
776 69.6963622491931
784 69.6963932851626
792 69.6964233146209
800 69.6964523674293
808 69.6964804728213
816 69.6965076593853
824 69.6965339550678
832 69.6965593871491
840 69.6965839822514
848 69.6966077663337
856 69.6966307647005
864 69.6966530019874
872 69.6966745021807
880 69.6966952886235
888 69.6967153840132
896 69.6967348104185
904 69.6967535892837
912 69.6967717414343
920 69.6967892871058
928 69.6968062459252
936 69.6968226369604
944 69.6968384786932
952 69.6968537890579
960 69.696868585451
968 69.6968828847227
976 69.6968967032233
984 69.6969100567832
992 69.6969229607499
1000 69.6969354299866
1008 69.6969474788902
1016 69.6969591214008
1024 69.6969703710239
1032 69.6969812408235
1040 69.6969917434513
1048 69.6970018911592
1056 69.6970116957918
1064 69.6970211688167
1072 69.6970303213373
1080 69.6970391640828
1088 69.697047707441
1096 69.6970559614538
1104 69.6970639358411
1112 69.697071640002
1120 69.6970790830269
1128 69.6970862736979
1136 69.6970932205143
1144 69.6970999317031
1152 69.6971064152008
1160 69.6971126786881
1168 69.6971187295964
1176 69.6971245751031
1184 69.697130222149
1192 69.697135677439
1200 69.6971409474573
1208 69.6971460384725
1216 69.6971509565465
1224 69.6971557075286
1232 69.6971602970758
1240 69.6971647306597
1248 69.6971690135601
1256 69.6971731508903
1264 69.6971771475803
1272 69.6971810084039
1280 69.697184737963
1288 69.697188340713
1296 69.6971918209631
1304 69.6971951828599
1312 69.6971984304237
1320 69.6972015675439
1328 69.6972045979596
1336 69.6972075253027
1344 69.69721035307
1352 69.6972130846466
1360 69.6972157232953
1368 69.6972182721814
1376 69.6972207343465
1384 69.6972231127445
1392 69.6972254102227
1400 69.6972276295269
1408 69.697229773322
1416 69.6972318441633
1424 69.6972338445414
1432 69.6972357768505
1440 69.6972376434026
1448 69.697239446437
1456 69.6972411881106
1464 69.6972428705137
1472 69.6972444956614
1480 69.6972460654997
1488 69.6972475819104
1496 69.697249046715
1504 69.697250461661
1512 69.6972518284548
1520 69.6972531487283
1528 69.6972544240642
1536 69.6972556559949
1544 69.6972568459936
1552 69.697257995493
1560 69.6972591058642
1568 69.6972601784424
1576 69.6972612145124
1584 69.6972622153185
1592 69.6972631820581
1600 69.697264115889
1608 69.6972650179427
1616 69.6972658892886
1624 69.6972667309742
1632 69.6972675440085
1640 69.6972683293771
1648 69.6972690880056
1656 69.6972698208138
1664 69.6972705286795
1672 69.6972712124496
1680 69.6972718729375
1688 69.6972725109552
1696 69.6972731272493
1704 69.697273722567
1712 69.697274297624
1720 69.6972748531036
1728 69.6972753896749
1736 69.6972759079852
1744 69.6972764086489
1752 69.6972768922746
1760 69.6972773594368
1768 69.697277810696
1776 69.6972782465971
1784 69.6972786676589
1792 69.6972790743899
1800 69.6972794672772
1808 69.6972798467895
1816 69.6972802133838
1824 69.6972805675039
1832 69.6972809095654
1840 69.6972812399862
1848 69.6972815591586
1856 69.697281867467
1864 69.6972821652872
1872 69.6972824529647
1880 69.69728273085
1888 69.6972829992798
1896 69.6972832585693
1904 69.6972835090388
1912 69.6972837509782
1920 69.6972839846833
1928 69.6972842104344
1936 69.6972844284998
1944 69.6972846391458
1952 69.6972848426212
1960 69.6972850391707
1968 69.6972852290299
1976 69.6972854124286
1984 69.697285589584
1992 69.6972857607081
2000 69.6972859260109
2008 69.6972860856893
2016 69.6972862399243
2024 69.6972863889163
2032 69.6972865328362
2040 69.6972866718555
2048 69.6972868061439
2056 69.6972869358615
2064 69.6972870611667
2072 69.6972871822042
2080 69.6972872991279
2088 69.6972874120653
2096 69.6972875211617
2104 69.6972876265432
2112 69.6972877283386
2120 69.6972878266736
2128 69.6972879216551
2136 69.697288013408
2144 69.6972881020379
2152 69.6972881876496
2160 69.6972882703494
2168 69.6972883502315
2176 69.6972884273997
2184 69.6972885019375
2192 69.697288573939
2200 69.6972886434934
2208 69.6972887106757
2216 69.6972887755731
2224 69.6972888382619
2232 69.6972888988203
2240 69.697288957315
2248 69.6972890138178
2256 69.6972890684004
2264 69.6972891211222
2272 69.6972891720546
2280 69.6972892212471
2288 69.6972892687661
2296 69.697289314673
2304 69.6972893590142
2312 69.6972894018478
2320 69.6972894432236
2328 69.6972894831879
2336 69.697289521794
2344 69.69728955909
2352 69.6972895951117
2360 69.6972896299097
2368 69.6972896635211
2376 69.6972896959925
2384 69.6972897273583
2392 69.6972897576522
2400 69.6972897869226
2408 69.6972898151901
2416 69.6972898424991
2424 69.6972898688769
2432 69.6972898943589
2440 69.6972899189702
2448 69.6972899427458
2456 69.6972899657116
2464 69.6972899879016
2472 69.6972900093287
2480 69.6972900300308
2488 69.6972900500277
2496 69.6972900693448
};
\addlegendentry{$P_2$}
\addplot [semithick, color2]
table {%
0 70.1474241150787
8 69.7139059276145
16 69.6907382766689
24 69.6839730906194
32 69.6780136712586
40 69.6726518770451
48 69.6678297322312
56 69.6634959945403
64 69.6596042863749
72 69.65611260031
80 69.6529828738252
88 69.6501806104714
96 69.6476745404513
104 69.6454363153799
112 69.6434402333155
120 69.6416629905862
128 69.6400834575786
136 69.6386824760077
144 69.637442675476
152 69.6363483073695
160 69.6353850944117
168 69.6345400943499
176 69.6338015763695
184 69.6331589090185
192 69.6326024586213
200 69.6321234970547
208 69.6317141181244
216 69.6313671616354
224 69.6310761445121
232 69.6308351982631
240 69.6306390122035
248 69.6304827819086
256 69.6303621624146
264 69.6302732257124
272 69.6302124221764
280 69.630176545482
288 69.6301627008004
296 69.6301682758711
304 69.6301909147581
312 69.630228493983
320 69.6302791008643
328 69.6303410138534
336 69.6304126846653
344 69.630492722019
352 69.6305798769258
360 69.6306730292881
368 69.6307711757327
376 69.6308734185858
384 69.6309789558134
392 69.6310870719532
400 69.6311971298273
408 69.6313085630266
416 69.6314208691173
424 69.6315336034546
432 69.6316463735486
440 69.6317588339869
448 69.6318706818124
456 69.6319816523255
464 69.6320915152948
472 69.6322000714755
480 69.6323071495158
488 69.6324126031037
496 69.6325163084124
504 69.6326181617514
512 69.6327180774816
520 69.6328159860878
528 69.6329118324694
536 69.6330055743405
544 69.6330971808746
552 69.6331866313626
560 69.6332739141028
568 69.6333590253425
576 69.63344196832
584 69.6335227524372
592 69.6336013924852
600 69.6336779079408
608 69.6337523223719
616 69.6338246628481
624 69.6338949594624
632 69.6339632448552
640 69.6340295538283
648 69.6340939229765
656 69.6341563903597
664 69.6342169952154
672 69.6342757776928
680 69.6343327786354
688 69.634388039366
696 69.6344416014973
704 69.6344935067954
712 69.6345437970039
720 69.6345925137528
728 69.6346396984283
736 69.6346853920717
744 69.634729635316
752 69.6347724683024
760 69.6348139306138
768 69.6348540612293
776 69.6348928984787
784 69.6349304799977
792 69.6349668427146
800 69.6350020228027
808 69.6350360556728
816 69.6350689759542
824 69.6351008174954
832 69.6351316133373
840 69.6351613957223
848 69.6351901960894
856 69.6352180450824
864 69.6352449725357
872 69.6352710074959
880 69.6352961782301
888 69.6353205122227
896 69.6353440361962
904 69.6353667761176
912 69.6353887572031
920 69.6354100039542
928 69.6354305401371
936 69.6354503888363
944 69.6354695724323
952 69.6354881126382
960 69.6355060305171
968 69.635523346471
976 69.6355400802955
984 69.6355562511576
992 69.6355718776392
1000 69.635586977735
1008 69.6356015688764
1016 69.6356156679407
1024 69.6356292912796
1032 69.6356424547098
1040 69.6356551735495
1048 69.6356674626312
1056 69.6356793362959
1064 69.6356908084287
1072 69.6357018924709
1080 69.6357126014141
1088 69.635722947834
1096 69.6357329438894
1104 69.6357426013484
1112 69.6357519315903
1120 69.6357609456213
1128 69.6357696540747
1136 69.6357780672409
1144 69.6357861950801
1152 69.6357940472017
1160 69.6358016329058
1168 69.6358089611891
1176 69.6358160407423
1184 69.6358228799691
1192 69.6358294869873
1200 69.6358358696487
1208 69.6358420355424
1216 69.635847992008
1224 69.6358537461283
1232 69.6358593047532
1240 69.6358646745071
1248 69.6358698617833
1256 69.6358748727723
1264 69.635879713445
1272 69.6358843895801
1280 69.635888906751
1288 69.6358932703522
1296 69.6358974856005
1304 69.6359015575183
1312 69.6359054909733
1320 69.6359092906745
1328 69.6359129611499
1336 69.6359165067974
1344 69.6359199318516
1352 69.6359232404101
1360 69.6359264364257
1368 69.6359295237305
1376 69.6359325060077
1384 69.6359353868334
1392 69.6359381696566
1400 69.6359408578025
1408 69.635943454497
1416 69.6359459628358
1424 69.6359483858345
1432 69.6359507263917
1440 69.6359529873076
1448 69.6359551712938
1456 69.6359572809638
1464 69.635959318846
1472 69.6359612873827
1480 69.6359631889312
1488 69.6359650257689
1496 69.6359668001016
1504 69.635968514047
1512 69.6359701696704
1520 69.6359717689482
1528 69.6359733137975
1536 69.6359748060747
1544 69.6359762475627
1552 69.6359776399969
1560 69.6359789850379
1568 69.6359802843038
1576 69.635981539349
1584 69.6359827516813
1592 69.6359839227509
1600 69.6359850539603
1608 69.6359861466782
1616 69.6359872022023
1624 69.6359882217996
1632 69.6359892066921
1640 69.635990158073
1648 69.6359910770655
1656 69.6359919647826
1664 69.635992822286
1672 69.6359936506031
1680 69.6359944507198
1688 69.6359952236167
1696 69.6359959702002
1704 69.6359966913743
1712 69.6359973880069
1720 69.6359980609246
1728 69.6359987109384
1736 69.6359993388327
1744 69.6359999453512
1752 69.63600053123
1760 69.6360010971667
1768 69.6360016438392
1776 69.6360021719075
1784 69.6360026820003
1792 69.6360031747341
1800 69.6360036506974
1808 69.6360041104587
1816 69.636004554572
1824 69.636004983574
1832 69.6360053979687
1840 69.6360057982622
1848 69.6360061849294
1856 69.6360065584368
1864 69.6360069192376
1872 69.6360072677521
1880 69.6360076044052
1888 69.6360079296029
1896 69.6360082437286
1904 69.636008547169
1912 69.6360088402768
1920 69.6360091234095
1928 69.6360093969055
1936 69.6360096610921
1944 69.6360099162904
1952 69.6360101628012
1960 69.6360104009224
1968 69.6360106309389
1976 69.6360108531284
1984 69.6360110677546
1992 69.6360112750747
2000 69.6360114753419
2008 69.636011668795
2016 69.6360118556554
2024 69.6360120361634
2032 69.6360122105262
2040 69.6360123789524
2048 69.6360125416478
2056 69.6360126988054
2064 69.6360128506173
2072 69.6360129972589
2080 69.6360131389173
2088 69.636013275746
2096 69.6360134079209
2104 69.6360135355967
2112 69.6360136589268
2120 69.6360137780646
2128 69.6360138931404
2136 69.6360140043046
2144 69.6360141116848
2152 69.6360142154086
2160 69.6360143156046
2168 69.6360144123878
2176 69.6360145058817
2184 69.6360145961893
2192 69.6360146834245
2200 69.6360147676952
2208 69.6360148490913
2216 69.6360149277196
2224 69.6360150036723
2232 69.6360150770437
2240 69.6360151479144
2248 69.636015216373
2256 69.6360152825042
2264 69.6360153463813
2272 69.6360154080906
2280 69.6360154676923
2288 69.6360155252658
2296 69.6360155808866
2304 69.6360156346105
2312 69.6360156865075
2320 69.6360157366383
2328 69.636015785059
2336 69.6360158318346
2344 69.6360158770226
2352 69.6360159206668
2360 69.6360159628283
2368 69.6360160035528
2376 69.6360160428951
2384 69.6360160808981
2392 69.6360161176031
2400 69.6360161530678
2408 69.6360161873166
2416 69.6360162204047
2424 69.6360162523655
2432 69.6360162832393
2440 69.6360163130592
2448 69.636016341867
2456 69.6360163696928
2464 69.636016396578
2472 69.6360164225402
2480 69.6360164476237
2488 69.6360164718526
2496 69.6360164952577
};
\addlegendentry{$P_3$}
\end{axis}

%% file: suppl_tex/p3_mean.tex
\definecolor{color0}{rgb}{0.12156862745098,0.466666666666667,0.705882352941177}
\definecolor{color1}{rgb}{1,0.498039215686275,0.0549019607843137}
\definecolor{color2}{rgb}{0.172549019607843,0.627450980392157,0.172549019607843}

\begin{axis}[
legend cell align={left},
legend style={
  fill opacity=0.8,
  draw opacity=1,
  text opacity=1,
  at={(0.03,0.97)},
  anchor=north west,
  draw=white!80!black
},
tick align=outside,
tick pos=left,
x grid style={white!69.0196078431373!black},
xlabel={t},
xmin=-124.8, xmax=2620.8,
xtick style={color=black},
y grid style={white!69.0196078431373!black},
ymin=6.01808257167024, ymax=6.13105612246075,
ytick style={color=black}
]
\addplot [semithick, color0]
table {%
0 6.03762098368379
8 6.04506383109624
16 6.04484243367716
24 6.04467063939237
32 6.04454257598251
40 6.04445322146042
48 6.04439809544305
56 6.04437317935329
64 6.04437486347061
72 6.04439990322067
80 6.04444538159952
88 6.04450867632003
96 6.04458743077781
104 6.04467952819726
112 6.0447830684834
120 6.04489634741368
128 6.04501783787993
136 6.04514617294475
144 6.0452801305148
152 6.04541861946327
160 6.04556066705612
168 6.04570540755493
176 6.04585207188285
184 6.04599997825354
192 6.04614852367314
200 6.04629717623318
208 6.04644546812202
216 6.04659298928826
224 6.04673938169597
232 6.04688433411745
240 6.04702757741423
248 6.04716888026123
256 6.04730804527359
264 6.04744490549916
272 6.04757932124293
280 6.04771117719308
288 6.04784037982021
296 6.04796685502585
304 6.04809054601592
312 6.04821141137912
320 6.048329423351
328 6.04844456624679
336 6.04855683504708
344 6.04866623412199
352 6.04877277608159
360 6.04887648074014
368 6.04897737418352
376 6.04907548793108
384 6.04917085818153
392 6.04926352513665
400 6.04935353239416
408 6.04944092640374
416 6.04952575598048
424 6.04960807186956
432 6.04968792635828
440 6.04976537292964
448 6.04984046595485
456 6.04991326041992
464 6.0499838116836
472 6.05005217526328
480 6.05011840664671
488 6.050182561126
496 6.05024469365294
504 6.05030485871209
512 6.05036311021151
520 6.05041950138772
528 6.05047408472455
536 6.05052691188389
544 6.05057803364778
552 6.05062749986942
560 6.05067535943368
568 6.05072166022485
576 6.05076644910139
584 6.0508097718772
592 6.05085167330782
600 6.050892197082
608 6.05093138581715
616 6.05096928105885
624 6.05100592328353
632 6.05104135190438
640 6.05107560527916
648 6.05110872072103
656 6.05114073451081
664 6.05117168191088
672 6.0512015971807
680 6.0512305135935
688 6.05125846345374
696 6.05128547811567
704 6.05131158800233
712 6.05133682262533
720 6.05136121060474
728 6.05138477968959
736 6.05140755677816
744 6.05142956793885
752 6.05145083843084
760 6.05147139272457
768 6.05149125452256
776 6.05151044677961
784 6.05152899172308
792 6.05154691087295
800 6.0515642250616
808 6.05158095445311
816 6.05159711856253
824 6.05161273627479
832 6.05162782586284
840 6.05164240500611
848 6.05165649080794
856 6.0516700998132
864 6.05168324802498
872 6.05169595092125
880 6.05170822347106
888 6.05172008015031
896 6.05173153495702
904 6.05174260142626
912 6.05175329264488
920 6.05176362126564
928 6.05177359952078
936 6.0517832392359
944 6.05179255184238
952 6.05180154839056
960 6.0518102395618
968 6.05181863568037
976 6.0518267467252
984 6.05183458234096
992 6.05184215184901
1000 6.05184946425808
1008 6.05185652827434
1016 6.05186335231151
1024 6.05186994450036
1032 6.05187631269797
1040 6.05188246449707
1048 6.05188840723455
1056 6.05189414799997
1064 6.05189969364368
1072 6.05190505078518
1080 6.05191022582019
1088 6.05191522492867
1096 6.05192005408155
1104 6.05192471904814
1112 6.05192922540258
1120 6.05193357853047
1128 6.05193778363522
1136 6.05194184574404
1144 6.05194576971401
1152 6.05194956023773
1160 6.05195322184875
1168 6.05195675892703
1176 6.05196017570418
1184 6.05196347626822
1192 6.05196666456864
1200 6.05196974442096
1208 6.05197271951135
1216 6.05197559340081
1224 6.05197836952958
1232 6.05198105122111
1240 6.05198364168608
1248 6.05198614402615
1256 6.05198856123777
1264 6.05199089621538
1272 6.05199315175551
1280 6.05199533055935
1288 6.0519974352365
1296 6.05199946830804
1304 6.0520014322091
1312 6.05200332929228
1320 6.05200516183022
1328 6.05200693201828
1336 6.05200864197722
1344 6.05201029375568
1352 6.05201188933279
1360 6.05201343062025
1368 6.05201491946493
1376 6.05201635765072
1384 6.05201774690108
1392 6.05201908888072
1400 6.05202038519779
1408 6.05202163740574
1416 6.05202284700522
1424 6.05202401544584
1432 6.05202514412794
1440 6.05202623440413
1448 6.05202728758118
1456 6.05202830492126
1464 6.05202928764375
1472 6.05203023692647
1480 6.05203115390714
1488 6.05203203968483
1496 6.05203289532131
1504 6.05203372184207
1512 6.05203452023783
1520 6.05203529146555
1528 6.05203603644966
1536 6.05203675608313
1544 6.0520374512286
1552 6.05203812271929
1560 6.05203877136007
1568 6.0520393979285
1576 6.0520400031756
1584 6.05204058782692
1592 6.05204115258332
1600 6.05204169812167
1608 6.05204222509602
1616 6.05204273413804
1624 6.05204322585792
1632 6.05204370084512
1640 6.05204415966903
1648 6.05204460287964
1656 6.05204503100824
1664 6.05204544456808
1672 6.05204584405487
1680 6.05204622994753
1688 6.05204660270872
1696 6.05204696278513
1704 6.05204731060855
1712 6.05204764659592
1720 6.05204797114988
1728 6.05204828465965
1736 6.05204858750099
1744 6.05204888003691
1752 6.05204916261819
1760 6.05204943558347
1768 6.05204969926005
1776 6.05204995396397
1784 6.05205020000055
1792 6.05205043766478
1800 6.05205066724158
1808 6.05205088900608
1816 6.05205110322419
1824 6.05205131015273
1832 6.05205151003969
1840 6.05205170312469
1848 6.05205188963925
1856 6.05205206980693
1864 6.05205224384375
1872 6.05205241195823
1880 6.05205257435202
1888 6.05205273121976
1896 6.05205288274944
1904 6.05205302912282
1912 6.05205317051524
1920 6.05205330709626
1928 6.05205343902962
1936 6.05205356647344
1944 6.05205368958051
1952 6.05205380849847
1960 6.05205392336976
1968 6.05205403433216
1976 6.05205414151866
1984 6.0520542450577
1992 6.05205434507355
2000 6.05205444168598
2008 6.05205453501082
2016 6.05205462515994
2024 6.05205471224141
2032 6.05205479635968
2040 6.05205487761551
2048 6.0520549561063
2056 6.0520550319262
2064 6.05205510516607
2072 6.05205517591371
2080 6.05205524425395
2088 6.05205531026862
2096 6.052055374037
2104 6.05205543563538
2112 6.05205549513769
2120 6.05205555261527
2128 6.05205560813694
2136 6.05205566176934
2144 6.05205571357672
2152 6.05205576362116
2160 6.05205581196272
2168 6.05205585865928
2176 6.05205590376685
2184 6.05205594733951
2192 6.05205598942945
2200 6.05205603008716
2208 6.05205606936136
2216 6.05205610729911
2224 6.05205614394595
2232 6.05205617934575
2240 6.05205621354099
2248 6.05205624657262
2256 6.05205627848025
2264 6.05205630930211
2272 6.05205633907521
2280 6.05205636783516
2288 6.05205639561645
2296 6.05205642245241
2304 6.05205644837521
2312 6.05205647341595
2320 6.05205649760456
2328 6.05205652097008
2336 6.05205654354053
2344 6.05205656534299
2352 6.0520565864035
2360 6.05205660674742
2368 6.05205662639902
2376 6.05205664538198
2384 6.05205666371897
2392 6.05205668143198
2400 6.05205669854231
2408 6.05205671507035
2416 6.05205673103603
2424 6.05205674645837
2432 6.05205676135598
2440 6.05205677574664
2448 6.0520567896476
2456 6.05205680307556
2464 6.05205681604665
2472 6.0520568285763
2480 6.05205684067962
2488 6.05205685237108
2496 6.05205686366471
};
\addlegendentry{$P_1$}
\addplot [semithick, color1]
table {%
0 6.02321773306981
8 6.02786641230223
16 6.02772669481239
24 6.02761618161514
32 6.02753164718145
40 6.02747025048956
48 6.02742942598862
56 6.02740685204482
64 6.02740042794273
72 6.02740825379949
80 6.02742861249946
88 6.02745995328523
96 6.02750087678847
104 6.02755012135061
112 6.02760655051809
120 6.02766914161427
128 6.02773697530204
136 6.02780922605756
144 6.0278851534819
152 6.02796409438174
160 6.02804545555521
168 6.02812870722346
176 6.02821337705235
184 6.02829904471361
192 6.02838533693832
200 6.02847192301924
208 6.02855851072247
216 6.02864484257213
224 6.02873069247478
232 6.02881586265296
240 6.02890018086055
248 6.02898349785425
256 6.02906568509827
264 6.02914663268146
272 6.02922624742724
280 6.02930445117935
288 6.02938117924703
296 6.02945637899597
304 6.029530008571
312 6.02960203573899
320 6.02967243684111
328 6.02974119584424
336 6.02980830348289
344 6.02987375648282
352 6.02993755685962
360 6.02999971128486
368 6.0300602305136
376 6.03011912886839
384 6.0301764237734
392 6.03023213533519
400 6.03028628596519
408 6.03033890004001
416 6.03039000359657
424 6.03043962405821
432 6.03048778998952
440 6.03053453087648
448 6.0305798769304
456 6.03062385891261
464 6.03066650797853
472 6.03070785553885
480 6.0307479331367
488 6.03078677233866
496 6.03082440463892
504 6.03086086137463
512 6.03089617365225
520 6.0309303722828
528 6.03096348772594
536 6.03099555004149
544 6.03102658884825
552 6.0310566332885
560 6.03108571199848
568 6.03111385308366
576 6.03114108409851
584 6.03116743203047
592 6.03119292328717
600 6.03121758368721
608 6.03124143845347
616 6.03126451220915
624 6.03128682897593
632 6.03130841217434
640 6.03132928462543
648 6.03134946855446
656 6.03136898559571
664 6.03138785679836
672 6.03140610263353
680 6.0314237430023
688 6.03144079724414
696 6.03145728414631
704 6.03147322195355
712 6.03148862837842
720 6.03150352061182
728 6.03151791533394
736 6.03153182872519
744 6.03154527647764
752 6.03155827380638
760 6.03157083546072
768 6.03158297573601
776 6.03159470848474
784 6.03160604712808
792 6.03161700466718
800 6.03162759369445
808 6.0316378264045
816 6.0316477146052
824 6.03165726972858
832 6.03166650284121
840 6.03167542465492
848 6.03168404553688
856 6.03169237551982
864 6.03170042431178
872 6.03170820130577
880 6.03171571558934
888 6.03172297595374
896 6.03172999090292
904 6.03173676866228
912 6.03174331718739
920 6.03174964417232
928 6.0317557570576
936 6.03176166303853
944 6.03176736907245
952 6.03177288188664
960 6.03177820798542
968 6.03178335365726
976 6.03178832498177
984 6.03179312783629
992 6.03179776790245
1000 6.03180225067255
1008 6.03180658145555
1016 6.03181076538315
1024 6.03181480741549
1032 6.03181871234662
1040 6.03182248481027
1048 6.03182612928478
1056 6.03182965009831
1064 6.03183305143372
1072 6.03183633733357
1080 6.03183951170441
1088 6.03184257832159
1096 6.0318455408333
1104 6.03184840276511
1112 6.03185116752378
1120 6.0318538384013
1128 6.03185641857876
1136 6.03185891112989
1144 6.03186131902484
1152 6.0318636451335
1160 6.03186589222881
1168 6.03186806299013
1176 6.03187016000634
1184 6.03187218577876
1192 6.0318741427242
1200 6.03187603317776
1208 6.03187785939558
1216 6.03187962355744
1224 6.0318813277694
1232 6.03188297406618
1240 6.03188456441372
1248 6.03188610071128
1256 6.03188758479394
1264 6.03188901843442
1272 6.03189040334563
1280 6.03189174118221
1288 6.03189303354284
1296 6.03189428197208
1304 6.03189548796195
1312 6.03189665295399
1320 6.0318977783409
1328 6.03189886546806
1336 6.03189991563526
1344 6.03190093009815
1352 6.03190191006987
1360 6.0319028567223
1368 6.03190377118766
1376 6.0319046545596
1384 6.03190550789483
1392 6.03190633221407
1400 6.0319071285034
1408 6.03190789771547
1416 6.03190864077053
1424 6.03190935855764
1432 6.03191005193567
1440 6.03191072173423
1448 6.0319113687549
1456 6.0319119937719
1464 6.03191259753327
1472 6.03191318076158
1480 6.03191374415482
1488 6.03191428838731
1496 6.03191481411052
1504 6.03191532195368
1512 6.03191581252471
1520 6.03191628641083
1528 6.03191674417936
1536 6.0319171863783
1544 6.03191761353706
1552 6.03191802616703
1560 6.03191842476221
1568 6.03191880979985
1576 6.03191918174095
1584 6.03191954103087
1592 6.03191988809984
1600 6.03192022336337
1608 6.03192054722299
1616 6.03192086006649
1624 6.03192116226849
1632 6.03192145419089
1640 6.0319217361833
1648 6.03192200858339
1656 6.03192227171741
1664 6.03192252590048
1672 6.03192277143702
1680 6.03192300862111
1688 6.03192323773687
1696 6.03192345905861
1704 6.03192367285151
1712 6.03192387937163
1720 6.03192407886624
1728 6.0319242715744
1736 6.0319244577269
1744 6.03192463754671
1752 6.03192481124927
1760 6.03192497904258
1768 6.0319251411277
1776 6.03192529769876
1784 6.0319254489433
1792 6.03192559504254
1800 6.0319257361715
1808 6.03192587249921
1816 6.03192600418901
1824 6.0319261313987
1832 6.03192625428062
1840 6.03192637298198
1848 6.03192648764503
1856 6.03192659840711
1864 6.03192670540098
1872 6.03192680875472
1880 6.03192690859225
1888 6.03192700503316
1896 6.03192709819295
1904 6.03192718818334
1912 6.03192727511203
1920 6.03192735908323
1928 6.03192744019757
1936 6.03192751855219
1944 6.03192759424099
1952 6.03192766735471
1960 6.03192773798088
1968 6.03192780620417
1976 6.03192787210631
1984 6.03192793576622
1992 6.0319279972603
2000 6.03192805666214
2008 6.03192811404293
2016 6.03192816947142
2024 6.03192822301405
2032 6.031928274735
2040 6.0319283246962
2048 6.03192837295751
2056 6.03192841957682
2064 6.03192846460996
2072 6.03192850811089
2080 6.03192855013178
2088 6.03192859072288
2096 6.03192862993297
2104 6.03192866780892
2112 6.03192870439618
2120 6.0319287397386
2128 6.03192877387848
2136 6.03192880685679
2144 6.03192883871303
2152 6.03192886948534
2160 6.03192889921068
2168 6.03192892792459
2176 6.03192895566152
2184 6.03192898245472
2192 6.03192900833626
2200 6.03192903333719
2208 6.03192905748745
2216 6.03192908081598
2224 6.03192910335078
2232 6.03192912511882
2240 6.03192914614621
2248 6.03192916645812
2256 6.03192918607891
2264 6.03192920503208
2272 6.03192922334039
2280 6.03192924102573
2288 6.03192925810931
2296 6.0319292746116
2304 6.0319292905524
2312 6.03192930595083
2320 6.03192932082528
2328 6.03192933519362
2336 6.03192934907307
2344 6.03192936248027
2352 6.03192937543124
2360 6.03192938794158
2368 6.0319294000262
2376 6.03192941169967
2384 6.03192942297592
2392 6.03192943386847
2400 6.03192944439042
2408 6.03192945455431
2416 6.0319294643724
2424 6.03192947385636
2432 6.03192948301766
2440 6.03192949186722
2448 6.03192950041565
2456 6.03192950867322
2464 6.03192951664985
2472 6.03192952435502
2480 6.03192953179802
2488 6.03192953898775
2496 6.03192954593284
};
\addlegendentry{$P_2$}
\addplot [semithick, color2]
table {%
0 6.08716089146002
8 6.10692638835159
16 6.1064151496685
24 6.10602753304808
32 6.10574891771295
40 6.10556656824323
48 6.10546912246371
56 6.10544640296562
64 6.10548928283189
72 6.10558957381675
80 6.10573992945929
88 6.10593375977419
96 6.10616515549536
104 6.10642882045522
112 6.10672001101996
120 6.10703448171393
128 6.10736843631461
136 6.10771848380969
144 6.108081598692
152 6.10845508513624
160 6.10883654465794
168 6.10922384690159
176 6.10961510324428
184 6.11000864293648
192 6.11040299153128
200 6.11079685137844
208 6.11118908398399
216 6.11157869405554
224 6.11196481507171
232 6.11234669622998
240 6.11272369064206
248 6.11309524465802
256 6.11346088821217
264 6.11382022609464
272 6.11417293006072
280 6.11451873169913
288 6.11485741598747
296 6.115188815471
304 6.11551280500511
312 6.11582929700892
320 6.11613823718198
328 6.11643960064039
336 6.11673338843314
344 6.11701962440202
352 6.11729835235421
360 6.11756963351665
368 6.11783354424608
376 6.11809017397132
384 6.11833962334442
392 6.11858200258207
400 6.11881742997876
408 6.11904603057501
416 6.11926793496662
424 6.11948327824064
432 6.11969219902676
440 6.11989483865206
448 6.12009134039012
456 6.12028184879459
464 6.12046650910943
472 6.12064546674797
480 6.12081886683469
488 6.12098685380258
496 6.12114957104171
504 6.12130716059245
512 6.12145976288067
520 6.12160751648893
528 6.12175055796102
536 6.121889021636
544 6.12202303950945
552 6.12215274111739
560 6.12227825344259
568 6.12239970083917
576 6.12251720497427
584 6.12263088478512
592 6.12274085644898
600 6.12284723336546
608 6.1229501261489
616 6.12304964263019
624 6.1231458878668
632 6.12323896415991
640 6.12332897107745
648 6.12341600548301
656 6.12350016156929
664 6.12358153089513
672 6.12366020242658
680 6.12373626258051
688 6.12380979527062
696 6.12388088195552
704 6.12394960168837
712 6.12401603116781
720 6.12408024478981
728 6.12414231470027
736 6.1242023108478
744 6.12426030103719
752 6.12431635098265
760 6.12437052436071
768 6.12442288286354
776 6.12447348625107
784 6.12452239240312
792 6.12456965737085
800 6.12461533542751
808 6.12465947911828
816 6.12470213930949
824 6.12474336523718
832 6.12478320455405
840 6.12482170337636
848 6.12485890632919
856 6.12489485659118
864 6.12492959593788
872 6.12496316478434
880 6.12499560222671
888 6.12502694608273
896 6.12505723293112
904 6.12508649815003
912 6.12511477595464
920 6.12514209943367
928 6.12516850058446
936 6.12519401034817
944 6.12521865864247
952 6.12524247439495
960 6.12526548557426
968 6.12528771922095
976 6.12530920147752
984 6.12532995761716
992 6.12535001207192
1000 6.12536938846016
1008 6.12538810961272
1016 6.12540619759877
1024 6.12542367375061
1032 6.12544055868756
1040 6.12545687233978
1048 6.12547263397049
1056 6.12548786219803
1064 6.12550257501692
1072 6.12551678981894
1080 6.1255305234123
1088 6.12554379204157
1096 6.12555661140579
1104 6.12556899667705
1112 6.12558096251768
1120 6.12559252309716
1128 6.1256036921087
1136 6.12561448278481
1144 6.12562490791291
1152 6.12563497985007
1160 6.12564471053716
1168 6.12565411151305
1176 6.12566319392795
1184 6.12567196855619
1192 6.12568044580907
1200 6.12568863574684
1208 6.12569654809059
1216 6.1257041922334
1224 6.12571157725155
1232 6.12571871191501
1240 6.12572560469779
1248 6.12573226378778
1256 6.12573869709659
1264 6.12574491226834
1272 6.12575091668943
1280 6.12575671749627
1288 6.12576232158445
1296 6.12576773561661
1304 6.12577296602992
1312 6.12577801904422
1320 6.12578290066906
1328 6.12578761671075
1336 6.12579217277929
1344 6.12579657429491
1352 6.12580082649468
1360 6.12580493443833
1368 6.12580890301461
1376 6.12581273694665
1384 6.12581644079805
1392 6.12582001897777
1400 6.12582347574563
1408 6.12582681521728
1416 6.12583004136901
1424 6.12583315804259
1432 6.12583616894972
1440 6.12583907767632
1448 6.12584188768707
1456 6.12584460232916
1464 6.12584722483658
1472 6.12584975833373
1480 6.1258522058391
1488 6.12585457026905
1496 6.12585685444118
1504 6.12585906107752
1512 6.12586119280797
1520 6.1258632521733
1528 6.1258652416282
1536 6.12586716354417
1544 6.12586902021242
1552 6.12587081384641
1560 6.12587254658461
1568 6.12587422049307
1576 6.12587583756773
1584 6.12587739973697
1592 6.12587890886381
1600 6.12588036674793
1608 6.12588177512824
1616 6.1258831356846
1624 6.12588445003988
1632 6.12588571976196
1640 6.12588694636559
1648 6.12588813131407
1656 6.12588927602118
1664 6.12589038185266
1672 6.12589145012802
1680 6.12589248212196
1688 6.12589347906612
1696 6.12589444214998
1704 6.12589537252314
1712 6.12589627129591
1720 6.1258971395409
1728 6.12589797829461
1736 6.12589878855816
1744 6.12589957129874
1752 6.12590032745086
1760 6.12590105791709
1768 6.12590176356958
1776 6.1259024452508
1784 6.12590310377463
1792 6.12590373992741
1800 6.12590435446876
1808 6.12590494813242
1816 6.12590552162739
1824 6.12590607563859
1832 6.12590661082753
1840 6.12590712783344
1848 6.12590762727385
1856 6.12590810974521
1864 6.12590857582393
1872 6.12590902606655
1880 6.12590946101104
1888 6.12590988117694
1896 6.12591028706612
1904 6.12591067916363
1912 6.12591105793771
1920 6.12591142384099
1928 6.12591177731069
1936 6.12591211876905
1944 6.12591244862411
1952 6.12591276727007
1960 6.12591307508759
1968 6.12591337244456
1976 6.1259136596963
1984 6.12591393718606
1992 6.12591420524562
2000 6.12591446419517
2008 6.12591471434428
2016 6.12591495599186
2024 6.1259151894268
2032 6.12591541492813
2040 6.12591563276534
2048 6.12591584319882
2056 6.12591604648016
2064 6.12591624285233
2072 6.1259164325501
2080 6.12591661580028
2088 6.12591679282183
2096 6.12591696382655
2104 6.12591712901872
2112 6.12591728859595
2120 6.12591744274904
2128 6.12591759166221
2136 6.12591773551363
2144 6.12591787447525
2152 6.1259180087132
2160 6.12591813838809
2168 6.12591826365492
2176 6.1259183846635
2184 6.12591850155856
2192 6.12591861447984
2200 6.12591872356245
2208 6.12591882893682
2216 6.12591893072894
2224 6.12591902906063
2232 6.12591912404948
2240 6.12591921580907
2248 6.12591930444917
2256 6.1259193900758
2264 6.12591947279135
2272 6.12591955269485
2280 6.12591962988179
2288 6.12591970444452
2296 6.12591977647224
2304 6.12591984605115
2312 6.1259199132645
2320 6.12591997819261
2328 6.12592004091319
2336 6.12592010150132
2344 6.12592016002947
2352 6.12592021656759
2360 6.12592027118344
2368 6.12592032394226
2376 6.12592037490729
2384 6.12592042413942
2392 6.12592047169756
2400 6.1259205176387
2408 6.12592056201767
2416 6.12592060488771
2424 6.1259206463
2432 6.1259206863042
2440 6.12592072494812
2448 6.12592076227798
2456 6.12592079833853
2464 6.12592083317294
2472 6.12592086682279
2480 6.1259208993284
2488 6.12592093072867
2496 6.12592096106118
};
\addlegendentry{$P_3$}
\end{axis}

%% file: suppl_tex/p4_mean.tex
\definecolor{color0}{rgb}{0.12156862745098,0.466666666666667,0.705882352941177}
\definecolor{color1}{rgb}{1,0.498039215686275,0.0549019607843137}
\definecolor{color2}{rgb}{0.172549019607843,0.627450980392157,0.172549019607843}

\begin{axis}[
legend cell align={left},
legend style={
  fill opacity=0.8,
  draw opacity=1,
  text opacity=1,
  at={(0.03,0.97)},
  anchor=north west,
  draw=white!80!black
},
tick align=outside,
tick pos=left,
x grid style={white!69.0196078431373!black},
xlabel={t},
xmin=-124.8, xmax=2620.8,
xtick style={color=black},
y grid style={white!69.0196078431373!black},
ymin=29.5486457647889, ymax=38.3524452926503,
ytick style={color=black}
]
\addplot [semithick, color0]
table {%
0 37.9522725868384
8 37.0954140304742
16 36.3199638004388
24 35.6216991837012
32 34.9930568359269
40 34.4274062222805
48 33.9187662086616
56 33.4617250280948
64 33.0513822852214
72 32.6832987426003
80 32.3534517528617
88 32.0581954001598
96 31.7942247153486
104 31.5585434694481
112 31.3484351372183
120 31.1614366846975
128 30.9953148820952
136 30.8480448815747
144 30.717790830677
152 30.6028883184982
160 30.5018284742043
168 30.4132435568626
176 30.3358938924528
184 30.2686560287901
192 30.2105119922363
200 30.1605395415985
208 30.1179033250645
216 30.081846855147
224 30.0516852249905
232 30.026798496674
240 30.0066256988793
248 29.9906593772173
256 29.9784406459423
264 29.969554694571
272 29.9636267073621
280 29.9603181575045
288 29.9593234414961
296 29.9603668223443
304 29.963199653224
312 29.9675978557865
320 29.9733596297617
328 29.9803033726364
336 29.98826579017
344 29.9971001802207
352 30.0066748740986
360 30.0168718209647
368 30.0275853022354
376 30.0387207641089
384 30.0501937574078
392 30.0619289749595
400 30.0738593775967
408 30.0859254006798
416 30.0980742338132
424 30.1102591670749
432 30.122438997665
440 30.1345774914792
448 30.1466428946074
456 30.1586074901723
464 30.1704471963899
472 30.1821412020789
480 30.19367163619
488 30.2050232682734
496 30.2161832370405
504 30.2271408044229
512 30.2378871328613
520 30.2484150836267
528 30.2587190342935
536 30.2687947135707
544 30.2786390519601
552 30.2882500466972
560 30.2976266397433
568 30.3067686075803
576 30.3156764617246
584 30.3243513589834
592 30.3327950205528
600 30.3410096591159
608 30.3489979132236
616 30.3567627882634
624 30.3643076034195
632 30.3716359440319
640 30.3787516188911
648 30.3856586219707
656 30.3923610982053
664 30.3988633128986
672 30.4051696244537
680 30.4112844600868
688 30.417212294249
696 30.4229576294791
704 30.4285249794918
712 30.4339188542237
720 30.4391437467204
728 30.4442041216161
736 30.4491044050913
744 30.4538489761478
752 30.4584421590679
760 30.4628882169454
768 30.4671913461695
776 30.4713556717738
784 30.4753852435518
792 30.4792840328736
800 30.4830559301058
808 30.4867047425974
816 30.4902341931382
824 30.493647918878
832 30.4969494706168
840 30.5001423124438
848 30.5032298216813
856 30.5062152890981
864 30.5091019193516
872 30.5118928316357
880 30.5145910605196
888 30.5171995569235
896 30.5197211892485
904 30.5221587446003
912 30.5245149301199
920 30.5267923744058
928 30.5289936289719
936 30.5311211698024
944 30.5331773989192
952 30.5351646459907
960 30.5370851699802
968 30.5389411607801
976 30.5407347408921
984 30.5424679670825
992 30.5441428320562
1000 30.5457612661094
1008 30.5473251387815
1016 30.5488362604823
1024 30.5502963841224
1032 30.5517072066975
1040 30.55307037086
1048 30.5543874664829
1056 30.5556600321652
1064 30.5568895567383
1072 30.558077480732
1080 30.5592251978038
1088 30.560334056142
1096 30.5614053598488
1104 30.5624403702793
1112 30.5634403073533
1120 30.5644063508386
1128 30.5653396415799
1136 30.5662412827311
1144 30.567112340942
1152 30.5679538474909
1160 30.56876679942
1168 30.56955216062
1176 30.5703108628935
1184 30.5710438069764
1192 30.571751863541
1200 30.5724358741692
1208 30.573096652293
1216 30.5737349841105
1224 30.5743516294725
1232 30.5749473227363
1240 30.5755227736131
1248 30.5760786679572
1256 30.5766156685725
1264 30.5771344159621
1272 30.577635529059
1280 30.578119605947
1288 30.5785872245415
1296 30.5790389432818
1304 30.5794753017468
1312 30.5798968213008
1320 30.5803040057033
1328 30.5806973416807
1336 30.5810772995118
1344 30.5814443335655
1352 30.5817988828382
1360 30.5821413714678
1368 30.5824722092364
1376 30.5827917920398
1384 30.5831005023684
1392 30.5833987097493
1400 30.5836867711835
1408 30.5839650315712
1416 30.584233824109
1424 30.5844934706984
1432 30.5847442823164
1440 30.5849865593899
1448 30.5852205921421
1456 30.5854466609554
1464 30.585665036683
1472 30.585875980987
1480 30.5860797466433
1488 30.5862765778353
1496 30.586466710457
1504 30.5866503723842
1512 30.5868277837545
1520 30.5869991572187
1528 30.5871646982041
1536 30.5873246051561
1544 30.5874790697711
1552 30.5876282772323
1560 30.5877724064224
1568 30.5879116301426
1576 30.5880461153173
1584 30.5881760232028
1592 30.588301509563
1600 30.588422724861
1608 30.5885398144545
1616 30.5886529187521
1624 30.5887621733766
1632 30.5888677093442
1640 30.5889696532212
1648 30.5890681272574
1656 30.5891632495495
1664 30.5892551341699
1672 30.5893438913177
1680 30.5894296274293
1688 30.5895124453421
1696 30.5895924443706
1704 30.5896697204611
1712 30.5897443662971
1720 30.5898164713992
1728 30.5898861222478
1736 30.589953402384
1744 30.5900183924975
1752 30.5900811705343
1760 30.5901418117888
1768 30.5902003889895
1776 30.5902569723962
1784 30.5903116298708
1792 30.5903644269735
1800 30.5904154270236
1808 30.5904646911878
1816 30.5905122785568
1824 30.5905582462072
1832 30.5906026492643
1840 30.590645540989
1848 30.5906869728236
1856 30.5907269944612
1864 30.5907656539033
1872 30.5908029975136
1880 30.5908390700844
1888 30.5908739148768
1896 30.5909075736843
1904 30.5909400868763
1912 30.5909714934445
1920 30.5910018310602
1928 30.5910311361038
1936 30.5910594437239
1944 30.591086787873
1952 30.5911132013415
1960 30.5911387158143
1968 30.5911633618844
1976 30.5911871691152
1984 30.5912101660556
1992 30.5912323802875
2000 30.5912538384502
2008 30.5912745662824
2016 30.591294588638
2024 30.5913139295326
2032 30.5913326121601
2040 30.5913506589262
2048 30.5913680914757
2056 30.5913849307104
2064 30.5914011968282
2072 30.5914169093341
2080 30.5914320870795
2088 30.5914467482574
2096 30.5914609104515
2104 30.5914745906463
2112 30.5914878052455
2120 30.5915005700989
2128 30.5915129005102
2136 30.5915248112715
2144 30.5915363166639
2152 30.5915474304784
2160 30.5915581660538
2168 30.5915685362569
2176 30.5915785535244
2184 30.5915882298673
2192 30.5915975768894
2200 30.5916066058073
2208 30.5916153274363
2216 30.5916237522399
2224 30.5916318903193
2232 30.5916397514397
2240 30.5916473450199
2248 30.5916546801682
2256 30.5916617656812
2264 30.5916686100523
2272 30.5916752214952
2280 30.5916816079296
2288 30.5916877770149
2296 30.5916937361524
2304 30.5916994924866
2312 30.5917050529189
2320 30.5917104241159
2328 30.591715612516
2336 30.5917206243426
2344 30.5917254656077
2352 30.5917301421141
2360 30.5917346594682
2368 30.5917390230871
2376 30.591743238204
2384 30.5917473098735
2392 30.5917512429744
2400 30.5917550422296
2408 30.5917587121864
2416 30.5917622572499
2424 30.5917656816707
2432 30.5917689895489
2440 30.5917721848571
2448 30.5917752714247
2456 30.5917782529491
2464 30.5917811330116
2472 30.5917839150622
2480 30.5917866024365
2488 30.5917891983556
2496 30.5917917059349
};
\addlegendentry{$P_1$}
\addplot [semithick, color1]
table {%
0 37.9515047230723
8 37.0953910947909
16 36.3200363671613
24 35.6216688259684
32 34.9928435747841
40 34.4269529200488
48 33.9180279824162
56 33.4606667616259
64 33.0499771256474
72 32.6815268226622
80 32.3512990509902
88 32.0556527393852
96 31.7912869022822
104 31.5552085651545
112 31.3447038459348
120 31.157311844412
128 30.9908010414058
136 30.843147948786
144 30.7125177830067
152 30.5972469611592
160 30.4958272408497
168 30.4068913443595
176 30.3291999241868
184 30.261629741725
192 30.2031629437895
200 30.1528773330745
208 30.1099375389454
216 30.0735870039962
224 30.0431407100947
232 30.0179785748573
240 29.9975394561784
248 29.9813157083096
256 29.968848238402
264 29.9597220171657
272 29.9535620017283
280 29.9500294326407
288 29.9488184706007
296 29.9496531415938
304 29.9522845621539
312 29.956488418983
320 29.9620626796206
328 29.9688255129844
336 29.9766134005689
344 29.9852794208099
352 29.9946916908502
360 30.0047319512442
368 30.0152942805818
376 30.026283928159
384 30.0376162539031
392 30.0492157657854
400 30.0610152458097
408 30.0729549564888
416 30.0849819204889
424 30.0970492667587
432 30.1091156370669
440 30.1211446474427
448 30.1331043995376
456 30.1449670373133
464 30.1567083449498
472 30.1683073821969
480 30.1797461537447
488 30.1910093095295
496 30.2020838731381
504 30.2129589957068
512 30.2236257330554
520 30.2340768438559
528 30.2443066069538
536 30.2543106560558
544 30.2640858302437
552 30.2736300387794
560 30.2829421389606
568 30.2920218257897
576 30.3008695323602
584 30.309486339994
592 30.3178738972223
600 30.3260343467737
608 30.333970259846
616 30.3416845769734
624 30.349180554886
632 30.356461718775
640 30.363531819495
648 30.370394795211
656 30.3770547370896
664 30.3835158586296
672 30.3897824683132
680 30.3958589452469
688 30.4017497175121
696 30.4074592429504
704 30.4129919921869
712 30.4183524336132
720 30.4235450202172
728 30.4285741780025
736 30.4334442958943
744 30.4381597169602
752 30.4427247308196
760 30.4471435671275
768 30.4514203900113
776 30.4555592933761
784 30.4595642969776
792 30.4634393431979
800 30.4671882944276
808 30.4708149310134
816 30.4743229496805
824 30.4777159624178
832 30.4809974957354
840 30.4841709902711
848 30.4872398007042
856 30.49020719594
864 30.4930763595232
872 30.4958503902596
880 30.4985323030282
888 30.5011250297325
896 30.5036314204065
904 30.5060542444142
912 30.5083961917605
920 30.5106598744892
928 30.5128478281231
936 30.514962513197
944 30.5170063168074
952 30.5189815542031
960 30.5208904704147
968 30.5227352418737
976 30.5245179780744
984 30.5262407232151
992 30.5279054578565
1000 30.5295141005586
1008 30.5310685095212
1016 30.5325704841945
1024 30.5340217668982
1032 30.5354240443933
1040 30.5367789494406
1048 30.538088062354
1056 30.5393529124929
1064 30.5405749797597
1072 30.5417556960541
1080 30.5428964466924
1088 30.5439985718011
1096 30.54506336769
1104 30.5460920881828
1112 30.5470859459201
1120 30.548046113633
1128 30.548973725365
1136 30.5498698776965
1144 30.550735630923
1152 30.5515720101807
1160 30.5523800065747
1168 30.553160578257
1176 30.5539146514833
1184 30.5546431216273
1192 30.5553468541781
1200 30.5560266857075
1208 30.5566834248048
1216 30.5573178529859
1224 30.5579307255773
1232 30.5585227725613
1240 30.5590946994212
1248 30.5596471879236
1256 30.5601808969214
1264 30.5606964630984
1272 30.5611945016961
1280 30.5616756072288
1288 30.5621403541627
1296 30.5625892976023
1304 30.5630229739034
1312 30.5634419013142
1320 30.563846580582
1328 30.5642374955217
1336 30.5646151135977
1344 30.5649798864588
1352 30.5653322504715
1360 30.5656726272318
1368 30.5660014240641
1376 30.5663190344899
1384 30.5666258387031
1392 30.5669222040123
1400 30.5672084852723
1408 30.5674850253104
1416 30.5677521553189
1424 30.5680101952615
1432 30.568259454243
1440 30.5685002308798
1448 30.5687328136455
1456 30.5689574812315
1464 30.5691745028559
1472 30.5693841386006
1480 30.5695866397147
1488 30.5697822489049
1496 30.5699712006385
1504 30.5701537214119
1512 30.5703300300294
1520 30.5705003378537
1528 30.5706648490665
1536 30.5708237609093
1544 30.5709772639149
1552 30.5711255421423
1560 30.5712687733866
1568 30.5714071293989
1576 30.571540776088
1584 30.5716698737287
1592 30.5717945771363
1600 30.571915035857
1608 30.5720313943626
1616 30.5721437922048
1624 30.572252364179
1632 30.5723572405013
1640 30.5724585469646
1648 30.5725564050712
1656 30.5726509321959
1664 30.5727422417128
1672 30.5728304431464
1680 30.572915642281
1688 30.5729979413235
1696 30.5730774389797
1704 30.5731542306071
1712 30.5732284083205
1720 30.5733000610904
1728 30.5733692748657
1736 30.5734361326741
1744 30.5735007147093
1752 30.5735630984371
1760 30.5736233586887
1768 30.5736815677445
1776 30.5737377954313
1784 30.5737921091929
1792 30.5738445741844
1800 30.5738952533351
1808 30.5739442074338
1816 30.5739914952057
1824 30.5740371733753
1832 30.5740812967269
1840 30.5741239181921
1848 30.5741650888944
1856 30.5742048582198
1864 30.5742432738723
1872 30.574280381928
1880 30.5743162269022
1888 30.5743508517871
1896 30.5743842981173
1904 30.5744166060114
1912 30.5744478142193
1920 30.5744779601781
1928 30.5745070800418
1936 30.5745352087408
1944 30.5745623800166
1952 30.5745886264556
1960 30.574613979545
1968 30.5746384696878
1976 30.574662126263
1984 30.5746849776417
1992 30.5747070512338
2000 30.5747283735125
2008 30.574748970056
2016 30.5747688655631
2024 30.574788083899
2032 30.5748066481133
2040 30.5748245804719
2048 30.5748419024837
2056 30.57485863492
2064 30.5748747978526
2072 30.5748904106645
2080 30.5749054920885
2088 30.5749200602049
2096 30.5749341324865
2104 30.574947725811
2112 30.57496085648
2120 30.5749735402436
2128 30.5749857923097
2136 30.5749976273775
2144 30.5750090596389
2152 30.5750201027975
2160 30.5750307701092
2168 30.5750410743595
2176 30.5750510279071
2184 30.5750606426874
2192 30.5750699302314
2200 30.575078901685
2208 30.5750875677958
2216 30.5750959389611
2224 30.5751040252187
2232 30.5751118362724
2240 30.5751193814817
2248 30.5751266698972
2256 30.5751337102603
2264 30.5751405110113
2272 30.5751470803113
2280 30.5751534260305
2288 30.57515955578
2296 30.5751654769141
2304 30.5751711965327
2312 30.5751767214933
2320 30.5751820584205
2328 30.5751872137121
2336 30.575192193552
2344 30.5751970039143
2352 30.5752016505652
2360 30.5752061390757
2368 30.5752104748287
2376 30.5752146630237
2384 30.5752187086839
2392 30.575222616657
2400 30.5752263916362
2408 30.5752300381402
2416 30.5752335605457
2424 30.5752369630764
2432 30.575240249807
2440 30.5752434246843
2448 30.575246491514
2456 30.5752494539693
2464 30.5752523156095
2472 30.5752550798622
2480 30.5752577500424
2488 30.5752603293503
2496 30.5752628208816
};
\addlegendentry{$P_2$}
\addplot [semithick, color2]
table {%
0 37.9511431312856
8 37.0954715094316
16 36.3200720010984
24 35.6218447523223
32 34.9932605848726
40 34.4276814208168
48 33.9191208912295
56 33.4621645529017
64 33.0519105132366
72 32.683918579867
80 32.3541654165411
88 32.0590045622118
96 31.7951305902395
104 31.5595468750668
112 31.3495365420703
120 31.1626362473639
128 30.9966124857589
136 30.8494401649733
144 30.7192832161782
152 30.6044770376819
160 30.5035125911981
168 30.4150219896105
176 30.3377654320683
184 30.2706193571304
192 30.2125656978264
200 30.1626821340304
208 30.1201332479886
216 30.0841624979709
224 30.0540849333906
232 30.0292805820228
240 30.0091884466824
248 29.9933010546472
256 29.9811595085525
264 29.9723489922646
272 29.9664946896839
280 29.9632580783214
288 29.9623335631228
296 29.9634454191677
304 29.9663450148768
312 29.97080828991
320 29.9766334644028
328 29.9836389583117
336 29.9916615016323
344 30.0005544179569
352 30.0101860655875
360 30.020438421718
368 30.0312057966462
376 30.0423936661259
384 30.0539176110549
392 30.0657023547188
400 30.077680888666
408 30.0897936791223
416 30.1019879466103
424 30.1142170120909
432 30.126439703539
440 30.1386198174471
448 30.1507256302675
456 30.1627294551978
464 30.1746072401978
472 30.1863382034595
480 30.1979045029039
488 30.2092909366188
496 30.2204846713976
504 30.2314749967775
512 30.2422531023111
520 30.2528118758759
528 30.2631457211341
536 30.273250392358
544 30.2831228450807
552 30.2927611010335
560 30.3021641261357
568 30.3113317202896
576 30.3202644178938
584 30.3289633981036
592 30.3374304039297
600 30.345667669343
608 30.3536778536583
616 30.3614639825106
624 30.3690293948204
632 30.3763776951601
640 30.3835127110561
648 30.3904384547298
656 30.3971590888821
664 30.4036788961119
672 30.4100022516535
680 30.4161335990999
688 30.4220774288337
696 30.4278382588883
704 30.4334206180461
712 30.4388290308927
720 30.4440680047147
728 30.4491420179868
736 30.4540555103402
744 30.4588128738457
752 30.463418445482
760 30.4678765006763
768 30.4721912477971
776 30.4763668235109
784 30.480407288908
792 30.4843166263258
800 30.488098736779
808 30.4917574379531
816 30.4952964626703
824 30.498719457817
832 30.5020299836431
840 30.5052315134069
848 30.5083274333274
856 30.5113210428053
864 30.5142155548716
872 30.5170140968436
880 30.5197197111699
888 30.5223353564131
896 30.5248639083863
904 30.527308161383
912 30.5296708295149
920 30.5319545481392
928 30.534161875324
936 30.5362952934068
944 30.5383572105698
952 30.5403499624537
960 30.5422758138107
968 30.5441369601451
976 30.5459355293976
984 30.5476735836066
992 30.5493531205867
1000 30.550976075586
1008 30.552544322943
1016 30.554059677718
1024 30.5555238973279
1032 30.5569386831351
1040 30.5583056820232
1048 30.5596264879644
1056 30.5609026435295
1064 30.5621356413977
1072 30.5633269258266
1080 30.564477894085
1088 30.5655898978596
1096 30.566664244641
1104 30.5677021990671
1112 30.568704984237
1120 30.5696737829977
1128 30.5706097391768
1136 30.5715139588169
1144 30.5723875113652
1152 30.5732314308098
1160 30.5740467168162
1168 30.5748343358161
1176 30.5755952220726
1184 30.5763302787045
1192 30.5770403786908
1200 30.5777263658472
1208 30.5783890557691
1216 30.5790292367479
1224 30.5796476706628
1232 30.5802450938339
1240 30.5808222178728
1248 30.5813797304739
1256 30.5819182962223
1264 30.5824385573452
1272 30.5829411344454
1280 30.5834266272218
1288 30.5838956151546
1296 30.5843486581968
1304 30.5847862973932
1312 30.5852090555266
1320 30.5856174377291
1328 30.5860119320564
1336 30.5863930100733
1344 30.586761127394
1352 30.5871167242195
1360 30.5874602258536
1368 30.5877920432069
1376 30.5881125732675
1384 30.5884221995826
1392 30.5887212927022
1400 30.5890102106165
1408 30.5892892991837
1416 30.5895588925276
1424 30.5898193134458
1432 30.5900708737836
1440 30.5903138748074
1448 30.5905486075527
1456 30.5907753531884
1464 30.5909943833278
1472 30.5912059603692
1480 30.5914103378007
1488 30.5916077604941
1496 30.5917984650106
1504 30.5919826798703
1512 30.5921606258345
1520 30.5923325161579
1528 30.5924985568517
1536 30.5926589469259
1544 30.5928138786235
1552 30.5929635376563
1560 30.5931081034181
1568 30.5932477492052
1576 30.5933826424202
1584 30.5935129447824
1592 30.5936388125035
1600 30.5937603964795
1608 30.5938778424871
1616 30.59399129134
1624 30.5941008790534
1632 30.5942067370222
1640 30.5943089921793
1648 30.5944077671294
1656 30.5945031803123
1664 30.594595346132
1672 30.5946843751085
1680 30.5947703739886
1688 30.5948534459097
1696 30.5949336904769
1704 30.5950112039168
1712 30.595086079185
1720 30.5951584060643
1728 30.5952282712889
1736 30.5952957586457
1744 30.5953609490621
1752 30.5954239207134
1760 30.5954847491161
1768 30.5955435072138
1776 30.5956002654738
1784 30.5956550919589
1792 30.5957080524232
1800 30.5957592103744
1808 30.5958086271605
1816 30.5958563620476
1824 30.5959024722822
1832 30.5959470131541
1840 30.5959900380826
1848 30.5960315986637
1856 30.596071744739
1864 30.5961105244534
1872 30.5961479843092
1880 30.5961841692334
1888 30.5962191226162
1896 30.5962528863771
1904 30.5962855010064
1912 30.5963170056133
1920 30.5963474379823
1928 30.5963768346033
1936 30.596405230731
1944 30.5964326604206
1952 30.5964591565616
1960 30.5964847509341
1968 30.5965094742241
1976 30.5965333560845
1984 30.5965564251507
1992 30.5965787090882
2000 30.5966002346173
2008 30.5966210275548
2016 30.5966411128303
2024 30.5966605145322
2032 30.5966792559254
2040 30.5966973594839
2048 30.5967148469182
2056 30.596731739194
2064 30.5967480565712
2072 30.5967638186143
2080 30.5967790442318
2088 30.5967937516736
2096 30.5968079585765
2104 30.596821681977
2112 30.5968349383294
2120 30.5968477435322
2128 30.596860112936
2136 30.5968720613784
2144 30.5968836031848
2152 30.5968947521881
2160 30.5969055217686
2168 30.5969159248327
2176 30.5969259738554
2184 30.5969356808848
2192 30.5969450575607
2200 30.5969541151341
2208 30.5969628644539
2216 30.5969713160161
2224 30.5969794799529
2232 30.5969873660602
2240 30.5969949837857
2248 30.5970023422661
2256 30.5970094503253
2264 30.5970163164832
2272 30.5970229489787
2280 30.5970293557565
2288 30.5970355444995
2296 30.5970415226322
2304 30.5970472973214
2312 30.5970528754901
2320 30.5970582638257
2328 30.5970634687865
2336 30.5970684966154
2344 30.5970733533432
2352 30.5970780447908
2360 30.5970825765825
2368 30.597086954152
2376 30.5970911827486
2384 30.5970952674434
2392 30.5970992131299
2400 30.5971030245462
2408 30.5971067062539
2416 30.5971102626717
2424 30.5971136980635
2432 30.5971170165425
2440 30.5971202220939
2448 30.5971233185588
2456 30.5971263096463
2464 30.5971291989492
2472 30.597131989928
2480 30.5971346859293
2488 30.5971372901839
2496 30.5971398058172
};
\addlegendentry{$P_3$}
\end{axis}